\documentclass{amsart}

\usepackage{amsmath}
\usepackage{amsfonts}
\usepackage{amssymb}
\usepackage{amsthm}
\usepackage{comment}
\usepackage{epsfig}
\usepackage{psfrag}
\usepackage{mathrsfs}
\usepackage{amscd}
\usepackage[all]{xy}
\usepackage{rotating}
\usepackage{lscape}
\usepackage{amsbsy}
\usepackage{verbatim}
\usepackage{moreverb}
\usepackage{color}
\usepackage{bbm}
\usepackage{eucal}
\usepackage{stmaryrd}
\usepackage{cleveref}

\usepackage{pdfpages}

\usepackage{float}

\usepackage{tikz-cd}
\usetikzlibrary{patterns,shapes.geometric,arrows,decorations.markings}
\usepackage{tikz-3dplot}

\usepackage{caption}
\usepackage{subcaption}

\colorlet{lightgray}{black!15}

\tikzset{->-/.style={decoration={
  markings,
  mark=at position .5 with {\arrow{>}}},postaction={decorate}}}
\tikzset{midarrow/.style={decoration={
    markings,
    mark=at position {#1} with {\arrow{>}}},postaction={decorate}}}

\newtheorem{theorem}{Theorem}[section]
\newtheorem{prop}[theorem]{Proposition}
\newtheorem{lemma}[theorem]{Lemma}
\newtheorem{cor}[theorem]{Corollary}

\theoremstyle{definition}
\newtheorem{definition}[theorem]{Definition}

\newtheorem{observation}[theorem]{Observation}

\newtheorem{terminology}[theorem]{Terminology}
\newtheorem{remark}[theorem]{Remark}
\newtheorem{example}[theorem]{Example}
\newtheorem{q}[theorem]{Question}
\newtheorem{notation}[theorem]{Notation}
\newtheorem{l.notation}[theorem]{Local Notation}

\newtheorem{convention}[theorem]{Convention}

\theoremstyle{remark}

\definecolor{orange}{rgb}{.95,0.5,0}
\definecolor{light-gray}{gray}{0.75}
\definecolor{brown}{cmyk}{0, 0.8, 1, 0.6}
\definecolor{plum}{rgb}{.5,0,1}

\DeclareMathOperator{\Link}{\sf Link}
\DeclareMathOperator{\Fin}{\sf Fin}

\DeclareMathOperator{\Vect}{\cV{\sf ect}}

\DeclareMathOperator{\pr}{\mathsf{pr}}

\DeclareMathOperator{\ev}{\mathsf{ev}}

\DeclareMathOperator{\Alg}{\sf Alg}

\DeclareMathOperator{\TwAr}{\sf TwAr}
\DeclareMathOperator{\TwAro}{{\sf TwAr}^\circ}

\DeclareMathOperator{\Aut}{\sf Aut}
\DeclareMathOperator{\colim}{{\sf colim}}

\DeclareMathOperator{\limit}{{\sf lim}}

\DeclareMathOperator{\Hom}{\sf Hom}
\DeclareMathOperator{\End}{\sf End}

\DeclareMathOperator{\Fun}{{\sf Fun}}
\DeclareMathOperator{\Iso}{\sf Iso}

\DeclareMathOperator{\Map}{{\sf Map}}

\DeclareMathOperator{\Ker}{\sf Ker}

\DeclareMathOperator{\exit}{\sf Exit}

\DeclareMathOperator{\Cat}{{\sf Cat}}
\DeclareMathOperator{\fCat}{{\sf fCat}}

\DeclareMathOperator{\Ar}{{\sf Ar}}

\DeclareMathOperator{\Diff}{{\sf Diff}}

\DeclareMathOperator{\op}{\mathsf{op}}

\DeclareMathOperator{\Bun}{\cB\mathsf{un}}

\DeclareMathOperator{\sk}{\mathsf{sk}}

\DeclareMathOperator{\cls}{\mathsf{cls}}
\DeclareMathOperator{\act}{\mathsf{act}}

\DeclareMathOperator{\emb}{\mathsf{emb}}
\DeclareMathOperator{\Cylo}{\mathsf{Cylo}}
\DeclareMathOperator{\Cylr}{\mathsf{Cylr}}

\DeclareMathOperator{\pcbl}{\mathsf{p.cbl}}

\DeclareMathOperator{\Top}{\mathsf{Top}}
\DeclareMathOperator{\Mfd}{{\cM}\mathsf{fd}}
\DeclareMathOperator{\cMfd}{{\sf c}{\cM}\mathsf{fd}}
\DeclareMathOperator{\Mfld}{\mathsf{Mfld}}

\DeclareMathOperator{\Emb}{\mathsf{Emb}}
\DeclareMathOperator{\Imm}{\mathsf{Imm}}
\DeclareMathOperator{\fImm}{\mathsf{fImm}}
\DeclareMathOperator{\EmbImm}{\mathsf{Emb-Imm}}
\DeclareMathOperator{\EmbfImm}{\mathsf{Emb-fImm}}
\DeclareMathOperator{\ImmImm}{\mathsf{Imm-Imm}}
\DeclareMathOperator{\ImmfImm}{\mathsf{Imm-fImm}}

\DeclareMathOperator{\Strat}{\cS\mathsf{trat}}

\DeclareMathOperator{\spaces}{\mathsf{Spaces}}
\DeclareMathOperator{\Spaces}{\spaces}

\DeclareMathOperator{\Disk}{{\mathsf{Disk}}}

\DeclareMathOperator{\Adj}{\mathsf{Adj}}

\DeclareMathOperator{\fr}{\sf fr}
\DeclareMathOperator{\sfr}{\sf sfr}

\DeclareMathOperator{\Bord}{\sf Bord}

\DeclareMathOperator{\Tang}{\sf Tang}
\DeclareMathOperator{\Ind}{\sf Ind}

\DeclareMathOperator{\BO}{\sf BO}
\DeclareMathOperator{\RP}{\RR\PP}

\def\ot{\otimes}

\DeclareMathOperator{\oo}{\infty}

\DeclareMathOperator{\disk}{\sf D}

\DeclareMathOperator{\tr}{\triangleright}
\DeclareMathOperator{\tl}{\triangleleft}

\newcommand{\lag}{\langle}
\newcommand{\rag}{\rangle}

\newcommand{\w}{\widetilde}
\newcommand{\un}{\underline}
\newcommand{\ov}{\overline}

\newcommand{\ra}{\rightarrow}
\newcommand{\la}{\leftarrow}
\newcommand{\xra}{\xrightarrow}
\newcommand{\xla}{\xleftarrow}

\def\cA{\mathcal A}\def\cB{\mathcal B}\def\cC{\mathcal C}\def\cD{\mathcal D}
\def\cE{\mathcal E}\def\cF{\mathcal F}\def\cG{\mathcal G}
\def\cI{\mathcal I}\def\cJ{\mathcal J}\def\cK{\mathcal K}
\def\cM{\mathcal M}\def\cP{\mathcal P}
\def\cR{\mathcal R}\def\cS{\mathcal S}\def\cT{\mathcal T}
\def\cU{\mathcal U}\def\cV{\mathcal V}\def\cW{\mathcal W}\def\cX{\mathcal X}

\def\CC{\mathbb C}\def\DD{\mathbb D}
\def\FF{\mathbb F}
\def\II{\mathbb I}
\def\NN{\mathbb N}\def\PP{\mathbb P}
\def\RR{\mathbb R}

\def\ZZ{\mathbb Z}

\def\sB{\mathsf B}\def\sC{\mathsf C}\def\sD{\mathsf D}
\def\sH{\mathsf H}
\def\sL{\mathsf L}
\def\sO{\mathsf O}
\def\sR{\mathsf R}\def\sT{\mathsf T}
\def\sV{\mathsf V}

\def\bDelta{\mathbf\Delta}

\def\fB{\frak B}\def\fC{\frak C}

\def\fU{\frak U}\def\fX{\frak X}
\def\fZ{\frak Z}

\def\bcD{\boldsymbol{\mathcal D}}

\DeclareMathOperator{\bTheta}{\boldsymbol{\Theta}}

\DeclareMathOperator{\Obj}{\mathsf{obj}}
\DeclareMathOperator{\obj}{\mathsf{obj}}

\DeclareMathOperator{\PShv}{\mathsf{PShv}}
\DeclareMathOperator{\uno}{\mathbbm{1}}

\DeclareMathOperator{\id}{\sf id}
\DeclareMathOperator{\Conf}{\sf Conf}
\DeclareMathOperator{\Gr}{\sf Gr}
\DeclareMathOperator{\Mor}{\sf mor}
\DeclareMathOperator{\mor}{\sf mor}
\DeclareMathOperator{\Dual}{\sf Dual}

\DeclareMathOperator{\Pic}{\sf Pic}

\DeclareMathOperator{\Free}{{\sf Free}}

\DeclareMathOperator{\cShv}{\sf cShv}
\DeclareMathOperator{\Morita}{\cM{\sf orita}}
\DeclareMathOperator{\Un}{\sf Un}

\newcommand{\bit}[1]{\textbf{\textit{#1}}}

\DeclareMathOperator{\lacts}{\curvearrowright}
\DeclareMathOperator{\racts}{\curvearrowleft}

\DeclareMathOperator{\n1}{\mathit{n}\text{-}1}
\DeclareMathOperator{\nk}{\mathit{n}\text{-}\mathit{k}}

\DeclareMathOperator{\Seg}{\sf Seg}

\DeclareMathOperator{\Opens}{\sf Opens}

\def\oC{\ov{\sC}}

\begin{document}
\dedicatory{To Pele Vega and Jack Aristarchus}

\title{The Tangle Hypothesis: Dimension 1}

\author{David Ayala \& John Francis}

\address{Department of Mathematics\\Montana State University\\Bozeman, MT 59717}
\email{david.ayala@montana.edu}
\address{Department of Mathematics\\Northwestern University\\Evanston, IL 60208}
\email{jnkf@northwestern.edu}
\thanks{DA was supported by the National Science Foundation under awards 1812055 and 1945639. JF was supported by the National Science Foundation under award 1812057. This material is based upon work supported by the National Science Foundation under Grant No. DMS-1440140, while the authors were in residence at the Mathematical Sciences Research Institute in Berkeley, California, during the Spring 2020 semester.}

\begin{abstract}
We introduce an $(\oo,1)$-category ${\sf Bord}_1^{\sf fr}(\mathbb{R}^n)$, the morphisms in which are framed tangles in $\RR^n\times \DD^1$. We prove that ${\sf Bord}_1^{\sf fr}(\mathbb{R}^n)$ has the universal mapping out property of the  1-dimensional Tangle Hypothesis of Baez--Dolan and Hopkins--Lurie: it is the rigid $\mathcal{E}_n$-monoidal $(\oo,1)$-category freely generated by a single object. Applying this theorem to a dualizable object of a braided monoidal $(\oo,1)$-category gives link invariants, generalizing the Reshetikhin--Turaev invariants.
\end{abstract}



\maketitle

\tableofcontents

\section*{Introduction}

The construction of knot invariants from representations of braid groups dates to at least to the 1980s and the Jones polynomial: Vaughan Jones's original work used braid group representations of Hecke algebra and von Neumann representations~\cite{jones1} and~\cite{jones2} to define polynomial link invariants. Building on these ideas, Turaev abstracted the role of Yang--Baxter equations~\cite{turaev}.  Then, in 1990, Reshetikhin and Turaev \cite{RT} constructed an invariant of framed links in $\RR^3$ given a finite-dimensional representation of a quantum group. This construction works even more generally, as in the work of Freyd--Yetter~\cite{freyd.yetter} and Shum~\cite{shum}, so as to construct from a dualizable object $V$ in a braided monoidal category $C$ a function
\[
\xymatrix{
\Bigl\{{\rm framed \ links}\Bigr\}_{\big/{\rm istpy}}
 \ar[r]^-{Z_V}& \Hom_C(\uno,\uno)}
\]
from the set of framed isotopy classes of links to the set of endomorphisms of the unit of the monoidal structure of $C$. In the case where $C$ is representations of the quantum group $U_\hbar({\sf sl}_2(\CC))$, this invariant specializes to the 2-variable Jones--Conway polynomial, after reparametrization.

\medskip

The 1-dimensional Tangle Hypothesis is a topological enhancement of this theorem of Shum, first put forward by Baez--Dolan \cite{baezdolan}, and given a rigorous $(\oo,1)$-categorically formulation by Hopkins–Lurie \cite{cobordism}. In this formulation, the set of framed isotopy classes is the set of components of a topological moduli space of framed tangles $\Tang_1^{\fr}(\RR^3)$, which has the homotopy type
\[
\Tang_1^{\fr}(\RR^3)\simeq\coprod_k \Emb^{\fr}\Bigl(
\underset{k}\amalg S^1, \RR^3
\Bigr)_{/\Diff\bigl(\underset{k}\amalg S^1\bigr)}
\]
as in Definition~\ref{def.tang.top.space}. This space of framed tangles has a very interesting topology: See \cite{budney}, \cite{sinha1}, \cite{sinha2}, \cite{vassiliev}, \cite{hatcher.knot}, and \cite{kontsevich}, among many other works. One can then ask for a refinement of Reshetikhin--Turaev's function: Given a dualizable object $V$ of a braided monoidal $(\oo,1)$-category $\cR$, one could ask for a map of spaces
\[
\xymatrix{
\Tang_1^{\fr}(\RR^3)\ar[r]^-{Z_V}&\Map_\cR(\uno,\uno)~.\\
}
\]
For $\cR$ enriched in chain complexes, such an assignment is equivalent to a morphism of chain complexes
\[
\xymatrix{
\sC_\ast\bigl(\Tang_1^{\fr}(\RR^3)\bigr)\ar[r]^-{Z_V}&\underline{\End}_\cR(\uno)\\
}
\]
where $\underline{\End}_\cR(\uno)$ is the chain complex of endomorphisms of the unit of $\cR$. Dualizing would give a map
\[
\xymatrix{
\underline{\End}_\cR(\uno)^\vee\ar[r]& \sC^\ast\bigl(\Tang_1^{\fr}(\RR^3)\bigr)~.\\
}
\]
and thus a characteristic class $\sigma_v\in \sH^k\bigl(\Tang_1^{\fr}(\RR^3)\bigr)$ for tangles for each class in the homology of the dual, $v \in \sH_{-k}\bigl(\underline{\End}_\cR(\uno)^\vee\bigr)$.

\medskip

The Tangle Hypothesis in ambient dimension 3 asserts that such a construction indeed exists and that, further, if one incorporates compatibility with boundary conditions on links, then this construction is essentially unique for each choice of $V\in \cR$. 

In ambient dimension $n>3$, any two links with the same number of connected components are isotopic.  
However, the topology of each connected component of $\Tang_1^{\fr}(\RR^n)$ is extremely nontrivial: See \cite{LTV}.  
Through considerations similar to those above, the Tangle Hypothesis in ambient dimension $n$ supplies a map
\[
\sH_{-k}\bigl(\underline{\End}_\cR(\uno)^\vee\bigr)
\longrightarrow
\sH^k\bigl(\Tang_1^{\fr}(\RR^n)\bigr)
\]
for each object $V\in \cR$ in a rigid $\cE_{\n1}$-monoidal $(\infty,1)$-category.

The following is the main result of this paper, the Tangle Hypothesis for 1-dimensional tangles in arbitrary ambient dimension. This result reduces to the identification of the free rigid braided monoidal $(1,1)$-category given by Freyd--Yetter~\cite{freyd.yetter} and Shum~\cite{shum} in the case where the target $\cR$ is a $(1,1)$-category.
\begin{theorem}[Tangle Hypothesis]\label{t4}
Let $\cR$ be a rigid $\cE_{\n1}$-monoidal $(\infty,1)$-category, with $n\geq 2$. Let $\Bord_1^{\fr}(\RR^{\n1})$ be the $\cE_{\n1}$-monoidal flagged $(\oo,1)$-category of framed tangles of Definition~\ref{def.bord}.
Evaluation at the object $\ast\in \Bord_1^{\fr}(\RR^{\n1})$ defines an equivalence between spaces
\[
\ev_\ast
\colon
\Map_{\Alg_{\n1}(\fCat_{(\oo,1)})} \bigl(
\Bord_1^{\fr}(\RR^{\n1})
,
\cR
\bigr)
\xra{~\simeq~}
\Obj(\cR)
~,\qquad
Z
\longmapsto
Z(\ast)
~.
\]
between the space of objects of $\cR$ and the space of $\cR$-valued $\cE_{\n1}$-monoidal functors on $\Bord_1^{\fr}(\RR^{\n1})$. Consequently, the univalent-completion
\[
\Bord_1^{\fr}(\RR^{\n1})^{\wedge}_{\sf unv}
\]
is the rigid $\cE_{\n1}$-monoidal $(\oo,1)$-category freely generated by a single object.
\end{theorem}

Taking the limit $n\mapsto\oo$ gives the 1-dimensional Cobordism Hypothesis, proved by Hopkins--Lurie \cite{cobordism}.

\begin{cor}[1-dimensional Cobordism Hypothesis]
Let $\cR$ be a rigid symmetric monoidal $(\infty,1)$-category, and let
\[
\Bord_1^{\fr}:= \varinjlim \Bord_1^{\fr}(\RR^{\n1})
\]
be the 1-dimensional cobordism category.
Evaluation at the object $\ast\in \Bord_1^{\fr}$ defines an equivalence between spaces
\[
\ev_\ast
\colon
\Map_{\Alg_{\sf Com}(\Cat_{(\oo,1)})} \bigl(
\Bord_1^{\fr}
,
\cR
\bigr)
\xra{~\simeq~}
\Obj(\cR)
~,\qquad
Z
\longmapsto
Z(\ast)
~.
\]
between the space of objects of $\cR$ and the space of $\cR$-valued symmetric monoidal functors on $\Bord_1^{\fr}$.
\end{cor}
As a consequence of the Tangle Hypothesis, we calculate the classifying spaces of $\Bord_1^{\fr}(\RR^{\n1})$ and $
\Bord_1^{\fr}$.

\begin{cor}
\label{t.inv.tqft}
For each $n\geq 2$, there is a canonical identification between $\cE_{\n1}$-spaces:
\[
\sB \Bord_1^{\fr}(\RR^{\n1})
~\simeq~
\Omega^{\n1} S^{\n1}
~.
\]
In the limit $n\mapsto \oo$, there is a canonical identification between $\cE_\infty$-spaces:
\[
\sB \Bord_1^{\fr}
~\simeq~
\Omega^{\infty} \Sigma^\infty  S^0
~.
\]

\end{cor}

\begin{remark}
The flagged $(\oo,1)$-categories $\Bord^{\fr}_1(\RR^{\n1})$ exhibit special behavior in those dimensions in which unknots in $\RR^{\n1}\times\DD[1]$ can be straightened up to a contractible choice. For $\Bord^{\fr}_1(\RR^1)$, this is Proposition~\ref{prop.Bord.discrete.gaunt}. For $\Bord^{\fr}_1(\RR^2)$, this is a consequence of a theorem of Hatcher~\cite{hatcher}. In higher dimensions, however, the space of long knots has a complicated homotopy type: The space of objects of $\Bord_1^{\fr}(\RR^{\n1})$ is not equivalent to the space of isomorphisms (i.e., the space of framed h-cobordism tangles, or long unknots), and so the functor to the univalent-completion $\Bord_1^{\fr}(\RR^{\n1}) \ra \Bord_1^{\fr}(\RR^{\n1})^{\wedge}_{\sf unv}$ is not an equivalence. (This is similar to, but distinct from, the phenomenon in which $\Bord_k^{\fr}$ is not univalent-complete for $k>3$. This is due to the existence of nontrivial h-cobordisms.) As $n$ tends to $\oo$, however, the space of long unknots becomes contractible, and for $n=\oo$ then $\Bord_1^{\fr} \ra (\Bord_1^{\fr})^{\wedge}_{\sf unv}$ is an equivalence as in the cases of $n=2$ and $n=3$.
\end{remark}

The interesting topology of spaces of long knots in $\RR^n$ for $3<n<\oo$ leads to the following question:

\begin{q}
Is there a direct construction of the $(\oo,1)$-categories $\Bord_1^{\fr}(\RR^{\n1})^{\wedge}_{\sf unv}$, for $3<n<\oo$? That is, can one construct $\Bord_1^{\fr}(\RR^{\n1})^{\wedge}_{\sf unv}$ without first constructing some other flagged $(\oo,1)$-category (namely, $\Bord_1^{\fr}(\RR^{\n1})$) and then taking its univalent-completion?
\end{q}

\subsection*{Acknowledgements} 
We thank Mike Hopkins for inspiring mathematics, personal support, and a lot of fun.

\subsection*{Outline of the proof}
Let $\Free_k^{\sf rig}(\cK)$ be the rigid $\cE_k$-monoidal $(\oo,1)$-category freely generated by an $(\oo,1)$-category $\cK$. The functor
\[
\Free_n^{\sf rig}: \Cat_{(\oo,1)} \ra \Alg_n^{\sf rig}(\Cat_{(\oo,1)})
\]
is equivalent to
\[
\Free_n^{\sf rig} \simeq \Ind_{\n1}^n\Free_{\n1}^{\sf rig}
\]
where $\Ind_{\n1}^n:\Alg_{\n1}(\Cat_{(\oo,1)})\ra \Alg_n(\Cat_{(\oo,1)})$ is the left adjoint to the functor defined by restriction along $\cE_{\n1}\ra \cE_n$. The 1-dimensional Tangle Hypothesis is equivalent to the assertion $\Free^{\sf rig}_n(\ast) \simeq \Bord_1^{\fr}(\RR^n)^{\wedge}_{\sf unv}$. Our proof divides into two subsidiary assertions:
\begin{enumerate}
\item  Base Case: $\Free^{\sf rig}_1(\ast)\simeq  \Bord_1^{\fr}(\RR^1)$; and
\item  Inductive Case: $\Ind_{\n1}^n\Bigl(\Bord_1^{\fr}(\RR^{\n1})\Bigr)^{\wedge}_{\sf unv} \simeq  \Bord_1^{\fr}(\RR^n)^{\wedge}_{\sf unv}$.
\end{enumerate}
To establish both the Base Case and the Inductive Case, we use the following technique: 
\begin{itemize}
\item Let us say we want to evaluate a functor $\cF$ on an $(\oo,1)$-category $\cC$, where the output $\cF$ is again an $(\oo,1)$-category. Let $\cC_\bullet$ denote the underlying simplicial space of $\cC$. There is a canonical functor
\[
\Bigl(\Seg\bigl(\cF(\cC_\bullet)\bigr)\Bigr)_{\sf unv}^\wedge
\longrightarrow
\cF(\cC)
\]
where
\begin{itemize}
\item $\cF(\cC_\bullet)$ is the simplicial object obtained by evaluating $\cF$ on each $\cC_p$, regarded as an $\oo$-groupoid;
\item $\Seg$ is an endofunctor on simplicial objects, the output of which often satisfies the Segal condition;
\item $(-)_{\sf unv}^\wedge$ is univalent-completion, which converts a Segal space into a univalent-complete Segal space (i.e., an $(\oo,1)$-category).
\end{itemize}
\end{itemize}

We give a formula for this $\Seg$ functor in Definition~\ref{def.Seg}, and prove in Lemma~\ref{lemma.Seg.universal} that this canonical functor is an equivalence so long as the output indeed satisfies the Segal condition (which it will not in general).

\medskip

To establish the base case, we start with a description of the free monoidal category generated by a single object with a right-dual, in Proposition~\ref{tDualBord}. Using our $\Seg$ functor, we then analyze the pushouts which define the process of adding left-duals and right-duals, resulting in Lemma~\ref{lemma.main.basecase.n2}.

\medskip

To establish the inductive step, we give a topological model for induction from $\cE_{\n1}$-algebras to $\cE_n$-algebras, in Lemma~\ref{lemma.II.is.Ind}. We then apply the same strategy of Segaling and univalent-completing to analyze the induction functor from $\cE_{\n1}$-monoidal flagged $(\oo,1)$-categories to $\cE_n$-monoidal flagged $(\oo,1)$-categories via the composite of steps in the following diagram:
\[
\xymatrix{
\Alg_{\n1}(\fCat_{(\oo,1)})\ar@{^{(}->}[d]_-{\rm forget}\ar[rrrr]^-{\Ind_{\n1}^n}&&&&\Alg_n(\fCat_{(\oo,1)})\\
\Alg_{\n1}(\Fun(\bDelta^{\op},\Spaces))\ar[rr]^{\rm Step \ (1)}&&\Alg_n(\Fun(\bDelta^{\op},\Spaces))\ar[rr]^-{\rm Step \ (2)}&&\Alg_n(\fCat_{(\oo,1)})\ar[u]_-{\rm Step \ (3)}}
\]
The proof is then completed by composing these three steps:

\begin{enumerate}
\item In Corollary~\ref{cor1}, we show
\[
\Ind_{\n1}^n\bigl(\Bord_1^{\fr}(\RR^{\n1})[p]\bigr)
\simeq \Tang_1^{\fr,\pitchfork[p]}(\RR^n\times\DD[p])^{\sf s.pr.emb}~,
\]
the induction on the space of framed tangles in $\RR^{\n1}\times \DD[p]$ is the space of slice-wise projection-embedding framed tangles in $\RR^n\times\DD[p]$, i.e., the tangles which can be sliced in such a way that they then project to embeddings in $\RR^{\n1}\times\DD[p]$.
\item In Lemma~\ref{lemma.proj.emb.to.proj.imm}, we show
\[
\Seg
\Bigl(
\Tang_1^{\fr}\bigl(\RR^n \times\DD[\bullet]\bigr)^{\sf s.pr.emb}_{\pitchfork[\bullet]}
\Bigr)[p]
\simeq
\Tang_1^{\fr}(\RR^n\times\DD[p])_{\pitchfork[p]}^{\sf pr.imm}~,
\]
the Segaling of the simplicial space of slice-wise projection-embedding framed tangles is the space of projection-immersion framed tangles, i.e., the tangles in $\RR^n\times\DD[p]$ which project to immersions in $\RR^{\n1}\times\DD[p]$.
\item
In Lemma~\ref{lemma.univ.cplt.of.proj.imm}, we show
\[
\Bigl(\Bord_1^{\fr}(\RR^n)^{\sf pr.imm}\Bigr)^\wedge_{\sf unv}
\simeq
\Bigl( \Bord_1^{\fr}(\RR^n) \Bigr)^\wedge_{\sf unv}~,\]
the univalent-completion of the Segal space of projection-immersion framed tangles is the $(\oo,1)$-category of framed tangles.
\end{enumerate}

\begin{remark}
These steps give a parametrized generalization of the familiar assertion that a framed knot or links admits a presentation in terms of braids. Combining these steps, there is an $(\oo,1)$-category which presents the space of framed links in $\RR^n$: The classifying space of this $(\oo,1)$-category is equivalent to the space of framed links in $\RR^n$, and an object of this $(\oo,1)$-category is a presentation of a framed link in $\RR^n$ as a stack of concatenated framed tangles in $\RR^{\n1}$. (Here, the stacking is the $\cE_{n}$-monoidal structure of $\Bord_1^{\fr}(\RR^{n})$, and the concatenation is the composition of $\Bord_1^{\fr}(\RR^{n})$.) Applying this in the case $n=3$, we find that the space of framed links in $\RR^3$ is a classifying space of an $(\oo,1)$-category whose objects are presentations and whose morphisms are relations.
\end{remark}

\subsection*{Linear overview of the paper}

We now give a sectional overview of the contents of this work.

{\bf Section 1} defines the moduli space of framed tangles.

{\bf Section 2} defines the flagged $(\oo,1)$-category $\Bord_1^{\fr}(\RR^{\n1})$, and shows that its space of morphisms are framed tangles in $\RR^{\n1}\times \DD^1$. It defines the $(\oo,1)$-category $\bcD^{\sfr}_{[\n1,n]}$ and shows that cosheaves on this category define rigid $\cE_{\n1}$-monoidal $(\oo,1)$-categories after restricting along a functor $\rho$ and univalent-completion.

{\bf Section 3} constructs a functor ${\sf Seg}$ on simplicial spaces. In opportune situations, $\Seg$ agrees with the left adjoint of the inclusion of Segal spaces into simplicial spaces. 

{\bf Section 4} shows that the twisted-arrow category of a poset of intervals localizes onto the twisted-arrow category of the simplex category.

{\bf Section 5} proves the 1-dimensional Tangle Hypothesis in ambient dimension 2: $\Bord_1^{\fr}(\RR^1)$ is the rigid monoidal $(\oo,1)$-category freely generated by a single object.

{\bf Section 6} establishes a topological model for induction from $\cE_{\n1}$-algebras to $\cE_n$-algebras.

{\bf Section 7} proves that $\Bord_1^{\fr}(\RR^n)$ is the $\cE_n$-monoidal flagged $(\oo,1)$-category obtained by applying the induction functor $\Ind_1^n:\Alg(\fCat_{(\oo,1)})\ra \Alg_{n}(\fCat_{(\oo,1)})$ to $\Bord_1^{\fr}(\RR^1)$.

{\bf Section 8} establishes a relative h-principle for embedding-immersions, used to evaluate the univalent-completion in Section 7.

{\bf Section 9} proves the 1-dimensional Cobordism Hypothesis as the $n\mapsto\oo$ limit of the Tangle Hypothesis.

{\bf Section 10} calculates the classifying spaces of the cobordism and tangle categories.

\subsection{Notation}
\label{sec.notation}

\begin{itemize}
\item $\Bord_1^{\fr}(\RR^{\n1})$ is the $\cE_{\n1}$-monoidal flagged $(\oo,1)$-category whose objects are $n$-framed points in $\RR^{\n1}$ and whose morphisms are framed tangles in $\RR^{\n1}\times\DD[1]$.
\item $\Bord_1^{\fr}(\RR^{\n1})_{\sf unv}^\wedge$ is the $\cE_{\n1}$-monoidal $(\oo,1)$-category which is the univalent-completion of $\Bord_1^{\fr}(\RR^{\n1})$.
\item For $p\geq 1$, $\DD[p]=(\{0,\ldots,p\}\subset[0,p])$ is the 1-disk stratified with $p-1$ points in the interior.
\item  $\Tang^{\fr}_1(M)$ is the space of framed 1-tangles in $M$. $\Tang^{\fr}_1(\RR^{\n1}\times\DD[p])$ is the  space of framed 1-tangles in $\RR^{\n1}\times\DD[p]$ which are transverse to $\RR^{\n1}\times \{0,\ldots,p\}$. See~\S\ref{sec.top.mod.tangles}.
\item $\Strat$ is the $(\oo,1)$-category of stratified spaces, whose morphisms are stratum-preserving maps. See~\cite{aft1}. 
\item $\Bun$ is the $(\oo,1)$-category of constructible bundles. See~\S\ref{sec.tang.bun}.
\item $\Mfd_{n}^{\sfr}$ is the $(\oo,1)$-category of solidly $n$-framed stratified spaces. See~\S\ref{sec.tang.bun}.
\item $\bcD^{\sfr}_{[\n1,n]}$ is a full $\oo$-subcategory of $\Mfd_{n}^{\sfr}$ of stratified spaces whose strata are diffeomorphic to $\RR^{\n1}$ or $\RR^n$. See \S\ref{sec.tang.bun}.
\item $\Disk^{\fr}_k$ is the $(\oo,1)$-category of finite disjoint unions of framed $k$-dimensional Euclidean spaces, whose morphisms are framed embeddings. See \S\ref{sec.n.ind}.
\item
The $(\oo,1)$-category of spaces, or $\infty$-groupoids, or $(\infty,0)$-categories, is
\[
\Spaces
~=:~
{\sf Gpds}_{\infty}
~=:~
\Cat_{(\infty,0)}
~.
\]
\item $\Cat_{(\oo,1)}$ is the $(\oo,1)$-category of small $(\oo,1)$-categories.
\item $\fCat_{(\oo,1)}$ is the $(\oo,1)$-category of small flagged $(\oo,1)$-categories, which is equivalent to the $(\oo,1)$-category $\Fun^{\Seg}(\bDelta^{\op}, \Spaces)$  of Segal presheaves on $\bDelta$. See also \cite{flagged}.
\item $\TwAr(\cK)$ is the twisted-arrow $(\oo,1)$-category of $\cK$. It is the unstraightening of the functor $\Map: \cK^{\op}\times\cK\ra \Spaces$. Limits indexed by $\TwAr(\cK)$ are ends.
\item $\TwAro(\cK)$ is the opposite of the twisted-arrow $(\oo,1)$-category of $\cK$. Colimits indexed by $\TwAro(\cK)$ are coends.
\item $\Seg$ is an endofunctor on simplicial objects. See~Definition~\ref{def.Seg}.
\item $\Alg_k(\cV)$ is the $(\oo,1)$-category of $\cE_k$-algebras in a symmetric monoidal $(\oo,1)$-category $\cV$.
\item $\Alg^{\sf rig}_k(\Cat_{(\oo,1)})$ is the $(\oo,1)$-category of rigid $\cE_k$-monoidal $(\oo,1)$-categories.
\item $\FF_k(V)$ is the $\cE_k$-algebra freely generated by $V$, an object of a symmetric monoidal $(\oo,1)$-category $\cV$.
\item 
$\sB$ is the classifying space functor: for $\cC$ a flagged $(\infty,1)$-category, $\sB \cC$ is its classifying space, which is the geometric realization of the simplicial space presenting $\cC$.
\item $\fB$ is the deloop functor, from monoidal $(\oo,1)$-categories to $(\oo,2)$-categories. See \S\ref{sec.BBord1}.
\end{itemize}
\begin{convention}
Unless otherwise specified, all stratified spaces in this work are assumed to be conically smooth in the sense of~\cite{aft1}. All stratified maps are, likewise, also assumed to be conically smooth.
\end{convention}

\section{Topological moduli spaces of framed tangles}

\subsection{$[\n1,n]$-manifolds}
The following class of stratified spaces will be used throughout this paper.

\begin{definition}
\label{d.new.manifold}
An \bit{$[\n1,n]$-manifold} is a pair $(M^{(\n1)} \subset M)$ consisting of:
\begin{itemize}
\item
a smooth manifold $M$ with boundary, such that each connected component is either a smooth $(\n1)$-manifold without boundary or a smooth $n$-manifold with boundary; 
and

\item
a properly embedded smooth $(\n1)$-dimensional submanifold $M^{(\n1)} \subset M$ that contains both the boundary $\partial M \subset M^{(\n1)}$ and each $(\n1)$-dimensional component of $M$.

\end{itemize}
When there is no ambiguity, we denote an $[\n1,n]$-manifold $(M^{(\n1)} \subset M)$ simply as the ambient manifold $M$.

\end{definition}

\begin{example}

\begin{enumerate}
\item[]

\item
A smooth $(\n1)$-manifold is an $[\n1,n]$-manifold.
In particular, $\RR^{\n1}$ is an $[\n1,n]$-manifold.

\item
We regard a smooth $n$-manifold $M$ with boundary as an $[\n1,n]$-manifold, where $M^{(\n1)} = \partial M$.
In particular, we regard $\RR^{\n1} \times \DD^1$ as an $[\n1,n]$-manifold.

\item
More generally, for each $p\geq 0$, 
we have
$\RR^{\n1}\times \DD[p]$ as an $[\n1,n]$-manifold, where $(\RR^{\n1}\times \DD[p])^{(\n1)} = \RR^{\n1}\times \{0,\dots,p\}$.

\end{enumerate}

\end{example}

\begin{remark}
\label{r1}
An $[\n1,n]$-manifold $M$ is a conically smooth stratified space that satisfies the following condition.
\begin{itemize}
\item[]
Each stratum is either $(\n1)$-dimensional or $n$-dimensional, and each $(\n1)$-dimensional stratum is contained in the closure of at most two $n$-dimensional strata.
\end{itemize}
This condition is equivalent to requiring that for each $x\in M$, there is a basic neighborhood about $x$ isomorphic with $\RR^n$, or $\RR^{\n1}\times \sC(\ast)$, or $\RR^{\n1}\times \sC(S^0)$. 

\end{remark}

\begin{observation}
Let $M$ be an $[\n1,n]$-manifold.
Its ambient smooth manifold admits a tangent vector bundle $\tau_M$, which is classified by map between spaces
\[
\tau_M
\colon
M
\longrightarrow
\BO(\n1)
\coprod
\BO(n)
~,\qquad
x
\longmapsto 
\sT_x M
~,
\]
in which the $(\n1)$-dimensional components of $M$ are carried to the first cofactor and the $n$-dimensional components of $M$ are carried to the second.

\end{observation}

\begin{definition}[Remark 2.37 of \cite{fact1}]
Let $M$ be an $[\n1,n]$-manifold.
A \bit{solid $n$-framing} on $M$ is a homotopy-commutative lift
\[
\xymatrix{
&&
S^{\n1}
\coprod 
\ast
\ar[d]
\\
M
\ar[rr]_-{\tau_M}
\ar@{-->}[urr]^-{\varphi}
&&
\BO(\n1) \coprod \BO(n)
.
}
\]
Here, the downward map is the coproduct of the standard map $S^{\n1} = \frac{\sO(n)}{\sO(\n1)} \to \BO(\n1)$ with the basepoint $\ast = \frac{\sO(n)}{\sO(n)} \to \BO(n)$.

\end{definition}

\begin{remark}
\label{r10}
Note the identifications $S^{\n1} \simeq {\sf Inj}(\RR^{\n1},\RR^n)_{/\sO(n)}$ and $\ast \simeq {\sf Inj}(\RR^n,\RR^n)_{/\sO(n)}$.
Consequently, a solid $n$-framing of an $[\n1,n]$-manifold $M$ is a vector bundle injection $\tau_M \hookrightarrow \epsilon^n_M$ into a trivial rank-$n$ vector bundle over $M$.
Indeed, 
Such an injection restricts over each $n$-dimensional connected component of $M$ as a trivialization of its tangent bundle.

\end{remark}

\subsection{Topological moduli of tangles in $M$}\label{sec.top.mod.tangles}

\begin{definition}[Tangles]
\label{d.tangle}
Let $M = (M^{(\n1)} \subset M)$ be an $[\n1,n]$-manifold in the sense of Definition~\ref{d.new.manifold}.
A \bit{$1$-tangle in $M$} is a compact smooth codimension-$(\n1)$ submanifold $W\subset M$ with boundary that is transverse to $M^{(\n1)} \subset M$ and $\partial W = W \cap \partial M$.
\end{definition}

\begin{terminology}
If the dimension is unspecified, a tangle $W\subset M$ refers to a $1$-tangle.
\end{terminology}

\begin{remark}
If the $[\n1,n]$-manifold $M$ is simply a smooth $(\n1)$-manifold, then a 1-tangle in $M$ is a compact 0-dimensional submanifold of $M$.
\end{remark}

\begin{remark}
Let $M$ be a smooth $n$-manifold.
Let $\partial \subset M$ be a codimension-1 properly embedded smooth submanifold that separates $M$.
Denote by $M'$ and $M''$ the closures of each separation: Each is a smooth $n$-manifold with boundary $\partial$.
Regard $M = (\partial \subset M)$ as an $[\n1,n]$-manifold.
Then a tangle $W\subset M$ is exactly a tangle in each of $W'\subset M'$ and $W''\subset M''$ which agree in $\partial$: $W' \cap \partial = W'' \cap \partial$.
\end{remark}

Let $M = (M^{(\n1)} \subset M)$ be an $[\n1,n]$-manifold.
Let $W$ be a compact $[0,1]$-dimensional manifold.
There is a topological space
\[
\Emb(W, M)
\]
whose underlying set consists of smooth embeddings $W \xra{g} M$ that are transverse to $M^{(\n1)}\subset M$ and that carry $0$-dimensional connected components of $W$ to $(\n1)$-dimensional connected components of $M$.
The topology on this set is the compact-open smooth topology (also known as the weak Whitney topology). 
There is a natural function
\[
\Emb(W, M) \longrightarrow \{W'\subset M| \ W' \text{ a tangle in } M\}
\]
which assigns to an embedding $g:W\hookrightarrow M$ the image $g(W)\subset M$.

\begin{definition}
Let $M$ be an $[\n1,n]$-manifold.
The topological space of tangles in $M$ consists of the set $\Tang_1(M) =\{W\subset M| \ W \text{ a  tangle in }M\}$ of tangles, endowed with the finest topology for which the function
\[
\Emb(W, M) \longrightarrow \Tang_1(M)
\]
is continuous for every $[0,1]$-dimensional manifold $W$.
In the case that $M$ is $(\n1)$-dimensional, a tangle in $M$ is a 0-dimensional submanifold therein, and we use the more suggestive notation
\[
\Tang_0(M)
\]
for the space of tangles in $M$.
\end{definition}

\begin{observation}
For each $W$, there is a natural action of $\Aut(W)$ on $\Emb(W,M)$. Letting $[W]$ range over all isomorphism classes of $[0,1]$-dimensional manifolds $W$,
there is a homeomorphism
\[
\coprod_{[W]}\Emb(W,M)_{/\Aut(W)} \overset{\cong}\longrightarrow \Tang_1(M)
\]
between the disjoint union of the quotient of embeddings spaces, and the topological space of tangles.
\end{observation}

\subsection{Topological spaces of framed tangles}

In this section, we fix a space $B \to \Gr_1(n) = \RR\PP^{\n1}$ over the Grassmannian of 1-planes in $\RR^n$.  
We fix an $[\n1,n]$-manifold $M$, together with a solid $n$-framing $\varphi$ on $M$.

The solid $n$-framing $\varphi$ on $M$ determines, for each tangle $W\subset M$, a Gauss map
\[
W
\xra{~{\sf Gauss}_{W\subset (M,\varphi)}~}
\Gr_1(n)
~,\qquad
w
\longmapsto
\begin{cases}
\Bigl(
\varphi\left( \sT_w W \right)
\subset
\RR^n 
\Bigr)
~,
&
{\sf dim}_w(W) = 1
\\
\Bigl(
\varphi(\sT_w M)^{\perp}
\subset
\RR^n
\Bigr)
~,
&
{\sf dim}_w(W) = 0
\end{cases}
\]
whose value on $w\in W$ depends on the dimension ${\sf dim}_w(W)$ of the connected component of $W$ containing $w$.

\begin{definition}\label{def.framed.tangle}
A $B$-framed tangle in $M$ is a tangle $W\subset M$ together with a lift of the Gauss map
\[
\xymatrix{
&&&
B
\ar[d]
\\
W
\ar[rrr]_-{{\sf Gauss}_{W\subset (M,\varphi)}}
\ar@{-->}[urrr]^-{\psi}
&&&
\Gr_1(n)~
}
\]
up to a specified homotopy. 
For $W$ a compact $[0,1]$-dimensional manifold, 
the topological space of $B$-framed embeddings of $W$ into $M$ is the homotopy pullback:
\[
\xymatrix{
\Emb^B(W,M)\ar[r]\ar[d]&\Map(W,B)\ar[d]\\
\Emb(W,M)\ar[r]&\Map(W,\Gr_1(n)).}
\]
\end{definition}

This facilitates the following definition of the topological space of $B$-framed tangles.
\begin{definition}
The topological space $\Tang_1^B(M)$ of \bit{$B$-framed 1-tangles in $M$} is the quotient
\[
\Tang_1^B(M) 
~:=~
\coprod_{[W]} \Emb^B(W,M)_{/\Aut(W)}
~,
\]
indexed by the set of isomorphism classes of compact $[0,1]$-dimensional manifolds $W$.

\end{definition}

In particular, a point of $\Tang_1^B(M)$ is classified by a triple: a tangle $W\subset M$; a map $W\ra B$; and a homotopy between ${\sf Gauss}_{W\subset (M,\varphi)}$ and the composite $W\ra B \ra \Gr_1(n)$.

We will be particularly concerned with the case of framed tangles.
\begin{definition}\label{def.tang.top.space}
The topological space of framed tangles $\Tang^{\fr}_1(M)$ is the space $\Tang_1^B(M)$ in which $B$ is a point and the map $\ast \ra \Gr_1(n)$ selects the line $\RR^{\{n\}}\subset \RR^n$ given by the last coordinate.

\end{definition}

In particular, a point of $\Tang_1^{\fr}(M)$ is classified by a pair: a tangle $W\subset M$; and a nullhomotopy of the Gauss map ${\sf Gauss}_{W\subset (M,\varphi)}:W \ra \Gr_1(n)$.

\section{The $(\oo,1)$-category of tangles in terms of $\Bun$}
In this section, we define the $\cE_{\n1}$-monoidal flagged $(\oo,1)$-category $\Bord_1^{\fr}(\RR^{\n1})$. Heuristically, an object in $\Bord_1^{\fr}(\RR^{\n1})$ consists of:
\begin{itemize}
\item a finite subset $S\subset\RR^{\n1}$;
\item a nullhomotopy of the constant map $S \xra{{\sf const}_{\RR^{\{n\}}}} \Gr_1(n)$.
\end{itemize}
Likewise, to define a morphism in $\Bord_1^{\fr}(\RR^{\n1})$, we endow $ \RR^{\n1}\times \DD^1$ with its standard $n$-framing: $\sT(\RR^{\n1}\times\DD^1) \cong\epsilon^{\n1} \oplus\epsilon^1 \cong\epsilon^n$. 
A morphism of $\Bord_1^{\fr}(\RR^{\n1})$ is then heuristically given by:
\begin{itemize}
\item a 1-tangle $W\subset \RR^{\n1}\times\DD^1$;
\item a nullhomotopy of the resulting Gauss map $W \ra \Gr_1(n)$.
\end{itemize}
This heuristic suggests that, whatever the technical definition of $\Bord_1^{\fr}(\RR^{\n1})$, its space of objects should have the homotopy type of $\Tang_0^{\fr}(\RR^{\n1})$ and its space of morphisms should have the homotopy type $\Tang_1^{\fr}(\RR^{\n1}\times \DD^1)$.
In particular, for any pair of 0-tangles $S_-,S_+\in \Tang_0^{\fr}(\RR^{\n1})$, there should result a homotopy equivalence between the space of morphisms of $\Bord_1^{\fr}(\RR^{\n1})$
\[
\Map_{\Bord_1^{\fr}(\RR^{\n1})}(S_-, S_+)
\]
and the fiber of the restriction map
\[
\Tang_1^{\fr}(\RR^{\n1}\times\DD^1) \longrightarrow \Tang_0^{\fr}(\RR^{\n1}\times\partial \DD^1) 
\]
over the point $S_- \sqcup S_+ \in \Tang_0^{\fr}(\RR^{\n1}\times\partial \DD^1)$.

\begin{remark}\label{rem.framings.zero.mflds}
Observe that for a given $S\subset\RR^{\n1}$, the space of nullhomotopies of $S \ra \Gr_1(n)$ is $\Omega(\Gr_1(n))^S$. Consequently, the set of equivalences classes of nullhomotopies is $\pi_1\Gr_1(n)$. 
For $n=2$, we see that there are $\ZZ$-many objects of $\Bord_1^{\fr}(\RR^1)$ whose underlying tangle is a singleton $\ast \subset \RR^1$. For $n\geq 3$, we see that there are $\ZZ/2$-many isomorphism classes of objects of $\Bord_1^{\fr}(\RR^{\n1})$ with underlying tangle a singleton $\ast \subset \RR^{\n1}$.
\end{remark}

\subsection{The $(\oo,1)$-categories $\bcD_{[k]}^{\sfr}$ and $\bcD_{[\n1,n]}^{\sfr}$}\label{sec.tang.bun}

Recall the $(\infty,1)$-category $\Mfd_n^{\sfr}$ from~\cite{fact1}.
An object is a conically smooth stratified space $M$ equipped with a solid $n$-framing. 
A morphism in $\Mfd_n^{\sfr}$ from $M$ to $M'$ consists of:
\begin{itemize}
\item a constructible bundle $X\xra{\pi} \Delta^1$ over the standardly stratified 1-simplex;
\item a solid $n$-framing on $\pi$, which is an injection $\sT_{X|\Delta^1}\hookrightarrow \epsilon^n_X$ of the vertical tangent constructible bundle into the trivial rank-$n$ bundle;
\item  isomorphisms
\[
X_{|\Delta^{\{0\}}}\cong M~, {\rm \ and} \qquad X_{|\Delta^{\{1\}}}\cong M'\
\] between $M$ and the special fiber of $X\ra \Delta^1$, and $M'$ and the generic fiber of $X\ra \Delta^1$.
\end{itemize}
There is a natural forgetful functor $\Mfd_n^{\sfr}\ra \Bun$ to the $(\infty,1)$-category $\Bun$ classifying constructible bundles, given by forgetting the solid $n$-framing. We think of $\Bun$ as parametrizing the deformation theory of stratified spaces: A morphism in $\Bun$ is a 1-parameter deformation of a stratified space, encoded as a constructible bundle over the 1-simplex. More generally, for $K$ a stratified space,
a $K$-point of $\Bun$ is $K$-parameter family of stratified spaces, encoded as constructible bundle over $K$. See \cite{strat}.

We recall the following special classes of morphisms in $\Bun$ and, thereafter, in $\Mfd_n^{\sfr}$. See \S6.5--6.6 of~\cite{strat} and \S1.4 of~\cite{fact1}.
\begin{itemize}
\item Closed-creation morphisms: The reversed mapping cylinder defines a monomorphism
\[
\Cylr:\Strat^{\pcbl}\hookrightarrow \Bun^{\op}~,
\]
where $\Strat^{\pcbl}$ is the $(\infty,1)$-category whose objects are stratified spaces and whose morphisms are proper constructible bundles. 
The $\Cylr$ functor assigns to a morphism $X_1 \ra X_0$ of $\Strat^{\pcbl}$ the morphism of $\Bun$ given by the constructible bundle
\[
\Cylr(X_1\ra X_0) := X_0\underset{X_1\times\Delta^{\{0\}}}\amalg X_1 \times \Delta^1\longrightarrow \Delta^1~.
\]
The image of $\Cylr$ consists of \bit{ closed-creation} morphisms, $\Bun^{\sf cls.crt}$. If the proper constructible bundle $X_1\ra X_0$ is, furthermore, an injection (i.e., $X_1 \to X_0$ is the inclusion of a closed union of strata of $X_0$), then the associated morphism of $\Bun$ is \bit{ closed}, and the essential image of closed embeddings is $\Bun^{\cls}$. 
If $X_1\ra X_0$ is surjective, the associated morphism of $\Bun$ is \bit{ creation}, and the essential image of creation morphisms is $\Bun^{\sf crt}$.

\item Refinement-embedding morphisms: The open mapping cylinder defines a monomorphism
\[
\Cylo:\Strat^{\sf open}\hookrightarrow \Bun~,
\]
where $\Strat^{\sf open}$ is the $(\infty,1)$-category whose objects are stratified spaces and whose morphisms are stratified maps $X_0\hookrightarrow X_1$ which are open embeddings of underlying topological spaces. The $\Cylo$ functor assigns to a morphism $X_0 \hookrightarrow X_1$ of $\Strat^{\sf open}$ the morphism of $\Bun$ given by the constructible bundle
\[
\Cylo(X_0\ra X_1) := X_0\times\Delta^1\underset{X_0\times\Delta^1\smallsetminus\Delta^{\{0\}}}\amalg X_1 \times \Delta^1\smallsetminus\Delta^{\{0\}}\longrightarrow \Delta^1~.
\]
The image of $\Cylr$ consists of \bit{refinement-embedding} morphisms, $\Bun^{\sf ref.emb}$. There are further $\oo$-subcategories of refinements and embeddings: $\Bun^{\sf ref}$ is the essential image of stratified maps $X_0\ra X_1$ which are isomorphisms of underlying topological spaces; $\Bun^{\sf emb}$ is the essential image of stratified maps $X_0\ra X_1$ which are isomorphisms onto its image as stratified spaces.

\item The $(\oo,1)$-category $\Bun^{\sf act}$ of \bit{ active} morphisms of $\Bun$ is the smallest $\oo$-subcategory containing both the refinement-embedding morphisms and the creation morphisms, $\Bun^{\sf ref.emb}$ and $\Bun^{\sf crt}$.

\end{itemize}
Every morphism in $\Bun$ admits a factorization as a closed-creation morphism followed by a refinement-embedding morphism.
Also, $\Bun$ admits a factorization system: Every morphism admits a unique factorization as a closed morphism followed by an active morphism.
A morphism of $\Mfd_n^{\sfr}$ is closed, creation, refinement, embedding, and active if the underlying morphism of $\Bun$ is so. 
This defines corresponding $\oo$-subcategories $\Mfd_n^{\sfr}$, as well a factorization system:
every morphism of $\Mfd_n^{\sfr}$ can be uniquely factored as closed morphism followed by an active morphism.

To rigorously construct $\Bord_1^{\fr}(\RR^{\n1})$ as an $\cE_{\n1}$-monoidal flagged $(\oo,1)$-category, we will define two functors
\[
 \bcD^{\sfr}_{[\n1]}\times\bDelta^{\op}\longrightarrow \bcD^{\sfr}_{[\n1,n]}
 \longrightarrow
 \Spaces~.
\]
\begin{definition}
The $(\oo,1)$-category $\bcD^{\sfr}_{[k]}$ is the full $\oo$-subcategory of $\Mfd_{k}^{\sfr}$ whose objects are solidly $k$-framed finite disjoint unions of $\RR^k$.
The $(\oo,1)$-category $\bcD^{\sfr}_{[\n1,n]}$ is the full $\oo$-subcategory of $\Mfd_{n}^{\sfr}$ whose objects are solidly $n$-framed $[\n1,n]$-manifolds with the following properties:
\begin{itemize}
\item
each stratum is diffeomorphic to either $\RR^{\n1}$ or $\RR^n$;

\item
the closure of each stratum is diffeomorphic with an open subspace of a closed disk.  

\end{itemize}
\end{definition}

\begin{example}
\label{eg.disk}
Here we explain how objects in $\bcD^{\sfr}_{[\n1,n]}$ arise as Poincar\'e duals of graphs embedded in framed $n$-manifolds.
Let $M$ be a framed $n$-manifold with boundary.
Let $\Gamma \subset M$ be a compact embedded graph that meets the boundary transversely.
Assume $\Gamma$ has no self-loops.
Let $D \subset M$ be a regular open neighborhood of $\Gamma$ -- it is an $n$-manifold with boundary $\partial D = D \cap \partial M$.
Being a regular neighborhood, there is a deformation retraction $D \xra{r} \Gamma$.
For each edge $E \subset \Gamma$, choose a point $e\in E$ in its interior.  
Then $r^{-1}(e) \subset D$ is a properly embedded smooth $(\n1)$-submanifold which is diffeomorphic with $\RR^{\n1}$ -- informally, $r^{-1}(e)$ slices $D$ transverse to $E$.
The pair $(\partial D \sqcup \underset{e} \bigsqcup r^{-1}\{e\} \subset D)$ is an $[\n1,n]$-manifold.  
The framing of $M$ determines the structure of a solid $n$-framing on this $[\n1,n]$-manifold.
In this way, $\left( \partial D \sqcup \underset{e} \bigsqcup r^{-1} \{e\} \subset D \right)$ may be regarded as an object in $\bcD^{\sfr}_{[\n1,n]}$.
See Figure~\ref{fig11}.

\end{example}

\begin{remark}
Every object in $\bcD^{\sfr}_{[\n1,n]}$ is a finite disjoint union of $\RR^{\n1}$, equipped with a solid $n$-framing, and objects constructed in Example~\ref{eg.disk}.

\end{remark}

\begin{figure}
[H]
  \includegraphics[width=\linewidth, trim={0 {.4in} 0 {.4in}}, clip]{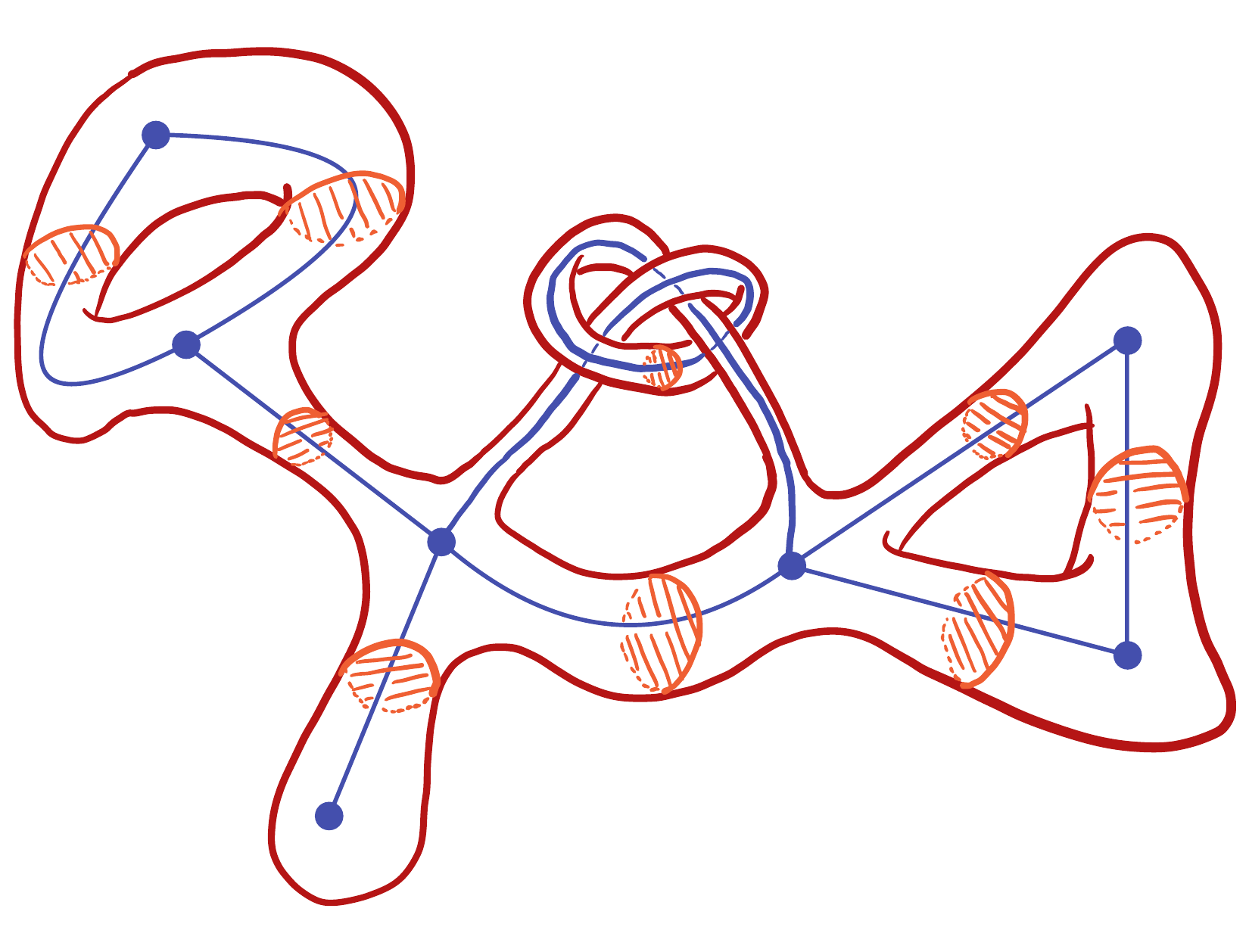}
  \caption{
  A Poincar\'e dual of an embedded graph in $\RR^3$.
  Here (in blue) is a compact embedded graph in $\RR^3$ with no self-loops. 
  Engulfing this graph is an (red) open regular neighborhood of it.
  Transverse to each edge is an (orange) open 2-disk that is properly embedded in the open regular neighborhood.  
  The open regular neighborhood, equipped with the open 2-disks therein and equipped with the solid 3-framing inherited from the standard framing of $\RR^3$, is an object in $\bcD^{\sfr}_{[2,3]}$.
  }
  \label{fig11}
\end{figure}

Let $k\geq 0$.
Recall the topological operad $\cE_k$ of little $k$-disks, which we regard as an $\infty$-operad: $\cE_k \ra \Fin_\ast$.  
Recall the notion of closed covers from \S3.5 of~\cite{fact1}, which are pullback diagrams in $\Bun$ of the form
\[
\xymatrix{
X
\ar[rr]
\ar[d]
&&
A
\ar[d]
\\
B
\ar[rr]
&&
A \cap B
,
}
\]
all of whose morphisms are closed.
Such a closed cover is equivalent to a pair $A,B \subset X$ of closed subspaces, each of which is a union of strata, such that $A \cup B = X$.
In particular, disjoint unions are closed covers, and disjoint unions constitute all of the closed covers in the particularly simple case of $\bcD_{[k]}^{\sfr}$.
\begin{notation}
\label{d.nAlg}
For $k \geq 0$ and $\cV$ be a symmetric monoidal $(\infty,1)$-category, then
 $\Alg_k(\cV)$ is the $(\oo,1)$-category of $\cE_k$-algebras in $\cV$:
\[
\Alg_k(\cV)
~:=~
\Fun^{\ot}(\cE_k,\cV)
~.
\]
For $k=1$, we further abbreviate $\Alg(\cV) := \Alg_1(\cV)$.

\end{notation}

\begin{remark}
Informally, an $\cE_k$-algebra is an object with a $(k\text{-}1)$-sphere family of binary multiplication rules.  
Remarkably, Dunn--Lurie additivity \cite{HA} states that an $\cE_k$-algebra in $\cV$ is precisely an object $A\in \cV$ with $k$ compatible associative algebra structures.
\end{remark}

\begin{observation}
\label{t14}
\begin{enumerate}
\item[]

\item
Each linear injection $\RR \to \RR^{\n1}$ determines a morphism between topological operads $\cE_1 \to \cE_{\n1}$ of little disks, and therefore maps of spaces
\begin{equation}
\label{e24}
S^{\mathit{n}\text{-}2}
\overset{\simeq}\longrightarrow
{\sf Hom}^{\sf lin.inj}(\RR , \RR^{\n1})
\longrightarrow
\Hom_{{\sf Operads}_{\infty}}(\cE_1 , \cE_{\n1})
~.
\end{equation}

\item
Let $\cV$ be a symmetric monoidal $(\infty,1)$-category.
The map~(\ref{e24}) corepresents a conservative functor between $(\infty,1)$-categories,
\begin{equation}
\label{e25}
S^{\mathit{n}\text{-}2}
\times
\Alg_{\n1}(\cV)
\longrightarrow
\Alg(\cV)
~,\qquad
(v,A)
\mapsto (A , \overset{v} \ot)
~,
\end{equation}
which is evidently functorial in the symmetric monoidal $(\infty,1)$-category $\cV$.

\end{enumerate}

\end{observation}

Algebras over little $n$-disks operads appear in the theory of functors on $(\oo,1)$-categories such as $\Mfd_n^{\sfr}$ as a consequence of the following proposition.

\begin{prop}
\label{t8}
Let $k \geq 0$.
Let $\cX$ be an $(\oo,1)$-category with finite products, regarded as a symmetric monoidal structure via its Cartesian monoidal structure.
There is a canonical equivalence between $(\oo,1)$-categories:
\[
\cShv_{\cX}(\bcD_{[k]}^{\sfr})
~\simeq~
\Alg_k(\cX)
~.
\]
\end{prop}
\begin{proof}
Taking path-components defines a functor $\bcD^{\sfr}_{[k]} \xra{\pi_0} {\sf Corr}(\Fin)$ to the category of spans among finite sets.
There is a standard monomorphism $\Fin_\ast \xra{I_+ \mapsto I} {\sf Corr}(\Fin)$ between categories, whose value on a morphism $(I_+ \xra{f} J_+)$ is the span $(I \hookleftarrow f^{-1}(J) \xra{f_|} J)$.  
Consider the $\infty$-subcategory $\Disk^{\fr}_{k,+}:= \Fin_\ast \underset{{\sf Corr}(\Fin)} \times \bcD^{\sfr}_{[k]}$.
This $(\infty,1)$-category $\Disk^{\fr}_{k,+}$ agrees with that of Definition 2.1.5 of \cite{zp}, and admits the following explicit description:
\begin{itemize}
\item an object $U_+$ is a space with a disjoint basepoint and where $U\in\Disk_k^{{\fr}}$ is a finite disjoint union of framed $k$-dimensional Euclidean spaces;
\item a morphism from $U_+$ to $V_+$ is a pointed map $g: U_+ \ra V_+$ such that the restriction
\[
g_|:U_{|V}\ra V
\]
is a framed embedding.
\end{itemize}
As Proposition~2.1.10 of \cite{zp}, we proved that for $\cX$ a Cartesian symmetric monoidal $(\infty,1)$-category, there is a fully-faithful embedding
\[
\Alg_k(\cX)
~\hookrightarrow~
\Fun(\Disk^{{\fr}}_{k,+},\cX)
~;
\]
the essential image consists of functors $F$ satisfying the following \emph{reduced Segal} condition: $F(\ast)\ra \uno_\cX$ is an equivalence, and the canonical morphisms
\[
F(U_+)\longrightarrow \prod_{i\in \pi_0U} F(U^i_+)
\]
are equivalences for all $U_+$, where $U^i\subset U$ is the connected component corresponding to the element $i\in \pi_0U$.

Restriction and right Kan extension give an adjunction
\[
\iota^\ast: \Fun(\bcD^{\sfr}_{[k]},\cX)\leftrightarrows \Fun(\Disk_{k,+}^{\fr},\cX): \iota_\ast
~,
\]
which we now show restricts as an equivalence between the full $\oo$-subcategory $\Alg_k(\cX)\subset\Fun(\Disk^{\fr}_{k,+},\cX)$ and the full $\oo$-subcategory of functors $F:\bcD^{\sfr}_{[k]}\ra \cX$ which preserve products. 
Let $F:\bcD^{\sfr}_{[k]}\ra \cX$ be a product-preserving functor.
We show that the unit $F \ra \iota_\ast\iota^\ast F$ is an equivalence. The proof that the counit is an equivalence is identical. For $U\in \bcD_{[k]}^{\sfr}$, the value of $\iota_\ast\iota^\ast F$ on $U$ is given by the limit of the composite functor
\[
(\iota_\ast\iota^\ast F)(U) =\limit\Bigl(\Disk_{k,+}^{\fr}\underset{ \bcD_{[k]}^{\sfr}}\times  (\bcD_{[k]}^{\sfr})^{U/}\ra \Disk_{k,+}^{\fr}\xra{\iota^\ast F}\cX\Bigr)~.
\]
By the closed-active factorization system \cite{fact1} on $\Bun$, and hence on $\bcD_{[k]}^{\sfr}$, there is a left adjoint to the inclusion
\[
\Disk_{k,+}^{\fr}\underset{ \bcD_{[k]}^{\sfr}}\times  (\bcD_{[k]}^{\sfr})^{U/^{\cls}}
\hookrightarrow
\Disk_{k,+}^{\fr}\underset{ \bcD_{[k]}^{\sfr}}\times  (\bcD_{[k]}^{\sfr})^{U/}
\]
of the full $\infty$-subcategory consisting of those $U \to V$ that are closed morphisms in $\bcD_{[k]}^{\sfr}$.
The adjoint sends a morphism $U\ra E$ to $U\ra E'$, the first factor in the closed-active factorization $U\xra{\cls} E'\xra{\sf act} E$ of the original morphism. Consequently, the inclusion is an initial functor, and therefore the right Kan extension can be computed after restricting to the $\infty$-undercategory. 

Lastly, we observe that the closed $\infty$-undercategory $\Disk_{k,+}^{\fr}\underset{ \bcD_{[k]}^{\sfr}}\times  (\bcD_{[k]}^{\sfr})^{U/^{\cls}}$ is equivalent to $\cP(\pi_0U)^{\op}$, the opposite of the poset of subsets of $\pi_0 U$. 
Since $F$ preserves products, this gives that the value of the limit indexed by $\cP(\pi_0U)^{\op}$ is exactly the product of the values of $F$ over the elements of $\pi_0U$, which is $F(U)$. 
\end{proof}

\begin{remark}
Compare Proposition~\ref{t8} with Proposition~C.1 of \cite{normed}, which is the $k\mapsto \oo$ limit of this result. 
The proofs of these two results are essentially identical.
\end{remark}

\subsection{Rigid $\cE_{\n1}$-monoidal categories from copresheaves on $\bcD_{[\n1,n]}^{\sfr}$}

Recall from~\cite{flagged} the pair of $(\infty,1)$-categories $\Cat_{(\oo,n)}\subset \fCat_{(\infty,n)}$ of $(\infty,2)$-categories and of flagged $(\infty,2)$-categories.  
Recall the full $\infty$-subcategory $\bTheta_n \subset \Cat_{(\oo,n)}$ from~\cite{rezk-n}, where $\bTheta_1 = \bDelta$.
Each object in $\bTheta_n$ is an $n$-category associated to a certain class of $n$-dimensional pasting diagram.
After the work~\cite{rezk-n}, it is proved in~\cite{flagged} that the full $\infty$-subcategory $\bTheta_n \subset \Cat_{(\oo,n)}$ strongly generates.  Specifically, the restricted Yoneda functor 
\[
\fCat_{(\infty,n)}
\longrightarrow
\PShv(\bTheta_n)
~,\qquad
\cC
\longmapsto
\bigl(
T
\mapsto 
\Hom_{\fCat_{(\infty,n)}} ( T , \cC )
\bigr)
~,
\]
is fully-faithful with essential image those presheaves on $\bTheta_n$ that satisfy the Segal condition.

\begin{definition}
Let $n \geq 2$.
The respective $(\infty,1)$-categories of \bit{$\cE_{\n1}$-monoidal $(\infty,1)$-categories} and of \bit{$\cE_{\n1}$-monoidal flagged $(\infty,1)$-categories} are
\[
\Alg_{\n1}(\Cat_{(\oo,1)})
~\subset~
\Alg_{\n1}(\fCat_{(\oo,1)})
~,
\]
those of $\cE_{\n1}$-algebras in $\Cat_{(\oo,1)}$ and in $\fCat_{(\oo,1)}$.  
A \bit{monoidal (flagged) $(\infty,1)$-category} is an $\cE_1$-algebra in (flagged) $(\infty,1)$-categories.
\end{definition}

\begin{definition}\label{def.duals}
\begin{enumerate}

\item[]

\item
Let $\cR\in \Alg(\Cat_{(\oo,1)})$ be a monoidal $(\infty,1)$-category.
An object $V \in \cR$ is \bit{right-dualizable} if left-tensoring with $V$
\[
\cR\xra{V\ot -}\cR
\]
has a right adjoint
\[
\cR\xla{-\ot V^R}\cR
\]
represented by right-tensoring with an object $V^R\in \cR$.
The monoidal $(\infty,1)$-category $\cR$ has \bit{right-duals} if every object of $\cR$ has a right-dual.

\item
An $\cE_{\n1}$-monoidal $(\infty,1)$-category $\cR\in \Alg_{\n1}(\Cat_{(\infty,1)})$ is \bit{rigid} if, for each $v\in S^{\mathit{n}\text{-}2}$, the monoidal $(\infty,1)$-category $(\cR , \overset{v}\ot)$ (see the notation of Observation~\ref{t14}(2)), has right-duals.  
The $(\infty,1)$-category of \bit{rigid $\cE_{\n1}$-monoidal $(\infty,1)$-categories} is the full $\infty$-subcategory
\[
\Alg^{\sf rig}_{\n1}( \Cat_{(\infty,1)} )
~\subset~
\Alg_{\n1}( \Cat_{(\infty,1)} )
\]
consisting of those $\cE_{\n1}$-monoidal $(\infty,1)$-categories $\cR$ that are rigid.

\end{enumerate}

\end{definition}

\begin{observation}
\label{t15}
A monoidal $(\infty,1)$-category $\cR$ is rigid if and only if each object in $\cR$ has both a right-dual and a left-dual.

\end{observation}

\begin{observation}
\label{t22}
The condition that an $\cE_{\n1}$-monoidal $(\infty,1)$-category is rigid can be detected on any of its underlying monoidal $(\infty,1)$-categories.
More precisely,
for each $v\in S^{\mathit{n}\text{-}2}$, the evident diagram among $(\infty,1)$-categories is a pullback:
\[
\xymatrix{
\Alg^{\sf rig}_{\n1}( \Cat_{(\infty,1)} )
\ar[rr]^-{\rm fully~faithful}
\ar[d]_-{(\ref{e25})_v}
&&
\Alg_{\n1}( \Cat_{(\infty,1)} )
\ar[d]^-{(\ref{e25})_v}
\\
\Alg^{\sf rig}( \Cat_{(\infty,1)} )
\ar[rr]^-{\rm fully~faithful}
&&
\Alg( \Cat_{(\infty,1)} )
.
}
\]

\end{observation}

Recall from~\cite{fact1} the \bit{cellular realization} functor, $\bTheta_n^{\op} \xra{\lag-\rag} \bcD_n^{\sfr}$.  In this paper, we only use this functor in the case that $n=1$.  In that case, each value of the cellular realization functor is a 0-disk, or a stratified closed 1-disk.  
We hereafter use the following suggestive notation.
\begin{notation}
\label{d100}
\[
\bDelta^{\op}
\xra{~\lag - \rag~}
\bcD_1^{\sfr}
~,\qquad
[p]
\longmapsto
\DD[p]
~.
\]
\end{notation}
Now, taking products of stratified spaces, and products of framings, defines a functor
\begin{equation}
\label{e2}
\rho
\colon
\bcD_{[\n1]}^{\sfr}
\times
\bDelta^{\op}
\xra{\id\times\langle-\rangle}
\bcD_{[\n1]}^{\sfr}
\times
\bcD_1^{\sfr}
\longrightarrow
\bcD_{[\n1,n]}^{\sfr}
~,\qquad
\bigl(
 U,[p]
\bigr)
\longmapsto
U \times \DD[p]
~.
\end{equation}

\begin{observation}
\label{t.closed.closed}
The functor $\rho$ of~(\ref{e2}) carries disjoint unions in the first variable to disjoint unions, and closed covers in the second variable to closed covers.
By Proposition~\ref{t8}, restriction along~(\ref{e2}) therefore defines a functor
\begin{equation}
\label{e3}
\cShv(\bcD_{[\n1,n]}^{\sfr} )
\xra{~\rho^\ast~}
\Alg_{\n1}( \fCat_{(\oo,1)})
~.
\end{equation}
\end{observation}

Consider the reflexive localization between $(\oo,1)$-categories:
\begin{equation}
\label{f109}
(-)^{\wedge}_{\sf unv}
\colon
\fCat_{(\infty,1)} 
~
\leftrightarrows
~
\Cat
\colon
{\rm inclusion}
~,
\end{equation}
in which the left adjoint is the forgetful functor, given by taking univalent-completion.
This left adjoint preserves finite products.
Therefore, the reflexive localization~(\ref{f109}) lifts to a reflexive localization between $(\oo,1)$-categories:
\begin{equation}
\label{f115}
(-)^{\wedge}_{\sf unv}
\colon
\Alg_{\n1}( \fCat_{(\infty,1)} )
~\rightleftarrows~
\Alg_{\n1}( \Cat_{(\infty,1)} )
\colon
{\rm inclusion}
~.
\end{equation}

We state the main result of this subsection here, and postpone its proof to the end of this subsection.
\begin{lemma}
\label{t7}
Let $n\geq 2$.
The univalent-completion of each value of the functor $\rho^\ast$ is an $\cE_{\n1}$-monoidal $(\infty,1)$-category that is rigid. That is, there is a unique dashed functor making the following diagram commute:
\[
\xymatrix{
\cShv(\bcD_{[\n1,n]}^{\sfr} )
\ar[d]_-{\rho^\ast}
\ar@{-->}[rr]
&&
\Alg_{\n1}^{\sf rig}( \Cat_{(\oo,1)})
\ar[d]^-{\rm fully~faithful}
\\
\Alg_{\n1}( \fCat_{(\oo,1)})
\ar[rr]^-{(-)^{\wedge}_{\sf unv}}
&&
\Alg_{\n1}( \Cat_{(\oo,1)})
.
}
\]
\end{lemma}

We prove Lemma~\ref{t7} by reducing to the case in which $n=2$, then invoking the following diagrammatic criteria for rigidity.

\begin{lemma}
\label{t105}
The univalent-completion of a monoidal flagged $(\infty,1)$-category $\cR$ is rigid provided the existence of the following data:
\begin{enumerate}
\item
maps between spaces
\[
\xymatrix{
\Obj(\cR) 
\ar@{-->}[rr]^-{L}
&&
\Obj(\cR)
&
\text{ and }
&
\Obj(\cR) 
\ar@{-->}[rr]^-{R}
&&
\Obj(\cR)
~;
}
\]

\item
fillers in diagrams among spaces
\[
%\begin{equation}
%\label{f121}
\xymatrix{
\ast
\ar[d]_-{\lag \uno \rag}
&
\Obj(\cR)
\ar@{-->}[d]_-{\eta_L}
\ar[l]_-{!}
\ar[r]^-{(\id , L)}
&
\Obj(\cR) \times \Obj(\cR)
\ar[d]^-{\ot}
&&
\Obj(\cR)\times \Obj(\cR)
\ar[d]_-{\ot}
&
\Obj(\cR)
\ar@{-->}[d]_-{\epsilon_L}
\ar[r]^-{!}
\ar[l]_-{(L , \id )}
&
\ast
\ar[d]^-{\lag \uno \rag}
\\
\Obj(\cR)
&
\Mor(\cR)
\ar[l]_-{\ev_s}
\ar[r]^-{\ev_t}
&
\Obj(\cR)
&
,
&
\Obj(\cR)
&
\Mor(\cR)
\ar[l]_-{\ev_s}
\ar[r]^-{\ev_t}
&
\Obj(\cR)
}
%\end{equation}
\]
and
\[
%\begin{equation}
%\label{f122}
\xymatrix{
\ast
\ar[d]_-{\lag \uno \rag}
&
\Obj(\cR)
\ar@{-->}[d]_-{\eta_R}
\ar[l]_-{!}
\ar[r]^-{(R , \id)}
&
\Obj(\cR) \times \Obj(\cR)
\ar[d]^-{\ot}
&&
\Obj(\cR)\times \Obj(\cR)
\ar[d]_-{\ot}
&
\Obj(\cR)
\ar@{-->}[d]_-{\epsilon_R}
\ar[r]^-{!}
\ar[l]_-{(\id , R)}
&
\ast
\ar[d]^-{\lag \uno \rag}
\\
\Obj(\cR)
&
\Mor(\cR)
\ar[l]_-{\ev_s}
\ar[r]^-{\ev_t}
&
\Obj(\cR)
&
,
&
\Obj(\cR)
&
\Mor(\cR)
\ar[l]_-{\ev_s}
\ar[r]^-{\ev_t}
&
\Obj(\cR)
~;
}
%\end{equation}
\]

\item
fillers in the diagrams among spaces
% https://q.uiver.app/#q=WzAsNyxbMCwwLCJcXG1vcihcXGNSKSBcXHRpbWVzIFxcbW9yKFxcY1IpIl0sWzAsMSwiXFxtb3IoXFxjUikiXSxbMSwxLCJcXG1vcihcXGNSKSBcXHVuZGVyc2V0e1xcb2JqKFxcY1IpfSBcXHRpbWVzIFxcbW9yKFxcY1IpIl0sWzIsMCwiXFxvYmooXFxjUikiXSxbMiwyLCJcXG1vcihcXGNSKSJdLFs0LDAsIlxcbW9yKFxcY1IpIFxcdGltZXMgXFxtb3IoXFxjUikiXSxbNCwxLCJcXG1vcihcXGNSKSJdLFszLDUsIiggXFxsYWcgXFxpZCBcXHJhZyAsIFxcZXBzaWxvbl9MICkiXSxbNSw2LCJcXG90aW1lcyJdLFszLDQsIlxcbGFnIFxcaWQgXFxyYWciLDAseyJsYWJlbF9wb3NpdGlvbiI6ODB9XSxbMiw2LCJcXHByXzIiXSxbMywwLCIoIFxcZXRhX0wgLCBcXGxhZyBcXGlkIFxccmFnICkiLDJdLFswLDEsIlxcb3RpbWVzIiwyXSxbMiwxLCJcXHByXzEiLDJdLFsyLDQsIlxcY2lyYyIsMl0sWzMsMiwiXFxhbHBoYV9MIiwyLHsibGFiZWxfcG9zaXRpb24iOjYwLCJzdHlsZSI6eyJib2R5Ijp7Im5hbWUiOiJkYXNoZWQifX19XV0=
\[\begin{tikzcd}
	{\mor(\cR) \times \mor(\cR)} && {\obj(\cR)} && {\mor(\cR) \times \mor(\cR)} \\
	{\mor(\cR)} & {\mor(\cR) \underset{\obj(\cR)} \times \mor(\cR)} &&& {\mor(\cR)} \\
	&& {\mor(\cR)}
	\arrow["\otimes"', from=1-1, to=2-1]
	\arrow["{( \eta_L , \lag \id \rag )}"', from=1-3, to=1-1]
	\arrow["{( \lag \id \rag , \epsilon_L )}", from=1-3, to=1-5]
	\arrow["{\alpha_L}"'{pos=0.6}, dashed, from=1-3, to=2-2]
	\arrow["{\lag \id \rag}"{pos=0.8}, from=1-3, to=3-3]
	\arrow["\otimes", from=1-5, to=2-5]
	\arrow["{\pr_1}"', from=2-2, to=2-1]
	\arrow["{\pr_2}", from=2-2, to=2-5]
	\arrow["\circ"', from=2-2, to=3-3]
\end{tikzcd}
~,\]
% https://q.uiver.app/#q=WzAsOSxbMCwwLCJcXG1vcihcXGNSKSBcXHRpbWVzIFxcbW9yKFxcY1IpIl0sWzAsMSwiXFxtb3IoXFxjUikiXSxbMSwxLCJcXG1vcihcXGNSKSBcXHVuZGVyc2V0e1xcb2JqKFxcY1IpfSBcXHRpbWVzIFxcbW9yKFxcY1IpIl0sWzIsMCwiXFxvYmooXFxjUikiXSxbMiwyLCJcXG1vcihcXGNSKSJdLFs0LDAsIlxcbW9yKFxcY1IpIFxcdGltZXMgXFxtb3IoXFxjUikiXSxbNCwxLCJcXG1vcihcXGNSKSJdLFsxLDAsIlxcb2JqKFxcY1IpIFxcdGltZXMgXFxvYmooXFxjUikiXSxbMywwLCJcXG9iaihcXGNSKSBcXHRpbWVzIFxcb2JqKFxcY1IpIl0sWzUsNiwiXFxvdGltZXMiXSxbMyw0LCJcXGxhZyBcXGlkIFxccmFnIiwwLHsibGFiZWxfcG9zaXRpb24iOjgwfV0sWzIsNiwiXFxwcl8yIl0sWzAsMSwiXFxvdGltZXMiLDJdLFsyLDEsIlxccHJfMSIsMl0sWzIsNCwiXFxjaXJjIiwyXSxbMywyLCJcXGJldGFfTCIsMix7ImxhYmVsX3Bvc2l0aW9uIjo2MCwic3R5bGUiOnsiYm9keSI6eyJuYW1lIjoiZGFzaGVkIn19fV0sWzMsNywiKCBMICwgXFxpZCApIiwyXSxbNywwLCJcXGxhZyBcXGlkIFxccmFnIFxcdGltZXMgXFxldGFfTCIsMl0sWzgsNSwiXFxlcHNpbG9uX0wgXFx0aW1lcyBcXGxhZyBcXGlkIFxccmFnIl0sWzMsOCwiKFxcaWQgLCBMKSJdXQ==
\[\begin{tikzcd}
	{\mor(\cR) \times \mor(\cR)} & {\obj(\cR) \times \obj(\cR)} & {\obj(\cR)} & {\obj(\cR) \times \obj(\cR)} & {\mor(\cR) \times \mor(\cR)} \\
	{\mor(\cR)} & {\mor(\cR) \underset{\obj(\cR)} \times \mor(\cR)} &&& {\mor(\cR)} \\
	&& {\mor(\cR)}
	\arrow["\otimes"', from=1-1, to=2-1]
	\arrow["{\lag \id \rag \times \eta_L}"', from=1-2, to=1-1]
	\arrow["{( L , \id )}"', from=1-3, to=1-2]
	\arrow["{(\id , L)}", from=1-3, to=1-4]
	\arrow["{\beta_L}"'{pos=0.6}, dashed, from=1-3, to=2-2]
	\arrow["{\lag \id \rag}"{pos=0.8}, from=1-3, to=3-3]
	\arrow["{\epsilon_L \times \lag \id \rag}", from=1-4, to=1-5]
	\arrow["\otimes", from=1-5, to=2-5]
	\arrow["{\pr_1}"', from=2-2, to=2-1]
	\arrow["{\pr_2}", from=2-2, to=2-5]
	\arrow["\circ"', from=2-2, to=3-3]
\end{tikzcd}
~,
\]
and
% https://q.uiver.app/#q=WzAsOSxbMCwwLCJcXG1vcihcXGNSKSBcXHRpbWVzIFxcbW9yKFxcY1IpIl0sWzAsMSwiXFxtb3IoXFxjUikiXSxbMSwxLCJcXG1vcihcXGNSKSBcXHVuZGVyc2V0e1xcb2JqKFxcY1IpfSBcXHRpbWVzIFxcbW9yKFxcY1IpIl0sWzIsMCwiXFxvYmooXFxjUikiXSxbMiwyLCJcXG1vcihcXGNSKSJdLFs0LDAsIlxcbW9yKFxcY1IpIFxcdGltZXMgXFxtb3IoXFxjUikiXSxbNCwxLCJcXG1vcihcXGNSKSJdLFsxLDAsIlxcb2JqKFxcY1IpIFxcdGltZXMgXFxvYmooXFxjUikiXSxbMywwLCJcXG9iaihcXGNSKSBcXHRpbWVzIFxcb2JqKFxcY1IpIl0sWzUsNiwiXFxvdGltZXMiXSxbMyw0LCJcXGxhZyBcXGlkIFxccmFnIiwwLHsibGFiZWxfcG9zaXRpb24iOjgwfV0sWzIsNiwiXFxwcl8yIl0sWzAsMSwiXFxvdGltZXMiLDJdLFsyLDEsIlxccHJfMSIsMl0sWzIsNCwiXFxjaXJjIiwyXSxbMywyLCJcXGFscGhhX1IiLDIseyJsYWJlbF9wb3NpdGlvbiI6NjAsInN0eWxlIjp7ImJvZHkiOnsibmFtZSI6ImRhc2hlZCJ9fX1dLFszLDcsIiggXFxpZCAsIFIgKSIsMl0sWzcsMCwiXFxldGFfUiBcXHRpbWVzIFxcbGFnIFxcaWQgXFxyYWciLDJdLFs4LDUsIlxcbGFnIFxcaWQgXFxyYWcgXFx0aW1lcyBcXGVwc2lsb25fUiJdLFszLDgsIihcXGlkICwgTCkiXV0=
\[\begin{tikzcd}
	{\mor(\cR) \times \mor(\cR)} & {\obj(\cR) \times \obj(\cR)} & {\obj(\cR)} & {\obj(\cR) \times \obj(\cR)} & {\mor(\cR) \times \mor(\cR)} \\
	{\mor(\cR)} & {\mor(\cR) \underset{\obj(\cR)} \times \mor(\cR)} &&& {\mor(\cR)} \\
	&& {\mor(\cR)}
	\arrow["\otimes"', from=1-1, to=2-1]
	\arrow["{\eta_R \times \lag \id \rag}"', from=1-2, to=1-1]
	\arrow["{( \id , R )}"', from=1-3, to=1-2]
	\arrow["{(\id , L)}", from=1-3, to=1-4]
	\arrow["{\alpha_R}"'{pos=0.6}, dashed, from=1-3, to=2-2]
	\arrow["{\lag \id \rag}"{pos=0.8}, from=1-3, to=3-3]
	\arrow["{\lag \id \rag \times \epsilon_R}", from=1-4, to=1-5]
	\arrow["\otimes", from=1-5, to=2-5]
	\arrow["{\pr_1}"', from=2-2, to=2-1]
	\arrow["{\pr_2}", from=2-2, to=2-5]
	\arrow["\circ"', from=2-2, to=3-3]
\end{tikzcd}
~,
\]
% https://q.uiver.app/#q=WzAsNyxbMCwwLCJcXG1vcihcXGNSKSBcXHRpbWVzIFxcbW9yKFxcY1IpIl0sWzAsMSwiXFxtb3IoXFxjUikiXSxbMSwxLCJcXG1vcihcXGNSKSBcXHVuZGVyc2V0e1xcb2JqKFxcY1IpfSBcXHRpbWVzIFxcbW9yKFxcY1IpIl0sWzIsMCwiXFxvYmooXFxjUikiXSxbMiwyLCJcXG1vcihcXGNSKSJdLFs0LDAsIlxcbW9yKFxcY1IpIFxcdGltZXMgXFxtb3IoXFxjUikiXSxbNCwxLCJcXG1vcihcXGNSKSJdLFszLDUsIiggXFxlcHNpbG9uX1IgLCBcXGxhZyBcXGlkIFxccmFnICkiXSxbNSw2LCJcXG90aW1lcyJdLFszLDQsIlxcbGFnIFxcaWQgXFxyYWciLDAseyJsYWJlbF9wb3NpdGlvbiI6ODB9XSxbMiw2LCJcXHByXzIiXSxbMywwLCIoIFxcbGFnIFxcaWQgXFxyYWcgLCBcXGV0YV9SICkiLDJdLFswLDEsIlxcb3RpbWVzIiwyXSxbMiwxLCJcXHByXzEiLDJdLFsyLDQsIlxcY2lyYyIsMl0sWzMsMiwiXFxiZXRhX1IiLDIseyJsYWJlbF9wb3NpdGlvbiI6NjAsInN0eWxlIjp7ImJvZHkiOnsibmFtZSI6ImRhc2hlZCJ9fX1dXQ==
\[\begin{tikzcd}
	{\mor(\cR) \times \mor(\cR)} && {\obj(\cR)} && {\mor(\cR) \times \mor(\cR)} \\
	{\mor(\cR)} & {\mor(\cR) \underset{\obj(\cR)} \times \mor(\cR)} &&& {\mor(\cR)} \\
	&& {\mor(\cR)}
	\arrow["\otimes"', from=1-1, to=2-1]
	\arrow["{( \lag \id \rag , \eta_R )}"', from=1-3, to=1-1]
	\arrow["{( \epsilon_R , \lag \id \rag )}", from=1-3, to=1-5]
	\arrow["{\beta_R}"'{pos=0.6}, dashed, from=1-3, to=2-2]
	\arrow["{\lag \id \rag}"{pos=0.8}, from=1-3, to=3-3]
	\arrow["\otimes", from=1-5, to=2-5]
	\arrow["{\pr_1}"', from=2-2, to=2-1]
	\arrow["{\pr_2}", from=2-2, to=2-5]
	\arrow["\circ"', from=2-2, to=3-3]
\end{tikzcd}
~.
\]

\end{enumerate}

\end{lemma}

\begin{proof}
Denote the unit morphism of the adjunction~(\ref{f115}):
\[
%\begin{equation}
%\label{f110}
\ov{(-)}
\colon
\cR 
\longrightarrow 
\cR^{\wedge}_{\sf unv}
~,\qquad
V
\longmapsto \ov{V}
~,
%\end{equation}
\]
which is a morphism between monoidal flagged $(\infty,1)$-categories.

Let $\ov{V} \in \Obj(\cR^{\wedge}_{\sf unv})$.
Using that the map $\Obj(\cR) \xra{\Obj\bigl( \ov{(-)} \bigr)} \Obj(\cR^{\wedge}_{\sf unv})$ is surjective, choose a lift $V \in \Obj(\cR)$ of $\ov{V} \in \Obj(\cR^{\wedge}_{\sf unv})$.
Consider the four 1-cells in $\cR$:
\[
\uno \xra{~\eta_L~} V \ot L(V)
~,~
L(V) \ot V \xra{~\epsilon_L(V)~} \uno
\qquad\text{ and }\qquad
\uno \xra{~\eta_R(V)~} R(V) \ot V
~,~
V \ot R(V) \xra{~\epsilon_R~} \uno
~.
\]
Consider the invertible 2-cells in $\cR$:
\[
\xymatrix{
V 
\ar[dr]_-{\eta_L(V) \ot \id_V }
\ar@(u,u)[rr]^-{\id_V}
&
\alpha_L(V)
&
V
&&
L(V) \ot V \ot L(V)
\ar[dr]^-{ \epsilon_L(V) \ot \id_{L(V)} }
&
\\
&
V \ot L(V) \ot V
\ar[ur]_-{\id_V \ot \epsilon_L(V)}
&
,
&
L(V)
\ar[ur]^-{ \id_{L(V)} \ot \eta_L(V) }
\ar@(d,d)[rr]_-{\id_{L(V)}}
&
\beta_L(V)
&
L(V)
}
\]
and
\[
\xymatrix{
R(V) 
\ar[dr]_-{\eta_R(V) \ot \id_{R(V)} }
\ar@(u,u)[rr]^-{\id_{R(V)}}
&
\alpha_R(V)
&
R(V)
&&
V \ot R(V) \ot V
\ar[dr]^-{\epsilon_R(V) \ot \id_V }
&
\\
&
R(V) \ot V\ot R(V)
\ar[ur]_-{\id_{R(V)} \ot \epsilon_R(V)}
&
,
&
V
\ar[ur]^-{ \id_V \ot \eta_R(V) }
\ar@(d,d)[rr]_-{\id_{V}}
&
\beta_R(V)
&
V
.
}
\]

Now, these diagrams in $\cR$ are carried by the morphism $\ov{(-)}$ between monoidal flagged $(\infty,1)$-categories to corresponding diagrams in $\cR^{\wedge}_{\sf unv}$.
By definition of left- and right-duals in a monoidal $(\infty,1)$-category, these data witness the values $\ov{L(V)}\in \cR^{\wedge}_{\sf unv}$ and $\ov{R(V)}\in \cR^{\wedge}_{\sf unv}$ as, respectively, left- and right-duals of $\ov{V}\in \cR^{\wedge}_{\sf unv}$.

\end{proof}

We prove Lemma~\ref{t7} by constructing specific diagrams in $\bcD_{[1,2]}^{\sfr}$ which any closed sheaf $\bcD_{[1,2]}^{\sfr} \xra{ \cF } \Spaces$ carries to the diagrams of Observation~\ref{t105}.
Constructing these diagrams in $\bcD_{[1,2]}^{\sfr}$ takes some explanation.
The hasty reader might simply inspect Figures~\ref{fig1}--\ref{fig8} for depictions of such diagrams in $\bcD_{[1,2]}^{\sfr}$.

Let $M$ be a $[1,2]$-manifold. 
Consider the space of solid 2-framings on $M$:
\[
\sfr_2(M)
~:=~
\Map_{/\BO(1) \amalg \BO(2)}\left( M , S^1 \amalg \ast \right)
~.
\]
Via Remark~\ref{r10}, the standard action $\sO(2) \lacts \RR^2$ determines an action 
\begin{equation}
\label{f106}
\sO(2)
~\lacts~ 
\sfr_2(M)
~.
\end{equation}
Should $M \subset \RR^2$ be given as subspace of a Euclidean 2-space, there is a natural solid 2-framing:
\[
\varphi_\natural^{M}
~\in~
\sfr_2(M)
~.
\]
In particular, there is a natural solid 2-framing on $\RR^1$ and on $\RR^1 \times \DD^1$ and, more generally, on $\RR^1 \times \DD[p]$ for each $p\geq 0$ (see Notation~\ref{d100}).
Generally, for $(D,\varphi)\in \bcD_{[1,2]}^{\sfr}$ an object, if $D \subset \RR^2$ and $\varphi= \varphi_\natural^D$, we simply denote the object $D := (D,\varphi) \in \bcD_{[1,2]}^{\sfr}$.
For each $p \geq 0$, there is a canonical $\sO(2)$-equivariant identification,
\[
S^1
\simeq
\sO(2)_{/\sO(1)}
~\simeq~
\sfr_2\bigl( \RR^1 \times \DD[p] \bigr)
\qquad
\text{ (if $p=0$) }
\qquad
\text{ and }
\qquad
\sO(2)
~\simeq~
\sfr_2\bigl( \RR^1 \times \DD[p] \bigr)
\qquad
\text{ (if $p>0$) }
~,
\]
that associates the identity element with the natural solid 2-framing: $\uno
\mapsto
\varphi^{\RR^1 \times \DD[p]}_{\natural}$.
In particular, for each $p\geq 0$, the space of solid 2-framings on $\RR^1\times \DD[p]$ is a 1-type;
for $\varphi , \psi \in \sfr_2\bigl( \RR^1 \times \DD[p] \bigr)$, 
there is a canonical $\ZZ$-torsor of identifications $\varphi \simeq \psi$.  

Denote the object
\[
\RR^1_{\sf rev}
~\in~
\bcD_{[1,2]}^{\sfr}
\]
which is $\RR^1$ with the solid 2-framing that is the action by $-\uno\in \sO(2)$ on the standard solid 2-framing on $\RR^1$.
Consider the continuous map between topological spaces:
\[
{\sf Rot}
\colon
\RR
\longrightarrow
\sO(2)
~,\qquad
t
\longmapsto
\begin{bmatrix} \cos(\pi t) & -\sin(\pi t) \\ \sin(\pi t) & \cos(\pi t)  \end{bmatrix}
~.
\]
Consider the composite continuous maps
\begin{equation}
\label{f107}
\Delta^2
\xra{(t_0,t_1,t_2)\mapsto t_1}
\RR
\xra{~\sf Rot~}
\sO(2)
\qquad\text{ and }\qquad
\Delta^2
\xra{(t_0,t_1,t_2)\mapsto -t_1}
\RR
\xra{~\sf Rot~}
\sO(2)
~.
\end{equation}
Once and for all, choose a framed diffeomorphism $f \colon \RR^1 \xra{\cong} (-\frac{1}{2},+\frac{1}{2})$ to an open interval.  
Consider the framed open embedding
\[
\mu
\colon
\RR^1 \sqcup \RR^1
~=~
\RR^1 \times \{\pm 1\}
\longrightarrow
\RR^1
~,\qquad
(v,r)
\longmapsto
f(v) + r
~.
\]
Consider the open embeddings
\begin{equation}
\label{f111}
\RR^1 \times \DD^1
\xra{~\eta'_L~}
\RR^1 \times \DD^1
~,\qquad
(v,r)
\longmapsto
\Bigl(
~
(1+f(v)) \sin(-\frac{\pi}{2} r)
~,~
(1+f(v)) \cos(-\frac{\pi}{2} r)-1
~
\Bigr)
~,
\end{equation}
and
\begin{equation}
\label{f113}
\RR^1 \times \DD^1
\xra{~\epsilon'_L~}
\RR^1 \times \DD^1
~,\qquad
(v,r)
\longmapsto
\Bigl(
~
(1-f(v)) \sin(\frac{\pi}{2} r)
~,~
1-(1-f(v)) \cos(\frac{\pi}{2} r)
~
\Bigr)
~,
\end{equation}
and
\begin{equation}
\label{f112}
\RR^1 \times \DD^1
\xra{~\eta'_R~}
\RR^1 \times \DD^1
~,\qquad
(v,r)
\longmapsto
\Bigl(
~
(1-f(v)) \sin(\frac{\pi}{2} r)
~,~
(1-f(v)) \cos(\frac{\pi}{2} r)-1
~
\Bigr)
~,
\end{equation}
and
\begin{equation}
\label{f114}
\RR^1 \times \DD^1
\xra{~\epsilon'_R~}
\RR^1 \times \DD^1
~,\qquad
(v,r)
\longmapsto
\Bigl(
~
(1+f(v)) \sin(-\frac{\pi}{2} r)
~,~
1-(1+f(v)) \cos(-\frac{\pi}{2} r)
~
\Bigr)
~.
\end{equation}
See Figures~\ref{fig1}--\ref{fig4}.
\begin{figure}[H]
  \includegraphics[width=\linewidth, trim={0 {3.6in} 0 {.5in}}, clip]{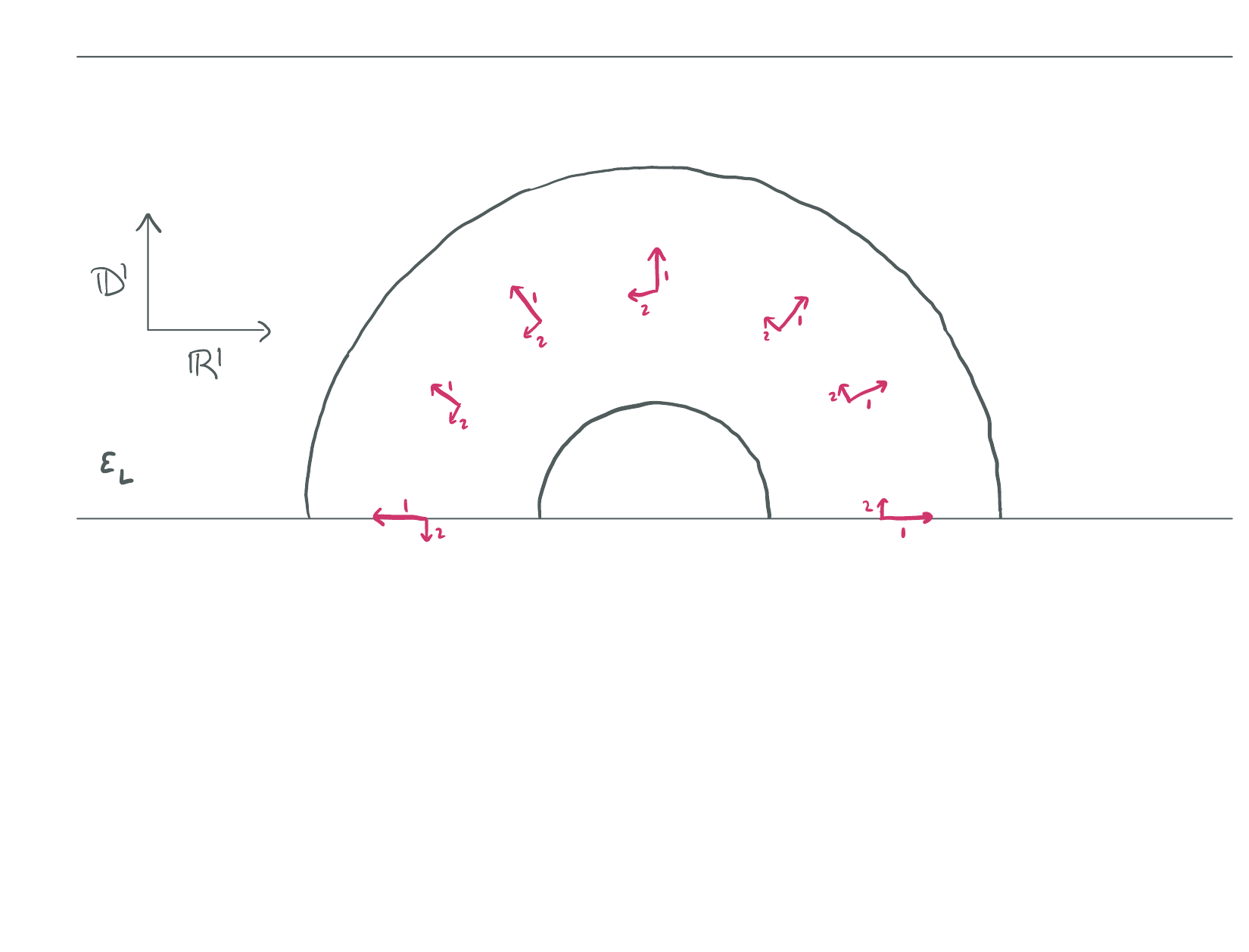}
  \caption{The open embedding $\epsilon'_L$.  
  It, together with the indicated framing, defines the morphism $\RR^1 \xra{\epsilon_L} \RR^1 \times \DD^1$ in $\bcD^{\sfr}_{[1,2]}$.}
  \label{fig1}
\end{figure}
\begin{figure}[H]
  \includegraphics[width=\linewidth, trim={0 {3.6in} 0 {.5in}}, clip]{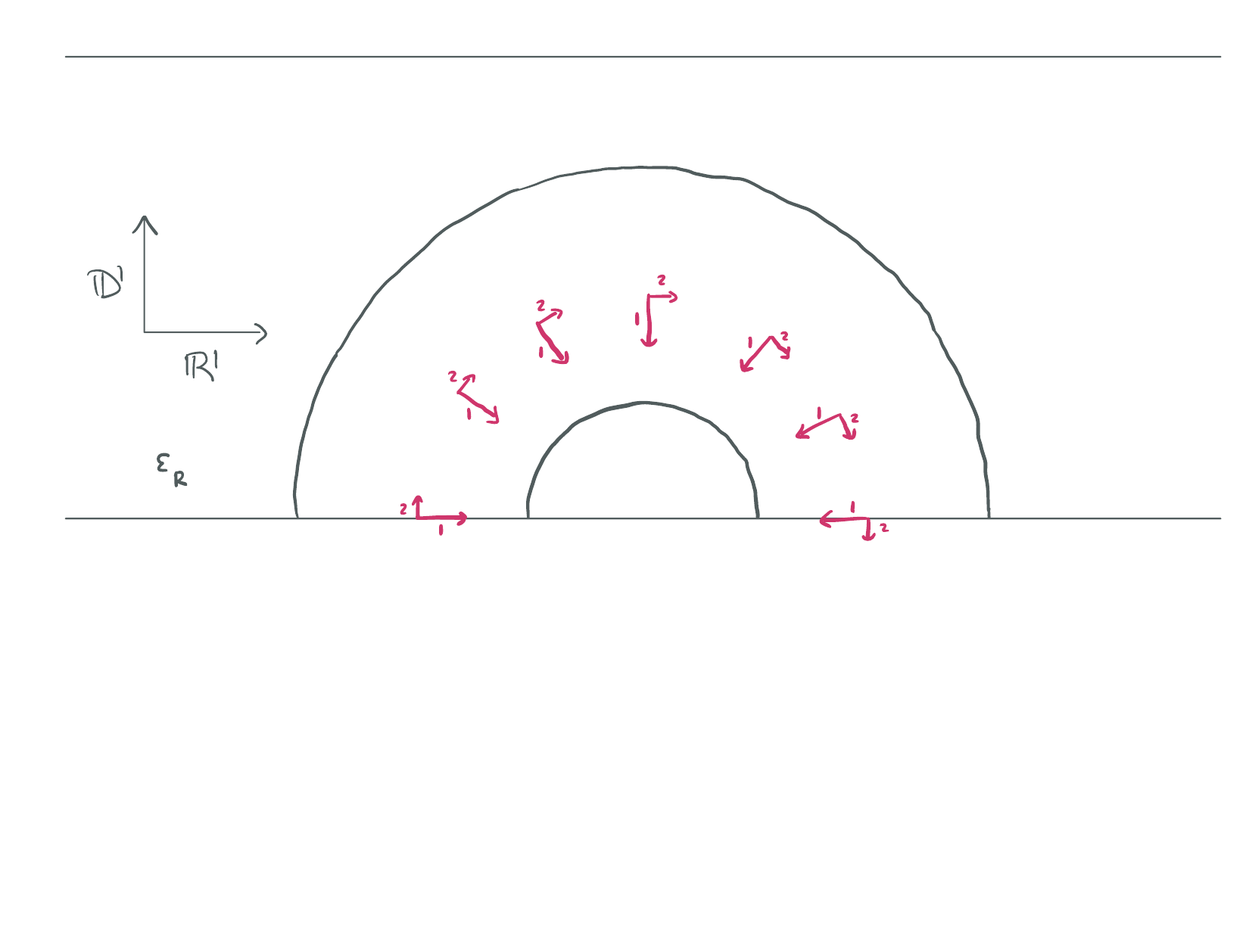}
  \caption{The open embedding $\epsilon'_R$.  
  It, together with the indicated framing, defines the morphism $\RR^1 \xra{\epsilon_R} \RR^1 \times \DD^1$ in $\bcD^{\sfr}_{[1,2]}$.}
  \label{fig2}
\end{figure}
\begin{figure}[H]
  \includegraphics[width=\linewidth, trim={0 {3.8in} 0 {.25in}}, clip]{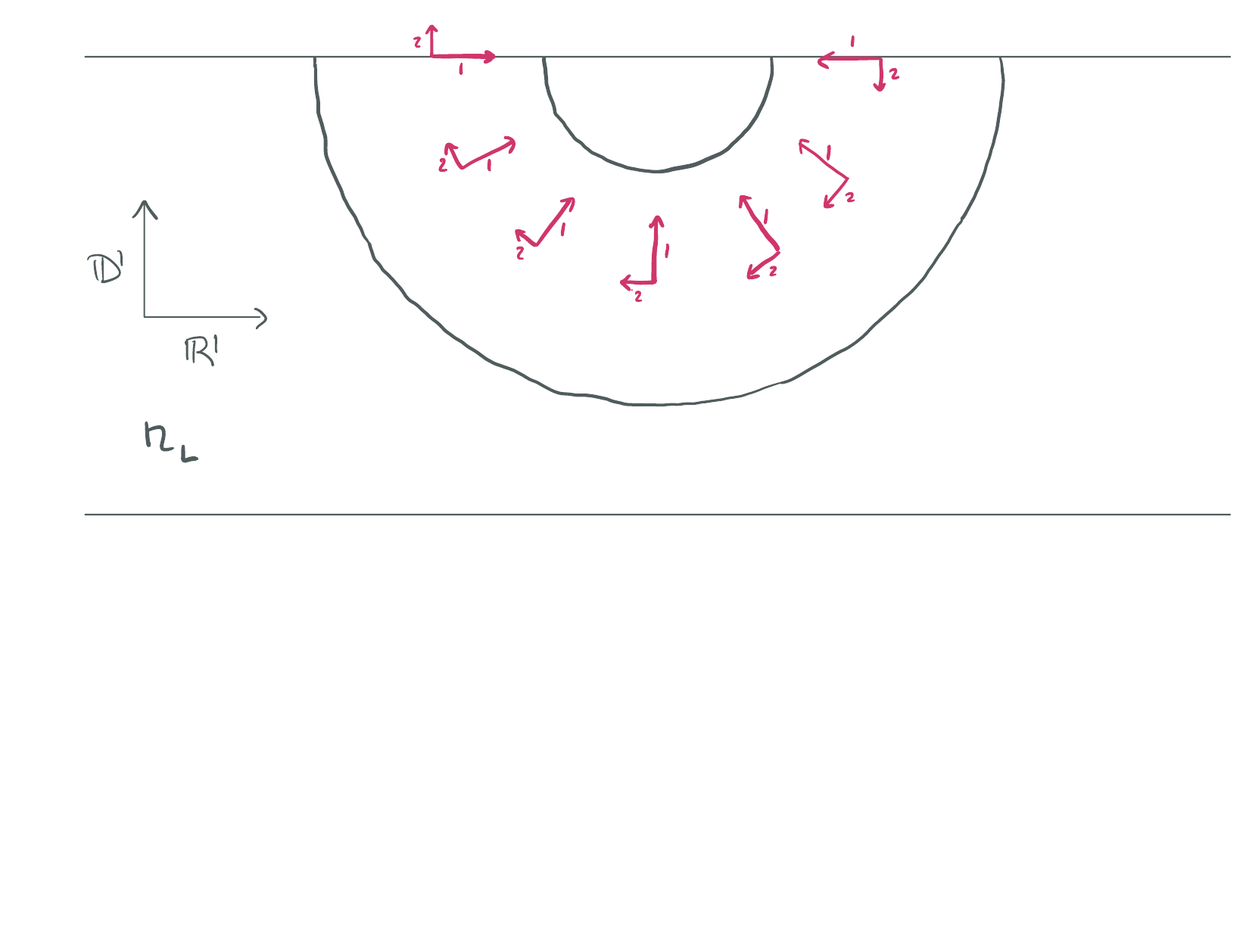}
  \caption{The open embedding $\eta'_L$.  
  It, together with the indicated framing, defines the morphism $\RR^1 \xra{\eta_L} \RR^1 \times \DD^1$ in $\bcD^{\sfr}_{[1,2]}$.}
  \label{fig3}
\end{figure}
\begin{figure}[H]
  \includegraphics[width=\linewidth, trim={0 {3.8in} 0 {.3in}}, clip]{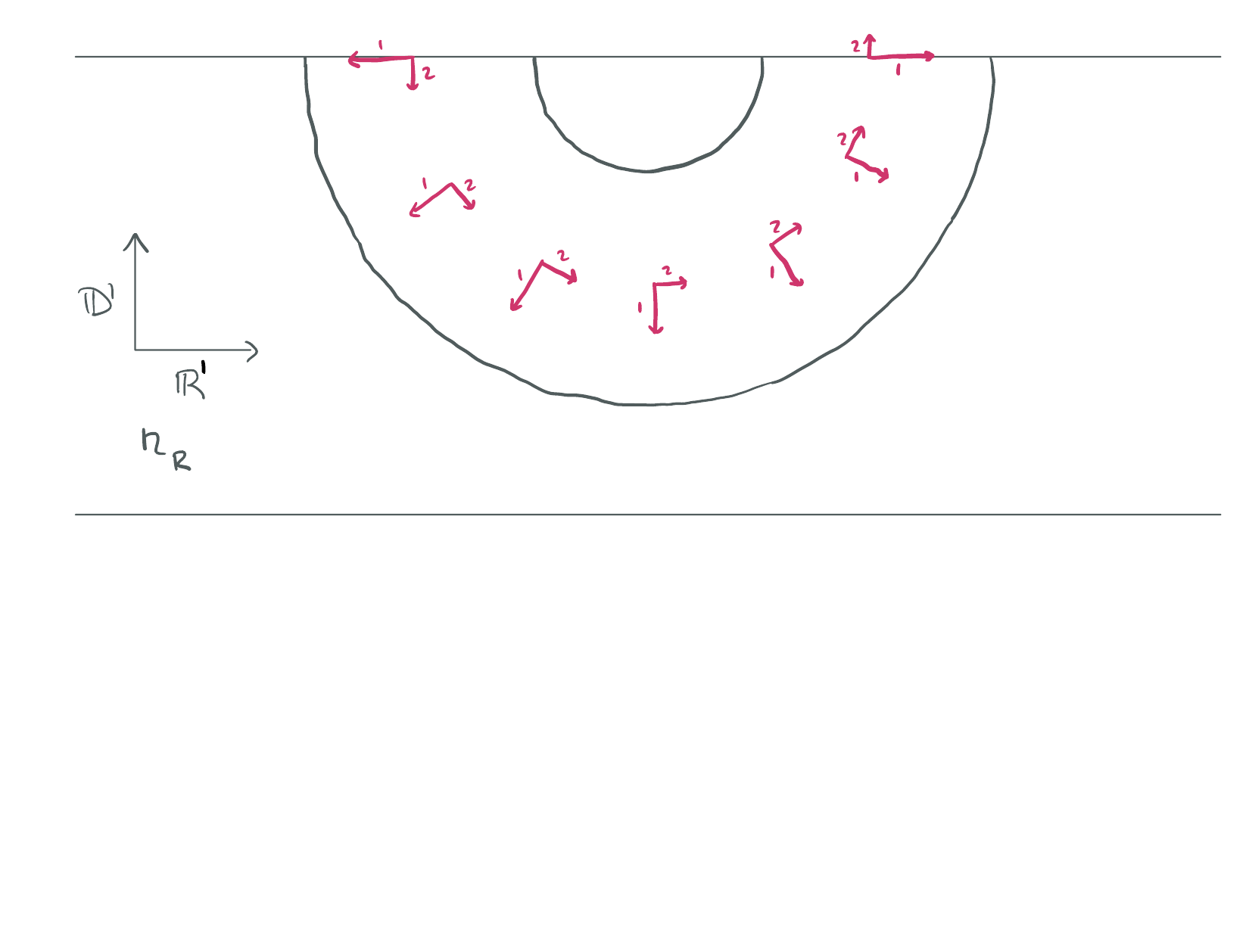}
  \caption{The open embedding $\eta'_R$.  
  It, together with the indicated framing, defines the morphism $\RR^1 \xra{\eta_R} \RR^1 \times \DD^1$ in $\bcD^{\sfr}_{[1,2]}$.}
  \label{fig4}
\end{figure}

Pulling back the standard solid 2-framing on $\RR^1 \times \DD^1$ along each of these open embeddings determines solid 2-framings on each domain $\RR^1 \times \DD^1$:
\[
(\eta'_L)^\ast \varphi_\natural^{\RR^1 \times \DD^1}
~,~
(\epsilon'_L)^\ast \varphi_\natural^{\RR^1 \times \DD^1}
~,~
(\eta'_R)^\ast \varphi_\natural^{\RR^1 \times \DD^1}
~,~
(\epsilon'_R)^\ast \varphi_\natural^{\RR^1 \times \DD^1}
~.
\]
Consider the two proper constructible embeddings:
\[
\RR^1 \xra{~\partial_\pm~}\RR^1 \times \DD^1 
~,\qquad
x \longmapsto (x,\pm 1)
~.
\]
These morphisms correspond to closed morphisms in $\bcD_{[1,2]}^{\sfr}$:
\[
\RR^1 \times \DD^1 
\xra{~\partial_\pm~}
\RR^1
~.
\]
Note that each of these solid 2-framings restricts along $\RR^1 \xra{\partial_\pm} \RR^1 \times \DD^1$ as, respectively, the standard solid 2-framing on $\RR^1$ and the reversed solid 2-framing on $\RR^1$:
\[
\varphi_{\natural}^{\RR^1}
~=~
\partial_-^\ast (\eta'_L)^\ast \varphi_\natural^{\RR^1 \times \DD^1}
~=~
\partial_-^\ast (\epsilon'_L)^\ast \varphi_\natural^{\RR^1 \times \DD^1}
~=~
\partial_-^\ast (\eta'_R)^\ast \varphi_\natural^{\RR^1 \times \DD^1}
~=~
\partial_-^\ast (\epsilon'_R)^\ast \varphi_\natural^{\RR^1 \times \DD^1}
~;
\]
and
\[
(-\uno)\cdot \varphi_{\natural}^{\RR^1}
~=~
\partial_+^\ast (\eta'_L)^\ast \varphi_\natural^{\RR^1 \times \DD^1}
~=~
\partial_+^\ast (\eta'_R)^\ast \varphi_\natural^{\RR^1 \times \DD^1}
~=~
\partial_+^\ast (\epsilon'_L)^\ast \varphi_\natural^{\RR^1 \times \DD^1}
~=~
\partial_+^\ast (\epsilon'_R)^\ast \varphi_\natural^{\RR^1 \times \DD^1}
~.
\]
Furthermore, via the action~(\ref{f106}), the 2-simplices~(\ref{f107}) supply paths in the fiber over $\varphi_\natural$ of the restriction map $\sfr_2(\RR^1 \times \DD^1) \xra{\partial_-^\ast} \sfr_2(\RR^1)$:
\begin{equation}
\label{f108}
\varphi_{\natural}^{\RR^1 \times \DD^1}
~\simeq~
(\eta'_L)^\ast \varphi_\natural^{\RR^1 \times \DD^1}
~,~
\varphi_{\natural}^{\RR^1 \times \DD^1}
~\simeq~
(\epsilon'_L)^\ast \varphi_\natural^{\RR^1 \times \DD^1}
~,~
\varphi_{\natural}^{\RR^1 \times \DD^1}
~\simeq~
(\eta'_R)^\ast \varphi_\natural^{\RR^1 \times \DD^1}
~,~
\varphi_{\natural}^{\RR^1 \times \DD^1}
~\simeq~
(\epsilon'_R)^\ast \varphi_\natural^{\RR^1 \times \DD^1}
~.
\end{equation}
These identifications~(\ref{f108}) endow each of the open embeddings~(\ref{f111}),~(\ref{f113}),~(\ref{f112}),~(\ref{f114}) 
with the structure of a solidly 2-framed open embedding.
In particular, each of the open embeddings~(\ref{f111}),~(\ref{f113}),~(\ref{f112}),~(\ref{f114})  
defines a morphism in $\bcD_{[1,2]}^{\sfr}$.
Precomposing each of these resulting morphisms in $\bcD_{[1,2]}^{\sfr}$ by the creation morphism $\RR^1 \to \RR^1 \times \DD^1$ results in the following four morphisms in $\bcD_{[1,2]}^{\sfr}$:
\begin{equation}
\label{f120}
\eta_L
~,~
\epsilon_L
~,~
\eta_R
~,~
\epsilon_R
~\in~
\Hom_{\bcD_{[1,2]}^{\sfr}}( \RR^1
,
\RR^1 \times \DD^1
)
~.
\end{equation}
See, again, Figures~\ref{fig1}--\ref{fig4}.  
Inspection of those figures readily verifies the following.
\begin{observation}
\label{t120}
The following diagrams in $\bcD^{\sfr}_{[1,2]}$ canonically commute:
\begin{equation}
\label{f101}
\xymatrix{
\emptyset
\ar[d]_-{!}
&
\RR^1
\ar[d]_-{\eta_L}
\ar[l]_-{!}
\ar[r]^-{(\id , L)}
&
\RR^1 \sqcup \RR^1
\ar[d]^-{\mu}
&&
\RR^1 \sqcup \RR^1
\ar[d]_-{\mu}
&
\RR^1
\ar[d]_-{\epsilon_L}
\ar[r]^-{!}
\ar[l]_-{(L , \id )}
&
\emptyset
\ar[d]^-{!}
\\
\RR^1
&
\RR^1 \times \DD^1
\ar[l]_-{\partial_-}
\ar[r]^-{\partial_+}
&
\RR^1
&
,
&
\RR^1
&
\RR^1 \times \DD^1
\ar[l]_-{\partial_-}
\ar[r]^-{\partial_+}
&
\RR^1
}
\end{equation}
and
\begin{equation}
\label{f102}
\xymatrix{
\emptyset
\ar[d]_-{!}
&
\RR^1
\ar[d]_-{\eta_R}
\ar[l]_-{!}
\ar[r]^-{(R , \id)}
&
\RR^1 \sqcup \RR^1
\ar[d]^-{\mu}
&&
\RR^1 \sqcup \RR^1
\ar[d]_-{\mu}
&
\RR^1
\ar[d]_-{\epsilon_R}
\ar[r]^-{!}
\ar[l]_-{(\id , R)}
&
\emptyset
\ar[d]^-{!}
\\
\RR^1
&
\RR^1 \times \DD^1
\ar[l]_-{\partial_-}
\ar[r]^-{\partial_+}
&
\RR^1
&
,
&
\RR^1
&
\RR^1 \times \DD^1
\ar[l]_-{\partial_-}
\ar[r]^-{\partial_+}
&
\RR^1
.
}
\end{equation}

\end{observation}

Again, recall the Notation~\ref{d100} for $\DD I$ for $I\in \bDelta$ a finite non-empty linearly ordered set.
For each $0 < i\leq 2$, consider the proper constructible embedding
\[
\partial_{\{i-1<i\}}
\colon
\RR^1
\times
\DD^1
~=~
\RR^1
\times
\DD\{i-1<i\}
~\hookrightarrow~
\RR^1
\times
\DD[2]
~. 
\]
Consider the proper constructible bundles
\[
u
\colon
\RR^1
\times
\DD^1
\xra{~\pr~}
\RR^1
\qquad
\text{ and }
\qquad
u_4
\colon
\RR^1
\times
\DD[4]
\xra{~\pr~}
\RR^1
~.
\]
All of these proper constructible bundles are canonically solidly 2-framed, and therefore define morphisms in $\bcD_{[1,2]}^{\sfr}$:
\[
\RR^1
\times
\DD[2]
\xra{~\partial_{\{i<j\}}~}
\RR^1
\times
\DD^1
\qquad
\text{ and }
\qquad
\RR^1
\xra{~u~}
\RR^1
\times
\DD^1
\qquad
\text{ and }
\qquad
\RR^1
\xra{~u_4~}
\RR^1
\times
\DD[4]
~.
\]
Also, the refinement $\RR^1 \times \DD[2] \xra{\partial_{\{0<2\}}} \RR^1 \times \DD\{0<2\}$ is canonically solidly 2-framed, and therefore defines a morphism in 
$\bcD_{[1,2]}^{\sfr}$:
\[
\RR^1
\times
\DD[2]
\xra{~\partial_{\{0<2\}}~}
\RR^1
\times
\DD^1
~.
\]

Consider the solidly 2-framed open embeddings
\begin{equation}
\label{f130}
\alpha'_L~,~\beta'_L~,~\alpha'_R~,~\beta'_R
~\colon~
\RR^1 \times \DD[4]
\longrightarrow
\RR^1 \times \DD[2]
\end{equation}
depicted in Figures~\ref{fig5}--\ref{fig8}.
Regard these as morphisms in $\bcD_{[1,2]}^{\sfr}$.
Precomposing each of the morphisms~(\ref{f130}) with $u_4$ determines morphisms in 
$\bcD_{[1,2]}^{\sfr}$:
\[
\alpha_L
~:=~
\alpha'_L \circ u_4
~,~
\beta_L
~:=~
\beta'_L \circ u_4
~,~
\alpha_R
~:=~
\alpha'_R \circ u_4
~,~
\beta_R
~:=~
\beta'_R \circ u_4
~\in~
\Hom_{\bcD_{[1,2]}^{\sfr}}\bigl( \RR^1 , \RR^1 \times \DD[2]  \bigr)
~.
\]
See, again, Figures~\ref{fig5}--\ref{fig8}.  
\begin{figure}[H]
  \includegraphics[width=\linewidth, trim={0 {3.7in} {2in} {.4in}}, clip]{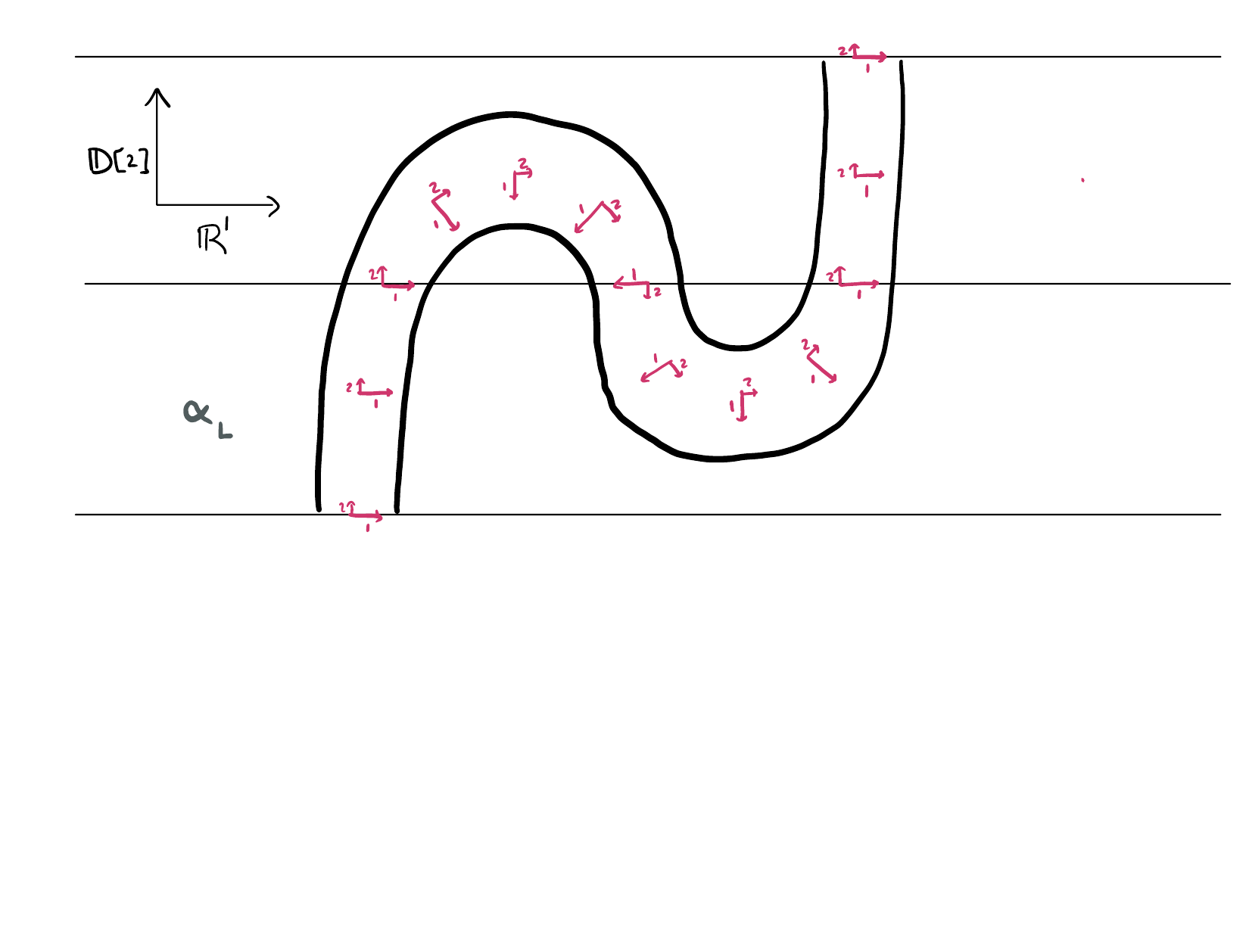}
  \caption{The open embedding $\alpha'_L$.  
  It, together with the indicated framing, defines the morphism $\RR^1 \xra{\alpha_L} \RR^1 \times \DD[2]$ in $\bcD^{\sfr}_{[1,2]}$.}
  \label{fig5}
\end{figure}
\begin{figure}[H]
  \includegraphics[width=\linewidth, trim={0 {3.65in} {2in} {.3in}}, clip]{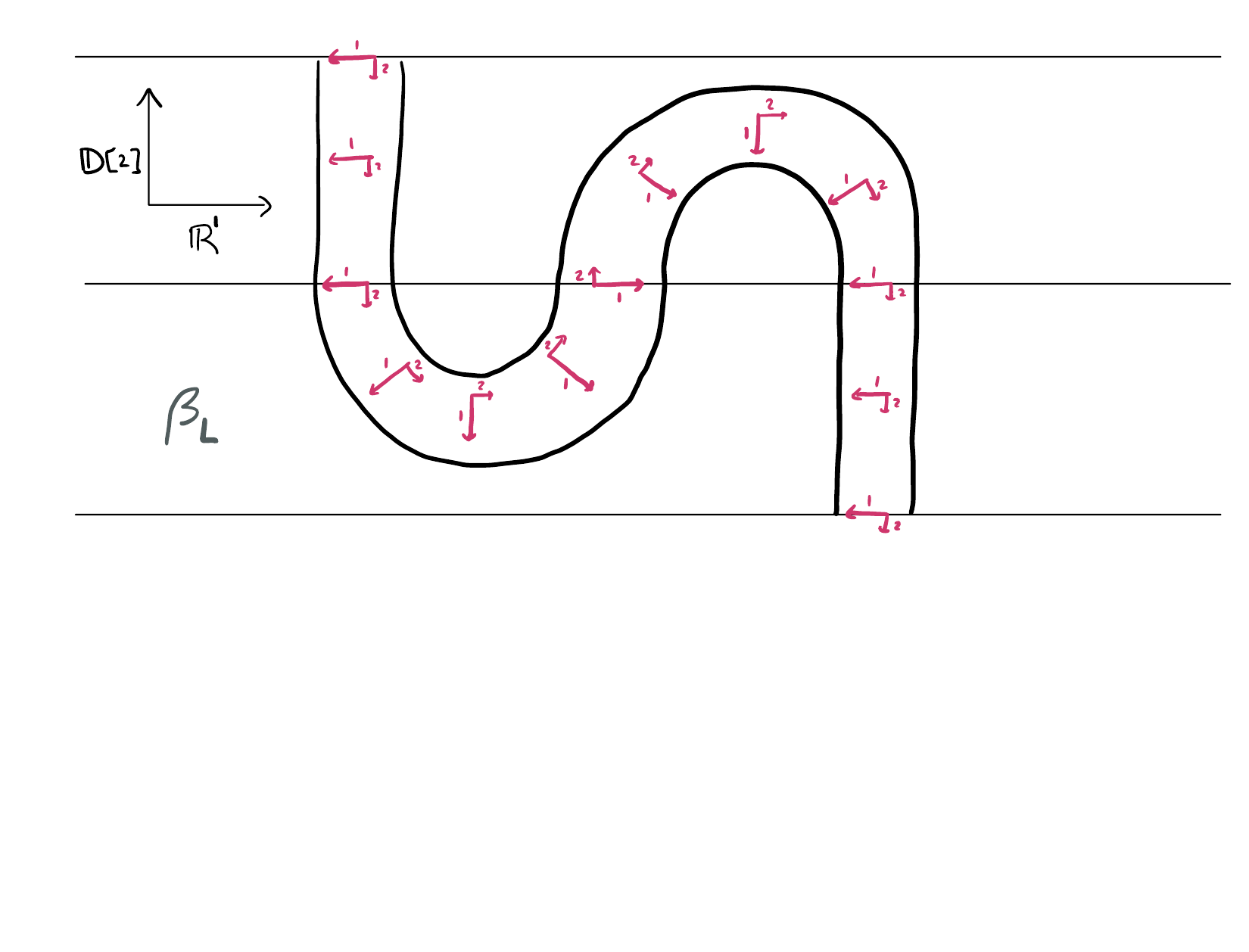}
  \caption{The open embedding $\beta'_L$.  
  It, together with the indicated framing, defines the morphism $\RR^1 \xra{\beta_L} \RR^1 \times \DD[2]$ in $\bcD^{\sfr}_{[1,2]}$.}
  \label{fig8}
\end{figure}
\begin{figure}[H]
  \includegraphics[width=\linewidth, trim={0 {3.6in} {2in} {.35in}}, clip]{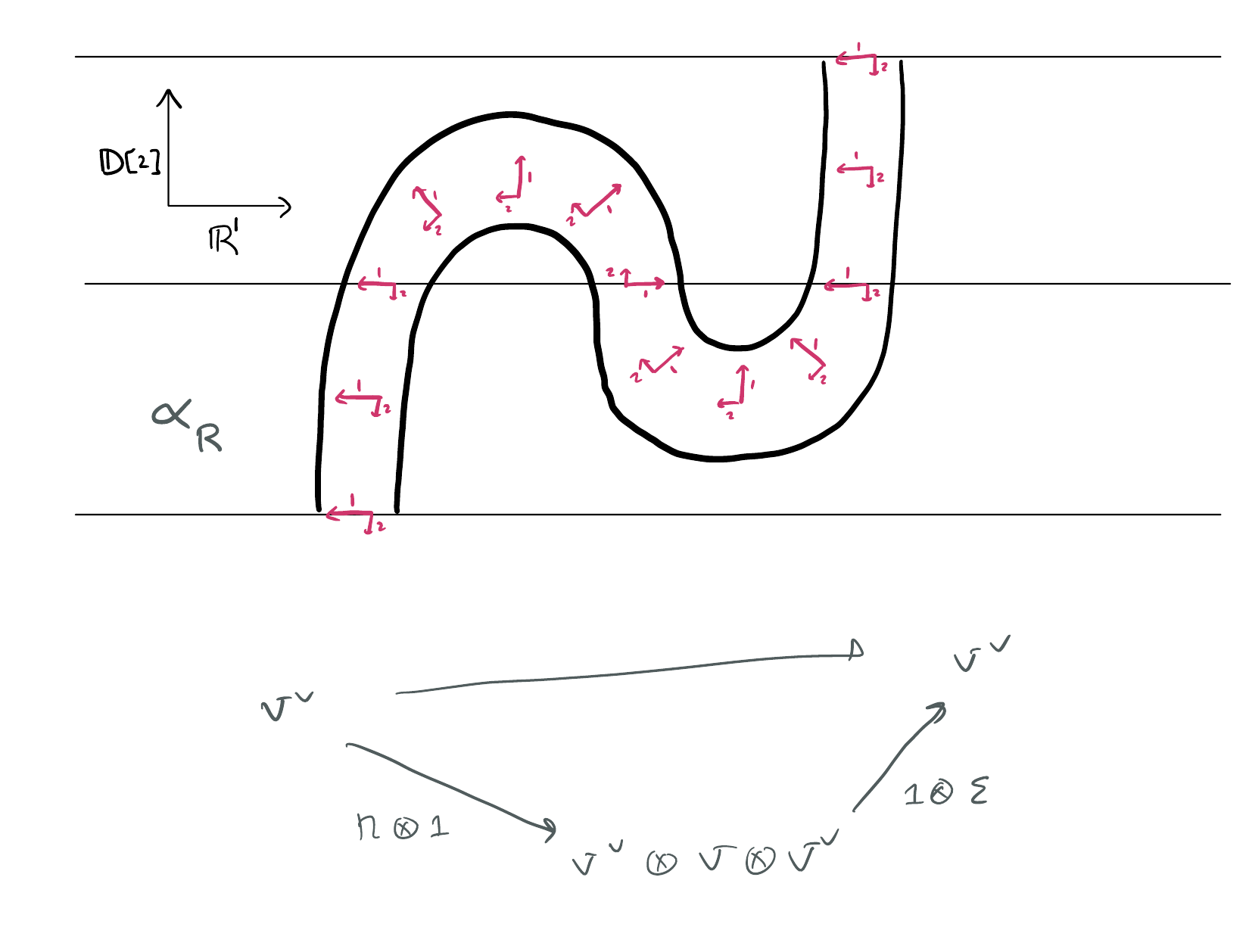}
  \caption{The open embedding $\alpha'_R$.  
  It, together with the indicated framing, defines the morphism $\RR^1 \xra{\alpha_R} \RR^1 \times \DD[2]$ in $\bcD^{\sfr}_{[1,2]}$.}
  \label{fig6}
\end{figure}
\begin{figure}[H]
  \includegraphics[width=\linewidth, trim={0 {3.7in} {2in} {.35in}}, clip]{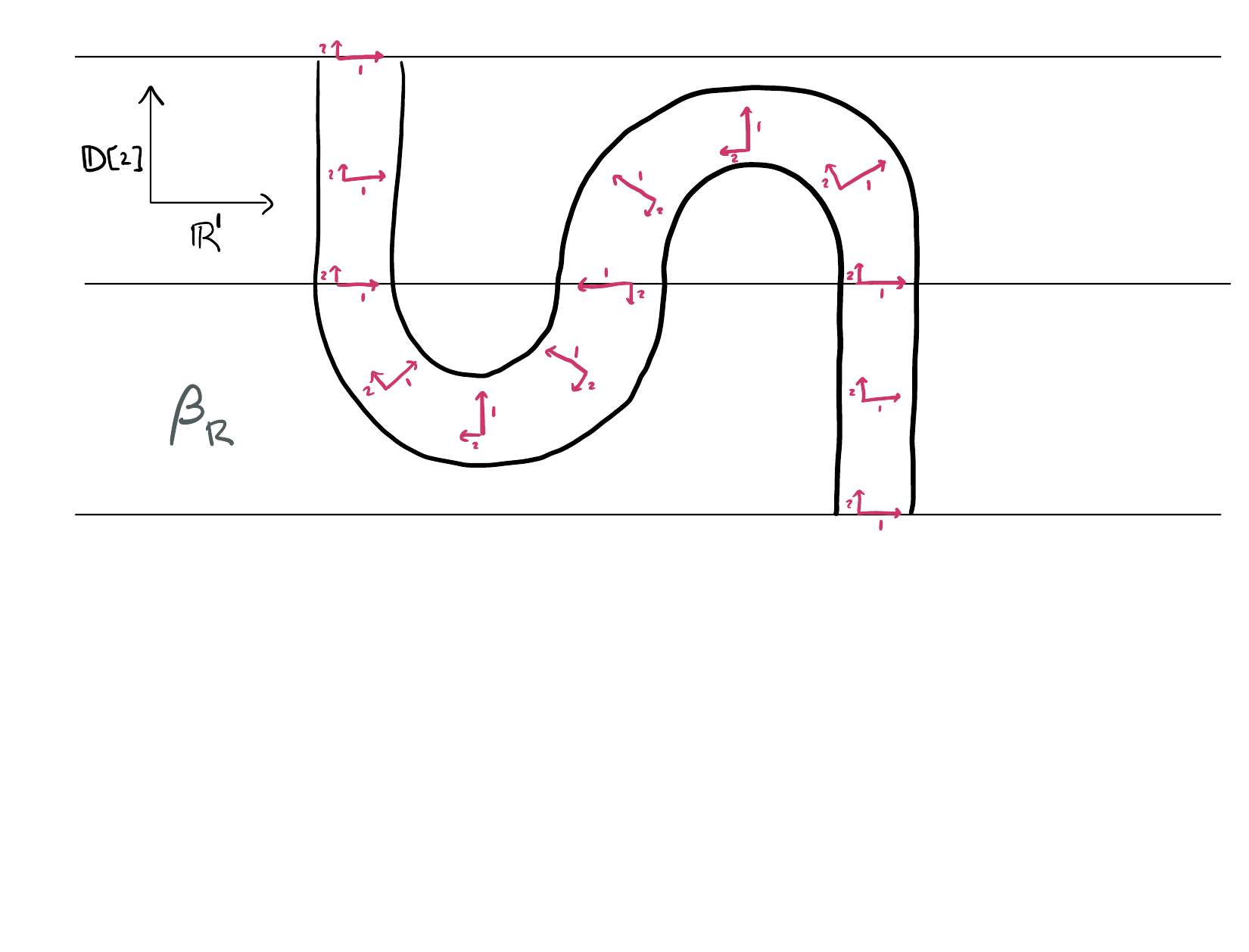}
  \caption{The open embedding $\beta'_R$.  
  It, together with the indicated framing, defines the morphism $\RR^1 \xra{\beta_R} \RR^1 \times \DD[2]$ in $\bcD^{\sfr}_{[1,2]}$.}
  \label{fig7}
\end{figure}
Inspecting Figures~\ref{fig5}--\ref{fig8} readily verifies the following.

\begin{observation}
\label{t121}
The morphisms constructed above organize as commutative diagrams in $\bcD_{[1,2]}^{\sfr}$:
\[
\xymatrix{
\RR^1 \times \DD^1
\sqcup
\RR^1 \times \DD^1
\ar[d]_-{\mu}
&
&
\RR^1
\ar[ll]_-{( \eta_L , u )}
\ar[rr]^-{( u , \epsilon_L)}
\ar[dl]_-{\alpha_L}
\ar[dd]_(.3){u}
&
&
\RR^1 \times \DD^1
\sqcup
\RR^1 \times \DD^1
\ar[d]^-{\mu}
\\
\RR^1 \times \DD^1
&
\RR^1 \times \DD[2]
\ar[l]_-{\partial_{\{0<1\}}}
\ar[rrr]^-{\partial_{\{1<2\}}}
\ar[dr]_-{\partial_{\{0<2\}}}
&
&
&
\RR^1 \times \DD^1
\\
&
&
\RR^1 \times \DD^1
&
&
,
}
\]
\[
\xymatrix{
\RR^1 \times \DD^1
\sqcup
\RR^1 \times \DD^1
\ar[d]_-{\mu}
&
\RR^1
\sqcup
\RR^1
\ar[l]_-{ u \sqcup \eta_L }
&
\RR^1
\ar[l]_-{( L ,  \id )}
\ar[r]^-{( \id , L )}
\ar[dl]_-{\beta_L}
\ar[dd]_(.3){u}
&
\RR^1
\sqcup
\RR^1
\ar[r]^-{ \epsilon_L \sqcup u }
&
\RR^1 \times \DD^1
\sqcup
\RR^1 \times \DD^1
\ar[d]^-{\mu}
\\
\RR^1 \times \DD^1
&
\RR^1 \times \DD[2]
\ar[l]_-{\partial_{\{0<1\}}}
\ar[rrr]^-{\partial_{\{1<2\}}}
\ar[dr]_-{\partial_{\{0<2\}}}
&
&
&
\RR^1 \times \DD^1
\\
&
&
\RR^1 \times \DD^1
&
&
,
}
\]
and
\[
\xymatrix{
\RR^1 \times \DD^1
\sqcup
\RR^1 \times \DD^1
\ar[d]_-{\mu}
&
\RR^1
\sqcup
\RR^1
\ar[l]_-{ \eta_R \sqcup u }
&
\RR^1
\ar[l]_-{( \id , R )}
\ar[r]^-{(R , \id  )}
\ar[dl]_-{\alpha_R}
\ar[dd]_(.3){u}
&
\RR^1
\sqcup
\RR^1
\ar[r]^-{u \sqcup \epsilon_R  }
&
\RR^1 \times \DD^1
\sqcup
\RR^1 \times \DD^1
\ar[d]^-{\mu}
\\
\RR^1 \times \DD^1
&
\RR^1 \times \DD[2]
\ar[l]_-{\partial_{\{0<1\}}}
\ar[rrr]^-{\partial_{\{1<2\}}}
\ar[dr]_-{\partial_{\{0<2\}}}
&
&
&
\RR^1 \times \DD^1
\\
&
&
\RR^1 \times \DD^1
&
&
,
}
\]
\[
\xymatrix{
\RR^1 \times \DD^1
\sqcup
\RR^1 \times \DD^1
\ar[d]_-{\mu}
&
&
\RR^1
\ar[ll]_-{( u, \eta_R )}
\ar[rr]^-{( \epsilon_R , u )}
\ar[dl]_-{\beta_R}
\ar[dd]_(.3){u}
&
&
\RR^1 \times \DD^1
\sqcup
\RR^1 \times \DD^1
\ar[d]^-{\mu}
\\
\RR^1 \times \DD^1
&
\RR^1 \times \DD[2]
\ar[l]_-{\partial_{\{0<1\}}}
\ar[rrr]^-{\partial_{\{1<2\}}}
\ar[dr]_-{\partial_{\{0<2\}}}
&
&
&
\RR^1 \times \DD^1
\\
&
&
\RR^1 \times \DD^1
&
&
.
}
\]

\end{observation}

\begin{proof}[Proof of Lemma~\ref{t7}]
Taking products with $\RR^{\mathit{n}\text{-}2}$ defines downward functors in a commutative square among $(\oo,1)$-categories:
\[
\xymatrix{
\bcD_{[1]}^{\sfr}
\times
\bDelta^{\op}
\ar[d]_-{(\RR^{\mathit{n}\text{-}2} \times -)}
\ar[rr]^-{\rho}
&&
\bcD_{[\n1,n]}^{\sfr}
\ar[d]^-{(\RR^{\mathit{n}\text{-}2} \times -)}
\\
\bcD_{[\n1]}^{\sfr}
\times
\bDelta^{\op}
\ar[rr]^-{\rho}
&&
\bcD_{[\n1,n]}^{\sfr}
.
}
\]
Notice that the downward functors carry closed covers to closed covers.
By Observation~\ref{t.closed.closed}, all of the functors in this diagram carry closed covers to closed covers.  
Therefore, applying $\Fun(-,\Spaces)$ to this commutative square results in a commutative square on the left in this diagram among $(\oo,1)$-categories:
\[
\xymatrix{
\cShv(\bcD_{[\n1,n]}^{\sfr} )
\ar[d]_-{(\RR^{\mathit{n}\text{-}2} \times -)^\ast}
\ar[rr]^-{\rho^\ast}
&&
\Alg_{\n1}( \fCat_{(\oo,1)})
\ar[d]^-{(\RR^{\mathit{n}\text{-}2} \times -)^\ast}
\ar[rr]^-{(-)^{\wedge}_{\sf unv}}
&&
\Alg_{\n1}( \Cat_{(\oo,1)})
\ar[d]^-{(\RR^{\mathit{n}\text{-}2} \times -)^\ast}
\\
\cShv(\bcD_{[\n1,n]}^{\sfr} )
\ar[rr]^-{\rho^\ast}
&&
\Alg( \fCat_{(\oo,1)})
\ar[rr]^-{(-)^{\wedge}_{\sf unv}}
&&
\Alg( \Cat_{(\oo,1)})
.
}
\]
The universal property of univalent-completions supplies a 2-cell in the right square witnessing it as lax-commuting.
Because univalent-completion $\fCat_{(\infty,1)} \to \Cat_{(\oo,1)}$ preserves products, this 2-cell is invertible, and therefore the right square canonically commutes.
Observation~\ref{t22} implies the diagram among $(\oo,1)$-categories
\[
\xymatrix{
\Alg^{\sf rig}_{\n1}( \Cat_{(\oo,1)})
\ar[d]_-{(\RR^{\mathit{n}\text{-}2} \times -)^\ast}
\ar[rr]^-{\rm inclusion}
&&
\Alg_{\n1}( \Cat_{(\oo,1)})
\ar[d]^-{(\RR^{\mathit{n}\text{-}2} \times -)^\ast}
\\
\Alg^{\sf rig}( \Cat_{(\oo,1)})
\ar[rr]^-{\rm inclusion}
&&
\Alg( \Cat_{(\oo,1)})
}
\]
is a pullback.  
Therefore, the statement of the lemma is true if and only if it is true for $n=2$.  
We are therefore reduced to the case in which $n=2$.

Next, let $(\bcD_{[1,2]}^{\sfr} \xra{\cF} \Spaces) \in \cShv(\bcD_{[1,2]}^{\sfr})$.
Denote the monoidal flagged $(\infty,1)$-category $\cR:= \rho^\ast \cF$.
Observe that $\cF$ carries 
\begin{enumerate}
\item
the morphisms in $\bcD_{[1,2]}^{\sfr}$,
\[
\RR^1 \xra{~L~} \RR^1_{\sf rev}
\qquad\text{ and }\qquad
\RR^1 \xra{~R~} \RR^1_{\sf rev}
~,
\]
to the maps between spaces
\[
\Obj(\cR) \xra{~L~} \Obj(\cR)
\qquad
\text{ and }
\qquad
\Obj(\cR) \xra{~R~} \Obj(\cR)
~,
\]
which are the data of Lemma~\ref{t120}(1);

\item
the diagrams in $\bcD_{[1,2]}^{\sfr}$ of Observation~\ref{t120} to the diagrams of Lemma~\ref{t105}(2);

\item
the diagrams in $\bcD_{[1,2]}^{\sfr}$ of Observation~\ref{t121} to the diagrams of Lemma~\ref{t105}(3).

\end{enumerate}
Therefore, by Lemma~\ref{t105}, the univalent-completion of the monoidal flagged $(\infty,1)$-category $\cR:=\rho^\ast \cF$ is rigid.

\end{proof}

\subsection{Tangle categories}

We now define $\Bord^{\fr}_1(\RR^{\n1})$, the $\cE_{\n1}$-monoidal flagged $(\oo,1)$-category of principal interest in this paper.
We do this through Lemma~\ref{t7}, by defining a functor
\[
\bcD^{\sfr}_{[\n1]}\times\bDelta^{\op}
\longrightarrow \Spaces
\]
that carries closed covers to pullbacks.
In particular, its space of objects $\obj(\Bord^{\fr}_1(\RR^{\n1}))$ is the value of this functor on $(\RR^{\n1},[0])$, and its space of morphisms $\mor(\Bord^{\fr}_1(\RR^{\n1}))$ is the value on $(\RR^{\n1}, [1])$.

\begin{definition}
\label{def.bord}
$\Bord_1^{\fr}(\RR^{\n1})$ is the composite functor
\[
\xymatrix{
\Bord_1^{\fr}(\RR^{\n1})
\colon 
\bcD^{\sfr}_{[\n1]}\times\bDelta^{\op}\ar[r]^-\rho & \bcD^{\sfr}_{[\n1,n]}\ar[rrr]^-{\Map_{ \bcD_{[\n1,n]}^{\sfr} }(\RR^{\n1},-)}&&&\Spaces
 ~.
}
\] 
\end{definition}

We have the following corollary of Lemma~\ref{t7}.

\begin{cor}\label{cor.Bord.has.duals}
$ \Bord_1^{\fr}(\RR^{\n1})$ is an $\cE_{\n1}$-monoidal flagged $(\oo,1)$-category. The univalent-completion $ \Bord_1^{\fr}(\RR^{\n1})^{\wedge}_{\sf unv}$ is a rigid $\cE_{\n1}$-monoidal $(\oo,1)$-category.
\end{cor}
\begin{proof}
Recall that closed covers are, in particular, limit diagrams in $\bcD^{\sfr}_{[\n1,n]}$.
Therefore, the second functor in the definition of $\Bord_1^{\fr}(\RR^{\n1})$ is a closed sheaf on $\bcD^{\sfr}_{[\n1,n]}$.
The result follows from Lemma~\ref{t7}.

\end{proof}

\begin{lemma}\label{lemma.corep.tang}
Let $X$ be a solidly $n$-framed $[\n1,n]$-manifold.
The space of morphisms in $\Mfd_n^{\sfr}$ from $\RR^{\n1}$ to $X$ is the moduli space of framed tangles in $X$. 
In particular, for $X = \RR^{\n1}\times\DD[p]$ and $p\geq 1$, and for $X = \RR^{\n1}$, we have equivalences
\[
\Map_{ \Mfd_n^{\sfr} }\bigl( \RR^{\n1} , \RR^{\n1}\times\DD[p] \bigr)
\simeq
\Tang_1^{\fr}(\RR^{\n1}\times\DD[p])
\]
and
\[
\Map_{ \Mfd_n^{\sfr} }\bigl( \RR^{\n1} , \RR^{\n1} \bigr)
\simeq
\Tang_0^{\fr}(\RR^{\n1})~.
\]
\end{lemma}

\begin{proof}
We denote by $\Mfd_{[\n1,n]}^{\sfr}\subset \Mfd_n^{\sfr}$ the full $(\oo,1)$-subcategory of $[\n1,n]$-manifolds. Composition in $\Mfd_{[\n1,n]}^{\sfr}$ defines a functor over $\Mfd_{[\n1,n]}^{\sfr}\times \Mfd_{[\n1,n]}^{\sfr}$:
\[
\xymatrix{
\Ar^{\sf cls.crt}(\Mfd_{[\n1,n]}^{\sfr})     \underset{\Mfd_{[\n1,n]}^{\sfr}} \times \Ar^{\sf emb}(\Mfd_{[\n1,n]}^{\sfr})     \ar[rd]_-{\ev_s\times \ev_t}   \ar[rr]^-{\sf comp}
&&
\Ar(\Mfd_{[\n1,n]}^{\sfr})      \ar[dl]^-{\ev_s\times \ev_t}
\\
&
\Mfd_{[\n1,n]}^{\sfr}\times \Mfd_{[\n1,n]}^{\sfr}
&
.
}
\]
where $\Ar^{\sf cls.crt}(\Mfd_{[\n1,n]}^{\sfr})$ is the full $\oo$-subcategory of $\Ar(\Mfd_{[\n1,n]}^{\sfr})$ whose objects are the closed-creation morphisms, and $\Ar^{\sf emb}(\Mfd_{[\n1,n]}^{\sfr})$ is the full $\oo$-subcategory of whose objects are the embeddings.
For each object $M\in \Mfd_{[\n1,n]}^{\sfr}$, the base change of this diagram along
\[
\Mfd_{[\n1,n]}^{\sfr} \xra{(\{M\} , {\sf id})} \Mfd_{[\n1,n]}^{\sfr} \times \Mfd_{[\n1,n]}^{\sfr}
\]
gives a functor over $\Mfd_{[\n1,n]}^{\sfr}$:
\begin{equation}\label{3}
\xymatrix{
(\Mfd_{[\n1,n]}^{\sfr})^{M/^{\sf cls.crt}}    \underset{\Mfd_{[\n1,n]}^{\sfr}} \times \Ar^{\sf emb}(\Mfd_{[\n1,n]}^{\sfr})    \ar[rd]_-{\ev_t}   \ar[rr]^-{{\sf comp}}
&&
(\Mfd_{[\n1,n]}^{\sfr})^{M/}    \ar[dl]^-{ \ev_t}
\\
&
\Mfd_{[\n1,n]}^{\sfr}
&
.
}
\end{equation}

The statement of the present lemma is implied by the following stronger assertion:
\begin{itemize}
\item[$\dagger$]
The composition functor in~(\ref{3}) is an equivalence between $(\oo,1)$-categories if $M$ is a smooth manifold equipped with a solid $n$-framing.
In particular, each morphism $M\to X$ in $\Mfd_{[\n1,n]}^{\sfr}$ admits a unique factorization
\[
M\xra{~\sf cls.crt~} E \xra{~\emb~} X
\]
as the reversed mapping cylinder on a proper constructible bundle followed by a stratified open embedding morphism.  
\end{itemize}
We first show that this assertion applied to the case $M=\RR^{\n1}$ gives the desired result. 
By intersecting tangles with closed unions of strata, the assignment $X \mapsto \Tang_1^{\fr}(X)$ is functorial with respect to closed morphisms in the argument $X$.
Consequently, this assignment carries closed covers to limits.
Meanwhile, because closed covers are, in particular, limit diagrams in $\Mfd_{[\n1,n]}^{\sfr}$, the functor $X\mapsto \Map_{\Mfd_{[\n1,n]}^{\sfr}}(\RR^{\n1},X)$ also carries closed covers to limits.  
Now, by definition of $\Mfd_{[\n1,n]}^{\sfr}$, every object therein admits a closed cover by $n$-manifold with boundary and $(\n1)$-manifolds without boundary.
We now prove the case where $X$ is a solidly $n$-framed $n$-manifold with boundary; the case of $X$ an $(\n1)$-manifold is nearly identical, but easier, and we leave the details to the reader.

Consider first the case in which $X$ is an $n$-dimensional manifold with boundary. 
The assertion then gives an equivalence
\[
\Map_{\Mfd_{[\n1,n]}^{\sfr}}(\RR^{\n1},X)
\simeq
\obj\Bigl(\Mfd_{[\n1,n]}^{\sfr})^{\RR^{\n1}/^{\sf cls.crt}}    \underset{\Mfd_{[\n1,n]}^{\sfr}} \times \bcD^{\sfr}_{[\n1,n]/^{\sf emb}X} 
\Bigr)~,
\]
the latter space being the fiber over $X$ of the functor 
\[
(\Mfd_{[\n1,n]}^{\sfr})^{\RR^{\n1}/^{\sf cls.crt}}    \underset{\Mfd_{[\n1,n]}^{\sfr}} \times \Ar^{\sf emb}(\Mfd_{[\n1,n]}^{\sfr})
\xra{\ev_t}
 \Mfd_{[\n1,n]}^{\sfr}~.
\]
We can now observe that there is a forgetful map
\begin{equation}
\label{reportW}
\obj\Bigl((\Mfd_{[\n1,n]}^{\sfr})^{\RR^{\n1}/^{\sf cls.crt}}    \underset{\Mfd_{[\n1,n]}^{\sfr}} \times \bcD^{\sfr}_{[\n1,n]/^{\sf emb}X} 
\Bigr)
\longrightarrow
\obj\Bigl(\Mfld_1^{\partial,\fr}\Bigr)
\end{equation}
to the moduli space of framed 1-manifolds with boundary. 
To see this, let $\RR^{\n1}\xra{\sf cls.crt} E \xra{\sf emb} X$ be a closed-creation morphism followed by an embedding morphism.
The first morphism is, by definition of closed-creation morphisms in $\Mfd_{[\n1,n]}^{\sfr}$, the datum of a proper fiber bundle $E \to \RR^{\n1}$ that respects solid $n$-framings.
The second morphism is, by definition of embedding morphisms in $\Mfd_{[\n1,n]}^{\sfr}$, an open embedding between manifolds with boundary. 
Denote by $W$ the fiber over $0\in \RR^{\n1}$ of this proper fiber bundle -- necessarily it is a compact manifold with boundary, and it has dimension-1 because $X$ is $n$-dimensional.
Using that the fiber bundle $E \to \RR^{\n1}$ respects solid $n$-framings, $W$ inherits the structure of a framing as a compact 1-manifold with boundary.
The map~(\ref{reportW}) assigns this compact framed 1-manifold $W$ to $\RR^{\n1}\xra{\sf cls.crt} E \xra{\sf emb} X$.

Let $W$ be a compact framed 1-manifold with boundary.
From the definition of~(\ref{reportW}), we identify its fiber over $W$ as a space of framed embeddings: 
\[
\xymatrix{
\obj\Bigl((\Mfd_{[\n1,n]}^{\sfr})^{\RR^{\n1}/^{\sf cls.crt}}    \underset{\Mfd_{[\n1,n]}^{\sfr}} \times \bcD^{\sfr}_{[\n1,n]/^{\sf emb}X} 
\Bigr)
\ar[d]_-{(\ref{reportW})}
&
\Emb^{\fr}(\RR^{\n1}\times W, X)\ar[d]\ar[l]\\
\obj\Bigl(\Mfld_1^{\partial,\fr}\Bigr)
&\ar[l]
\{W\}~.
}
\]
The associated action of the loop space $\Omega_{\{W\}}\bigl(\obj(\Mfld_1^{\partial,\fr})\bigl)$ on the fiber is then the natural action of the framed diffeomorphisms of $W$ on the space of embeddings:
\[
\Omega_{\{W\}}\bigl(\obj(\Mfld_1^{\partial,\fr})\bigl)\simeq \Diff^{\fr}(W)
\circlearrowright
\Emb^{\fr}(\RR^{\n1}\times W, X)
~.\]
This implies the equivalence
\[
\Map_{\Mfd_{[\n1,n]}^{\sfr}}(\RR^{\n1},X)
\simeq
\coprod_{[W]}\Emb^{\fr}(\RR^{\n1}\times W, X)_{/\Diff^{\fr}(W)}
\]
where $[W]$ ranges over isomorphism classes of framed 1-manifolds with boundary.

We now give a matching identification of $\Tang^{\fr}_1(X)$. There is likewise a forgetful map
\[
\Tang^{\fr}_1(X)
\longrightarrow 
\obj\bigl(\Mfld_1^{\partial,\fr}\bigr)
~,\qquad
(W \subset X)
\longmapsto 
W
~,
\]
since a trivialization of the Gauss map associated to a tangle $W\subset X$ defines a framing of the 1-manifold $W$. The fiber of this forgetful map is equivalent to the homotopy fiber of the (oriented lift of the) Gauss map
\[
\Emb(W, X)
\longrightarrow
\Map(W, \sV_1(n))
~,\qquad
(W \xra{f} X) \longmapsto (W \xra{{\sf Gauss}_f} \sV_1(n))
~,
\]
which is equivalent to the space of framed embeddings
\[
\Emb^{\fr}(\RR^{\n1}\times W, X)\simeq {\sf fiber}_\gamma\Bigl(\Emb(W, X)\ra
\Map(W, V_1(n))\Bigr)
\]
where $\gamma$ is the framing of $W$. The loop space $\Omega_{\{W\}}\obj\bigl(\Mfld_1^{\partial,\fr}\bigr) \simeq \Diff^{\fr}(W)$ again acts by the canonical action. This gives an identification
\[
\Tang^{\fr}_1(X)
\simeq
\coprod_{[W]}\Emb^{\fr}(\RR^{\n1}\times W, X)_{/\Diff^{\fr}(W)}
\]
and thus proves the equivalence
\[
\Tang^{\fr}_1(X)
\simeq
\Map_{\Mfd_{[\n1,n]}^{\sfr}}(\RR^{\n1},X)
\]
in the case of $X\in\Mfd_{[\n1,n]}^{\sfr}$ an $n$-dimensional manifold with boundary. 

The second case, of $X$ an $(\n1)$-manifold without boundary, follows mutatis mutandis, replacing the target of the forgetful maps with the moduli space of compact 0-dimensional manifolds. 
We leave the details to the reader.

We now prove the assertion $(\dagger)$, which will then complete the proof of the present lemma.

The composition functor~(\ref{3}) is an equivalence if and only if, for each $[p]\in \bDelta$ the map between the fibers of the map,
\[
\Map\left( [p] ,
(\Mfd_{[\n1,n]}^{\sfr})^{M/^{\sf cls.crt}} \underset{\Mfd_{[\n1,n]}^{\sfr}}\times \Ar^{\emb}(\Mfd_{[\n1,n]}^{\sfr})
\right)
\longrightarrow
\Map\left([p] , (\Mfd_{[\n1,n]}^{\sfr})^{M/} \right)~,
\]
is a weak homotopy equivalence.
That is, it suffices to show that for each choice of basepoint in the domain of this map, and each $q\geq 0$, 
the map between sets of connected components of based maps from the $q$-sphere,
\[
\Map_{\ast}\Bigl(  S^q ,
\Map\bigl( [p] ,
(\Mfd_{[\n1,n]}^{\sfr})^{M/^{\sf cls.crt}} \underset{\Mfd_{[\n1,n]}^{\sfr}}\times \Ar^{\emb}(\Mfd_{[\n1,n]}^{\sfr})
\bigr)
\Bigr)
\longrightarrow
\Map_{\ast}\Bigl(  S^q ,
\Map\bigl([p] , (\Mfd_{[\n1,n]}^{\sfr})^{M/} \bigr)
\Bigr)~,
\]
is bijective on $\pi_0$.
We show below that the map between connected components of non-based maps
\begin{equation}\label{333}
\Map\Bigl(  S^q ,
\Map\bigl( [p] ,
(\Mfd_{[\n1,n]}^{\sfr})^{M/^{\sf cls.crt}} \underset{\Mfd_{[\n1,n]}^{\sfr}}\times \Ar^{\emb}(\Mfd_{[\n1,n]}^{\sfr})
\bigr)
\Bigr)
\longrightarrow
\Map\Bigl(  S^q ,
\Map\bigl([p] , (\Mfd_{[\n1,n]}^{\sfr})^{M/} \bigr)
\Bigr)
~,
\end{equation}
is bijective on $\pi_0$.
The case of based maps follows nearly identically, and we leave the modification to the reader.
By adjunction, $\pi_0$-bijectivity of~(\ref{333}) is implied by bijectivity of the map between sets of connected components of spaces of functors,
\begin{equation}\label{334}
\Map\Bigl( \exit(K) ,
(\Mfd_{[\n1,n]}^{\sfr})^{M/^{\sf cls.crt}} \underset{\Mfd_{[\n1,n]}^{\sfr}}\times \Ar^{\emb}(\Mfd_{[\n1,n]}^{\sfr})
\Bigr)
\longrightarrow
\Map\Bigl( \exit(K) , (\Mfd_{[\n1,n]}^{\sfr})^{M/} \Bigr)~,
\end{equation}
for each compact stratified space of the form $K = \Delta^p \times S^q$.
Here $\Delta^p$ is the standardly stratified $p$-simplex, and $S^q$ is the unstratified $q$-sphere.
In this way, $\pi_0$-bijectivity of the map~(\ref{333}) is implied by $\pi_0$-bijectivity of the map~(\ref{334}) for each compact stratified space $K$.
This is to equivalent to showing that for each functor $\exit(K) \ra (\Mfd_{[\n1,n]}^{\sfr})^{M/}$, there exists a lift
\begin{equation}\label{lift}
\xymatrix{
&(\Mfd_{[\n1,n]}^{\sfr})^{M/^{\sf cls.crt}} \underset{\Mfd_{[\n1,n]}^{\sfr}}\times \Ar^{\emb}(\Mfd_{[\n1,n]}^{\sfr})\ar[d]\\
\exit(K)\ar@{-->}[ur]\ar[r]&(\Mfd_{[\n1,n]}^{\sfr})^{M/}}
\end{equation}
that is unique up to homotopy.

By adjunction, a functor $\exit(K)\ra (\Mfd_{[\n1,n]}^{\sfr})^{M/}$ is equivalent to a functor
\[
\exit(K)^{\tl}~\simeq ~\exit\bigl(\ov\sC(K)\bigr)\longrightarrow \Mfd_{[\n1,n]}^{\sfr}
\]
together with an identification of the value on the cone-point as $M$. 
Here, $\oC(K) := \{0\}   \underset{K\times \{0\}}\amalg  K\times [0,1]$ is the closed cone on $K$, regarded as a stratified space via the stratification $[0,1] \xra{t\mapsto \lceil t \rceil} [1]$.
The assumed compactness of $K$ implies compactness of $\ov{\sC}(K)$.
The latter classifies the following data:
\begin{itemize}
\item
a constructible bundle $X \xra{\pi} \oC(K)$;

\item
an $\sfr_n$-structure on $\pi$, by which we mean a lift 
$\sT_\pi \colon \exit(X) \to \Vect^{\sf inj}_{/\RR^n} \to \Vect^{\sf inj}$ of the fiberwise tangent classifier (see Remark~2.37 of~\cite{fact1});

\item
an equivalence $M \simeq X_0$ over $\Vect^{\sf inj}_{/\RR^n}$ between $M$ and $X_0$, the fiber of $X$ over $\{0\}\subset \oC(K)$.

\end{itemize}

So fix such data.  
A lift in~(\ref{lift}) is equivalent to the following data:
\begin{itemize}
\item a constructible bundle $\widetilde{X}\ra \oC(K\times \Delta^1)$ over the closed cone;

\item an $\sfr_n$-structure on the constructible bundle $\widetilde{X}\ra \oC(K\times \Delta^1)$;

\item an identification 
over $\Vect^{\sf inj}_{/\RR^n}$ of its restriction:
\[
\xymatrix{
X       \ar[d]      \ar@{=}[rr]
&&
\widetilde{X}_{|\oC(K\times \Delta^{\{1\}})}   \ar[d]    
\\
\oC(K)   \ar@{=}[rr]
&& 
\oC(K\times \Delta^{\{1\}}) ;
}
\]
\end{itemize}
and subject to the conditions that:
\begin{itemize}
\item 
for each point $k\in K$, the pulled-back constructible bundle 
\[
\widetilde{X}_{|\oC(\{k\}\times\Delta^{\{0\}})} 
\longrightarrow 
\oC\bigl(\{k\}\times\Delta^{\{0\}}\bigr) = \oC\bigl(\{k\}\bigr) \cong \Delta^1
\]
is classified by a closed-creation morphism in $\Mfd_{[\n1,n]}^{\sfr}$;
\item 
for each point $k\in K$, the pulled-back constructible bundle 
\[
\widetilde{X}_{|\{k\}\times \Delta^1} \longrightarrow \{k\}\times \Delta^1\times \{1\} = \Delta^1
\]
is classified by an embedding morphism in $\Mfd_{[\n1,n]}^{\sfr}$.
\end{itemize}
We first construct such a lift~(\ref{lift}). 
Take the stratified space $\Link_{X_0}(X)$, 
which is the link of the closed constructible subspace $X_0\subset X$. This link lies in a commutative diagram among stratified spaces:
\[
\xymatrix{
M=X_0 \ar[d]&\ar[d]\ar[l]_-{\sf p.cbl} \Link_{X_0}(X)\ar[r]^-{\sf open} &X_{|K}\ar[d]\\
\{0\}&\ar[l]
\Link_{\{0\}}\bigl(\oC(K)\bigr)   \ar[r]^-{\cong} 
& 
K .
}
\]
From the definition of links, the fact that $X_0=M$ is trivially stratified results in two further simplifications.
\begin{enumerate}
\item
The proper constructible bundle $\Link_{X_0}(X)\ra X_{|K\times\Delta^1}$ is, further, a {\it fiber bundle} of stratified spaces.
(Heuristically, the fibers vary continuously.)

\item
The open map $\Link_{X_0}(X)\ra X_{|K\times\Delta^1}$ is, further, an {\it open stratified embedding}.
(The refinement map onto its image is an isomorphism.)

\end{enumerate}
This allows us to now make use of the reversed and open mapping cylinders, as in Theorem~6.6.15 of \cite{strat}. The map $\Link_{X_0}(X)\ra M$ is proper constructible fiberwise over $K$, and taking the reversed mapping cylinder fiberwise over $K$ gives a constructible bundle
\[
{\sf Cylr}^{\sf fib}_K\Bigl(M \la   \Link_{X_0}(X) \Bigr) 
\longrightarrow
\oC(K)~.
\]
The map  $\Link_{X_0}(X)\ra X_{|K}$ is open fiberwise over $K$, and taking the open mapping cylinder fiberwise over $K$ gives a constructible bundle
\[
{\sf Cylo}^{\sf fib}_K\Bigl( \Link_{X_0}(X) \ra  X_{|K} \Bigr) 
\longrightarrow
K\times \Delta^1~ .
\]
These constructible bundles base change over $K = K\times \Delta^{\{0\}}$ identically.
Using that $\Bun$ is an $(\oo,1)$-category, this pair of concatenating constructible bundles extends to a constructible bundle 
\[
\w{X} \longrightarrow \oC(K\times \Delta^1)\supset \oC(K)\underset{K} \amalg K\times \Delta^1~.
\]
With point (2) above, 
the construction of this constructible bundle supplies the sought data of a lift~(\ref{lift}), sans $\sfr_n$-structures.
We now extend the given $\sfr_n$-structure on the constructible bundle $X\to \oC(K)$ to one on $\w{X}\to \oC(K\times \Delta^1)$.

Since the $\sfr_n$ tangential structure pulls back along open maps, point~(2) just above grants that there is a unique  $\sfr_n$-structure on the constructible bundle $\Link_{X_0}(X) \to K$ for which the open embedding of point~(2) is one between $\sfr_n$-structured constructible bundles.
Furthermore, the $\sfr_n$-structure on $X \to \oC(K)$ determines a $\sfr_n$-structure on ${\sf Cylr}\bigl(X_0 \la \Link_{X_0}(X)\bigr)  \to \oC(K)$ that restricts to the  $\sfr_n$-structure on $\Link_{X_0}(X) \to K$.
Composition in the $(\oo,1)$-category $\Mfd^{\sfr}_{n}$ endows this extension $\w{X} \to \oC(K\times \Delta^1)$ with a  $\sfr_n$-structure extending the given one over $\oC(K\times \Delta^{\{1\}})$.
We have exhibited a sought lift~(\ref{lift}).

We lastly show that every factorization is equivalent to that constructed above. 
Let $F\ra \oC(K\times\Delta^1)$ be a constructible bundle together a  $\sfr_n$-structure on it and identifications over $\oC(K)$, defining another such lift~(\ref{lift}).
This is classified by a diagram in $\Fun\bigl(\exit(K),\Mfd_{[\n1,n]}^{\sfr}\bigr)$,
\[
\xymatrix{
&&F_{|K\times\Delta^{\{0\}}}\ar[d]^{\emb}\\
X_0\ar[urr]^-{\sf cls.crt}\ar[rr]_-{\widetilde{X}}&&X_{|K} ,
}
\]
in which $X_0$ is understood
as the constant functor $\exit(K)\ra \Mfd_{[\n1,n]}^{\sfr}$ valued at $X_0$. 
By taking iterated links, we obtain a diagram in $\Fun\bigl(\exit(K),\Bun\bigr)$:
\[
\xymatrix{
&&\ar[d]_-{(3)}^-{\sf cls.crt}\Link_{X_0}(F_{|K\times\Delta^{\{0\}}})\ar[r]_-{(1)}^-{\sf open}&F_{|K\times\Delta^{\{0\}}}\ar[dd]^-{\emb}\\
X_0 \ar[drr]_-{\sf cls.crt}\ar[urr]^-{\sf cls.crt}&&\ar[d]_-{(2)}^-{\sf open}\Link_{\Link_{X_0}(F_{|K\times\Delta^{\{0\}}})}\bigl(\Link_{X_0}(F_{|K\times\Delta^1})\bigr)
\\
&&
\Link_{X_0}(X_{|K})      \ar[r]_-{(1)}^-{\sf open}
&
X_{|K}    .
}
\]
As argued above, this diagram admits a unique lift to one in $\Fun\bigl(\exit(K),\Mfd_{[\n1,n]}^{\sfr}\bigr)$, extending the given lift of $c_1 \to \Fun\bigl(\exit(K),\Mfd_{[\n1,n]}^{\sfr}\bigr)$ classifying $X\to \oC(K)$.
Also as argued above, using that each stratum of $M=X_0$ is both open and closed, the construction of links gives that the horizontal open morphisms labeled as (1) must be embedding morphisms.
Furthermore, the assumption that the morphism $F_{|K\times \Delta^{\{0\}}} \to X_{|K}$ in $\Fun\bigl(\exit(K),\Mfd_{[\n1,n]}^{\sfr})$ is by embedding morphisms gives that the top horizontal morphism is in fact an equivalence. 
Since the morphism (2) factors a closed-creation through a closed-creation, it too must be an equivalence. 
This implies the closed-creation morphism (3) must be an equivalence.
We conclude an equivalence in $\Fun\bigl(\exit(K),\Mfd_{[\n1,n]}^{\sfr}\bigr)$,
\[
\Link_{X_0}(X_{|K})~\simeq~F_{|K\times \Delta^{\{0\}}}~,
\]
under the constant functor at $X_0$ and over the restriction of the given functor classifying $X_{|K}$.

\end{proof}

We now have the following explicit description of $\Bord_1^{\fr}(\RR^{\n1})$. 
Recall the notation $\FF_{\n1}(V)$ for the free $\cE_{\n1}$-algebra on an object $V$.

\begin{cor}\label{cor.obj.mor.Bord.free}
The object and morphism spaces of the $\cE_{\n1}$-monoidal flagged $(\oo,1)$-category $\Bord_1^{\fr}(\RR^{\n1})$ are
\label{cor.bord.is.tang}
\[
\Obj\bigl( \Bord_1^{\fr}(\RR^{\n1}) \bigr) \simeq \Tang_0^{\fr}(\RR^{\n1})\simeq \FF_{\n1}(\Omega \RR\PP^{\n1})
\]
and
\[
\Mor\bigl(  \Bord_1^{\fr}(\RR^{\n1}) \bigr) \simeq \Tang^{\fr}_1(\RR^{\n1}\times\DD[1])~.
\]

\end{cor}
\begin{proof}
Combining Definition~\ref{def.bord} and Lemma~\ref{lemma.corep.tang}, we have the equivalences
\[
\Obj\bigl( \Bord_1^{\fr}(\RR^{\n1}) \bigr)
:=
\Map_{ \Mfd_{n}^{\sfr} }\bigl( \RR^{\n1} , \RR^{\n1} \bigr)
\simeq
\Tang_0^{\fr}(\RR^{\n1})
\]
and
\[
\Mor\bigl(  \Bord_1^{\fr}(\RR^{\n1}) \bigr)
:=
\Map_{ \Mfd_{n}^{\sfr} }\bigl( \RR^{\n1} , \RR^{\n1}\times\DD[p] \bigr)
\simeq
\Tang_1^{\fr}(\RR^{\n1}\times\DD[p])~.
\]
It only remains to establish the equivalence
\[
\Tang_0^{\fr}(\RR^{\n1})
\simeq
\FF_{\n1}(\Omega \RR\PP^{\n1})~.
\]
To see this, recall from Remark~\ref{rem.framings.zero.mflds} that the space of $n$-framings on a finite subset $S \subset \RR^{\n1}$ is 
$(\Omega\Gr_1(n))^S = (\Omega\RR\PP^{\n1})^S$. 
Consequently, there is a canonical map $\FF_{\n1}(\Omega \RR\PP^{\n1})\ra \Tang_0^{\fr}(\RR^{\n1})$ which is immediately seen to be an equivalence.
\end{proof}

\subsection{Tangles in 2-dimensional space, and the categorical delooping $\fB\Bord_1^{\fr}(\RR^1)$}\label{sec.BBord1}

The case of the monoidal $(\oo,1)$-category $\Bord_1^{\fr}(\RR^1)$ is of particular importance in our argument, because the Tangle Hypothesis in 2-dimensional ambient space forms the base case of the induction on the ambient dimension in \S5.

\begin{prop}\label{prop.Bord.discrete.gaunt}
$\Bord_1^{\fr}(\RR^1)$ is a monoidal category. 
That is, the monoidal flagged $(\oo,1)$-category $\Bord_1^{\fr}(\RR^1)$ has discrete mapping spaces, and the natural map from its space of objects to its space of isomorphisms
\[
\obj\bigl(\Bord_1^{\fr}(\RR^1)\bigr)
\longrightarrow
\Bord_1^{\fr}(\RR^1)^\sim
\]
is an equivalence.
\end{prop}
\begin{proof}
Given 2-framed configurations $S_0\subset \RR^1$ and $S_1\subset \RR^1$, the space of morphisms $\Map_{\Bord_1^{\fr}(\RR^1)}(S_0, S_1)$ is the fiber over $S_0\amalg S_1\subset \RR^1\times \partial \DD[1]$ of the map
\[
\Tang_1^{\fr}(\RR^1\times\DD[1])\ra \Tang_0^{\fr}(\RR^1\times\partial\DD[1])~.
\]
To show that this fiber is discrete, it suffices to show that each of these spaces is discrete. This, in turn, follows from showing that both the spaces $\Emb^{\fr}(W, \RR^1\times\DD[1])_{/\Diff^{\fr}(W)}$ and $\Emb^{\fr}(S, \RR^1\times\partial\DD[1])_{/\Diff^{\fr}(S)}$ are discrete for all 2-framed 1-dimensional and 0-dimensional manifolds $W$ and $S$, respectively. Note first that for $W\cong S^1$, the space $\Emb^{\fr}(S^1, \RR^1\times\DD[1])$ is empty: The Gauss map for any embedding $S^1\hookrightarrow \RR^1\times\DD[1]$ is not nullhomotopic (being twice the generator of $\pi_1\Gr_1(2)\cong\ZZ$), so there does not exist a trivialization of this Gauss map. For $W\cong \DD^1$, it is a classical result 
(implied by a theorem of Smale~\cite{smale.2.sphere})
that the space $\Emb(\DD^1, \RR^1\times\DD[1])$ is discrete. The spaces of framings are likewise discrete, by the equivalence $\Omega \RR\PP^1\simeq \ZZ$.

We lastly show univalent-completeness. The inclusion of identity morphisms into isomorphisms is given by the map
\[
\Tang_0^{\fr}(\RR^1) \longrightarrow \Tang_1^{\fr}(\RR^1\times\DD[1])
\]
assigning a set $S\subset \RR^1$ to the straight-line tangles $S\times\DD[1]\subset \RR^1\times\DD[1]$. Again by discreteness of the space $\Emb(\DD^1, \RR^1\times\DD[1])$, it follows that this inclusion is an monomorphism and the essential image consists of untangles---which are the isomorphisms in $\Bord_1^{\fr}(\RR^1)$.
\end{proof}

Recall that given a monoidal $(\oo,1)$-category $\cC$, there exists a pointed $(\oo,2)$-category $\ast\in\fB\cC$, the categorical deloop of $\cC$. 
It is presented by the simplicial $(\infty,1)$-category
\[
{\sf Bar}_\bullet \cC
\colon
\bDelta^{\op}
\longrightarrow
\Cat_{(\infty,1)}
~.
\]
Delooping determines a functor between $(\infty,1)$-categories
\begin{equation}
\label{e.deloop}
\Alg(\Cat_{(\oo,1)}) \xra{~\fB~} \Cat_{(\infty,2)}^{\ast/}
\end{equation}
from monoidal $(\infty,1)$-categories to pointed $(\infty,2)$-categories. This functor is a fully-faithful with image consisting of those pointed $(\infty,2)$-categories $\ast \in \cC$ for which the morphism $\ast \to \cC$ is surjective on objects.
Furthermore, the deloop $\fB\cC$ is uniquely characterized by the following:
\begin{itemize}
\item the datum of an equivalence $\cC \simeq \Omega\fB\cC$ of monoidal $(\oo,1)$-categories, where $\Omega\fB\cC$ is the limit of the diagram $\ast\ra \fB\cC\la \ast$ with monoidal structure given by composition of loops;
\item the condition that the underlying $\oo$-groupoid category of $\fB\cC$ is connected.
\end{itemize}

We now give a convenient geometric description of $\fB\Bord_1^{\fr}(\RR^1)$, which we identify with the categorical deloop in Lemma~\ref{t.deloop}.
In~\S3.8 of~\cite{fact1}, we constructed a functor 
\[
\bTheta_n^{\op}
\xra{~\langle -\rangle~}
\Mfd^{\sfr}_n
\]
whose value on $T$ is the stratified space $\langle T\rangle$ depicting the pasting diagram of $T$ -- it is a naturally subspace of $\RR^n$, from which it inherits a solid n-framing. This allows for the following definition. Recall first that $\Mfd_n^{\sfr}$ is a pointed $(\oo,1)$-category, the zero-object of which is given by the empty set $\emptyset$.
\begin{definition}\label{def.BBord}
$\fB^{\n1}\Bord_1^{\fr}(\RR^{\n1})$ is the pointed presheaf on $\bTheta_n$ given by the composite functor
\[
\bTheta_n^{\op}
\xra{~\langle -\rangle~}
\Mfd^{\sfr}_n
\xrightarrow{~\Map_{\Mfd_n^{\sfr}}( \RR^{\n1} , -)~}
\Spaces_\ast
\]
whose distinguished point $\emptyset\in\fB^{\n1}\Bord_1^{\fr}(\RR^{\n1})(T)$ is given by the empty tangle, which is the composite $\RR^{\n1}\ra \emptyset\ra \langle T\rangle$ in $\Mfd_n^{\sfr}$.
\end{definition}

\begin{observation}\label{rem.T.pt.BBord}
By Lemma~\ref{lemma.corep.tang}, for $T\in \bTheta_n$, the space
\[
\fB^{\n1}\Bord_1^{\sf fr}(\RR^{\n1})(T)
~\simeq~
\left\{
~\left(~
W \subset \langle T\rangle ~,~ \varphi
~\right)~
\right\}
\]
is a moduli space of codimension-$(\n1)$ properly embedded tangles $W \subset \langle T\rangle$ equipped with a trivialization $\varphi$ of the Gauss map $W \xrightarrow{\tau_W} {\sf Gr}_1(n)$ defined by the solid $n$-framing of $\langle T\rangle\subset\RR^n$.

\end{observation}

\begin{prop}
$\fB\Bord_1^{\fr}(\RR^{1})$ is a pointed $(\oo,2)$-category. 
That is, the $\bTheta_2$-presheaf $\fB \Bord_1^{\fr}(\RR^{1})$ belongs to the full $\oo$-subcategory $\Cat_{(\infty,2)}^{\ast/}\subset \Fun(\bTheta_2^{\op},\Spaces_\ast)$
\end{prop}

\begin{proof}
The functor $\langle -\rangle$ sends Segal covers to closed covers, and $\Map_{\Mfd_2^{\sfr}}( \RR^{1} , -)$ is a closed sheaf. Consequently, the composite functor satisfies the Segal condition and therefore presents an $(\oo,2)$-category.
\end{proof}

\begin{lemma}
\label{t.deloop}
The pointed $(\oo,2)$-category $\fB\Bord_1^{\fr}(\RR^1)$ is the categorical deloop of the monoidal $(\oo,1)$-category $\Bord_1^{\fr}(\RR^1)$.
\end{lemma}
\begin{proof}
Consider the functor $\bcD^{\sfr}_1
\xra{~(-)^{\geq 1}~}
\bcD^{\sfr}_{[1]}$ given by $D\mapsto D^{\geq 1}:= D \smallsetminus \sk_0(D)$, deletion of the 0-dimensional strata.
The open embedding $D^{\geq 1} \hookrightarrow D$ for each object $D\in \bcD^{\sfr}_1$ determines a natural transformation $(-)^{\geq 1} \to \id$.
Consider the wreath functor $\bDelta \times \bDelta \xra{\wr} \bTheta_2$, which is given by $\left( [p] , [q] \right) \mapsto [p] \wr [q] := [p]\left( [q] , \dots , [q] \right)$.

Observe that the natural transformation $(-)^{\geq 1} \to \id$ above determines a lax-commutative diagram:
\[
\xymatrix{
\bDelta^{\op}\times\bDelta^{\op}\ar[r]^-\wr\ar[d]_-{\langle-\rangle^{\geq 1}\times\id}&\bTheta^{\op}_2\ar[d]^{\langle-\rangle}\\
\bcD_{[1]}^{\sfr}\times\bDelta^{\op}\ar[r]_-{\id\times\langle-\rangle}\ar@{=>}[ur]&\Mfd_2^{\sfr}
.
}
\]

Now, the emptyset is the only 1-tangle in a 0-dimensional manifold.
Consequently, for $X$ a solidly 2-framed stratified space, each tangle $W\subset X$ necessarily avoids the 0-dimensional strata of $X$:
\[
W
~\subset~
X^{\geq 1}
~:=~
X \smallsetminus \sk_0 (X)
~.
\]
The open embedding $X^{\geq 1} \hookrightarrow X$ thereby determines an equivalence between spaces of framed tangles:
\[
\Tang_1^{\fr}(X^{\geq 1}) \simeq \Map_{\Mfd_2^{\sfr}}(\RR^1, X^{\geq 1}) 
\xra{~\simeq~} 
\Map_{\Mfd_2^{\sfr}}(\RR^1, X)\simeq \Tang_1^{\fr}(X)
~.
\]
In particular, the canonical map
\[
\Bord_1^{\fr}(\RR^1)(\langle[p]\rangle^{\geq 1}\times\langle [q]\rangle)
\xra{~\simeq~}
\fB\Bord_1^{\fr}(\RR^1)(\langle[p]\wr [q]\rangle)
\]
is an equivalence.
In this way, we achieve commutativity of the outer diagram of space-valued functors on $\bDelta^{\op}\times\bDelta^{\op}$
\[
\xymatrix{
&\bTheta^{\op}_2\ar[dr]^{\langle-\rangle}\ar@/^1pc/[drrrr]^-{\fB\Bord_1^{\fr}(\RR^1)}\\
\bDelta^{\op}\times\bDelta^{\op}\ar[ur]^-\wr\ar[dr]_-{\langle-\rangle^{\geq 1}\times\id}&\Uparrow&\Mfd_2^{\sfr}\ar[rrr]^-{\Map_{\Mfd_2^{\sfr}}(\RR^1,-)}&&&\Spaces\\
&\bcD_{[1]}^{\sfr}\times\bDelta^{\op}\ar[ur]_-{\id\times\langle-\rangle}\ar@/_1pc/[urrrr]_-{\Bord_1^{\fr}(\RR^1)}
.
}
\]
Finally, we note that the bottom composite
\[
\Bord_1^{\fr}(\RR^1)\circ (\langle-\rangle^{\geq 1}\times\id) = {\sf Bar}_\bullet\Bord_1^{\fr}(\RR^1)
\]
is the bar construction of the monoidal $(\oo,1)$-category $\Bord_1^{\fr}(\RR^1)$. This identifies the top composite as the categorical deloop of $\Bord_1^{\fr}(\RR^1)$.
\end{proof}

Recall from \cite{barwick.schommer.pries} that an $(\oo, 2)$-category $\cC$ is gaunt if, for each $0 \leq k \leq 2$, the space of $k$-morphisms $\cC(c_k)$ is discrete.

\begin{prop}\label{prop.BBord.gaunt}
$\fB\Bord_1^{\fr}(\RR^1)$ is a gaunt $(\oo,2)$-category.
\end{prop}
\begin{proof}
We prove that $\fB \Bord_1^{\sf fr}(\RR^1)(c_k)$ is a discrete space for $0\leq k\leq 2$. 
By Observationn~\ref{rem.T.pt.BBord}, the space of $k$-morphisms of $\fB \Bord_1^{\sf fr}(\RR^1)$ can be expressed as a coproduct over diffeomorphism classes of $({\mathit{k}\text{-}1})$-manifolds with boundary:
\[
\fB \Bord_1^{\sf fr}(\RR^1)(c_k) \simeq
\coprod_{[W]} \Emb^{\fr}(W, \DD^k)_{/\Diff(W)}\simeq
\coprod_{[W]} {\sf fiber}_\ast\Bigl(\Emb(W, \DD^k)\ra
\Map(W, \Gr_1(2))\Bigr)_{/\Diff(W)}~.
\]
For the $[W]$-component to be nonempty, there must exist a trivialization of the Gauss map of an embedding $W\subset \DD^k$, which disallows the case of $W\cong S^1\subset \DD^2$. Consequently, $W$ is diffeomorphic to a finite disjoint union of disks $\DD^{{\mathit{k}\text{-}1}}$. 
Since $k\leq 2$, the space $\Emb(W,\DD^k)$ is discrete, and the space $\Map(W,\Gr_1(2))$ is a 1-type. Consequently, the fiber is discrete. Again since $k\leq 2$, the diffeomorphism group $\Diff(W)$ is equivalent to a symmetric group if $W$ is 0-dimensional, or a product of the symmetric group on $\pi_0W$ and $(\ZZ/2)^{\pi_0W}$ if $W$ is 1-dimensional. In particular, it is discrete. Further, by inspection the diffeomorphism group $\Diff(W) \simeq \Sigma_{\pi_0W}\times(\ZZ/2)^{\pi_0W}$ acts freely on the above fiber, which implies that ${\sf fiber}_\ast\bigl(\Emb(W, \DD^k)\ra
\Map(W, \Gr_1(2))\bigr)_{/\Diff(W)}$ is again discrete. This implies that $\fB \Bord_1^{\sf fr}(\RR^1)(c_k)$ is discrete, completing the proof.
\end{proof}

\section{Segaling}

In this section, we construct a functor $\Seg: \Fun(\bDelta^{\op},\cX) \ra \Fun(\bDelta^{\op},\cX)$ for $\cX$ a presentable $(\oo,1)$-category in which colimits are universal -- see Definition~\ref{def.Seg}. The main result of this section is Lemma~\ref{lemma.Seg.universal}, which states that if $\cC\in \Fun(\bDelta^{\op},\cX)$ is such that $\Seg(\cC)$ is Segal, then $\Seg(\cC)$ is the Segal completion of $\cC$.

In this section, we use the following.
\begin{notation}
\label{d.circ}
For $a \in \cA$ an object in an $(\infty,1)$-category, 
$a^\circ \in \cA^{\op}$ is the corresponding object in the opposite $(\infty,1)$-category. 

\end{notation}

\begin{definition}
The category $\Ar^{\cls}(\bDelta^{\op})$ is the full subcategory of $\Ar(\bDelta^{\op})$ whose objects are closed morphisms $A^\circ \ra C^\circ$, i.e., those morphisms such that $C\ra A$ is a convex inclusion.
\end{definition}

We forthwith regard $\Ar^{\cls}(\bDelta^{\op})$ as a category over $\bDelta^{\op}$ via the source-evaluation functor $\ev_s$.

\begin{lemma}\label{lemma.Arcls.coCart}
$\Ar^{\cls}(\bDelta^{\op})\xra{\ev_s}\bDelta^{\op}$ is a Cartesian fibration, with monodromy given by convex-hull of image. In particular, $\ev_s$ is an exponentiable fibration.
\end{lemma}
\begin{proof}
We define a functor $\bDelta_{/^{\cls}-}:\bDelta \ra \Cat_{(\oo,1)}$. This sends $A\in \bDelta$ to the category $\bDelta_{/^{\cls}A}$ of closed morphisms $C\subset A$. To a map $f:A \ra B$, we assign the covariant functor ${\sf Hull}(f):\bDelta_{/^{\cls}A} \ra \bDelta_{/^{\cls}B}$ which sends $C\subset A$ to
\[
{\sf Hull}(f)(C\subset A) := \bigl({\sf Hull}(f(C))\subset B\bigr)
\]
the convex hull of the image $f(C)\subset B$. 
The unstraightening of this functor, $\Un(\bDelta_{/^{\cls}-})\ra \bDelta$, is a coCartesian fibration, and its opposite $\Un(\bDelta_{/^{\cls}-})^{\op}\ra \bDelta^{\op}$ is a Cartesian fibration. We claim there is an equivalence $\Ar^{\cls}(\bDelta)\simeq \Un(\bDelta_{/^{\cls}-})$, which will prove the result after further noting the equivalence $\Ar^{\cls}(\bDelta)^{\op} \cong \Ar^{\cls}(\bDelta^{\op})$. By the factorization system on morphisms in an unstraightening, a morphism in $\Un(\bDelta_{/^{\cls}-})$ from $(C_1\subset A_1)$ to $(C_0\subset A_0)$ consists of a morphism $f:A_1 \ra A_0$ in $\bDelta$ and a morphism ${\sf Hull}(f)(C_1\subset A_1) \ra C_0$ in $\bDelta_{/^{\cls}A_0}$. However, there exists such a morphism ${\sf Hull}(f)(C_1\subset A_1) \ra C_0$ if and only if the image $f(C_1)\subset A_0$ is contained in the image $C_0\subset A_0$. Since $\bDelta_{/^{\cls}A_0}$ is equivalent to a poset, this morphism ${\sf Hull}(f)(C_1\subset A_1) \ra C_0$ is unique if it exists. Consequently, a morphism from $(C_1\subset A_1)$ to $(C_0\subset A_0)$ in  $\Un(\bDelta_{/^{\cls}-})$ is uniquely determined by a commutative diagram
\[
\xymatrix{
C_1\ar[r]\ar[d]_-{\cls}&C_0\ar[d]^-{\cls}\\
A_1\ar[r]&A_0}
\]
in $\bDelta$. This is exactly a morphism in $\Ar^{\cls}(\bDelta)$ from $(C_1\subset A_1)$ to $(C_0\subset A_0)$. 
\end{proof}

For $\cK$ an $(\infty,1)$-category, recall from~\S\ref{sec.notation}
the $(\infty,1)$-category $\TwAro(\cK)$.
Source-evaluation and target-evaluation define functors
\[
\cA
\xla{~\ev_s~}
\TwAro(\cA)
\xra{~\ev_t~}
\cA^{\op}
~.
\]
Consider the fiber product
\[
\TwAro(\bDelta^{\op})\underset{\bDelta^{\op}}\times \Ar^{\cls}(\bDelta^{\op})~.
\]
Note that an object of the fiber product consists of a pair $(A^\circ\ra B^\circ)\in \TwAro(\bDelta^{\op})$ and $(C\subset A)^\circ\in \Ar^{\cls}(\bDelta^{\op})$, where $C\subset A$ is a convex inclusion (or equivalently, $A^\circ \ra C^\circ$ is a closed morphism). A morphism
\[
\bigl(A_0^\circ\xra{\varphi_0} B_0^\circ, (C_0\subset A_0)^\circ\bigr) \longrightarrow \bigl(A_1^\circ\xra{\varphi_1} B_1^\circ, (C_1\subset A_1)^\circ\bigr)
\]
consists of a morphism in $\TwAro(\bDelta^{\op})$ given by a commutative diagram in $\bDelta^{\op}$
\[
\xymatrix{
A_0^\circ\ar[d]_-{\varphi_0}\ar[r]^-{f^\circ}&A_1^\circ\ar[d]^-{\varphi_1}\\
B_0^\circ&\ar[l]_-{g^\circ}B_1^\circ}
\]
and a morphism in $\Ar^{\cls}(\bDelta^{\op})$ given by a commutative diagram in $\bDelta^{\op}$
\[
\xymatrix{
A_0^\circ\ar[d]_-{\cls}\ar[r]^-{f^\circ}&A_1^\circ\ar[d]^-{\cls}\\
C_0^\circ\ar[r]&C_1^\circ}
\]
which is the same as a morphism $f:A_1\ra A_0$ in $\bDelta$ for which the image $f(C_1)$ is contained in the image $C_0\subset A_0$.

\begin{lemma}\label{lemma.S}
There exists a functor
\[
\TwAro(\bDelta^{\op})\underset{\bDelta^{\op}}\times \Ar^{\cls}(\bDelta^{\op})\overset{S}\longrightarrow \bDelta^{\op}
\]
uniquely determined by the following values:
\begin{itemize}
\item On objects, $S$ sends $\bigl(A^\circ\xra{\varphi} B^\circ, (C\subset A)^\circ\bigr)$ to the convex subset of $A$:
\[
C^\varphi:=\Bigl[{\sf sup}\bigl\{\varphi(b)\big| \ \varphi(b)\leq {\sf min}(C)\bigr\},{\sf inf}\bigl\{\varphi(b)\big| \ {\sf max}(C)\leq \varphi(b)\bigr\}\Bigr]~.
\]
Here we use the convention ${\sf sup}(\emptyset) := {\sf min}(A)$ and ${\sf inf}(\emptyset) := {\sf max}(A)$.
\item
On morphisms, $S$ sends a morphism given by the opposite of
\[
\xymatrix{
C_0\ar[d]_-{\cls}&\ar[l]C_1\ar[d]^-{\cls}\\
A_0&\ar[l]_-{f}A_1\\
B_0\ar[u]^-{\varphi_0}\ar[r]^-{g}&B_1\ar[u]_-{\varphi_1}}
\]
to (the opposite of) the map in $\bDelta$ given by the restriction of $f$:
\[
\xymatrix{
C_1^{\varphi_1}:=\Bigl[{\sf sup}\bigl\{\varphi_1(b_1)\big| \ \varphi_1(b_1)\leq {\sf min}(C_1)\bigr\}, {\sf inf}\bigl\{\varphi_1(b_1)\big| \ {\sf max}(C_1)\leq \varphi_1(b_1)\bigr\}\Bigr]
\ar[d]^-{f_|}\\
C_0^{\varphi_0}:=\Bigl[{\sf sup}\bigl\{\varphi_0(b_0)\big| \ \varphi_0(b_0)\leq {\sf min}(C_0)\bigr\}, {\sf inf}\bigl\{\varphi_0(b_0)\big| \ {\sf max}(C_0)\leq \varphi_0(b_0)\bigr\}\Bigr]}
\]
\end{itemize}
\end{lemma}
\begin{proof}
To show that $S$ is well-defined, it suffices to verify the inequalities
\[
f\Bigl({\sf sup}\bigl\{\varphi_1(b_1)\big| \ \varphi_1(b_1)\leq {\sf min}(C_1)\bigr\}\Bigr)
\geq
{\sf sup}\bigl\{\varphi_0(b_0)\big| \ \varphi_0(b_0)\leq {\sf min}(C_0)\bigr\}
\]
and
\[
f\Bigl({\sf inf}\bigl\{\varphi_1(b_1)\big| \ {\sf max}(C_1)\leq \varphi_1(b_1)\bigr\}\Bigr)
\leq
{\sf inf}\bigl\{\varphi_0(b_0)\big| \ {\sf max}(C_0)\leq \varphi_0(b_0)\bigr\}
~.
\]
We only verify the first inequality; the verification of the second is identical.
Note the inequalities
\[
f\Bigl({\sf sup}\bigl\{\varphi_1(b_1)\big| \ \varphi_1(b_1)\leq {\sf min}(C_1)\bigr\}\Bigr)
\geq
{\sf sup}\bigl\{f\varphi_1(b_1)\big| \ f\varphi_1(b_1)\leq {\sf min}(f(C_1))\bigr\}
\geq
{\sf sup}\bigl\{f\varphi_1(b_1)\big| \ f\varphi_1(b_1)\leq {\sf min}(C_0)\bigr\}
\]
where the first follows from $f{\sf sup}\geq {\sf sup}f$, and the second follows from the inequality ${\sf min}(C_0)\leq f({\sf min}(C_1))$.
We have the further inequality
\[
{\sf sup}\bigl\{f\varphi_1(b_1)\big| \ f\varphi_1(b_1)\leq {\sf min}(C_0)\bigr\}
\geq
{\sf sup}\bigl\{f\varphi_1g(b_0)\big| \ f\varphi_1g(b_0)\leq {\sf min}(C_0)\bigr\}
=
{\sf sup}\bigl\{\varphi_0(b_0)\big| \ \varphi_0(b_0)\leq {\sf min}(C_0)\bigr\}
\]
since $\varphi_0(b_0) = f\varphi_1g(b_0)$. Composing inequalities gives the result.
\end{proof}

\begin{remark}\label{remark.phi}
Note that the construction of $S$ determines, for each map $B\xra{\varphi} A$ in $\bDelta$, an endofunctor
\[
\bDelta_{/^{\cls}A}\xra{\varphi} \bDelta_{/^{\cls}A}
\]
sending $C\subset A$ to $C^\varphi \subset A$. Note the natural containment $C\subset C^\varphi \subset A$, for all $C\in \bDelta_{/^{\cls}A}$.

\end{remark}

\begin{observation}\label{obs.key}
We give an explicit presentation of the values of the functor $\varphi$ of Remark~\ref{remark.phi}. Given $\varphi^\circ:[a]^\circ\ra [b]^\circ$ in $\bDelta^{\op}$ and $0<i\leq a$, we have
\[
  \{i-1<i\}^\varphi =
    \begin{cases}
     [0,\varphi(0)] & \text{if $i\leq \varphi(0)$}\\
     [\varphi(b),a] & \text{if $\varphi(b)\leq i-1$}\\
     [\varphi(j-1),\varphi(j)] & \text{if $\varphi(j-1)\leq i-1<i\leq \varphi(j)$}
    \end{cases}       
\]
and
\[
  \{i\}^\varphi =
    \begin{cases}
     [0,\varphi(0)] & \text{if $i< \varphi(0)$}\\
     [\varphi(b),a] & \text{if $\varphi(b)< i$}\\
     [\varphi(j-1),\varphi(j)] & \text{if $\varphi(j-1)<i< \varphi(j)$}\\
     \{\varphi(j)\} & \text{if $i=\varphi(j)$}
    \end{cases}      
\]
\end{observation}

Since $\ev_s:\Ar^{\cls}(\bDelta^{\op})\ra \bDelta^{\op}$ is an exponentiable fibration by Lemma~\ref{lemma.Arcls.coCart}, this guarantees existence of the relative functor $(\oo,1)$-category $\Fun^{\sf rel}_{\bDelta^{\op}}( \Ar^{\cls}(\bDelta^{\op}), \un{\cX})$.  
See Definition 1.7 of~\cite{fibrations}.
Recall that $\Fun^{\sf rel}_{\bDelta^{\op}}( \Ar^{\cls}(\bDelta^{\op}), \un{\cX})$ is an $(\infty,1)$-category over $\bDelta^{\op}$ characterized as follows:
for each $(\infty,1)$-category $\cK \to \bDelta^{\op}$ over $\bDelta^{\op}$, a lift
\[
\xymatrix{
&&
\Fun^{\sf rel}_{\bDelta^{\op}}( \Ar^{\cls}(\bDelta^{\op}), \un{\cX})
\ar[d]
\\
\cK
\ar[rr]
\ar@{-->}[urr]
&&
\bDelta^{\op}
}
\]
is the datum of a functor
\[
\cK
\underset{\bDelta^{\op}}
\times
\Ar^{\cls}(\bDelta^{\op})
\dashrightarrow
\cX
~.
\]

\begin{observation}
Let $\cE \xra{\pi} \cB$ be an exponentiable fibration, and let $\cX$ be an $(\oo,1)$-category. An object $\Fun^{\sf rel}_\cB(\cE,\un{\cX})$ is a pair $(b\in \cB, \cE_{|b}\xra{\cF}\cX)$. For $(b_0, \cF_0)$ and $(b_1, \cF_1)$ objects of $\Fun^{\sf rel}_\cB(\cE,\un{\cX})$, a morphism $(b_0, \cF_0)\ra (b_1, \cF_1)$ consists of a morphism $f:b_0\ra b_1$ in $\cB$ together with a filler
\[
\xymatrix{
\cE_{|b_0}\ar[d]\ar[rrd]\\
\cE_{|f}\ar@{-->}[rr]&&\cX\\
\cE_{|b_1}\ar[u]\ar[urr]}
\]
where $\cE_{|f} = \cE\underset{\cB}\times[1]$ is the pullback of $\cE$ along the functor $[1]\xra{\langle f\rangle}\cB$ classifying the morphism $f:b_0\ra b_1$.
\end{observation}
\begin{observation}\label{obs.morphisms.relFun}
If $\cE\ra \cB$ is a Cartesian fibration, then there is an equivalence over $[1]$
\[
\xymatrix{
\Cylr(\cE_{|b_0}\xla{f^\ast}\cE_{|b_1})
\ar[rr]^-{\simeq}\ar[dr]&&\ar[dl]
\cE_{|f}
\\
&[1]}
\]
between the pullback of $\cE$ along $[1]\xra{\langle f\rangle}\cB$, and the right-cylinder of the Cartesian monodromy functor $f^\ast$. Consequently a morphism $(b_0, \cF_0)\ra (b_1, \cF_1)$ is equivalent to a natural transformation $\Phi$:
\[
\xymatrix{
\cE_{|b_0}\ar[drr]^-{\cF_0}\\
&\Phi\Downarrow&\cX\\
\ar[uu]^-{f^\ast}\cE_{|b_1}\ar[rru]_-{\cF_1}}
\]
\end{observation}
\begin{observation}
If $\cE\ra \cB$ is a Cartesian fibration, then $\Fun^{\sf rel}_\cB(\cE,\un{\cX})\ra \cB$ is a coCartesian fibration. Furthermore, the coCartesian monodromy functor over $f:b_0\ra b_1$ is given by precomposition with the Cartesian monodromy functor $f^\ast$ associated to $\cE\ra \cB$:
\[
\Fun(\cE_{|b_0},\cX)\xra{(-\circ f^\ast)}\Fun(\cE_{|b_1},\cX)~.
\]
\end{observation}

\begin{definition}\label{def.S}
The functor
\[
\xymatrix{
\TwAro(\bDelta^{\op}) \ar[rr]^-S\ar[dr]_-{\ev_s}&&\Fun^{\sf rel}_{\bDelta^{\op}}( \Ar^{\cls}(\bDelta^{\op}), \un{\bDelta^{\op}})\ar[dl]\\
&\bDelta^{\op}}
\]
is adjoint to the functor
\[
\TwAro(\bDelta^{\op})\underset{\bDelta^{\op}}\times \Ar^{\cls}(\bDelta^{\op})\xra{S} \bDelta^{\op}
\]
constructed in Lemma~\ref{lemma.S}.
\end{definition}

\begin{observation}\label{obs.S.description}
On objects, the functor $S: \TwAro(\bDelta^{\op})\ra\Fun^{\sf rel}_{\bDelta^{\op}}( \Ar^{\cls}(\bDelta^{\op}), \un{\bDelta^{\op}})$ sends $(A^\circ\xra{\varphi^\circ} B^\circ)\in  \TwAro(\bDelta^{\op})$ to the pair $\bigl(A^\circ, (\bDelta_{/^{\cls}A}\bigr)^{\op}\xra{\varphi}\bDelta^{\op}\bigr)$, where $\varphi$ is the functor defined as the composite
\[
\xymatrix{
\bigl(\bDelta_{/^{\cls}A}\bigr)^{\op}\ar[r]^-\varphi&\bigl(\bDelta_{/^{\cls}A}\bigr)^{\op}\ar[r]^-{\rm forget}&\bDelta^{\op}}
\]
sending $(C\subset A)^\circ$ to $C^\varphi$ as in Remark~\ref{remark.phi}. To a morphism $(f,g): (A_0^\circ \xra{\varphi_0}B_0^\circ)\ra(A_1^\circ \xra{\varphi_1}B_1^\circ)$, defined by a commutative diagram in $\bDelta$
\[
\xymatrix{
A_1\ar[r]^-{f}&A_0\\
\ar[u]_-{\varphi_0}B_1&\ar[l]^-{g}\ar[u]^-{\varphi_1}B_0}
\]
the functor $S$ assigns the morphism in $\Fun^{\sf rel}_{\bDelta^{\op}}( \Ar^{\cls}(\bDelta^{\op}), \un{\bDelta^{\op}})$ (see Observation~\ref{obs.morphisms.relFun}) defined by the natural transformation
\[
\xymatrix{
\bigl(\bDelta_{/^{\cls}A_0}\bigr)^{\op}\ar[drr]^-{\varphi_0}\\
&\Downarrow&\bDelta^{\op}\\
\ar[uu]^-{{\sf Hull}(f)}\bigl(\bDelta_{/^{\cls}A_1}\bigr)^{\op}\ar[rru]_-{\varphi_1}}
\]
which to $C_1\subset A_1$ assigns the morphism in $\bDelta^{\op}$ given by the opposite of the map
\[
f_{|}:C_1^{\varphi_1}\ra {\sf Hull}(f(C_1))^{\varphi_0}~.\] This is well-defined, by Lemma~\ref{lemma.S}.
\end{observation}

\begin{definition}\label{def.LIM}
The functor
\[
\Fun^{\sf rel}_{\bDelta^{\op}}( \Ar^{\cls}(\bDelta^{\op}), \un{\cX})
\xra{\sf LIM}
\cX
\]
assigns to a $\cK$-point $\cF$ of $\Fun^{\sf rel}_{\bDelta^{\op}}( \Ar^{\cls}(\bDelta^{\op}), \un{\cX})$, given by
\[
\cK\xra{\bigl\langle 
\cK 
\underset{\bDelta^{\op}} \times
\Ar^{\cls}(\bDelta^{\op})
\xra{\cF}\cX\bigr\rangle}\Fun^{\sf rel}_{\bDelta^{\op}}( \Ar^{\cls}(\bDelta^{\op}), \un{\cX})~,
\]
the $\cK$-point of $\cX$ defined by the composite functor
\[
\cK\longrightarrow
\cK 
\underset{\bDelta^{\op}} \times
\Ar^{\cls}(\bDelta^{\op})
\overset{\cF}\longrightarrow
\cX
\]
where $\cK\longrightarrow
\cK 
\underset{\bDelta^{\op}} \times
\Ar^{\cls}(\bDelta^{\op})$
is the pullback along $\cK\ra\bDelta^{\op}$ of the functor $\bDelta^{\op} \ra \Ar^{\cls}(\bDelta^{\op})$ given by identity morphisms: i.e., $A^\circ$ is sent to identity morphism $A^\circ \ra A^\circ$.
\end{definition}

\begin{lemma}\label{lemma.leftKan.id}
The left Kan extension
\[
\xymatrix{
\ar[d]_-{\ev_s}\TwAro(\bDelta^{\op})\ar[r]^-S&\Fun^{\sf rel}_{\bDelta^{\op}}( \Ar^{\cls}(\bDelta^{\op}), \un{\bDelta^{\op}})\ar[r]^-{\sf LIM}&\bDelta^{\op}\\
\bDelta^{\op}\ar@{-->}[urr]_-{(\ev_s)_!({\sf LIM}\circ S)}&}
\]
is equivalent to the identity functor: $(\ev_s)_!({\sf LIM}\circ S) \simeq {\sf id}_{\bDelta^{\op}}$.
Furthermore, this diagram commutes.  
\end{lemma}
\begin{proof}
Let $\varphi^\circ: A^\circ \ra B^\circ$ be an object of $\TwAro(\bDelta^{\op})$. We claim that the value of ${\sf LIM}\circ S$ on this object is
\[
({\sf LIM}\circ S) (A^\circ \xra{\varphi^\circ} B^\circ) = A^\circ~.
\]
To see this, recall first that the value $S( A^\circ \xra{\varphi^\circ} B^\circ) \in \Fun^{\sf rel}_{\bDelta^{\op}}( \Ar^{\cls}(\bDelta^{\op}), \un{\bDelta^{\op}})$ is given by the pair $\bigl(A^\circ, (\bDelta_{/^{\cls}A}\bigr)^{\op}\xra{\varphi}\bDelta^{\op}\bigr)$. Consequently, the object $({\sf LIM}\circ S) (A^\circ \xra{\varphi^\circ} B^\circ)$ is given by the composite
\[
\ast\longrightarrow (\bDelta_{/^{\cls}A}\bigr)^{\op}\xra{\varphi}\bDelta^{\op}
\]
where $\ast \ra (\bDelta_{/^{\cls}A}\bigr)^{\op}$ selects the identity morphism $A^\circ \ra A^\circ$. Consequently, we obtain
\[
({\sf LIM}\circ S) (A^\circ \xra{\varphi^\circ} B^\circ) = (A^\varphi )^{\circ}
\]
and (recall Lemma~\ref{lemma.S}) we have the further simplification:
\[
A^\varphi=\Bigl[{\sf sup}\bigl\{\varphi(b)\big| \ \varphi(b)\leq {\sf min}(A)\bigr\},{\sf inf}\bigl\{\varphi(b)\big| \ {\sf max}(A)\leq \varphi(b)\bigr\}\Bigr] = [{\sf min}(A), {\sf max}(A)] = A~.
\]

Consequently, the value $(\ev_s)_!({\sf LIM}\circ S) (A^\circ)$ is the colimit of the constant functor
\[
\TwAro(\bDelta^{\op})\underset{\bDelta^{\op}}\times \bDelta^{\op}_{/A^\circ} \longrightarrow \bDelta^{\op}
\]
taking constant value $A^\circ$ on every object $(A^\circ \ra B^\circ)$. Since $\TwAro(\bDelta^{\op})\underset{\bDelta^{\op}}\times \bDelta^{\op}_{/A^\circ}$ is a connected category and $\bDelta^{\op}$ is discrete, the value of this colimit is $A^\circ$.

To finish, inspection of the functor ${\sf LIM}$ reveals that the diagram commutes.
\end{proof}

\begin{cor} For $\cC:\bDelta^{\op}\ra \cX$, there is a canonical commutative diagram
\[
\xymatrix{
\ar[d]_-{\ev_s}\TwAro(\bDelta^{\op})\ar[r]^-S&\Fun^{\sf rel}_{\bDelta^{\op}}( \Ar^{\cls}(\bDelta^{\op}), \un{\bDelta^{\op}})\ar[r]^-\cC&\Fun^{\sf rel}_{\bDelta^{\op}}( \Ar^{\cls}(\bDelta^{\op}), \un{\cX})\ar[r]^-{\sf LIM}&\cX\\
\bDelta^{\op}\ar[urrr]_-{\cC}&}
\]
which witnesses $\cC$ as the left Kan extension $(\ev_s)_!({\sf LIM}\circ\cC\circ S)$
\end{cor}
\begin{proof}
By Lemma~\ref{lemma.leftKan.id} and the construction of ${\sf LIM}$ in Definition~\ref{def.LIM}, the diagram commutes. Also, $\ev_s$ is a Cartesian fibration whose fibers have contractible classifying spaces, therefore it is a localization (by Corollary~3.20 of \cite{fibrations}, using that Cartesian fibrations are right-initial). Lastly, we observe that a commutative diagram
\[
\xymatrix{
\cT\ar[rr]^-F\ar[d]_-G&&\cX\\
\cD\ar[urr]_-C}
\]
automatically witnesses $C$ as the left Kan extension $G_!F$ if $G$ is a localization. The result follows.
\end{proof}

\begin{definition}
Let $\cX$ be an $(\oo,1)$-category. The Morita $(\oo,1)$-category of $\cX$
\[
\Morita[\cX]\subset \Fun^{\sf rel}_{\bDelta^{\op}}( \Ar^{\cls}(\bDelta^{\op}), \un{\cX})
\]
is the full $\oo$-subcategory of the relative functors consisting of those objects $(\bDelta_{/^{\cls}A})^{\op}\xra{\cF}\cX$ which send Segal covers to limits. \end{definition}

\begin{remark}
If $\cX$ has finite limits, then the $(\oo,1)$-category $\Morita[\cX]$ is the unstraightening of a functor $\bDelta^{\op}\xra{{\sf Morita}[\cX]}\Cat_{(\oo,1)}$, 
similar to the developments in~\cite{rune.morita1} and~\cite{rune.morita2}.
\end{remark}

\begin{observation}\label{observation3}
The functor $S$ of Definition~\ref{def.S} factors through $\Morita[\bDelta^{\op}]$:
\[
\xymatrix{
&\Morita[\bDelta^{\op}]\ar[d]\\
\TwAro(\bDelta^{\op})\ar[r]^-S\ar@{-->}[ur]&\Fun^{\sf rel}_{\bDelta^{\op}}( \Ar^{\cls}(\bDelta^{\op}), \un{\bDelta^{\op}})
.
}
\]
\end{observation}

\begin{observation}\label{observation4}
Given a functor $\cC:\bDelta^{\op}\ra \cX$ such that the restriction $\bDelta^{\cls,\op} \ra \bDelta^{\op}\xra{\cC}\cX$ carry Segal covers to limits, there is a factorization
\[
\xymatrix{
\Morita[\bDelta^{\op}]\ar[d]\ar@{-->}[r]&\Morita[\cX]\ar[d]\\
\Fun^{\sf rel}_{\bDelta^{\op}}( \Ar^{\cls}(\bDelta^{\op}), \un{\bDelta^{\op}})\ar[r]^-\cC&\Fun^{\sf rel}_{\bDelta^{\op}}( \Ar^{\cls}(\bDelta^{\op}), \un{\cX})
.
}
\]
\end{observation}

\begin{notation}
For $A^\circ \in \bDelta^{\op}$, 
\[
b_A
\colon
(\bDelta^{\leq 1}_{/^{\cls}A})^{\op}  \hookrightarrow (\bDelta_{/^{\cls}A})^{\op}
\]
is the inclusion of the full subcategory consisting of those $A^\circ\ra [p]^\circ$ with $p\leq 1$.
\end{notation}

\begin{lemma}\label{lemma.def.corr}
For $(A^\circ,\cF) \in \Fun^{\sf rel}_{\bDelta^{\op}}( \Ar^{\cls}(\bDelta^{\op}), \un{\cX})$, then $(A^\circ,\cF)$ belongs to the $\oo$-subcategory $\Morita[\cX]$ if and only if the natural map
\[
\cF \ra b_{A\ast}b_A^\ast \cF
\]
is an equivalence.
\end{lemma}
\begin{proof}
Consider a convex inclusion $[p]\subset A$ defining an object of $\bDelta_{/^{\cls}A}$. We prove that $\cF[p]\ra (b_{A\ast}b_A^\ast \cF)[p]$ is an equivalence. The right Kan extension $b_{A\ast}$ is computed by the limit of the functor $b_A^\ast \cF: (\bDelta^{\leq 1}_{/^{\cls}[p]})^{\op}\ra \cX$, which is exactly the iterated fiber product $\cF(\{0<1\})\times_{\cF(\{0\})}\ldots \times_{\cF(\{p-1\})}\cF(\{p-1<p\})$. We thus see that $\cF$ is Segal if and only if it computed by the right Kan extension from closed morphisms.
\end{proof}

\begin{definition}\label{def.CX}
The $\oo$-subcategory $\cM[\cX]\subset \Fun^{\sf rel}_{\bDelta^{\op}}( \Ar^{\cls}(\bDelta^{\op}), \un{\cX})$ consists of those morphisms $(A^\circ\xra{\varphi^\circ}B^\circ, \eta)$
\[
\xymatrix{
\bigl(\bDelta_{/^{\cls}A}\bigr)^{\op}\ar[rr]^-{\cF_A}&&\cX\\
\bigl(\bDelta_{/^{\cls}B}\bigr)^{\op}\ar[u]^-{{\sf Hull}(\varphi)}\ar[urr]^-{ \Downarrow\eta}_{\cF_B}}
\]
such that the canonical morphism
\[
b_B^\ast{\sf Hull}(\varphi)^\ast\cF_A \longrightarrow b_B^\ast{\sf Hull}(\varphi)^\ast b_{A\ast}b_A^\ast\cF_A
\]
is an equivalence.
\end{definition}

\begin{lemma}
There is an inclusion $\Morita[\cX]\subset \cM[\cX]$ of $\oo$-subcategories of $\Fun^{\sf rel}_{\bDelta^{\op}}( \Ar^{\cls}(\bDelta^{\op}), \un{\cX})$.
\end{lemma}
\begin{proof}
By Lemma~\ref{lemma.def.corr}, an object $(A^\circ, \cF_A)\in \Fun^{\sf rel}_{\bDelta^{\op}}( \Ar^{\cls}(\bDelta^{\op}), \un{\cX})$ belongs to $\Morita[\cX]$ if and only $\cF_A \ra b_{A\ast}b_A^\ast \cF_A$ is an equivalence, from which the result immediately follows.
\end{proof}

\begin{lemma}
For $\cC:\bDelta^{\op}\ra \cX$ a simplicial object in $\cX$, the essential image of the composite functor $\cC\circ S$ lies in $\cM[\cX]$. If $\cC$ is Segal, then the essential image of $\cC\circ S$ lies in $\Morita[\cX]$. That is, there are unique factorizations:
\[
\xymatrix{
&& \Morita[\cX]\ar[d]\\
&&\cM[\cX]\ar[d]\\
\TwAro(\bDelta^{\op})\ar[r]_-S\ar@{-->}[urr]\ar@{-->}[uurr]^{\exists ~ {\rm if}~\cC \ {\rm Segal}}&\Fun^{\sf rel}_{\bDelta^{\op}}( \Ar^{\cls}(\bDelta^{\op}), \un{\bDelta^{\op}})\ar[r]_-\cC&\Fun^{\sf rel}_{\bDelta^{\op}}( \Ar^{\cls}(\bDelta^{\op}), \un{\cX})}
\]
\end{lemma}
\begin{proof}
The factorization for $\cC$ Segal follows from Observation~\ref{observation3} and Observation~\ref{observation4}. It remains to prove the first factorization, that for any simplicial object $\cC$ the composite $\cC\circ S$ belongs to $\cM[\cX]$. That is, we show that for any morphism in $\TwAro(\bDelta^{\op})$, the image under $\cC\circ S$ satisfies the condition of Definition~\ref{def.CX}. Let $f=(f,g)$ be such a morphism, consisting of a commutative diagram
\[
\xymatrix{
A_0^\circ\ar[d]_-{\varphi_0^\circ}\ar[r]^-{f^\circ}&A_1^\circ\ar[d]^-{\varphi_1^\circ}\\
B_0^\circ&\ar[l]^-{g^\circ}B_1^\circ}
\]
in $\bDelta^{\op}$. The functor $\cC\circ S$ carries this morphism to the natural transformation $\Phi_f$
\[
\xymatrix{
\bigl(\bDelta_{/^{\cls}A_0}\bigr)^{\op}\ar[r]^-{\varphi_0}&\bigl(\bDelta_{/^{\cls}A_0}\bigr)^{\op}\ar[r]^-{\rm forget}&\bDelta^{\op}\ar[dr]^-\cC\\
&\Phi_f\Downarrow&&\cX\\
\bigl(\bDelta_{/^{\cls}A_1}\bigr)^{\op}\ar[r]^-{\varphi_1}\ar[uu]^-{{\sf Hull}(f)}&\bigl(\bDelta_{/^{\cls}A_1}\bigr)^{\op}\ar[r]^-{\rm forget}&\bDelta^{\op}\ar[ur]_-\cC}
\] 
where $\Phi_f$ takes the value, for each $C_1\subset A_1$ a convex subset, as the morphism in $\cX$:
\[
\xymatrix{
\cC\Bigl({\sf Hull}(f(C_1))^{\varphi_0}\Bigr)\ar@{=}[r]&\cC\Bigl(
\bigl[
{\sf sup}\{\varphi_0(b_0)|\varphi_0(b_0)\leq f({\sf min} (C_1))\},
{\sf inf}\{\varphi_0(b_0)| \varphi_0(b_0\geq f({\sf max} (C_1))\}
\bigr]^\circ
\Bigr)
\ar[d]^-{\cC(f_|^\circ)}\\
\cC(C_1^{\varphi_1})\ar@{=}[r]&\cC\Bigl(
\bigl[
{\sf sup}\{\varphi_1(b_1)|\varphi_1(b_1)\leq {\sf min} (C_1)\},
{\sf inf}\{\varphi_1(b_1)| \varphi_1(b_1)\geq {\sf max}(C_1)\}
\bigr]^\circ
\Bigr)
}
\]
Denote by $\cF$ the top horizontal composite: $\cF :=\cC\circ {\rm forget}\circ\varphi_0$. Let $C_1\subset A_1$ be an object of $\bDelta^{\leq 1}_{/^{\cls}A_1}$. Then
\[
b_{A_1}^\ast {\sf Hull}(f)^\ast \cF (C_1^\circ) 
\simeq
\cC\Bigl(
\bigl[
{\sf sup}\{\varphi_0|\varphi_1(b_0)\leq f(C_1)\},
{\sf inf}\{\varphi_0(b_0)|\varphi_0(b_0)\geq f(C_1)\}
\bigr]^\circ
\Bigr)
\]
while
\[
b_{A_1}^\ast {\sf Hull}(f)^\ast b_{A_0\ast} b_{A_0}^\ast\cF (C_1^\circ) 
\simeq
b_{A_0\ast} b_{A_0}^\ast\cF\bigl({\sf Hull}(f(C_1))^\circ\bigr)
\]
\[
\simeq
{\sf lim}\Bigl(
\bigl(\bDelta^{\leq 1}_{/^{\cls}{\sf Hull}(f(C_1))}\bigr)^{\op}
\hookrightarrow
\bigl(\bDelta_{/^{\cls}A_0}\bigr)^{\op}
\xra{\varphi_0}
\bigl(\bDelta_{/^{\cls}A_0}\bigr)^{\op}
\xra{\rm forget}
\bDelta^{\op}
\xra{\cC}
\cX
\Bigr)
\]
is the limit of the composite functor which is given by
\[
(C_0 \subset{\sf Hull}(f(C_1)))^\circ \mapsto
\cC\Bigl(
\bigl[
{\sf sup}\{
\varphi_0(b_0)|\varphi_0(b_0)\leq {\sf min}(C_0)\},
{\sf inf}\{
\varphi_0(b_0)|\varphi_0(b_0)\geq {\sf max}(C_0)\}
\bigr]
\Bigr)~.
\]
We show that the canonical morphism in $\cX$ defined by the unit of the adjunction $(b_{A_0}^\ast \dashv b_{A_0\ast})$
\[
b_{A_1}^\ast{\sf Hull}(f)^\ast\cF(C_1^\circ)
\xra{\sf unit}
b_{A_1}^\ast{\sf Hull}(f)^\ast b_{A_0\ast}b_{A_0}^\ast\cF(C_1^\circ)
\]
is an equivalence, again for $(C_1\subset A_1) \in\bDelta^{\leq 1}_{/^{\cls}A_1}$. To establish this, we claim that there is a localization functor
\[
\bDelta^{\leq 1}_{/^{\cls}{\sf Hull}(f(C_1))}\overset{Q}\longrightarrow \bDelta^{\leq 1}_{/^{\cls}f(C_1)}
\]
together with a factorization
\[
\xymatrix{
\ar@{-->}[d]_-{Q^{\op}}\bigl(\bDelta^{\leq 1}_{/^{\cls}{\sf Hull}(f(C_1))}\bigr)^{\op} \ar[r]&\bigl(\bDelta_{/^{\cls}A_0}\bigr)^{\op}\ar[r]^-{\varphi_0}
&\bigl(\bDelta_{/^{\cls}A_0}\bigr)^{\op}
\\
\bigl(\bDelta^{\leq 1}_{/^{\cls}f(C_1)}\bigr)^{\op}\ar[r]&\bigl(\bDelta_{/^{\cls}A_0}\bigr)^{\op}\ar[ur]_-{\varphi_0}~.
}
\]
Here, the functor $Q$ is defined to take values on $I \in \bDelta^{\leq 1}_{/^{\cls}{\sf Hull}(f(C_1))}$
\[
  Q(I) =
    \begin{cases}
     \{ {\sf min}(f(C_1))\} & \text{if $I = \{ {\sf min}(f(C_1))\} $}\\
     \{ {\sf max}(f(C_1))\} & \text{if $I = \{ {\sf max}(f(C_1)) \} $}\\
      f(C_1) & \text{otherwise}
    \end{cases}       
    ~.
\]

Let $I\subset {\sf Hull}(f(C_1))$ be an object of $\bDelta^{\leq 1}_{/^{\cls}{\sf Hull}(f(C_1))}$. Observe the following containments of subsets of $A_0$:
\[
\{\varphi_0(b_0)\in \varphi_0(B_0)|{\sf min}(f(C_1))<\varphi_0(b_0)\leq {\sf min}(I)\}
\overset{(1)}\subset
\Bigl({\sf Hull}(f(C_1))\smallsetminus \{{\sf min}(f(C_1))\}\Bigr)\cap \varphi_0(B_0)
\]
\[
\overset{(2)}\subset
\Bigl({\sf Hull}(f(C_1))\smallsetminus \{{\sf min}(f(C_1))\}\Bigr)\cap f(A_1)
\overset{(3)}\subset
\Bigl(f(C_1))\smallsetminus \{{\sf min}(f(C_1))\}\Bigr)
\overset{(4)}\subset
\{{\sf max}(f(C_1))\}
\]
where containment (1) is immediate from the definition of the hull; (2) follows from the containment $\varphi_0(B_0)\subset f(A_1)$; (3) follows from the equality ${\sf Hull}(f(S))\cap f(A_1) = f(S)$, for any subset $S\subset A_1$; and (4) follows from the assumption on the cardinality $1\leq |C_1|\leq 2$.

Therefore:
\[
  {\sf sup}\{\varphi_0(b_0)|\varphi_0(b_0)\leq {\sf min}(I)\} =
    \begin{cases}
      {\sf max}(C_1) & \text{${\sf max}(C_1)\in \varphi_0(B_0)$ and $I=\{{\sf max}(f(C_1)) \}$ }\\
      {\sf sup}\{\varphi_0(b_0)\in\varphi_0(B_0)|\varphi_0(b_0)\leq {\sf min}(f(C_1))\}  & \text{otherwise}
    \end{cases}       
\]
and similarly
\[
  {\sf inf}\{\varphi_0(b_0)|\varphi_0(b_0)\geq {\sf max}(I)\} =
    \begin{cases}
      {\sf min}(C_1) & \text{${\sf min}(C_1)\in \varphi_0(B_0)$ and $I=\{{\sf min}(f(C_1)) \}$}\\
        {\sf inf}\{\varphi_0(b_0)\in\varphi_0(B_0)|\varphi_0(b_0)\geq {\sf max}(f(C_1))\}  & \text{otherwise}
    \end{cases}       
\]
Therefore,
\[
  I^{\varphi_0}=
    \begin{cases}
         \{ {\sf max}(f(C_1))\} & \text{if ${\sf max}(C_1)\in \varphi_0(B_0)$ and $I=\{{\sf max}(f(C_1)) \}$ }\\
     \{ {\sf min}(f(C_1))\} & \text{if ${\sf min}(C_1)\in \varphi_0(B_0)$ and $I=\{{\sf min}(f(C_1)) \}$ }\\
      f(C_1)^{\varphi_0} & \text{otherwise}
    \end{cases}    
    ~.   
\]

This shows the factorization through $Q:\bDelta^{\leq 1}_{/^{\cls}{\sf Hull}(f(C_1))}\ra \bDelta^{\leq 1}_{/^{\cls}f(C_1)}
$. We now show that $Q$ is a localization. 
The source $\bDelta^{\leq 1}_{/^{\cls}{\sf Hull}(f(C_1))}$ is equivalent to the category 
\[
\Bigl\{ \{h_0\}\ra \{h_0 < h_1\} \la \{h_1\}\ra \{h_1< h_2\}\la  \ldots \ra \{h_{q-1}< h_q\}\la \{h_q\}\Bigr\}
\]
where $h_0 = {\sf min}(f(C_1))$ and $h_q = {\sf max}(f(C_1))$. The functor $Q$ is then given by:
\[
  Q(\{h_{i-1}<h_i\}) = f(C_1)\\
  \]
  \[
  Q(\{h_i\}) =
    \begin{cases}
     \{ {\sf min}(f(C_1))\} & \text{if $k=0$}\\
     \{ {\sf max}(f(C_1))\} & \text{if $k=q$}\\
      f(C_1) & \text{otherwise}
    \end{cases}       
\]
And so $Q$ witnesses a localization with respect to the set of morphisms whose source is not $\{h_0\}$ or $\{h_q\}$. Since localizations are final and initial, we obtain in particular that the functor $Q^{\op}$ is initial. We thus obtain that the canonical morphism
\[
\xymatrix{
\cC(f(C_1)^{\varphi_0})\simeq
{\sf lim}\Bigl(
\bigl(\bDelta^{\leq 1}_{/^{\cls}f(C_1)}\bigr)^{\op}
\xra{\varphi_0}
\bigl(\bDelta_{/^{\cls}A_0}\bigr)^{\op}
\xra{\cC}
\cX
\Bigr)\ar[d]\\
{\sf lim}\Bigl(
\bigl(\bDelta^{\leq 1}_{/^{\cls}{\sf Hull}(f(C_1))}\bigr)^{\op}
\xra{\varphi_0}
\bigl(\bDelta_{/^{\cls}A_0}\bigr)^{\op}
\xra{\cC}
\cX
\Bigr)\simeq b_{A_0\ast} b_{A_0}^\ast\cF\bigl({\sf Hull}(f(C_1))^\circ\bigr)
}
\]
is an equivalence. That is, the unit morphism
\[
\cC(f(C_1)^{\varphi_0})\simeq b_{A_1}^\ast{\sf Hull}(f)^\ast\cF(C_1^\circ)
\ra b_{A_1}^\ast{\sf Hull}(f)^\ast b_{A_0\ast}b_{A_0}^\ast\cF(C_1^\circ)\simeq \cC(f(C_1)^{\varphi_0})
\] is an equivalence, which completes the proof.
\end{proof}

\begin{lemma}
If $\cX$ has finite limits, the inclusion $\Morita[\cX]\ra \cM[\cX]$ admits a left adjoint $\sL$, which commutes with projection to $\bDelta^{\op}$.
\end{lemma}
\begin{proof}
Let $(A^\circ, (\bDelta_{/^{\cls}A})^{\op}\xra\cF\cX)$ be an object of $\cM[\cX]$. 
By Lemma~2.17 of~\cite{fibrations}, it suffices to show $\Morita[\cX]^{(A^\circ,\cF)/}$ has an initial object, $\sL(A^\circ, \cF)$.
Define $\sL(A^\circ, \cF) := (A^\circ, (\bDelta_{/^{\cls}A})^{\op}\xra{b_{A\ast}b_A^\ast\cF}\cX)$. This object exists, since $\cX$ has finite limits. By the equivalence $b_{A\ast}b_A^\ast \simeq {\sf id}_{(\bDelta_{/^{\cls}A})^{\op}}$, we have 
\[
b_{A\ast}b_A^\ast \cF \simeq b_{A\ast}b_A^\ast b_{A\ast}b_A^\ast \cF~.
\]
Therefore $\sL(A^\circ, \cF) $ is an object of $\Morita[\cX]$. By construction, there is a canonical morphism $(A^\circ, \cF)\ra \sL(A^\circ,\cF)$ over ${\sf id}_{A^\circ}$ in $\bDelta^{\op}$. Therefore, we have defined an object $\sL(A^\circ,\cF)\in \Morita[\cX]^{(A^\circ,\cF)/}$.

We now show that the object $\sL(A^\circ,\cF)$ is initial in $ \Morita[\cX]^{(A^\circ,\cF)/}$. So let
\[
(A^\circ\xra{\varphi^\circ}B^\circ,  (\bDelta_{/^{\cls}B})^{\op}\xra\cG\cX, \Phi)
\]
be an object of $\Morita[\cX]^{(A^\circ,\cF)/}$, where $\Phi$ is a natural transformation
\[
\xymatrix{
\bigl(\bDelta_{/^{\cls}A}\bigr)^{\op}\ar[rr]^-{\cF}&&\cX\\
\bigl(\bDelta_{/^{\cls}B}\bigr)^{\op}\ar[u]^-{{\sf Hull}(\varphi)}\ar[urr]^-{ \Downarrow\Phi}_{\cG}}~.
\]
We must show that there is a unique factorization in $\Fun((\bDelta_{/^{\cls}B})^{\op},\cX)$
\[
\xymatrix{
&{\sf Hull}(\varphi)^\ast b_{A\ast}b_A^\ast\cF\ar@{-->}[dr]^-{\exists!}\\
{\sf Hull}(\varphi)^\ast\cF\ar[ur]\ar[rr]^-\Phi&&\cG~.}
\]
Consider the solid diagram in $\Fun((\bDelta_{/^{\cls}B})^{\op},\cX)$:
\[
\xymatrix{
\ar[dd]_-\Phi{\sf Hull}(\varphi)^\ast\cF\ar[dr]^-{{\sf Hull}(\varphi)^\ast{\sf unit}}\ar[rr]^-{\sf unit}
&& 
b_{B\ast}b_B^\ast{\sf Hull}(\varphi)^\ast\cF\ar[dr]^-{b_{B\ast}b_B^\ast{\sf Hull}(\varphi)^\ast{\sf unit}}\ar[dd]^(.3){b_{B\ast}b_B^\ast\Phi}
\\
&
{\sf Hull}(\varphi)^\ast b_{A\ast}b_A^\ast\cF
\ar@{-->}[dl]^-{\exists!}\ar[rr]^(.3){\sf unit}
&&
b_{B\ast}b_B^\ast{\sf Hull}(\varphi)^\ast b_{A\ast}b_A^\ast\cF\\
\cG\ar[rr]^{\sf unit}
&&
b_{B\ast}b_B^\ast\cG}
.
\]
By definition of $\Morita[\cX]$, the unit morphism $\cG\ra b_{B\ast}b_B^\ast\cG$ is an equivalence. By definition of $\cM[\cX]$, the top left morphism $b_{B\ast}b_B^\ast{\sf Hull}(\varphi)^\ast{\sf unit}$ is an equivalence. The existence and uniqueness of the dashed filler  then follows.
\end{proof}

\begin{definition}\label{def.Seg}
Let $\cX$ be a presentable $(\oo,1)$-category in which colimits are universal. The endofunctor $\Seg:\Fun(\bDelta^{\op}, \cX)\ra\Fun(\bDelta^{\op}, \cX)$ sends a simplicial object $\cC$ to the left Kan extension
\[
\xymatrix{
\ar[d]_-{\ev_s}\TwAro(\bDelta^{\op})\ar[r]^-S&\Morita[\bDelta^{\op}]\ar[r]^-\cC&\cM[\cX]\ar[r]^-\sL&\Morita[\cX]\ar[r]^-{\sf LIM}&\cX\\
\bDelta^{\op}\ar@{-->}[urrrr]_-{\Seg(\cC)=(\ev_{s})_!({\sf LIM}\circ\sL\circ\cC\circ S)}}
\]
\end{definition}
\begin{observation}\label{obs.Segal.basecase.contact}
Given $\cC: \bDelta^{\op}\ra \cX$, the value
\[
\xymatrix{
\TwAro(\bDelta^{\op})\ar[r]^-S&\Morita[\bDelta^{\op}]\ar[r]^-\cC&\cM[\cX]\ar[r]^-\sL&\Morita[\cX]\ar[r]^-{\sf LIM}&\cX}
\]
is given by
\begin{itemize}
\item $(\cC\circ S)([a]^\circ\xra{\varphi^\circ}[b]^\circ): (\bDelta^{\op}_{/^{\cls}[a]})^{\op}\ra \cX$ is the functor sending $(C\subset[a])^\circ$ to $\cC((C^\varphi)^\circ)$.
\item $(\sL\circ \cC\circ S)([a]^\circ\xra{\varphi^\circ}[b]^\circ): (\bDelta^{\op}_{/^{\cls}[a]})^{\op}\ra \cX$ is the functor sending $(C\subset[a])^\circ$ to 
\[
\cC((\{c_0<c_1\}^\varphi)^\circ)\underset{\cC((\{c_1\}^\varphi)^\circ)}\times\ldots\underset{\cC((\{c_{q\text{-}1}\}^\varphi)^\circ)}\times\cC((\{c_{q\text{-}1}<c_q\}^\varphi)^\circ)
\]
where $C= \{c_0<\ldots < c_q\}$.
\item The composite $({\sf LIM}\circ\sL\circ \cC\circ S)([a]^\circ\xra{\varphi^\circ}[b]^\circ)\in \cX$ is then the value
\[
\cC((\{0<1\}^\varphi)^\circ)\underset{\cC((\{1\}^\varphi)^\circ)}\times\ldots\underset{\cC((\{a\text{-}1\}^\varphi)^\circ)}\times\cC((\{a\text{-}1<a\}^\varphi)^\circ)
\]
\[
\simeq \cC([0,\varphi(0)]^\circ)\underset{\simeq \cC(\{\varphi(0)\}^\circ)}\times \cC([\varphi(0),\varphi(1)]^\circ)
\underset{\cC(\{\varphi(1)\}^\circ)}\times\ldots\underset{\cC(\{\varphi(b\text{-}1)\}^\circ)}\times
\cC([\varphi(b\text{-}1),\varphi(b)]^\circ)
\underset{\cC(\{\varphi(b)\}^\circ)}\times
\cC([\varphi(b),a]^\circ)
\]
by Observation~\ref{obs.key}.
\end{itemize}
\end{observation}

\begin{observation} 
The unit of the adjunction $\cM[\cX]\leftrightarrows \Morita[\cX]$ determines a canonical morphism $\eta_\cC:\cC \ra\Seg(\cC)$ for any $\cC \in \Fun(\bDelta^{\op},\cX)$. 
Using that colimits in $\cX$ are universal, if $\cC$ is Segal, then $\eta_\cC$ is an equivalence. 
\end{observation}

\begin{definition}
A monad $T$ on an $(\oo,1)$-category $\cV$ is idempotent if the structure morphisms
\[
\xymatrix{
T\circ {\sf id}_\cV\ar[rd]^-{{\sf id}_T \circ {\sf unit}}\\
&T\circ T\ar[r]^-{{\sf mult}_T}&T\\
{\sf id}_\cV \circ T\ar[ur]^-{{\sf unit}\circ {\sf id}_T}}
\]
are equivalences. The $(\oo,1)$-category ${\sf Monad}^{\sf idemp}(\cV)$ is the full $\oo$-subcategory of ${\sf Monad}(\cV):=\Alg(\Fun(\cV,\cV))$ consisting of the idempotent monads.
\end{definition}
\begin{observation}
Since the diagram
\[
\xymatrix{
T\circ {\sf id}_\cV\ar[rd]_-{{\sf id}_T \circ {\sf unit}}\ar[drrr]^-{{\sf id}_T}\\
&T\circ T\ar[rr]^-{{\sf mult}_T}&&T\\
{\sf id}_\cV \circ T\ar[ur]^-{{\sf unit}\circ {\sf id}_T}\ar[urrr]_-{{\sf id}_T}}
\]
commutes, $T$ is idempotent if any one of the three structure morphisms are equivalences.
\end{observation}

\begin{definition}
Let $\cV$ be an $(\oo,1)$-category. The full $\oo$-subcategory ${\sf RFull}(\cV)\subset \Cat_{(\infty,1)/\cV}$ has objects those functors $R: \cV_0 \ra \cV$ which are fully-faithful right-adjoints. That is, $R:\cV_0\ra \cV$ belongs to ${\sf RFull}(\cV)$ if and only if there exists a left adjoint $L:\cV \ra \cV_0$ for which the counit of the adjunction $LR \ra {\sf id}_{\cV_0}$ is an equivalence.

\end{definition}

\begin{lemma}\label{lemma.idemp.monads}
Let $\cV$ be an $(\oo,1)$-category. There is a natural equivalence of spaces
\[
\obj\Bigl({\sf RFull}(\cV)\Bigr)
\simeq
\obj\Bigl({\sf Monad}^{\sf idemp}(\cV)\Bigr)
\]
between the space of fully-faithful right-adjoints to $\cV$, and the space of idempotent monads on $\cV$.
\end{lemma}
\begin{proof}
Let $\Cat^\sR$ be the $(\oo,1)$-category of $(\oo,1)$-categories whose morphisms are right adjoints. There is a functor
\[
f: \Cat^\sR_{/\cV} \longrightarrow {\sf Monad}(\cV)
\]
sending a right adjoint $R:\cU\ra \cV$ to the monad $T: = RL$, where $L$ is the left adjoint of $R$. To a morphism in $\Cat^\sR_{/\cV}$, which is a right adjoint $R':\cU' \ra \cU$ over $\cV$ with left adjoint $L'$, this functor assigns the map of monads given by the unit of $L'\dashv R'$:
\[
T = RL \xra{{\sf unit}_{L'\dashv R'}}RR'L'L = T'~.
\]
Now assume the functor $R:\cU\ra \cV$ is fully-faithful, and thus the counit map $LR \ra {\sf id}_{\cU}$ is an equivalence. Then the multiplication morphism $T\circ T \ra T$ is an equivalence, since it given by the counit of $LR$: $T\circ T = R(LR)L \ra LR \simeq RL=T$. Consequently, this restricts to a functor ${\sf RFull}(\cV) \ra {\sf Monad}^{\sf idemp}(\cV)$.

We construct a map $g: \obj({\sf Monad}^{\sf idemp}(\cV))\ra \obj({\sf RFull}(\cV))$. To a monad $T$, we assign the $(\oo,1)$-category $\cV_T$ of $T$-local objects, the full $\oo$-subcategory of those objects $x$ of $\cV$ on which the unit of $T$ is an equivalence:
\[
\cV_T:=\{x\in\cV|\ x\xra{\sim} Tx\}\subset \cV~.
\]
The inclusion $\cV_T \hookrightarrow \cV$ has a left adjoint, which is $T$ itself: For any morphism $x\ra t$ where $t$ is $T$-local, there is a functorial commutative diagram
\[
\xymatrix{
x\ar[d]\ar[r]&t\ar[d]^-\sim\\
Tx\ar[r]\ar@{-->}[ur]&Tt}
\]
and thus a canonical factorization of the morphism $x\ra t$ through the unit morphism $x\ra Tx$. Consequently, $Tx$ is initial in $\cV_T^{x/}$, and the functor $T: \cV \ra \cV_T$ defines a left adjoint to the inclusion $\cV_T \hookrightarrow \cV$. 
By this construction of $g$, observe an equivalence ${\sf id} \simeq f\circ g$, of self-maps of $\obj({\sf Monad}^{\sf idemp}(\cV))$.

We lastly exhibit a natural equivalence ${\sf id} \simeq g\circ f$, of self-maps of $\obj({\sf RFull}(\cV))$. To a fully-faithful right-adjoint $R:\cU\hookrightarrow \cV$ with left adjoint $L$, we claim there is a unique factorization
\[
\xymatrix{
\cU \ar@{^{(}->}[dr]_-R\ar@{-->}[rr]&&\cV_{RL}\ar@{^{(}->}[dl]\\
&\cV}
\]
where $\cV_{RL}\subset \cV$ is the full $\oo$-subcategory of $x\in\cV$ for which $x\ra RL x$ is an equivalence. However, $x$ belongs to $\cV_{RL}$ if and only if $x$ is in the essential image of $R$: so a unique functor $\cU \ra \cV_{RL}$ exists over $\cV$, and is an equivalence.
\end{proof}

We now come to the culminating result of this section, which describes the Segal-completion functor in special cases.
\begin{lemma}\label{lemma.Seg.universal}
Let $\cX$ be a presentable $(\oo,1)$-category in which colimits are universal, and let $\cC\in\Fun(\bDelta^{\op},\cX)$ be a simplicial object in $\cX$. If $\Seg(\cC)$ is Segal, then $\eta_\cC:\cC \ra \Seg(\cC)$ exhibits $\Seg(\cC)$ as the universal Segal object under $\cC$.
\end{lemma}
\begin{proof}
The fully-faithful inclusion of Segal functors $R:\Fun^{\Seg}(\bDelta^{\op},\cX) \hookrightarrow \Fun(\bDelta^{\op},\cX)$ has a left adjoint $L$ by the adjoint functor theorem, since the inclusion preserves limits and filtered colimits. Let $\cV:= \Seg^{-1}\bigl(\Fun^{\Seg}(\bDelta^{\op},\cX)\bigr)\subset \Fun(\bDelta^{\op},\cX)$ be the full $\oo$-subcategory of all $\cC$ such that $\Seg(\cC)$ is Segal.

Observe that on $\cV$ the functor $\Seg$ has the structure of an idempotent monad. The functor $RL$ additionally has the structure of an idempotent monad on $\Fun(\bDelta^{\op},\cX)$, which restricts to an idempotent monad on $\cV$. Consequently, both $\Seg$ and $RL$ define idempotent monads on $\cV$, both of which correspond to fully-faithful right adjoints $\cV_{\Seg}\hookrightarrow \cV$ and $\cV_{RL}\hookrightarrow \cV$, defined by the $\Seg$-local and $RL$-local objects, respectively. 
Observe that these $(\oo,1)$-categories are both equivalent to $\Fun^{\Seg}(\bDelta^{\op},\cX)$, however. Consequently, by Lemma~\ref{lemma.idemp.monads}, we have an equivalence $\Seg\simeq L$ on $\cV$. That is, for any $\cC\in \cV$, then $\Seg(\cC)$ is the universal Segal object under $\cC$.
\end{proof}

\section{Twisted-arrow categories of 1-disks}

In this section, we prove Lemma~\ref{localize.D.Delta}, that the twisted arrow category of a poset of disks localizes onto the twisted arrow category of $\bDelta$. 

\begin{definition}
Let $I$ be a 1-manifold with boundary. The category $\disk_I$ is the full sub-poset of $\Opens(I)$ consisting of all $U\subset I$ such that:
\begin{itemize}
\item each component of $U$ is diffeomorphic to either $\RR$ or $\RR_{\geq 0}$;
\item $U$ contains the boundary $\partial I$;
\item $U$ has finitely many components.
\end{itemize}
\end{definition}

Recall the following result from, e.g., \cite{oldfact}.
\begin{lemma}
Let $I$ be an oriented closed interval. There is functor
\[
\disk_I \overset{\pi}\longrightarrow \bDelta^{\op}
\]
sending an open $U\subset I$ to the finite set $\pi U := \pi_0(I\smallsetminus U)$, ordered according to the orientation of $I$. This functor is a localization with respect to the set of morphisms $U \ra V$ which are isotopy equivalences or, equivalently, the morphisms which induce bijections $\pi U\simeq \pi V$.
\end{lemma}

Our goal in this section, Lemma~\ref{localize.D.Delta}, is to prove the analogous localization after taking categories of twisted arrows. This will rely on the following two preparatory results, Lemma~\ref{lemma.I.final} and Lemma~\ref{lemma.points.twar.over}.

\begin{lemma}\label{lemma.I.final}
Let $I$ be a 1-manifold with boundary.  
Let $\cJ \subset \TwAr(\sD_I)$ be the subcategory consisting of all objects $U\hookrightarrow V$, and whose morphisms $(U\hookrightarrow V)\ra (U'\hookrightarrow V')$ are those commutative diagrams
\[
\xymatrix{
U\ar[d]&U'\ar[l]\ar[d]\\
V\ar[r]&V'}
\]
for which both $U\la U'$ and $V\ra V'$ are isotopy equivalences. 
Let $\cJ'\subset \cJ$ be the full subcategory consisting of those objects $U\hookrightarrow V$ for which $V$ contains the closure of $U$: $\overline{U}\subset V$.
Under these conditions, the inclusion
\[
\cJ' \hookrightarrow \cJ
\]
is a final functor.
\end{lemma}
\begin{proof}
By Quillen's Theorem A, it suffices to check that for any object $U\hookrightarrow V$, the undercategory
\[
\cJ'\underset{\cJ}\times\cJ^{U\ra V/}
\]
has contractible classifying space. There is a full subcategory $\cC \subset \cJ'\underset{\cJ}\times\cJ^{U\ra V/}$ consisting of those $U'\ra V'$ under $U\ra V$ for which $V\ra V'$ is an isomorphism.  The inclusion $\cC \subset \cJ'\underset{\cJ}\times\cJ^{U\ra V/}$ admits a right adjoint, sending $U'\ra V'$ to $U'\ra V$: the unit of the adjunction is given by
\[
\xymatrix{
U\ar[d]&U'\ar[l]\ar[d]&U'\ar[d]\ar[l]_-=\\
V\ar[r]^-=&V\ar[r]&V'}
\]
for each $U'\ra V'$ under $U\ra V$. Hence the classifying space of $\cJ'\underset{\cJ}\times\cJ^{U\ra V/}$ is equivalent to that of $\cC$. The category $\cC$ is clearly seen to be equivalent to $\sD_U^{\op}$, via evaluation at source. The category $\sD_U$ has an initial object, given by the empty inclusion $\emptyset\ra U$, hence its classifying space is contractible. 
\end{proof}

The next result uses the following.
\begin{notation}
\label{d.fun.w}
Let $\cK$ be an $(\infty,1)$-category; let $\cW \subset \cC$ be an $\infty$-subcategory of an $(\infty,1)$-category.
We denote the $\infty$-subcategory
\[
\Fun^{\cW}( \cK , \cC)
~\hookrightarrow~
\Fun(\cK,\cC)
\]
which is characterized by the following property.
\begin{itemize}
\item
A natural transformation $F \xra{\eta} G$ belongs to this $\infty$-subcategory if and only if, for each $x\in \cK$, the morphism $F(x) \xra{\eta(x)} G(x)$ in $\cC$ belongs to $\cW$.

\end{itemize}

\end{notation}

\begin{lemma}\label{lemma.points.twar.over}
Let $I$ be a closed oriented interval, and let $U\ra V$ be a morphism of $\sD_I$ for which $V$ contains the closure of $U$: i.e., $U\ra V$ is an object of $\cJ\subset \TwAr(\sD_I)$. For all $[p]\in\bDelta$, the natural functor
\[
\Fun^\cJ\Bigl([p], \TwAr\bigl(({\sD_I}_{/V})^{U\ra V/}\bigr)\Bigr)
\longrightarrow
\Hom\Bigl([p], \TwAr\bigl(({\bDelta^{\op}}_{/\pi V})^{\pi U\ra \pi V/}\bigr)\Bigr)
\]
induces an equivalence between the classifying space of $\Fun^\cJ\Bigl([p], \TwAr\bigl(({\sD_I}_{/V})^{U\ra V/}\bigr)\Bigr)$ and the set of $p$-simplices of $\TwAr\bigl(({\bDelta^{\op}}_{/\pi V})^{\pi U\ra \pi V/}\bigr)$.

\end{lemma}

\begin{proof}
We prove this result by induction on $[p]$. First, note the equivalence $({\sD_I}_{/V})^{U\ra V/} \cong \sD_V^{U/}$.

{\bf Base Case}: We show that 
\[
\Fun^\cJ\Bigl([0], \TwAr\bigl( \sD_V^{U/}\bigr)\Bigr)
\longrightarrow
\obj\Bigl(\TwAr\bigl(({\bDelta^{\op}}_{/\pi V})^{\pi U\ra \pi V/}\bigr)\Bigr)
\]
induces an equivalence between the classifying space of the source and the set of objects of the target. We analyze the following commutative diagram, defined by evaluating twisted arrows at their target:
\[
\xymatrix{
\Fun^{\cJ'}\Bigl([0], \TwAr\bigl(\sD_V^{U/}\bigr)\Bigr)\ar[r]^-{\rm final}\ar[d]_-{\ev_t}^{\rm coCart}
&\Fun^\cJ\Bigl([0], \TwAr\bigl(\sD_V^{U/}\bigr)\Bigr)\ar[d]_-{\ev_t}^{\rm coCart}
\ar[r]&
\obj\Bigl(\TwAr\bigl(({\bDelta^{\op}}_{/\pi V})^{\pi U\ra \pi V/}\bigr)\Bigr)\ar[d]_-{\ev_t}^{\rm coCart}\\
\bigl(\sD_V^{U/}\bigr)^{\cJ'}\ar[r]^-{\rm final}&\bigl(\sD_V^{U/}\bigr)^\cJ\ar[r]&({\bDelta^{\op}}_{/\pi V})^{\pi U\ra \pi V/}}
\]
Here, $\bigl(\sD_V^{U/}\bigr)^\cJ \subset \sD_V^{U/}$ is the subcategory consisting of those morphisms $V' \subset V''$ that are isotopy equivalences; 
$\bigl(\sD_V^{U/}\bigr)^{\cJ'} \subset \bigl(\sD_V^{U/}\bigr)^\cJ $ is the full subcategory consisting of those objects $V'$ such that there is containment of the closure: $\ov{U} \subset V'$.
We claim that the inclusion
\[
\bigl(\sD_V^{U/}\bigr)^{\cJ'}\longrightarrow\bigl(\sD_V^{U/}\bigr)^\cJ
\]
is final, as indicated in the diagram. By Quillen's Theorem A, it suffices to show that for every $V' \in \bigl(\sD_V^{U/}\bigr)^{\cJ}$, the poset
\[
\Bigl(\bigl(\sD_V^{U/}\bigr)^{\cJ'}\Bigr)^{V'/}
\]
has contractible classifying space. 
Explicitly, this poset consists of those $D\subset V$ such that $D$ contains both $V'$ and the closure of $U$, and the inclusion $V'\ra D$ is an isotopy equivalence. 
Using that the ambient space $I$ is a closed interval, this poset has finite nonempty colimits: Given any finite nonempty diagram
\[
G: \cK \ra \Bigl(\bigl(\sD_V^{U/}\bigr)^{\cJ'}\Bigr)^{V'/}
\]
the union $D := \bigcup_{k\in \cK} G(k)$ is again an element of $\Bigl(\bigl(\sD_V^{U/}\bigr)^{\cJ'}\Bigr)^{V'/}$, and is the colimit of the diagram. This implies that the classifying space of $\Bigl(\bigl(\sD_V^{U/}\bigr)^{\cJ'}\Bigr)^{V'/}$ is contractible, as claimed.

We have thus established the finality of the functor $\bigl(\sD_V^{U/}\bigr)^{\cJ'}\longrightarrow\bigl(\sD_V^{U/}\bigr)^\cJ$. Since finality is preserved under pullbacks along coCartesian fibrations (e.g., Corollary~5.15 of \cite{fibrations}), to establish the base case, we have reduced to showing that
\[
\Fun^{\cJ'}\Bigl([0], \TwAr\bigl(\sD_V^{U/}\bigr)\Bigr)
\longrightarrow
\obj\Bigl(\TwAr\bigl(({\bDelta^{\op}}_{/\pi V})^{\pi U\ra \pi V/}\bigr)\Bigr)
\]
induces an equivalence on classifying spaces. 
To do so, it suffices to show first that the map on the bases of the $\ev_t$ coCartesian fibrations
\[
\bigl(\sD_V^{U/}\bigr)^\cJ\longrightarrow({\bDelta^{\op}}_{/\pi V})^{\pi U\ra \pi V/}
\]
is an equivalence on classifying spaces, and second that for every $\ev_t$-fibers over any object $(\ov{U}\subset V' \subset V) \in \bigl(\sD_V^{U/}\bigr)^{\cJ'}$, the map
\[
\ev_t^{-1}\{\ov{U}\subset V' \subset V\} \longrightarrow \ev_t^{-1}\{\pi U\subset \pi V' \subset \pi V\}
\]
is an equivalence on classifying spaces. For the first statement, this follows from the proofs of Proposition~2.19 and Lemma~3.11 from \cite{oldfact}: The only alteration is that we now have an undercategory, and the proofs carry over mutatis mutandis. For the second statement, we can identify the fibers:
\[
\ev_t^{-1}\{\ov{U}\subset V' \subset V\} \cong \bigl(\sD_{V'}^{U/}\bigr)^{\cJ, \op}
\]
and
\[
\ev_t^{-1}\{\pi U\subset \pi V' \subset \pi V\} \cong ({\bDelta^{\op}}_{/\pi V'})^{\pi U\ra \pi V'/,\op}~.
\]
We now have the equivalence on classifying spaces, since this statement is the opposite of that just established by Proposition~2.19 and Lemma~3.11 from \cite{oldfact}, and taking opposites does not change the classifying space.

{\bf Inductive Step}: We assume the equivalence on $[k]$-points for $k<p$. For each choice of an object $(U'\hookrightarrow V')\in \TwAr\bigl((\sD_{I/V})^{U\ra V/}\bigr)$, we have a pullback square:
\[
\xymatrix{
\Fun^\cJ\Bigl([p-1], \TwAr\bigl(({\sD_I}_{/V'})^{U'\ra V'/}\bigr)\Bigr)
\ar[d]\ar[r]&\Fun^\cJ\Bigl([p], \TwAr\bigl(({\sD_I}_{/V})^{U\ra V/}\bigr)\Bigr)
 \ar[d]^-{\ev_p}\\
\ast\ar[r]^-{\langle U'\ra V'\rangle}&\Fun^\cJ\Bigl(\{p\}, \TwAr\bigl(({\sD_I}_{/V})^{U\ra V/}\bigr)\Bigr)
}
\]
where $\ev_p$ is a coCartesian fibration. Note further the equivalences
\[
\TwAr(\sD_I)_{/U\ra V} \cong \TwAr\bigl(((\sD_I)_{/V})^{U\ra V/}\bigr) \cong \TwAr\bigl((\sD_V)^{U/}\bigr) 
\]
and thus the equivalence
\[
\Fun^\cJ([p-1], \TwAr(\sD_I)_{/U\ra V}) \cong \Fun^\cJ([p-1], \TwAr((\sD_V)^{U/}))~.
\]

Consider the diagram
\[
\xymatrix{
\Fun^\cJ\Bigl([p], \TwAr\bigl(({\sD_I}_{/V})^{U\ra V/}\bigr)\Bigr)'\ar[d]^-{\rm coCart}_-{\ev_p}\ar[r]^-{\rm final}&\Fun^\cJ\Bigl([p], \TwAr\bigl(({\sD_I}_{/V})^{U\ra V/}\bigr)\Bigr) \ar[d]_-{\ev_p}^-{\rm coCart}\ar[r]&\ar[d]_-{\ev_p}^-{\rm coCard}\Hom\Bigl([p], \TwAr\bigl(({\bDelta^{\op}}_{/\pi V})^{\pi U\ra \pi V/}\bigr)\Bigr)\\
\Fun^{\cJ'}\Bigl(\{p\}, \TwAr\bigl(({\sD_I}_{/V})^{U\ra V/}\bigr)\Bigr)\ar[r]^-{\rm Lemma~\ref{lemma.I.final}}_-{\rm final}&\Fun^\cJ\Bigl(\{p\}, \TwAr\bigl(({\sD_I}_{/V})^{U\ra V/}\bigr)\Bigr)\ar[r]&\obj(\TwAr(\bDelta^{\op}))}
\]
where the left square is defined as a pullback.
Finality of the lower left horizontal functor follows from Lemma~\ref{lemma.I.final}: Using Quillen's Theorem A, finality of the fully-faithful  inclusion $\cJ' \to \cJ$ (Lemma~\ref{lemma.I.final}) implies finality of the functor between any undercategories for any object in $\cJ$.
Again using that the pullback of a final functor along a coCartesian fibration is again final, it suffices to show that 
\[
\Fun^\cJ\Bigl([p], \TwAr\bigl(({\sD_I}_{/V})^{U\ra V/}\bigr)\Bigr)'\longrightarrow
\Hom\Bigl([p], \TwAr\bigl(({\bDelta^{\op}}_{/\pi V})^{\pi U\ra \pi V/}\bigr)\Bigr)\] induces an equivalence on the classifying space.

Since each $\ev_p$ is a coCartesian fibration and the 
the base case ensures the functor between base categories induces an equivalence between classifying spaces,
we reduce to the statement on the fibers, to showing that for $U\ra V$ belonging to $\cJ'$, then
\[
\Fun^\cJ\Bigl([p-1], \TwAr\bigl(({\sD_I}_{/V'})^{U'\ra V'/}\bigr)\Bigr)
\longrightarrow
\Map([p-1], \TwAr(\bDelta^{\op})_{/\pi V'}^{\pi V' \ra \pi U'/})
\]
induces an equivalence on the classifying space. This equivalence is exactly given by the inductive hypothesis.
\end{proof}

The following is the main result of this section.

\begin{lemma}\label{localize.D.Delta}
Let $I$ be a closed oriented interval. The functor
\[
\TwAr(\sD_I)\overset{\pi}\longrightarrow \TwAr(\bDelta^{\op})
\]
is a localization.
\end{lemma}
\begin{proof}
We apply the following sufficiency condition (Theorem 3.8 of \cite{mazel.gee.localization})
for establishing when a functor $\pi:\cC \ra \cD$ is a localization:
\begin{itemize}\label{condition.bullet}
\item Let $\cJ= \pi^{-1}\cD^\sim$ be the $\oo$-subcategory of $\cC$ which maps to isomorphisms in $\cD$, and let $\Fun^\cJ(\cK,\cC)\subset\Fun(\cK,\cC)$ be the $\oo$-subcategory of functors in which the morphisms are natural transformations contained in $\cJ$. If for each $[p]\in \bDelta$, the natural map
\[
\Fun^\cJ([p],\cC) \longrightarrow \Map([p],\cD)
\]
exhibits the space of $p$-simplices of $\cD$ as the classifying space of $\Fun^\cJ([p],\cC)$, then $\pi:\cC\ra\cD$ is a localization.
\end{itemize}
We apply this condition to the functor $\TwAr(\sD_I)\longrightarrow \TwAr(\bDelta^{\op})$, and prove by induction on $p$ that the functor
\[
\Fun^\cJ([p], \TwAr(\sD_I)) \longrightarrow \Map([p], \TwAr(\bDelta^{\op}))
\]
induces an equivalence from the classifying space of $\Fun^\cJ([p], \TwAr(\sD_I))$.

Consider the following diagram:
\[
\xymatrix{
\Fun^{\cJ}([p], \TwAr(\sD_I))'\ar[r]\ar[d]_-{\ev_p}^-{\rm coCart}&\Fun^\cJ([p], \TwAr(\sD_I)) \ar[r]\ar[d]_-{\ev_p}^-{\rm coCart}& \Map([p], \TwAr(\bDelta^{\op}))\ar[d]_-{\ev_p}^-{\rm coCart}
\\
\cJ' \ar[r]^-{\rm final}&\cJ \ar[r]&\obj\bigl(\TwAr(\bDelta^{\op})\bigr)}
\] 
where $\Fun^{\cJ}([p], \TwAr(\sD_I))'$ is the pullback of $\Fun^{\cJ}([p], \TwAr(\sD_I))$ along $\cJ' \ra \cJ$. By Lemma~\ref{lemma.I.final}, the inclusion $\cJ'\ra \cJ$ is final. Since $\ev_p$ is a coCartesian fibration, and finality is preserved by pullbacks along coCartesian fibrations, then $\Fun^{\cJ}([p], \TwAr(\sD_I))' \ra \Fun^{\cJ}([p], \TwAr(\sD_I))$ is again final, and the classifying spaces of the two categories are equivalent. 

Consequently, it suffices to show that $\cJ \ra \obj\bigl(\TwAr(\bDelta^{\op})\bigr)$ induces an equivalence on classifying spaces, and that for every object $(U\ra V) \in \cJ'$, the functor on the $\ev_p$-fibers
\[
\Fun^\cJ\Bigl([p-1], \TwAr\bigl((\sD_I)_{/U\ra V}\bigr)\Bigr)
\longrightarrow
\Hom\Bigl([p-1], \TwAr\bigl((\bDelta^{\op})_{/\pi U\ra \pi V}\bigr)\Bigr)
\]
induces an equivalence on classifying spaces. This second assertion is given by Lemma~\ref{lemma.points.twar.over}. To show that $\cJ \ra \obj\bigl(\TwAr(\bDelta^{\op})\bigr)$ induces an equivalence on classifying spaces, we can use the coCartesian fibrations defined by $\ev_t$:
\[
\xymatrix{
\cJ\ar[d]^-{\ev_t}\ar[r]&\obj\bigl(\TwAr(\bDelta^{\op})\bigr)\ar[d]^-{\ev_t}\\
(\sD_I)^{\sf istpy}\ar[r]&\obj(\bDelta^{\op})}
\]
where $(\sD_I)^{\sf istpy}\subset \sD_I$ is the sub-poset with morphisms those inclusions which are isotopy equivalence. 
By Proposition~2.19 and Lemma~3.11 from \cite{oldfact}, we have that the functor on the base, $(\sD_I)^{\sf istpy}\ra\obj(\bDelta^{\op})$, induces an equivalence on classifying spaces.
Over $V\in (\sD_I)^{\sf istpy}$, the functor between fibers is identified as $(\sD_V)^{\sf istpy,\op}\ra\obj\bigl((\bDelta^{\op}_{/\pi V})^{\op}\bigr)$, which also induces an equivalence on classifying spaces.
Using that $\ev_t$ is a coCartesian fibration, we conclude that the set $\obj\bigl(\TwAr(\bDelta^{\op}))$ is the classifying space of $\cJ$, which completes the proof.
\end{proof}

\section{The Tangle Hypothesis in ambient dimension 2}

The main result of this section is Theorem~\ref{thm.n2}, that
\[
\Bord_1^{\fr}(\RR^{1})
\]
is the rigid monoidal $(\oo,1)$-category freely generated by a single object. This implies \Cref{t4} in the special case of $n=2$. We prove this result by induction on the number of duals of the distinguished object. The proof of the base case, in which case the object has a single right-dual, occupies the next section.

\subsection{Base case}

For $0< i \leq\oo$, we will say that an object $V$ of a monoidal $(\oo,1)$-category $\cR$ has $i$ right-duals (respectively, left-duals) if $V$ has a right-dual $V^R$, and $V^R$ has $i-1$ right-duals. For $i=0$, the condition is vacuously satisfied.

\begin{definition}
For $0\leq i,j\leq \oo$, $\Dual^{[i,j]} \in \Alg(\Cat_{(\oo,1)})$ corepresents the functor
\[
\Alg(\Cat_{(\oo,1)})\longrightarrow \spaces
\]
which sends a monoidal $(\oo,1)$-category $\cR$ to the subspace of $\Obj(\cR)$ consisting of those objects $V$ such that $V$ admits $i$ left-duals and $j$ right-duals. 
We denote the universal such object $D\in  \Dual^{[i,j]}$.
We denote
\[
\Dual := \Dual^{[0,1]}
~.
\]
\end{definition}

Recall Corollary~\ref{cor.obj.mor.Bord.free}, which in particular identifies the monoidal space of objects 
\[
\obj\bigl(\Bord_1^{\fr}(\RR^1)\bigr)
\simeq
\FF_1(\Omega\RP^1)
\]
with the free $\cE_1$-space generated by the set $\ZZ\simeq \Omega\RR\PP^1$. This allows the following definition.
For $0\leq i,j\leq \infty$, we denote the subset:
\[
[-i,j]
~:=~
\left\{
a\in \ZZ
\mid 
-i \leq a \leq j
\right\}
~\subset~
\ZZ
~.
\]

\begin{definition}\label{def.Bord.ij}
For $0\leq i,j\leq \oo$, the monoidal full $\infty$-subcategory
\[
\Bord^{[i,j]}
~\subset~
\Bord_1^{\fr}(\RR^1)
\]
is that whose space of objects is the subspace
\[
\coprod_{r\geq 0} \Conf_r(\RR^1) \underset{\Sigma_r}{\times}[-i,j]^r
\subset
\coprod_{r\geq 0} \Conf_r(\RR^1) \underset{\Sigma_r}{\times}\ZZ^r
\simeq
\FF_1(\Omega\RP^1)
\simeq
\obj\bigl(\Bord_1^{\fr}(\RR^1)\bigr)~.
\]
A \bit{$[i,j]$-framing} of a tangle $W\subset \RR^1 \times \DD^1$ is a framing of this tangle such that the resulting framed tangle is a morphism in $\Bord^{[i,j]}$.

\end{definition}
In other words, an object in $\Bord^{[i,j]}$ is a finite subset of $\RR^1$ labeled by integers within $[-i,j]$.

We single out two cases of this definition, which again follow from Corollary~\ref{cor.obj.mor.Bord.free}.

\begin{cor}\label{t.bord.obj}
There are canonical equivalences
\[
\FF_1(\ast)\simeq \Dual^{[0,0]}\simeq \Bord^{[0,0]}
\]
and
\[
\FF_1\left( \left\{0,1 \right\} \right)
\simeq
\Obj\left( \Bord^{[0,1]} \right)
\]
where $\FF_1(\ast)$ is the monoidal category freely generated by a single object $\ast$, and $\FF_1(\{0,1\})$ is the monoidal space freely generated by the 2-element set $\{0,1\}$.
\end{cor}
\qed

\begin{lemma}
The object $\left( \{0\} , 0 \right) \in \Bord^{[i,j]}$ has $i$ left-duals and $j$ right-duals.
Equivalently, there is a unique monoidal functor
\begin{equation}
\label{eDualBord}
\Dual^{[i,j]}
\longrightarrow
\Bord^{[i,j]}
~,\qquad
D
\longmapsto
\left( \{0\} , 0 \right) 
~.
\end{equation}
\end{lemma}
\begin{proof}
After Corollary~\ref{cor.Bord.has.duals}, 
$\left( \{0\} , 0 \right) \in \Bord_1^{\fr}(\RR^1)$ has $a$ left-duals and $b$ right-duals for all $0\leq a,b\leq \infty$.
By construction, the the $a^{\rm th}$ left-dual and $b^{\rm th}$ right-dual are contained in $\Bord^{[i,j]}$ if and only if $a\leq i$ and $b\leq j$.
Consequently, $\left( \{0\} , 0 \right) \in\Bord^{[i,j]}$ has $i$ left-duals and $j$ right-duals.
\end{proof}

Consider the \bit{walking adjunction}, which is the $(\infty,2)$-category $\Adj$ corepresenting an adjunction among $(\infty,2)$-categories.
This $(\infty,2)$-category is the subject of the work~\cite{riehl.verity} which establishes the following remarkable features of $\Adj$.
To state it, we recall the following (see Appendix~A of~\cite{trace} for further details).  
\begin{itemize}
\item
$\sO$ is the category in which an object is a finite linearly ordered set and a morphism is an order preserving map between such.
Join of linearly ordered sets defines a monoidal structure on $\sO$. 
This monoidal category is the \bit{walking monad}, which corepresents monads.
Its opposite $\sO^{\op}$ is the \bit{walking comonad}, which corepresents comonads.

\item
$\sO_+$ (respectively $\sO_-$) is the category in which an object is a non-empty finite linearly ordered set and a morphism is an order preserving map between such that preserves maxima (respectively minima).
Join of linearly ordered sets defines a left (respectively right) module structure $\sO \lacts \sO_+$ (respectively $\sO_- \racts \sO$).  

\item
For each $I^{\tr} \in \sO_+$, the category $\Fun^{+\infty/}(I^{\tr} , \{0<+\infty\})$ of maxima-preserving functors is a non-empty finite linearly ordered set, and the functor $\Fun^{+\infty/}(J^{\tr},\{0<+\infty\}) \xra{f^\ast} \Fun^{+\infty/}( I^{\tr} , \{0<+\infty\})$ associated to each morphism $I^{\tr} \xra{f} J^{\tr}$ in $\sO_+$ preserves minima.  
This organizes as a functor
\[
\sO_+^{\op}
\xra{~\Fun^{+\infty/}(-,\{0<+\infty\})~}
\sO_-
~,
\]
which can be checked to be an equivalence (an inverse is given by $\Fun_{/-\infty}( - , \{-\infty<0\})$.

\end{itemize}
\begin{theorem}[\cite{riehl.verity}, \cite{schanuel.street}]\label{t.Adj}
The $(\infty,2)$-category $\Adj$ admits the following description.
\begin{itemize}
\item
Its space of objects is canonically identified as
\[
\Obj(\Adj)
~\simeq~
\{-,+\}
~,
\]
a 0-type with two path-components.

\item
Its monoidal categories of endomorphisms of each of its two objects are canonically identified as the walking monad and the walking comonad:
\[
\un{\End}_{\Adj}(-)
\underset{\simeq}{\xla{~(\sR\circ \sL)^{\circ I} ~\mapsfrom~ I~}}
\sO
\qquad
\text{ and }
\qquad
\un{\End}_{\Adj}(+)
\underset{\simeq}{\xla{~(\sL\circ \sR)^{\circ I} ~\mapsfrom~ I~}}
\sO^{\op}
~.
\]

\item
Its $\bigl(\un{\End}_{\Adj}(-), \un{\End}_{\Adj}(+)\bigr)$-bimodule category of morphisms from $-$ to $+$, 
and its $\bigl(\un{\End}_{\Adj}(-), \un{\End}_{\Adj}(+)\bigr)$-bimodule category of morphisms from $+$ to $-$, are respectively
\[
\un{\Hom}_{\Adj}(-,+)
\underset{\simeq}{\xla{~\sL \circ (\sR \circ \sL)^{\circ J} ~\mapsfrom~ J^{\tl}~}}
\sO_-~ \simeq ~\sO_+^{\op}
\qquad
\text{ and }
\qquad
\un{\Hom}_{\Adj}(+,-)
\underset{\simeq}{\xla{~(\sR \circ \sL)^{\circ J} \circ \sR ~\mapsfrom~ J^{\tr}~}}
\sO_+ ~\simeq ~\sO_-^{\op}
~.
\]

\end{itemize}
In particular, the $(\infty,2)$-category $\Adj$ is gaunt, and its maximal $(\infty,1)$-category $\Adj_{\leq 1} \subset \Adj$ is the free $(\infty,1)$-category on the directed graph
\[\begin{tikzcd}
	{-} &&& {+}
	\arrow["{\sf L}", shift left=3, from=1-1, to=1-4]
	\arrow["{\sf R}", shift left=2, from=1-4, to=1-1]
\end{tikzcd}
~.
\]

\end{theorem}

The last statement of Theorem~\ref{t.Adj} supplies an identification between $(\infty,1)$-categories
\begin{equation}
\label{e.graph}
\ast
\underset{\Obj(\Adj} \coprod
\Adj_{\leq 1}
\simeq
\fB \FF_1\left( \{0,1\} \right)
~.
\end{equation}

\begin{observation}
\label{tAdjDual}
The universal properties of $\Dual$, and the walking adjunction $\Adj$, are such that they participate in a pushout diagram among $(\infty,2)$-categories:
\begin{equation}
\label{eAdjDual}
\xymatrix{
\Obj(\Adj)
\ar[rr]
\ar[d]
&&
\Adj
\ar[d]
\\
\ast
\ar[rr]
&&
\fB \Dual
~.
}
\end{equation}
In particular, 
the identification~(\ref{e.graph}) supplies a morphism between monoidal $(\infty,1)$-categories from the free monoidal category on the 2-element set $\{0,1\}$,
\[
\FF_1\left( \{0,1\} \right)
\longrightarrow
\Dual
~,
\]
that is essentially surjective.  
\end{observation}

\begin{remark}
Therefore, in spite of $\fB \Dual$ fitting into a pushout involving gaunt $(\infty,2)$-categories, it may not be gaunt -- indeed, the gauntness condition on an $(\infty,2)$-category is not preserved under the formation of pushouts. 
This issue complicates direct ``object-morphism''-style approaches to proving Proposition~\ref{tDualBord}.
Gauntness will, however, follow from the gauntness of $\Bord^{[0,1]}$, implied by Proposition~\ref{prop.BBord.gaunt}, through Proposition~\ref{tDualBord}.

\end{remark}

\begin{definition}
\label{dd2}
Let $A$ be a set.
\begin{enumerate}
\item A \bit{word (in the alphabet $A$)} is a finite linearly ordered set $(I,\leq)$ together with a map between sets $I \xra{\omega} A$.  
When no confusion can arise, we often denote such a word $\left( \left( I, \leq \right) , \omega \right)$ simply as $(I,\omega)$, or as $\omega$.

\item
Say a word $(I,\omega)$ is the \bit{empty word} if $I = \emptyset$.

\item
Say a word $(I,\omega)$ is \bit{alternating} if, for each non-maximal element $i\in I \smallsetminus {\sf Max}(I)$, the two values $\omega(i) \neq \omega\left( {\sf suc}(i) \right)$ do not agree where ${\sf suc}(i)\in I$ is the successor of $i$ in $I$.

\item
For $J$ a finite linearly ordered set, and for $(I_j,\omega_j)_{j\in J}$ a $J$-indexed sequence of words, their \bit{($J$-fold) concatenation} is the word 
\[
\underset{j\in J} \bigstar \omega_j
~:=~
\Bigl(
\underset{j\in J} \bigstar I_j
~,~
\underset{j\in J} \coprod I_j 
\xra{ ( \omega_j)_{j\in J} }
A
\Bigr)
\]
whose underlying linearly ordered set is the $J$-fold join.
For $J = \{1<\dots<r\}$, this will often be denoted simply as
\[
\omega_r \cdots \omega_1
~.
\]

\item
An isomorphism between two words $(I,\omega) \cong (I' , \omega')$ is an isomorphism between linearly ordered sets $I \xra{\alpha}I'$ such that $\omega' \circ \alpha = \omega$.  

\item
The \bit{space of words (in the alphabet $A$)} is the groupoid of words (in the alphabet $A$) and isomorphisms between such.  

\end{enumerate}

\end{definition}

\begin{observation}
\label{s4}
Let $A$ be a set.
The space of words in the alphabet $A$ is a 0-type, and concatenation of words defines a monoidal structure on this space.
Consequently, the space of words in the alphabet $A$ is the free $\cE_1$-algebra on the space $A$.

\end{observation}

\begin{observation}
\label{s2}
After Corollary~\ref{t.bord.obj}, the monoid of objects in $\Bord^{[0,1]}$ is that of words in the alphabet $\{0,1\}$.
\end{observation}

\subsection{Proof of $\Dual \simeq \Bord^{[0,1]}$}

This subsection is devoted to a proof of the following.
\begin{prop}
\label{tDualBord}
The monoidal functor 
\[
\Dual \xra{~(\ref{eDualBord})~} \Bord^{[0,1]}
\]
is an equivalence between monoidal $(\infty,1)$-categories.

\end{prop}

After Observation~\ref{tAdjDual}, Corollary~\ref{t.bord.obj} yields the following.
\begin{cor}
\label{s20}
There is a canonical lift among $(\infty,2)$-categories:
\[
\xymatrix{
&&
\fB \Dual
\ar[d]^-{(\ref{eDualBord})}
\\
\fB \Obj\left( \Bord^{[0,1]} \right)
\ar[rr]^-{\rm inclusion}
\ar@{-->}[urr]
&&
\fB \Bord^{[0,1]}
.
}
\]

\end{cor}

The next technical result supplies lifts of 1-morphisms along the monoidal functor $\Dual \xra{(\ref{eDualBord})} \Bord^{[0,1]}$.

\begin{lemma}
\label{s.1}
Consider a solid diagram $(\infty,1)$-categories,
% https://egretwalker.github.io/quivermore/#q=WzAsNCxbMCwwLCJub25lIiwiXFxwYXJ0aWFsIGNeMiJdLFsxLDAsIm5vbmUiLCJcXGNCIFxcRHVhbCJdLFsxLDEsIm5vbmUiLCJcXGZCIFxcQm9yZF8xXntbMCwxXX0oXFxSUl4xKSJdLFswLDEsIm5vbmUiLCJjXjIiXSxbMCwzLCJub25lIiwie1xccm0gaW5jbHVzaW9ufSIsMix7InN0eWxlIjp7InRhaWwiOnsibmFtZSI6Imhvb2siLCJzaWRlIjoidG9wIn19fV0sWzAsMSwibm9uZSJdLFszLDIsIm5vbmUiXSxbMSwyLCJub25lIiwiKFxccmVme2VEdWFsQm9yZH0pIl0sWzMsMSwibm9uZSIsIiIsMCx7InN0eWxlIjp7ImJvZHkiOnsibmFtZSI6ImRhc2hlZCJ9fX1dXQ==
\begin{equation}
\label{f100}
\begin{tikzcd}
	& {\Dual} \\
	{c_1} & {\Bord^{[0,1]}}
	\arrow[from=2-1, to=2-2]
	\arrow["{(\ref{eDualBord})}", from=1-2, to=2-2]
	\arrow[dashed, from=2-1, to=1-2]
\end{tikzcd}
~.
\end{equation}
Via Corollary~\ref{t.bord.obj}, each of the composite morphisms
\[
\omega_s
\colon
c_0
\xra{~s~}
c_1
\longrightarrow
\Bord^{[0,1]}
\qquad
\text{ and }
\qquad
\omega_t
\colon
c_0
\xra{~t~}
c_1
\longrightarrow
\Bord^{[0,1]}
\]
is a word in the alphabet $\{0,1\}$.
There is a canonical filler in~(\ref{f100}) provided both $\omega_s$ and $\omega_t$ are alternating (and possibly empty), and provided the following conditions are satisfied.
\begin{enumerate}

\item
If neither $\omega_s$ nor $\omega_t$ is empty, then their extremal letters agree.  

\item
If $\omega_s$ is empty and $\omega_t$ is not empty, then the first letter of $\omega_t$ is $0$ and the last letter of $\omega_t$ is $1$.

\item
If $\omega_t$ is empty and $\omega_s$ is not empty, then the first letter of $\omega_s$ is $1$ and the last letter of $\omega_s$ is $0$.

\end{enumerate}
\end{lemma}

\begin{proof}

Observation~\ref{tAdjDual} supplies a composite functor
\begin{equation}
\label{f.1}
\Adj
\longrightarrow
\fB \Dual
\xra{~(\ref{eDualBord})~}
\fB \Bord_1^{\sf fr}(\RR^1)
~.
\end{equation}
Inspection of the maximal $(\infty,1)$-subcategory of $\Adj$ (Theorem~\ref{t.Adj}) supplies a unique lift of $\partial c_2 \hookrightarrow c_2 \to \fB \Bord_1^{\sf fr}(\RR^1)$ along (\ref{f.1}).  
Inspection of the space of 2-morphisms of $\fB \Bord_1^{\sf fr}(\RR^1)$ (Lemma~\ref{cor.obj.mor.Bord.free}) and the space of 2-morphisms of $\Adj$ (Theorem~\ref{t.Adj}) supplies a further unique lift of $c_2 \to \fB \Bord_1^{\sf fr}(\RR^1)$ along~(\ref{f.1}).  
This unique lift determines the canonical lift of the statement of the lemma.

\end{proof}

The remainder of this subsection is a proof of Proposition~\ref{tDualBord}.
We prove the result by constructing a filler in this diagram among monoidal $(\infty,1)$-categories:
% https://q.uiver.app/#q=WzAsNSxbMSwwLCJ7XFxzZiBGcmVlfV97XFxjRV8xfShcXGFzdCkiXSxbMiwxLCJcXER1YWwiXSxbMiwyLCJcXEJvcmQiXSxbMCwyLCJcXEJvcmQiXSxbMCwxLCJcXER1YWwiXSxbMCwxXSxbMSwyLCIoXFxyZWZ7ZUR1YWxCb3JkfSkiLDFdLFswLDRdLFs0LDMsIihcXHJlZntlRHVhbEJvcmR9KSIsMV0sWzMsMiwiPSIsMV0sWzMsMSwiRiIsMSx7InN0eWxlIjp7ImJvZHkiOnsibmFtZSI6ImRhc2hlZCJ9fX1dXQ==
\begin{equation}
\label{efiller}
\begin{tikzcd}
	& {\FF_1(\ast)} \\
	\Dual && \Dual \\
	\Bord^{[0,1]} && \Bord^{[0,1]}
	\arrow[from=1-2, to=2-3]
	\arrow["{(\ref{eDualBord})}", from=2-3, to=3-3]
	\arrow[from=1-2, to=2-1]
	\arrow["{(\ref{eDualBord})}", swap, from=2-1, to=3-1]
	\arrow["{=}", from=3-1, to=3-3]
	\arrow["L", dashed, from=3-1, to=2-3]
\end{tikzcd}
~.
\end{equation}
The universal property of $\Dual$ will imply that the composite morphism $\Dual \xra{ L \circ (\ref{eDualBord}) } \Dual$ is the identity provided that the functor $L$ is such that the upper square of this diagram is commutative.

Via~(\ref{e.deloop}), the diagram~(\ref{efiller}) among monoidal $(\infty,1)$-categories can be codified in terms of its deloop, as a diagram among pointed $(\infty,2)$-categories.
Via the presentation of $(\infty,2)$-categories as presheaves on $\bTheta_2$ established in \cite{rezk-n}, the diagram~(\ref{efiller}) can be codified as a diagram among presheaves on $\bTheta_2$.
Via the straightening-unstraightening equivalence ${\sf RFib}_{\bTheta_2} \simeq \PShv(\bTheta_2)$ of \S2 of \cite{HTT}, the diagram~(\ref{efiller}) can in turn be codified as the following diagram among right fibrations over $\bTheta_2$:
\begin{equation}
\label{ethetas}
% https://q.uiver.app/#q=WzAsNSxbMSwwLCIoXFxiVGhldGFfMilfey9cXGZCe1xcc2YgRnJlZX1fe1xcY0VfMX0oXFxhc3QpfSJdLFsyLDEsIihcXGJUaGV0YV8yKV97L1xcZkIgXFxEdWFsfSJdLFsyLDIsIihcXGJUaGV0YV8yKV97L1xcZkIgXFxCb3JkfSJdLFswLDIsIihcXGJUaGV0YV8yKV97L1xcZkIgXFxCb3JkfSJdLFswLDEsIihcXGJUaGV0YV8yKV97L1xcZkIgXFxEdWFsfSJdLFswLDFdLFsxLDIsIihcXGJUaGV0YV8yKV97L1xcZkIgKFxccmVme2VEdWFsQm9yZH0pfSJdLFswLDRdLFs0LDMsIihcXGJUaGV0YV8yKV97L1xcZkIgKFxccmVme2VEdWFsQm9yZH0pfSIsMl0sWzMsMiwiPSIsMl0sWzMsMSwiRiIsMSx7InN0eWxlIjp7ImJvZHkiOnsibmFtZSI6ImRhc2hlZCJ9fX1dXQ==
\begin{tikzcd}
	& {(\bTheta_2)_{/\fB\FF_1(\ast)}} \\
	{(\bTheta_2)_{/\fB \Dual}} && {(\bTheta_2)_{/\fB \Dual}} \\
	{(\bTheta_2)_{/\fB \Bord^{[0,1]}}} && {(\bTheta_2)_{/\fB \Bord^{[0,1]}}}
	\arrow[from=1-2, to=2-3]
	\arrow["{(\bTheta_2)_{/\fB (\ref{eDualBord})}}", from=2-3, to=3-3]
	\arrow[from=1-2, to=2-1]
	\arrow["{(\bTheta_2)_{/\fB (\ref{eDualBord})}}"', from=2-1, to=3-1]
	\arrow["{=}"', from=3-1, to=3-3]
	\arrow["(\bTheta_2)_{/L}", dashed, from=3-1, to=2-3]
\end{tikzcd}
~.
\end{equation}

\subsubsection{$(\bTheta_2)_{/L}$ on 1-categories}
Consider the full $\infty$-subcategories $\bTheta_2 \subset \Cat_{(\oo,2)}\supset \Cat_{(\oo,1)}$.
Their intersection is the simplex category: $\bDelta \simeq \bTheta_2 \cap \Cat_{(\oo,1)}$.
In other words, an object $T\in \bTheta_2$ is a 1-category if and only if $T = [p]$ for some $p\geq 0$.
Consider, then, the full subcategory
$\bDelta_{/\fB \Bord^{[0,1]}} \subset (\bTheta_2)_{/\fB \Bord^{[0,1]}}$ consisting of those $\left(T \to  \fB \Bord^{[0,1]}\right)$ such that $T$ is a 1-category (i.e., $T=[p]$ for some $p\geq 0$).
We define the lift $(\bTheta_2)_{/L}$ over this full subcategory to be that given by Corollary~\ref{s20}.

\subsubsection{$(\bTheta_2)_{/L}$ on 2-cells}
\label{sec.2.cell}
We now explicitly construct the value of $(\bTheta_2)_{/L}$ on each object $\left( T \xra{\lag (W,\varphi) \rag} \fB \Bord^{[0,1]} \right)$ in which $T=c_2 = [1]([1])$ is the 2-cell.
So choose such an object $\left( c_2 \xra{\lag (W ,\varphi)\rag} \fB \Bord^{[0,1]} \right)$, which is a tangle $W \subset \DD^2$ in the hemispherically stratified 2-disk equipped with a $[0,1]$-framing.
(As we proceed, we omit the framing $\varphi$ from our notation.)

Let us overview our logic.  
Consider the composite morphism between flagged $(\infty,2)$-categories $\partial c_2 \to c_2 \xra{\lag W \rag} \fB \Bord^{[0,1]}$.
Through Corollary~\ref{t.bord.obj}, this selects a pair of words in the alphabet $\{0,1\}$.
Optimistically, these two words are alternating, for in that case Lemma~\ref{s.1} could be applied to achieve the value of $(\bTheta_2)_{/L}$ on this object $\left( c_2 \xra{\lag W \rag} \fB \Bord^{[0,1]} \right)$.
However, of course, these two words need not be alternating.
Our strategy, therefore, is to consider the minimal factorization of these two words by alternating words.
We then minimally factor the 2-cell, commensurately with this factorization of its boundary, as a 2-categorical composition of 2-cells.  
We then lift each of these constituent 2-cells via Corollary~\ref{t.bord.obj}, which supplies a lift of the original 2-cell.  
This logic can be organized as the existence of a flagged $(\infty,2)$-category $G(\lag W \rag)$ fitting into a diagram among flagged $(\infty,2)$-categories:
\begin{equation}
\label{f3}
% https://egretwalker.github.io/quivermore/#q=WzAsNCxbMiwwLCJub25lIiwiXFxmQiBcXER1YWwiXSxbMiwyLCJub25lIiwiXFxmQiBcXEJvcmRfMV57WzAsMV19KFxcUlJeMSkiXSxbMCwyLCJub25lIiwiVCJdLFsxLDEsIm5vbmUiLCJHXFxsZWZ0KCBcXGxlZnRcXGxhZyBXIFxccmlnaHRcXHJhZyBcXHJpZ2h0KSJdLFswLDEsIm5vbmUiLCIoXFxyZWZ7Pz8gZHVhbCB0byBib3JkfSkiXSxbMiwzLCJub25lIl0sWzMsMCwibm9uZSJdLFsyLDEsIm5vbmUiLCJcXGxhZyBXIFxccmFnIiwyXSxbMywxLCJub25lIiwiIiwxLHsic3R5bGUiOnsiYm9keSI6eyJuYW1lIjoiZGFzaGVkIn19fV1d
\begin{tikzcd}
	&& {\fB \Dual} \\
	& {G\left( \left\lag W \right\rag \right)} \\
	c_2 && {\fB \Bord^{[0,1]}}
	\arrow["{(\ref{eDualBord})}", from=1-3, to=3-3]
	\arrow["q", from=3-1, to=2-2]
	\arrow[from=2-2, to=1-3]
	\arrow["{\lag W \rag}"', from=3-1, to=3-3]
	\arrow[from=2-2, to=3-3]
\end{tikzcd}
~.
\end{equation}
We then take the value of $(\bTheta_2)_{/L}$ on the object $\left( c_2 \xra{\lag W \rag} \fB \Bord^{[0,1]} \right) \in (\bTheta_2)_{/\fB \Bord^{[0,1]}}$ to be the composite diagonal morphism: $\left( T \to G(\lag W \rag) \to \fB \Dual \right) \in (\bTheta_2)_{/\fB \Dual}$.

Recall that the cellular realization of the 2-cell, $\lag c_2 \rag = \DD^2$, is the hemispherically stratified 2-disk, equipped with its solid 2-framing.

By definition of $\Bord^{[0,1]}$, the morphism $c_2 \xra{\lag W \rag} \fB \Bord^{[0,1]}$ is a $[0,1]$-framed tangle $W \subset \DD^2$.  
See Figure~\ref{fig9}.

\begin{figure}[H]
  \includegraphics[width=\linewidth, trim={0 {0in} {0in} {0in}}, clip]{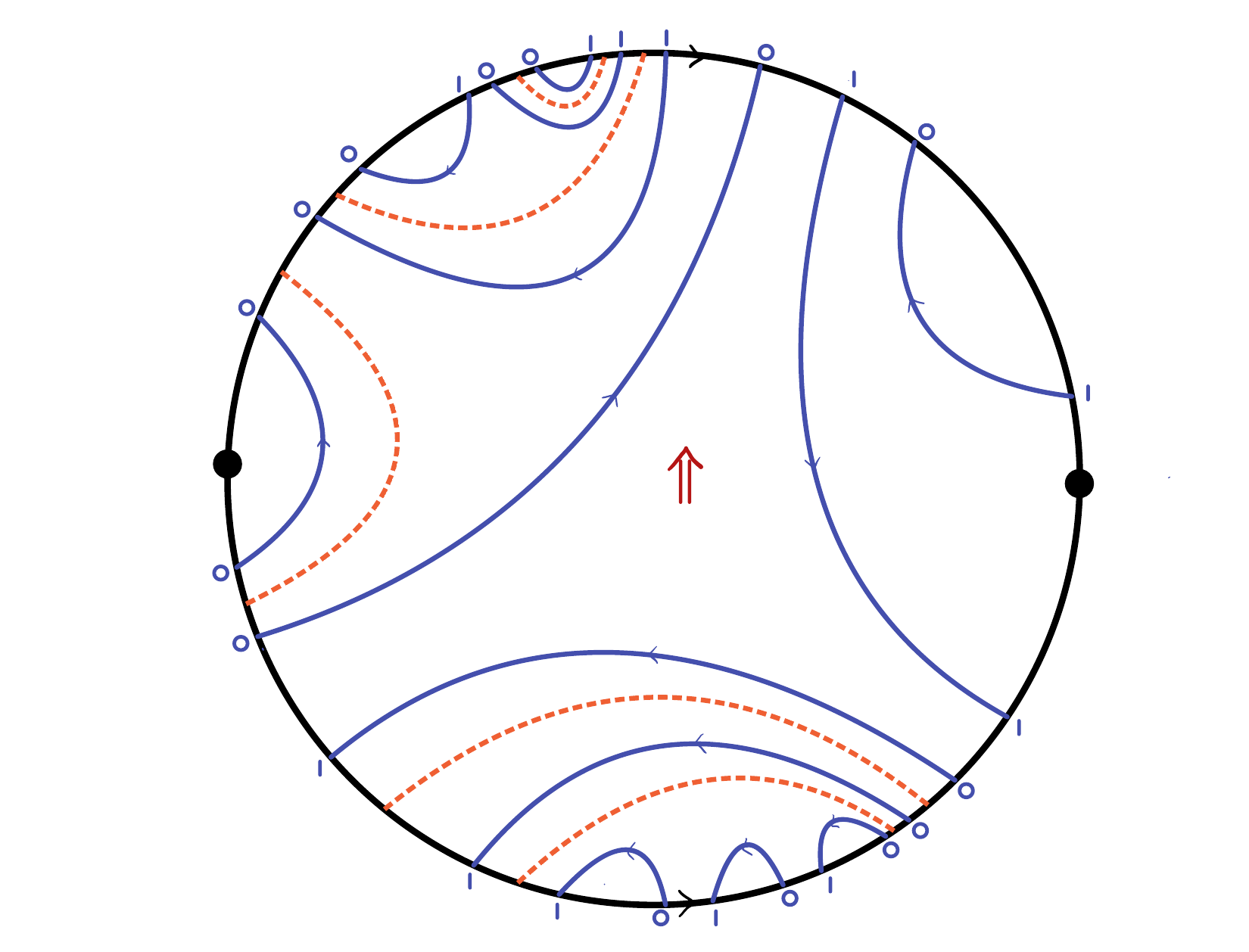}
  \caption{Depicted (in solid blue) is a $[0,1]$-framed tangle in the (solid black) hemispherically stratified 2-disk equipped with the (red) solid 2-framing.  This datum is a 2-cell $c_2 \xra{\lag W \rag} \fB \Bord^{[0,1]}$.  Also depicted is the (dashed orange) union $A$ of arcs.}
  \label{fig9}
\end{figure}

As {\bf Step 1}, we construct a finite disjoint union of arcs $A\subset \DD^2 \smallsetminus W$ with the following features.
\begin{itemize}
\item
The end-points of each arc in $A$ are contained in the boundary of $\DD^2$.

\item
Recall that the $[0,1]$-framing of $W \subset \DD^2$ determines two words $W \cap \partial_\pm \DD^2$ in the alphabet $\{0,1\}$.
The set of end-points of the arcs in $A$ minimally factors these words into alternating words.

\end{itemize}
Constructing $A$ involves several case-by-case analyses, 
ultimately exploiting that the closure of each connected component of $\DD^2 \smallsetminus W$ is a 2-disk.

The $[0,1]$-framing of $W\subset \DD^2$ determines an orientation on $W$ with the following property and implication.
\begin{itemize}
\item[($\dagger$)]
Let $W_\alpha \subset W$ be a connected component.
If either $W_\alpha \cap \partial_+ \DD^2 = \emptyset$ or $W_\alpha \cap \partial_- \DD^2 = \emptyset$, then $W_\alpha \subset \DD^2$ is oriented-isotopic relative to its boundary to a horizontal straight line oriented opposite to the first coordinate direction.

\item[($\dagger\dagger$)]
Each connected component of $W$ is a (closed) arc whose set of end-points is a subset of $(\DD^2)^{(1)}$, the 1-dimensional strata of the hemispherically stratified (closed) 2-disk.
In particular, each connected component of $W$ separates $\DD^2$.
Say a connected component $W_\alpha \subset W$ \bit{crosses} if $W_\alpha \cap \partial_- \DD^2 \neq \emptyset \neq  W_\alpha \cap \partial_+ \DD^2$, which is to say its end-points do not both belong to the same hemisphere of $\partial \DD^2$.

\end{itemize}
Let $C \subset \DD^2 \smallsetminus W$ be a connected component of the complement of this tangle.
Its closure $\ov{C} \subset \DD^2 \subset \RR^2$ is a compact subspace.
Consider the stratification of $\ov{C}$ such that, for each $0\leq i \leq 2$, each connected component of the intersection $C \cap (\DD^2)^{(i)}$ and each connected component of the intersection $\ov{C} \cap W \cap (\DD^2)^{(i)}$ is a stratum of $\ov{C}$, where $(\DD^2)^{(i)}$ is the $i$-dimensional strata of $\DD^2$.
This stratified space $\ov{C}$ is a compact 2-manifold with corners;
the closure of each 1-dimensional face of $\ov{C}$ is isomorphic with $\DD^1$, as a 1-manifold with boundary.
The closure of some of its 1-dimensional faces are connected components of $W$.
Also, $\partial \ov{C} \smallsetminus W = \partial C$ is the boundary of $C$, as a 2-manifold with corners.  
This boundary $\partial C \subset \partial \DD^2$ is an open subspace of the boundary of the 2-disk, as a 2-manifold with corners.  
Each 1-dimensional face in $\partial C$ is a connected open interval of the upper-hemisphere $\partial_+ \DD^2$ or lower-hemisphere $\partial_- \DD^2$ of the boundary $\partial \DD^2$.
Therefore, the finite set of 1-dimensional faces of $C$ admits a partition indexed by the set $\{+,-\}$ -- corresponding to which hemisphere a face belongs to -- and each partition inherits a linear order from the $1^{st}$ coordinate direction of the solid 2-framing of $\DD^2$.
Note that $\ov{C}$ is isomorphic, as a manifold with corners, with a convex subspace of $\RR^2$.

Now, say a connected component $C\subset \DD^2 \smallsetminus W$ is \bit{coherent} if there exists an orientation of the boundary topological 1-manifold $\partial \overline{C}$ that extends the orientation of $W \cap \partial \ov{C} \subset W$. 
Say such a $C$ is \bit{incoherent} if it is not coherent.  
Let $C \subset \DD^2 \smallsetminus W$ be an incoherent connected component.
Choose a (closed) arc $A_C \subset C$ with the following properties.  
See Figure~\ref{fig9}.

\begin{itemize}

\item[(A.1)]
The arc separates: The complement $C \smallsetminus A_C$ has exactly two connected components.
Necessarily, then, the set of end-points $\partial A_C \subset \partial \DD^2$ are in the boundary of the 2-disk.  

\item[(A.2)]
There is containment $S^0 \cap C \subset \partial A_C$.
In other words, each 0-dimensional stratum of the hemispherically stratified 2-disk $\DD^2$ that is contained in $C$ is an end-point of the arc $A_C$.

\item[(A.3)]
There exists an orientation on $\partial \ov{C} \smallsetminus A_C$, which is the boundary of $\ov{C}$ remove the set of end-points of the arc $A_C$, that extends the orientation of $W \cap \partial (\ov{C} \smallsetminus A_C) \subset W$.

\end{itemize}
We now argue that such an arc $A_C \subset C$ exists.
Note that $C$ incoherent implies $W \neq \emptyset$ is not empty.  
Fix an embedding $\ov{C} \xra{\varphi} \RR^2$ whose image in $\RR^2$ is a convex 2-submanifold with corners.
Isotope $\varphi$ as needed to ensure that the unique affine arc between two points each in distinct faces of $\ov{C}$ only intersects the boundary of $\ov{C}$ at its end-points.
Any such arc separates $C$.  
We now construct $A_C$ through cases.
In each case, the named $A_C$ will satisfy the desired property in light of ($\dagger$).\begin{itemize}
\item
Suppose $S^0 \cap C = S^0$.
Then take $A_C$ to be the preimage by $\varphi$ of the unique affine arc whose set of end-points is $\varphi(S^0)$.  

\item
Suppose $S^0 \cap C = \{e_1\}$, the singleton consisting of the first standard basis vector of $\RR^2$.
Then there is exactly one connected component $W_\alpha \subset \ov{C} \cap W$ that crosses.
Denote $\{p_\pm\} := W_\alpha \cap \partial_\pm \DD^2$.
\begin{itemize}
\item
Suppose $W_\alpha$ is oriented such that the orientation is inward at $p_-$ and outward at $p_+$.
In this case, select an element $x$ in the 1-dimensional face of $C \cap \partial_- \DD^2$ whose closure contains $p_+$.
Take $A_C$ to be the preimage by $\varphi$ of the unique affine arc whose set of end-points is $\varphi(\{x,e_1\})$.

\item
Suppose $W_\alpha$ is oriented such that the orientation is outward at $p_-$ and inward at $p_+$.
In this case, select an element $x$ in the 1-dimensional face of $C \cap \partial_- \DD^2$ whose closure contains $p_-$.
Take $A_C$ to be the preimage by $\varphi$ of the unique affine arc whose set of end-points is $\varphi(\{x,e_1\})$.  

\end{itemize}

\item
Suppose $S^0 \cap C = \{-e_1\}$, the singleton consisting of the negation of the first standard basis vector of $\RR^2$.
Then there is exactly one connected component $W_\alpha \subset \ov{C} \cap W$ that crosses.
Denote $\{p_\pm\} := W_\alpha \cap \partial_\pm \DD^2$.\begin{itemize}
\item
Suppose $W_\alpha$ is oriented such that the orientation is inward at $p_-$ and outward at $p_+$.
In this case, select an element $y$ in the 1-dimensional face of $C \cap \partial_- \DD^2$ whose closure contains $p_-$.
Take $A_C$ to be the preimage by $\varphi$ of the unique affine arc whose set of end-points is $\varphi(\{-e_1,y\})$.  

\item
Suppose $W_\alpha$ is oriented such that the orientation is outward at $p_-$ and inward at $p_+$.
In this case, select an element $y$ in the 1-dimensional face of $C \cap \partial_- \DD^2$ whose closure contains $p_+$.
Take $A_C$ to be the preimage by $\varphi$ of the unique affine arc whose set of end-points is $\varphi(\{-e_1,y\})$.  

\end{itemize}

\item
Suppose $S^0 \cap C = \emptyset$.
In this case, there are either exactly 0 or 2 connected components of $W \cap \ov{C}$ that cross.  
\begin{itemize}
\item
Suppose no connected component of $W \cap \ov{C}$ crosses.
Using the supposition that $S^0 \cap C = \emptyset$, then every connected component of $W \cap \ov{C}$ has end-points in the same hemisphere.  
Assume, without loss in generality, that this hemisphere is $\partial_+ \DD^2$.  
The supposition that $C$ is incoherent ensures the boundary $\partial C$ has more than one connected component.  
In the linear order on the set of connected components of the boundary $\partial C$, select an element $x$ in the minimal connected component and an element $y$ in the maximal connected component.  
Take $A_C$ to be the preimage by $\varphi$ of the unique affine arc whose set of end-points is $\varphi(\{x,y\})$.  

\item
Suppose exactly two connected components $W_\alpha , W_\beta \subset W \cap \ov{C}$ cross.
Denote $\{p_\pm\} := W_\alpha \cap \partial_\pm \DD^2$ and $\{q_\pm\}:= W_\beta \cap \partial \DD^2$.  
\begin{itemize}
\item
Suppose the orientation of $W_\alpha$ is inward at $p_-$ and the orientation of $W_\beta$ is inward at $q_-$.
In the linear order on the set of connected components of $\partial C \cap \partial_- \DD^2$, select an element $x$ in the maximal connected component.
In the linear order on the set of connected components of $\partial C \cap \partial_+ \DD^2$, select an element $y$ in the minimal connected component.
Take $A_C$ to be the preimage by $\varphi$ of the unique affine arc whose set of end-points is $\varphi(\{x,y\})$.  

\item
Suppose the orientation of $W_\alpha$ is inward at $p_-$ and the orientation of $W_\beta$ is inward at $q_+$.
In the linear order on the set of connected components of $\partial C \cap \partial_+ \DD^2$, select an element $x$ in the minimal connected component and an element $y$ in the maximal connected component.
Take $A_C$ to be the preimage by $\varphi$ of the unique affine arc whose set of end-points is $\varphi(\{x,y\})$.  

\item
Suppose the orientation of $W_\alpha$ is inward at $p_+$ and the orientation of $W_\beta$ is inward at $q_-$.
In the linear order on the set of connected components of $\partial C \cap \partial_- \DD^2$, select an element $x$ in the minimal connected component and an element $y$ in the maximal connected component.
Take $A_C$ to be the preimage by $\varphi$ of the unique affine arc whose set of end-points is $\varphi(\{x,y\})$.  

\item
Suppose the orientation of $W_\alpha$ is inward at $p_+$ and the orientation of $W_\beta$ is inward at $q_+$.
In the linear order on the set of connected components of $\partial C \cap \partial_- \DD^2$, select an element $x$ in the minimal connected component.
In the linear order on the set of connected components of $\partial C \cap \partial_+ \DD^2$, select an element $y$ in the maximal connected component.
Take $A_C$ to be the preimage by $\varphi$ of the unique affine arc whose set of end-points is $\varphi(\{x,y\})$.  

\end{itemize}

\end{itemize}

\end{itemize}
Using that the map $\sO(k) \xra{\simeq} \Diff(\DD^k)$ is an equivalence for $0\leq k \leq 2$ (\cite{smale.2.sphere}), the moduli space of such arcs $A_C \subset C$ is contractible.

The containment $A_C \subset C$ for each incoherent connected component of $\DD^2\smallsetminus W$ implies $A_C$ and $A_{C'}$ are disjoint if $C$ and $C'$ are distinct.
Take $A$ to be the disjoint union of arcs, one for each incoherent connected component of $\DD^2 \smallsetminus W$:
\[
A
~:=~
\underset{C ~{\rm incoherent}} \bigsqcup A_C 
~\subset~
\DD^2
~,
\]
which is a properly embedded 1-dimensional submanifold with boundary of $\DD^2$.
This concludes Step 1.

As {\bf Step 2}, we use the union of arcs $A \subset \DD^2 \smallsetminus W$ to construct the flagged $(\infty,2)$-category $G(\lag W \rag)$.
As {\bf Step 2a}, we construct an $(\infty,1)$-category the deloop of the free monoid generated by the factors in the alternating factorization of the words $W \cap \partial_\pm \DD^2$ in the alphabet $\{0,1\}$.
As {\bf Step 2b}, we define $G(\lag W \rag)$ by freely adjoining finitely many 2-cells to this $(\infty,1)$-category.

The complement $\partial_\pm \DD^2 \smallsetminus A \subset \partial_\pm \DD^2$ is a finite union of open arcs.  
Consider the finite sets
\[
E_\pm
~:=~
\pi_0\left(
\partial_\pm \DD^2 \smallsetminus A 
\right)
\qquad
\text{ and }
\qquad
E
~:=~
E_- \amalg E_+
\]
of such open arcs.  
Consider the $(\infty,1)$-category $\fB \FF_1(E)$, which is the deloop of the free monoid generated by the finite set $E$. 
This concludes Step 2a.

We now implement Step 2b.
The complement $\DD^2 \smallsetminus A$ is a finite disjoint union of contractible subspaces of $\DD^2$.
Consider the finite set of such:
\[
F
~:=~
\pi_0\left(
\DD^2 \smallsetminus A 
\right)
~.
\]

For each $D \in F$ consider the subsets
\[
s(D)
~:=~
\left\{
e\in E_- \mid e \subset D
\right\}
~\subset~
E_-
\qquad
\text{ and }
\qquad
t(D)
~:=~
\left\{
e\in E_+ \mid e \subset D
\right\}
~\subset~
E_+
~.
\]
The linear orders on each of $E_\pm$ inherited from the $1^{st}$-coordinate direction of the 2-framing on $\DD^2 \subset \RR^2$ determines a linear order on each of these subsets.
With these linear orders, we regard $s(D)$ as a word in the alphabet $E_-$ and $t(D)$ as a word in the alphabet $E_+$.
Invoking Observation~\ref{s4}, we then have two composite maps
\[
s
\colon
F
\longrightarrow
\FF_1(E_-)
\hookrightarrow
\FF_1(E)
\hookleftarrow
\FF_1(E_+)
\longleftarrow
F
\colon
t
\]
from $F$ to the underlying 0-type of the free associative algebra on the set $E$.  
Define the flagged $(\infty,2)$-category $G\left( \left\lag W \right\rag \right)$ as the pushout:
% https://egretwalker.github.io/quivermore/#q=WzAsNCxbMSwxLCJub25lIiwiRyJdLFswLDEsIm5vbmUiLCJHX3tcXGxlcSAxfSJdLFswLDAsIm5vbmUiLCIoXFxwYXJ0aWFsIGNeMilee1xcYW1hbGcgRn0iXSxbMSwwLCJub25lIiwiKGNeMilee1xcYW1hbGcgRn0iXSxbMiwxLCJub25lIiwiKHMsdCkiLDJdLFsxLDAsIm5vbmUiXSxbMiwzLCJub25lIiwiIiwwLHsic3R5bGUiOnsidGFpbCI6eyJuYW1lIjoiaG9vayIsInNpZGUiOiJ0b3AifX19XSxbMywwLCJub25lIl0sWzAsMiwibm9uZSIsIiIsMSx7InN0eWxlIjp7Im5hbWUiOiJjb3JuZXIifX1dXQ==
\[\begin{tikzcd}
\label{f1}
	{(\partial c_2)^{\amalg F}} & {(c_2)^{\amalg F}} \\
	\fB \FF_1(E) & G\left( \left\lag W \right\rag \right)	
	\arrow["{(s,t)}"', from=1-1, to=2-1]
	\arrow[from=2-1, to=2-2]
	\arrow[hook, from=1-1, to=1-2]
	\arrow[from=1-2, to=2-2]
	\arrow["\lrcorner"{anchor=center, pos=0.125, rotate=180}, draw=none, from=2-2, to=1-1]
\end{tikzcd}
~.
\]
Put differently, $G\left( \left\lag W \right\rag \right)$ is the flagged $(\infty,2)$-category in which: there is a unique object; its space of 1-endomorphisms is the 0-type of words in the alphabet $E$; for each $D \in F$, there is a (freely) generating 2-morphism from the word $s(D)$ to the word $t(D)$.
See Figure~\ref{fig10}.
This concludes Step 2b.

\begin{figure}[H]
  \includegraphics[width=\linewidth, trim={0 {0in} {0in} {0in}}, clip]{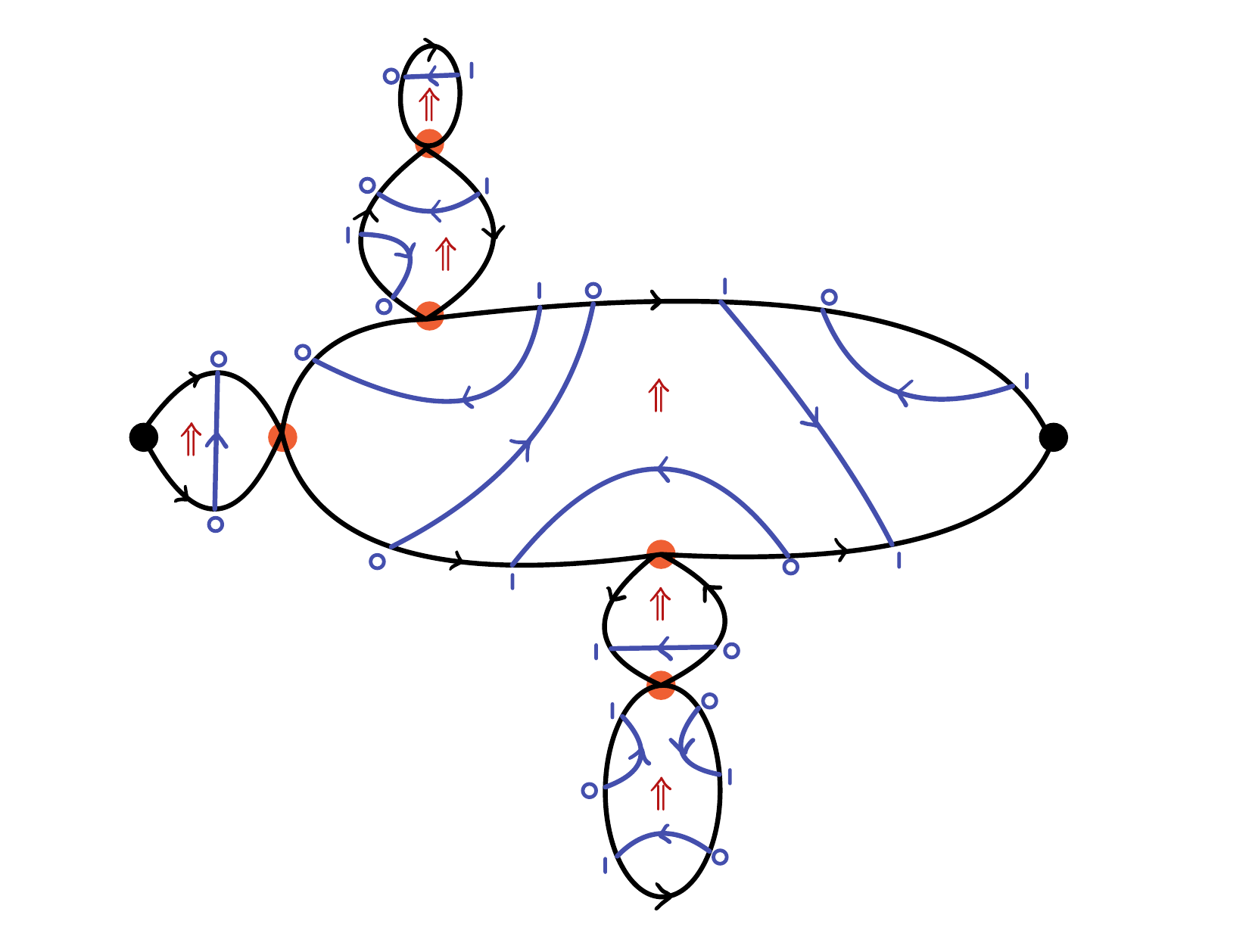}
  \caption{Depicted is the flagged $(\infty,2)$-category $G(\lag W \rag)$, equipped with its morphism $G(\lag W \rag) \to \fB \Bord^{[0,1]}$, associated to the 2-cell $c_1 \xra{\lag W \rag} \fB \Bord^{[0,1]}$ depicted in Figure~\ref{fig9}.  The flagged $(\infty,2)$-category $G(\lag W \rag)$ is that freely generated by the 0-, 1-, and 2-dimensional cells in this picture.  The picture is achieved from that of Figure~\ref{fig9} by collapsing each (dashed orange) arc.}
  \label{fig10}
\end{figure}

As {\bf Step 3}, we construct a morphism 
\begin{equation}
\label{s12}
c_2
\xra{~q~}
G(\lag W \rag)
\end{equation}
between flagged $(\infty,2)$-categories, as in the diagram~(\ref{f3}).
Informally, $q$ selects the maximal composite 2-cell in $G(\lag W \rag)$ (see Figure~\ref{fig10}).

Through Observation~\ref{s4}, the linear orders on the finite sets $E_-$ and $E_+$ respectively define words $s,t \in \FF_1(E)$.

We therefore have a morphism between flagged $(\infty,2)$-categories:
\[
(s,t)
\colon
\partial c_2
\longrightarrow
\fB \FF_1(E)
\longrightarrow
G(\lag W \rag)
~.
\]

We wish to fill the diagram among flagged $(\infty,2)$-categories:
% https://egretwalker.github.io/quivermore/#q=WzAsNCxbMCwwLCJub25lIiwiXFxwYXJ0aWFsIGNeMiJdLFsxLDAsIm5vbmUiLCJUX3tcXGxlcSAxfSJdLFsyLDAsIm5vbmUiLCJHKFxcbGFnIFcgXFxyYWcpIl0sWzAsMSwibm9uZSIsImNeMiJdLFsxLDIsIm5vbmUiLCJxX3t8VF97XFxsZXEgMX19Il0sWzAsMSwibm9uZSJdLFswLDMsIm5vbmUiXSxbMywyLCJub25lIiwiIiwyLHsic3R5bGUiOnsiYm9keSI6eyJuYW1lIjoiZGFzaGVkIn19fV1d
\begin{equation}
\label{s27}
\begin{tikzcd}
	{\partial c_2} & {G(\lag W \rag)} \\
	{c_2}
	\arrow["q_{|\partial c_2}", from=1-1, to=1-2]
	\arrow[from=1-1, to=2-1]
	\arrow[dashed, swap, "q", from=2-1, to=1-2]
\end{tikzcd}
~.
\end{equation}
We do this by induction on the number of connected components of $A$, which is the number of incoherent connected components $C\subset \DD^2 \smallsetminus W$.  
Suppose $A = \emptyset$ is empty.
Then each of the finite sets $E_\pm = \ast$ is a singleton, so that $E$ has cardinality 2, and $F$ is a singleton, and so the canonical morphism $c_2 = (c_2)^{\amalg F} \to G(\lag W \rag)$ is the sought functor $q$.
Next, suppose $A$ has at least one connected component.
Select a connected component $A_C \subset A$.
Isotope the tangle $W\subset \DD^2$ and the arc $A_C\subset \DD^2$ to assume $A_C$ is an affine arc that is not parallel to the $2^{nd}$ coordinate direction of $\DD^2$.
Projecting the $1^{st}$ coordinate direction of $\DD^2$ onto $A_C$ therefore determines a framing of $A_C$.  
Consider the union $\Gamma := \partial \DD^2 \cup A_C$.
The framing of $A_C$, and the given framings of $\partial_{\pm} \DD^2$ also via the $1^{st}$-coordinate direction of the solid 2-framing of $\DD^2$, give $\Gamma$ the structure of a finite \emph{directed} graph: Its set of vertices is the union $S^0 \cup \partial A_C$, which has cardinality at least 2 and at most 4; its set of directed edges is $\{A_C\} \cup \pi_0( \partial \DD^2 \smallsetminus \left(S^0 \cup \partial A_C) \right)$.
Denote by $\fC(\Gamma)$ the quiver generated by this finite directed graph: It is the $(\infty,1)$-category freely generated by $\Gamma$.  
By design of the arc $A_C$, each edge of $\Gamma$ determines a word in the alphabet $E$, with the edge $A_C$ determining the empty word.  
We therefore have a morphism between flagged $(\infty,2)$-categories
\begin{equation}
\label{s23}
\fC(\Gamma)
\longrightarrow
\fB \FF_1(E)
\longrightarrow
G(\lag W \rag)
~.
\end{equation}

Now, using that the arc $A_C$ separates $\DD^2$, denote each connected component of the complement $H_- \sqcup H_+ := \DD^2 \smallsetminus A_C$.  
The closure $\ov{H}_\pm \subset \DD^2$ is a compact 2-manifold with corners.
Each face of the boundary $\partial \ov{H}_\pm$ has a given framing.
Further, $\ov{H}_\pm$ inherits a solid 2-framing from that of $\DD^2$, and along each face of the boundary $\partial \ov{H}_\pm$, the projection of the $1^{st}$ coordinate of this solid 2-framing agrees with the given framing.
Choose a piecewise smooth homeomorphism $\DD^2 \xra{\psi_\pm} \ov{H}_\pm$ that is a diffeomorphism away from the end-points of $A_C$, and that respects solid 2-framings and the framings of the faces.  
Denote $W_\pm := \psi^{-1}_\pm(W) \subset \DD^2$, which are tangles in $\DD^2$, and they inherit a $[0,1]$-framing from that of $W \subset \DD^2$.
Consequently, they are selected by morphisms $c_2 \xra{\lag W_\pm \rag} \fB \Bord^{[0,1]}$.
Furthermore, $A_\pm := \psi^{-1}_\pm(A \smallsetminus A_C)\subset \DD^2$ is a disjoint union of arcs, one for each incoherent connected component $\DD^2 \smallsetminus W_\pm$, each that satisfies the characteristic conditions~(A.1)-(A.3) above.  
By construction, there are strictly fewer connected components of each of $A_\pm$ than there are of $A$.
By induction on the number of connected components of $A$, 
we therefore have fillers
\[
\begin{tikzcd}
	{\partial c_2} & {G(\lag W_\pm \rag)} \\
	{c_2}
	\arrow["(q_\pm)_{|\partial c_2}", from=1-1, to=1-2]
	\arrow[from=1-1, to=2-1]
	\arrow[dashed, swap, "q_\pm", from=2-1, to=1-2]
\end{tikzcd}
~.
\]
The inclusions $\partial \ov{H}_\pm \hookrightarrow \Gamma$ and $\partial \DD^2 \hookrightarrow \Gamma$ determine morphisms between $(\infty,1)$-categories:
\[
(\partial c_2)^{\amalg \{\pm\}}
\xra{~h^0_\pm~}
\fC(\Gamma)
\xla{~f^0~}
\partial c_2
~.
\]
Inspecting the definition of each of $G(\lag W_\pm \rag)$ and $G(\lag W \rag)$ as a pushout, observe the canonical fillers in a commutative diagram among flagged $(\infty,2)$-categories
% https://q.uiver.app/#q=WzAsNSxbNCwxLCJHKFxcbGFnIFcgXFxyYWcpIl0sWzIsMSwiRyhcXGxhZyBXXy0gXFxyYWcpIFxcYW1hbGcgRyhcXGxhZyBXXysgXFxyYWcpIl0sWzIsMCwiXFxmQyhcXEdhbW1hKSJdLFswLDEsIihjXjIpXntcXGFtYWxnIFxce1xccG1cXH19Il0sWzAsMCwiKFxccGFydGlhbCBjXjIpXntcXGFtYWxnIFxce1xccG1cXH19Il0sWzIsMCwiKFxccmVme3MyM30pIl0sWzEsMCwiIiwyLHsic3R5bGUiOnsiYm9keSI6eyJuYW1lIjoiZGFzaGVkIn19fV0sWzMsMSwicV8tIFxcYW1hbGcgcV8rIiwxXSxbNCwzLCIiLDIseyJzdHlsZSI6eyJ0YWlsIjp7Im5hbWUiOiJob29rIiwic2lkZSI6InRvcCJ9fX1dLFs0LDIsIiIsMCx7InN0eWxlIjp7ImJvZHkiOnsibmFtZSI6ImRhc2hlZCJ9fX1dLFs0LDEsIihxXy0pX3t8XFxwYXJ0aWFsIGNeMn0gXFxhbWFsZyAocV8rKV97fFxccGFydGlhbCBjXjJ9IiwxXV0=
\begin{equation}
\label{s26}
\begin{tikzcd}
	{(\partial c_2)^{\amalg \{\pm\}}} && {\fC(\Gamma)} \\
	{(c_2)^{\amalg \{\pm\}}} && {G(\lag W_- \rag) \amalg G(\lag W_+ \rag)} && {G(\lag W \rag)}
	\arrow["h^0_\pm", from=1-1, to=1-3]
	\arrow[hook, from=1-1, to=2-1]
	\arrow["{(q_-)_{|\partial c_2} \amalg (q_+)_{|\partial c_2}}"{description}, from=1-1, to=2-3]
	\arrow["{(\ref{s23})}", from=1-3, to=2-5]
	\arrow[swap, "{q_- \amalg q_+}", from=2-1, to=2-3]
	\arrow[dashed, from=2-3, to=2-5]
\end{tikzcd}
~.
\end{equation}
Consider the solid diagram among flagged $(\infty,2)$-categories,
% https://q.uiver.app/#q=WzAsNixbMCwxLCIoY14yKV57XFxhbWFsZyBcXHtcXHBtXFx9fSJdLFswLDAsIihcXHBhcnRpYWwgY14yKV57XFxhbWFsZyBcXHtcXHBtXFx9fSJdLFsxLDAsIlxcZkMoXFxHYW1tYSkiXSxbMSwxLCJcXGNDIl0sWzIsMCwiXFxwYXJ0aWFsIGNeMiJdLFsyLDEsImNeMiJdLFsxLDIsImheMF97XFxwbX0iXSxbMSwwLCIiLDAseyJzdHlsZSI6eyJ0YWlsIjp7Im5hbWUiOiJob29rIiwic2lkZSI6InRvcCJ9fX1dLFswLDMsImhfXFxwbSJdLFsyLDNdLFszLDEsIiIsMSx7InN0eWxlIjp7Im5hbWUiOiJjb3JuZXIifX1dLFs0LDIsImZeMCIsMl0sWzQsNSwiIiwwLHsic3R5bGUiOnsidGFpbCI6eyJuYW1lIjoiaG9vayIsInNpZGUiOiJ0b3AifX19XSxbNSwzLCJmIiwyLHsic3R5bGUiOnsiYm9keSI6eyJuYW1lIjoiZGFzaGVkIn19fV1d
\begin{equation}
\label{s24}
\begin{tikzcd}
	{(\partial c_2)^{\amalg \{\pm\}}} & {\fC(\Gamma)} & {\partial c_2} \\
	{(c_2)^{\amalg \{\pm\}}} & \cC & {c_2}
	\arrow["{h^0_{\pm}}", from=1-1, to=1-2]
	\arrow[hook, from=1-1, to=2-1]
	\arrow[from=1-2, to=2-2]
	\arrow["{f^0}"', from=1-3, to=1-2]
	\arrow[hook, from=1-3, to=2-3]
	\arrow["{h_\pm}", from=2-1, to=2-2]
	\arrow["\lrcorner"{anchor=center, pos=0.125, rotate=180}, draw=none, from=2-2, to=1-1]
	\arrow["f"', dashed, from=2-3, to=2-2]
\end{tikzcd}
\end{equation}
in which $\cC$ defined as the pushout of the left square.
Being a pushout, the diagram~(\ref{s26}) supplies a canonical morphism between flagged $(\infty,2)$-categories
\begin{equation}
\label{s25}
\cC
\longrightarrow
G(\lag W \rag)
~.
\end{equation}
Observe that a 2-categorical composition of the two 2-cells $(c_2)^{\amalg \{\pm\}} \xra{h_\pm} \cC$ in $\cC$ supplies a filler in~(\ref{s24}).
Finally, define the sought filler~(\ref{s27}) to be the composition
\[
q
\colon
c_2
\xra{~f~}
\cC
\xra{~(\ref{s25})~}
G(\lag W \rag)
~.
\]
This concludes Step 3.

As {\bf Step 4}, we construct a factorization
\begin{equation}
\label{s7}
% https://egretwalker.github.io/quivermore/#q=WzAsMyxbMiwxLCJub25lIiwiXFxmQiBcXER1YWwiXSxbMCwxLCJub25lIiwiVCJdLFsxLDAsIm5vbmUiLCJHXFxsZWZ0KCBcXGxlZnRcXGxhZyBXIFxccmlnaHRcXHJhZyBcXHJpZ2h0KSJdLFsxLDIsIm5vbmUiXSxbMSwwLCJub25lIiwiXFxsYWcgVyBcXHJhZyIsMl0sWzIsMCwibm9uZSIsIiIsMCx7InN0eWxlIjp7ImJvZHkiOnsibmFtZSI6ImRhc2hlZCJ9fX1dXQ==
\begin{tikzcd}
	& {G\left( \left\lag W \right\rag \right)} \\
	c_2 && {\fB \Bord^{[0,1]}}
	\arrow["q", from=2-1, to=1-2]
	\arrow["{\lag W \rag}"', from=2-1, to=2-3]
	\arrow[dashed, from=1-2, to=2-3]
\end{tikzcd}
~.
\end{equation}

For each $e\in E$, the intersection $W \cap e \subset e$ is a finite set.  
The $1^{st}$-coordinate of the solid 2-framing on $\DD^2 \subset \RR^2$ determines an orientation on $e$, which thereafter induces a linear ordering on this finite set $W \cap e$.  
Furthermore, the $[0,1]$-framing of $W \subset \DD^2$ determines a map $W \cap e \to \{0,1\}$.  
In this way, through Observation~\ref{s4}, we have a map
\begin{equation}
\label{s6}
E
\longrightarrow
\FF_1(\{0,1\})
\underset{\rm Lem~\ref{t.bord.obj}}
\hookrightarrow
\Bord^{[0,1]}
~.
\end{equation}
By design, the map~(\ref{s6}) takes values in words in the alphabet $\{0,1\}$ that are \emph{alternating}.
Via the monoidal structure on the $(\infty,1)$-category $\Bord^{[0,1]}$, this map uniquely extends as a monoidal functor
\begin{equation}
\label{s11}
\FF_1(E)
\longrightarrow
\Bord^{[0,1]}
~.
\end{equation}
By definition of the categorical deloop, this is a morphism between flagged $(\infty,2)$-categories:
\begin{equation}
\label{s8}
\fB \FF_1(E)
\longrightarrow
\fB \Bord^{[0,1]}
~.
\end{equation}
Next, let $D\in F$.
Note that, by design, the map
\begin{equation}
\label{s10}
F
\amalg
F
\xra{~ \lag s,t \rag~}
\FF_1(E)
\xra{~(\ref{s6})~}
\FF_1(\{0,1\})
\end{equation}
takes values in words in the alphabet $\{0,1\}$ that are \emph{alternating}.
Consider the intersection of the closure of $D$ with the tangle $W$:
\[
W_D
~:=~
W \cap \ov{D}
~.
\]
Consider the quotient of the closure of $D$ by each arc contained in it:
\[
R_D
~:=~
\ov{D}_{/\sim}
~,
\]
where $x\sim y$ if $x=y$ or both $x$ and $y$ belong to the same connected component of $A \cap \ov{D}$.
As $\ov{D} \subset \RR^2$ is isomorphic with a polygon, then $R_D$ is isomorphic with a polygon collapse some (but not all) of its edges, and is therefore a compact surface with corners that refines a closed 2-disk.  
Furthermore, $R_D$ inherits a solid 2-framing from that of $\ov{D} \subset \DD^2$.
Using that the compact subspaces $W,A \subset \DD^2$ are disjoint, the quotient map $\ov{D} \to R_D$ restricts to $W_D \subset \ov{D}$ as an isomorphism onto its image, which we again denote as
\[
W_D
~\subset~
R_D
~.
\]
This subspace is a tangle in a solidly 2-framed stratified manifold with corners refining a closed 2-disk, and it inherits a $[0,1]$-framing from that of $W_D \subset \ov{D}$.  
By definition of $\Bord^{[0,1]}$, this tangle $W_D \subset R_D$ together with this $[0,1]$-framing define a morphism between flagged $(\infty,2)$-categories: 
\[
c_2
\xra{~\lag W_D \rag~}
\fB \Bord^{[0,1]}
~.
\]
By design, the resulting diagram among flagged $(\infty,2)$-categories
% https://egretwalker.github.io/quivermore/#q=WzAsNCxbMSwxLCJub25lIiwiXFxmQiBcXEJvcmRfMV57WzAsMV19KFxcUlJeMSkiXSxbMSwwLCJub25lIiwiKGNeMilee1xcYW1hbGcgUH0iXSxbMCwwLCJub25lIiwiKFxccGFydGlhbCBjXjIpXntcXGFtYWxnIFB9Il0sWzAsMSwibm9uZSIsIlxcZkIge1xcc2YgRnJlZX1fe1xcY0VfMX0oRSkiXSxbMiwxLCJub25lIiwiIiwwLHsic3R5bGUiOnsidGFpbCI6eyJuYW1lIjoiaG9vayIsInNpZGUiOiJ0b3AifX19XSxbMSwwLCJub25lIiwiKFxcbGFnIFdfcCBcXHJhZyApX3twXFxpbiBQfSJdLFsyLDMsIm5vbmUiLCIocyx0KSIsMl0sWzMsMCwibm9uZSIsIihcXHJlZntzOH0pIiwyXV0=
\[
\begin{tikzcd}
	{(\partial c_2)^{\amalg F}} & {(c_2)^{\amalg F}} \\
	{\fB \FF_1(E)} & {\fB \Bord^{[0,1]}}
	\arrow[hook, from=1-1, to=1-2]
	\arrow["{(\lag W_D \rag )_{D\in F}}", from=1-2, to=2-2]
	\arrow["{(s,t)}"', from=1-1, to=2-1]
	\arrow["{(\ref{s8})}"', from=2-1, to=2-2]
\end{tikzcd}
\]
commutes.  
By definition of the flagged $(\infty,2)$-category $G\left( \left\lag W \right\rag \right)$ as a pushout, this diagram is a morphism between flagged $(\infty,2)$-categories:
\begin{equation}
\label{s9}
G\left( \left\lag W \right\rag \right)
\longrightarrow
\fB \Bord^{[0,1]}
~.
\end{equation}
Using that $W = \underset{D\in F} \bigcup W_D \subset \underset{D\in F} \bigcup \ov{D} = \DD^2$, the diagram among flagged $(\infty,2)$-categories
% https://egretwalker.github.io/quivermore/#q=WzAsMyxbMiwxLCJub25lIiwiXFxmQiBcXEJvcmRfMV57WzAsMV19KFxcUlJeMSkiXSxbMCwxLCJub25lIiwiVCJdLFsxLDAsIm5vbmUiLCJHKFxcbGFnIFcgXFxyYWcpIl0sWzEsMCwibm9uZSIsIlxcbGFnIFcgXFxyYWciLDJdLFsxLDIsIm5vbmUiLCIiLDAseyJzdHlsZSI6eyJib2R5Ijp7Im5hbWUiOiJkYXNoZWQifX19XSxbMiwwLCJub25lIiwiKFxccmVme3M5fSkiXV0=
\begin{equation}
\label{s21}
\begin{tikzcd}
	& {G(\lag W \rag)} \\
	c_2 && {\fB \Bord^{[0,1]}}
	\arrow["{\lag W \rag}"', from=2-1, to=2-3]
	\arrow["q", from=2-1, to=1-2]
	\arrow["{(\ref{s9})}", from=1-2, to=2-3]
\end{tikzcd}
\end{equation}
commutes.
This concludes Step 4.
Note how, by design, this factorization of $\lag W \rag$ has the property that its restriction to each of the 1-cells $c_1 \xra{s,t} c_2$ is the coarsest \emph{alternating} factorization of the words selected by the composite morphisms $c_1 \xra{s,t} c_2 \xra{\lag W \rag} \fB \Bord^{[0,1]}$.

As {\bf Step 5}, we construct a factorization
\begin{equation}
\label{f2}
% https://egretwalker.github.io/quivermore/#q=WzAsMyxbMSwyLCJub25lIiwiXFxmQiBcXEJvcmRfMV57WzAsMV19KFxcUlJeMSkiXSxbMCwxLCJub25lIiwiRyhcXGxhZyBXIFxccmFnKSJdLFsxLDAsIm5vbmUiLCJcXGZCIFxcRHVhbCJdLFsxLDAsIm5vbmUiLCIoXFxyZWZ7czl9KSJdLFsyLDAsIm5vbmUiLCIoXFxyZWZ7ZUR1YWxCb3JkfSkiXSxbMSwyLCJub25lIiwiIiwwLHsic3R5bGUiOnsiYm9keSI6eyJuYW1lIjoiZGFzaGVkIn19fV1d
\begin{tikzcd}
	& {\fB \Dual} \\
	{G(\lag W \rag)} \\
	& {\fB \Bord^{[0,1]}}
	\arrow["{(\ref{s9})}", from=2-1, to=3-2]
	\arrow["{(\ref{eDualBord})}", from=1-2, to=3-2]
	\arrow[dashed, from=2-1, to=1-2]
\end{tikzcd}
~.
\end{equation}
By definition of $G(\lag W \rag)$ as a pushout, this is to construct fillers (a) and (b) as in the diagram among flagged $(\infty,2)$-categories:
% https://egretwalker.github.io/quivermore/#q=WzAsNixbMSwxLCJub25lIiwiRyhcXGxhZyBXIFxccmFnKSJdLFsxLDAsIm5vbmUiLCIoY14yKV57XFxhbWFsZyBQfSJdLFswLDAsIm5vbmUiLCIoXFxwYXJ0aWFsIGNeMilee1xcYW1hbGcgUH0iXSxbMCwxLCJub25lIiwiXFxmQiB7XFxzZiBGcmVlfV97XFxjRV8xfShFKSJdLFszLDIsIm5vbmUiLCJcXGZCIFxcQm9yZF8xXntbMCwxXX0oXFxSUl4xKSJdLFszLDAsIm5vbmUiLCJcXGZCIFxcRHVhbCJdLFsyLDEsIm5vbmUiLCIiLDAseyJzdHlsZSI6eyJ0YWlsIjp7Im5hbWUiOiJob29rIiwic2lkZSI6InRvcCJ9fX1dLFsxLDAsIm5vbmUiXSxbMiwzLCJub25lIiwiKHMsdCkiLDJdLFszLDAsIm5vbmUiXSxbNSw0LCJub25lIiwiKFxccmVme2VEdWFsQm9yZH0pIl0sWzAsNCwibm9uZSIsIihcXHJlZntzOX0pIiwxXSxbMyw1LCJub25lIiwie1xccm0gKGEpfSIsMSx7InN0eWxlIjp7ImJvZHkiOnsibmFtZSI6ImRvdHRlZCJ9fX1dLFsxLDUsIm5vbmUiLCJ7XFxybSAoYil9IiwxLHsic3R5bGUiOnsiYm9keSI6eyJuYW1lIjoiZG90dGVkIn19fV0sWzAsNSwibm9uZSIsIlxcZXhpc3RzICEiLDEseyJzdHlsZSI6eyJib2R5Ijp7Im5hbWUiOiJkYXNoZWQifX19XV0=
\[\begin{tikzcd}
	{(\partial c_2)^{\amalg F}} & {(c_2)^{\amalg F}} && {\fB \Dual} \\
	{\fB \FF_1(E)} & {G(\lag W \rag)} \\
	&&& {\fB \Bord^{[0,1]}}
	\arrow[hook, from=1-1, to=1-2]
	\arrow[from=1-2, to=2-2]
	\arrow["{(s,t)}"', from=1-1, to=2-1]
	\arrow[from=2-1, to=2-2]
	\arrow["{(\ref{eDualBord})}", from=1-4, to=3-4]
	\arrow["{(\ref{s9})}"{description}, from=2-2, to=3-4]
	\arrow["{{\rm (a)}}"{description}, dotted, from=2-1, to=1-4]
	\arrow["{{\rm (b)}}"{description}, dotted, from=1-2, to=1-4]
	\arrow["{\exists !}"{description}, dashed, from=2-2, to=1-4]
\end{tikzcd}
~.
\]
Using that the flagged $(\infty,2)$-category is, in fact, a flagged $(\infty,1)$-category, Lemma~\ref{s20} supplies a unique filler (a).
Using that the map~(\ref{s10}) takes values in alternating words, 
Lemma~\ref{s.1} supplies the filler (b).
This concludes Step 5.
In doing so we have constructed the diagram among $(\infty,2)$-categories~(\ref{f3}), and have constructed the value of $(\bTheta_2)_{/L}$ on 2-cells in $\fB \Bord^{[0,1]}$.

\subsubsection{$(\bTheta_2)_{/L}$ on general objects}
We now define the value of $(\bTheta_2)_{/L}$ on a general object $\left( T \xra{\lag W \rag} \fB \Bord^{[0,1]} \right) \in (\bTheta_2)_{/\fB \Bord^{[0,1]}}$.
Write $T= [p]([q_1],\dots,[q_p])$.
If $p=0$, then $T$ is a 1-category, which was considered in a previous subsection.
So suppose $p>0$.
The canonical diagram
\[
\xymatrix{
\{1\}()
\ar[rr]
\ar[d]
&&
\{1<\dots<p\}([q_2],\dots,[q_p])
\ar[d]
\\
\{0<1\}([q_1])
\ar[rr]
&&
[p]([q_1],\dots,[q_p])
}
\]
is a pushout in $\bTheta_2$, and therefore also defines a pushout in $(\bTheta_2)_{/\fB \Bord^{[0,1]}}$.
Therefore, the value of $(\bTheta_2)_{/\fB \Bord^{[0,1]}} \xra{(\bTheta_2)_{/L}} (\bTheta_2)_{/\fB \Dual}$ on $T$ can be defined as a pushout of the values of $(\bTheta_2)_{/L}$ on the other terms in this square. 
Via induction on $p$, we are reduced to the cases in which $T=[1]([q])$ or $T = c_0$.
Suppose $T=[1]([q])$.
If $q=0$, then $T$ is a 1-category, and it is an instance of a case considered in the previous subsection.
So suppose $q>0$.
The canonical diagram
\[
\xymatrix{
[1](\{1\})
\ar[rr]
\ar[d]
&&
[1](\{1<\dots<q\})
\ar[d]
\\
[1](\{0<1\})
\ar[rr]
&&
[1]([q])
}
\]
is a pushout in $\bTheta_2$, and therefore also defines a pushout in $(\bTheta_2)_{/\fB \Bord^{[0,1]}}$.
Therefore, the value of $(\bTheta_2)_{/\fB \Bord^{[0,1]}} \xra{(\bTheta_2)_{/L}} (\bTheta_2)_{/\fB \Dual}$ on $T$ can be defined as a pushout of the values of $(\bTheta_2)_{/L}$ on the other terms in this square.
Via induction on $q$, we are reduced to the cases in which $T=c_k$ for $0\leq k \leq 2$.
These cases were addressed in the previous subsections.  
In this way, we have defined the value of $(\bTheta_2)_{/\fB \Bord^{[0,1]}} \xra{(\bTheta_2)_{/L}} (\bTheta_2)_{/\fB \Dual}$ on each object $\left( T \xra{\lag W \rag} \fB \Bord^{[0,1]} \right)$.

\subsubsection{Functoriality of $(\bTheta_2)_{/L}$}
We now argue that the above-defined values of $(\bTheta_2)_{/L}$ assemble as a functor over $(\bTheta_2)_{/ \fB \Bord^{[0,1]}}$.

We first make some observations about the subspace $A \subset \DD^2 \smallsetminus W$ of subsection~\S\ref{sec.2.cell}.
Let $\left(T \xra{\lag W \rag} \fB \Bord^{[0,1]} \right)$ be an object in 
$(\bTheta_2)_{/\fB \Bord^{[0,1]}}$.
Write $T = [p]( [q_1], \dots, [q_p])$.  
For each $0<i\leq p$, and each $0<j\leq q_i$, the inclusion $c_2 \simeq \{i-1<i\}(\{j-1<j\}) \hookrightarrow T$ determines an embedding between stratified spaces $\DD^2 = \lag c_2 \rag \hookrightarrow \lag T \rag$ from the hemispherically stratified 2-disk. 
Denote the image of this embedding as $\DD^2_{i,j} \subset \lag T \rag$.
Similarly, for each $0<i\leq p$, and each $0\leq j\leq q_i$, the inclusion $c_1 \simeq \{i-1<i\}(\{j\}) \hookrightarrow T$ determines an embedding between stratified spaces $\DD^1 = \lag c_1 \rag \hookrightarrow \lag T \rag$ from the hemispherically stratified 1-disk. 
Denote the image of this embedding as $\DD^1_{i,j} \subset \lag T \rag$.
For each $0<i\leq p$, and each $0<j\leq q_i$, recognize the upper- and lower-hemispheres:
\[
\DD^1_{i,j-1}
~=~
\partial_- \DD^2_{i,j}
\qquad
\text{ and }
\qquad
\DD^1_{i,j}
~=~
\partial_+ \DD^2_{i,j}
~.
\]
Denote the intersection of the $[0,1]$-framed tangle $W\subset \lag T \rag$ with each of these subspaces as
\[
W_{i,j}
~:=~
W \cap \DD^2_{i,j}
~\subset~
\DD^2_{i,j}
~,
\]
which are, themselves, $[0,1]$-framed tangles.

Recall from the argument of subsection~\ref{sec.2.cell} that the factorization~(\ref{s21}) restricts along each of the generating 1-cells of $c_2$ as the coarsest alternating factorization of word in the alphabet $\{0,1\}$ selected by each such 1-cell.
As such, $\partial_+ \DD^2_{i,j} \cap A_{i,j} \subset \partial_+ \DD^2 = \DD^1_{i,j}$ is isotopic in $\partial_- \DD^2_{i,j+1} \cap A_{i,j+1} \subset \partial_- \DD^2  = \DD^1_{i,j}$.
Invoking the isotopy extension theorem, choose isotopies of each pair $(\partial A_{i,j} \subset A_{i,j}) \subset \left( \partial( \DD^2_{i,j} \smallsetminus W_{i,j}) \subset \DD^2_{i,j} \smallsetminus W \right)$
to ensure the following properties.
\begin{itemize}
\item
For each $0<i\leq p$ and each $0<j<q_i$, there is an equality between finite subsets of $\DD^1_{i,j}$:
\[
\partial_+ \DD^2_{i,j} \cap A_{i,j} 
~=~
\partial_- \DD^2_{i,j+1} \cap A_{i,j+1}
~.
\]

\item
For each $0<i\leq p$, the union of arcs $\underset{0<j\leq q_i} \bigcup A_{i,j} \subset 
\left \lag \{ i-1<i\}([q_i]) \right \rag$ is a smooth properly embedded 1-submanifold with boundary that is transverse to each stratum of $\left \lag \{ i-1<i\}([q_i]) \right \rag$.

\end{itemize}
In this way, we can choose a finite disjoint union of (closed) arcs $A \subset \lag T \rag \smallsetminus W$ that is transverse to the strata of $\lag T \rag$, each of whose end-points lies in the boundary of $\lag T \rag \subset \RR^2$, and such that, for each $0<i\leq p$ and $0<j\leq q_p$, the intersection $\DD^2_{i,j} \cap A$ is the disjoint union of arcs of the subsection~\S\ref{sec.2.cell}.

We now explain how the values of $(\bTheta_2)_{/L}$ established in the previous subsection canonically assemble as a functor.
Let $[k] \xra{ \lag T_0 \to \cdots \to T_k \to \fB \Bord^{[0,1]} \rag} (\bTheta_2)_{/ \fB \Bord^{[0,1]} }$ be a functor from an object $[k]\in \bDelta$.  
The values of $(\bTheta_2)_{/L}$ on object supplies lifts in the diagram among flagged $(\infty,2)$-categories
% https://q.uiver.app/#q=WzAsNSxbMCwxLCJUXzAiXSxbMSwxLCJcXGNkb3RzIl0sWzIsMSwiVF9rIl0sWzMsMSwiXFxmQiBcXEJvcmRfMV57WzAsMV19KFxcUlJeMSkiXSxbMywwLCJcXGZCIFxcRHVhbCJdLFswLDFdLFsxLDJdLFsyLDNdLFs0LDNdLFsyLDQsIlxcY2RvdHMiLDAseyJsYWJlbF9wb3NpdGlvbiI6MCwic3R5bGUiOnsiYm9keSI6eyJuYW1lIjoiZGFzaGVkIn19fV0sWzAsNCwiIiwyLHsic3R5bGUiOnsiYm9keSI6eyJuYW1lIjoiZGFzaGVkIn19fV1d
\begin{equation}
\label{s22}
\begin{tikzcd}
	&&& {\fB \Dual} \\
	{T_0} & \cdots & {T_k} & {\fB \Bord^{[0,1]}}
	\arrow[from=1-4, to=2-4]
	\arrow[dashed, from=2-1, to=1-4]
	\arrow[from=2-1, to=2-2]
	\arrow[from=2-2, to=2-3]
	\arrow["\cdots"{pos=0}, dashed, from=2-3, to=1-4]
	\arrow[from=2-3, to=2-4]
\end{tikzcd}
\end{equation}
such that each right-angled triangle therein commutes.
We must explain how this is a commutative diagram.
Through the functor $\bTheta_2^{\op} \xra{\lag - \rag} \cMfd_2^{\sf sfr}$, the functor $[k] \xra{ \lag T_\bullet \to \fB \Bord^{[0,1]} \rag} (\bTheta_2)_{/ \fB \Bord^{[0,1]} }$ is the data of a proper constructible bundle $X \to \Delta^k$ equipped with a fiberwise solid 2-framing, equipped with a compact subspace $W \subset X$ that is a fiberwise (over the standardly stratified simplex $\Delta^k$) tangle equipped with a fiberwise $[0,1]$-framing together with, for each $0\leq \ell \leq k$, an identification of the fiber over $\Delta^{\{\ell\}} \subset \Delta^k$,
\[
(W_{|\Delta^{\{\ell\}}} \subset X_{|\Delta^{\{\ell\}}}) 
~\cong~ 
(W_\ell \subset \lag T_\ell \rag )
\]
with the $[0,1]$-framed tangle selected by $T_\ell \to  \fB \Bord^{[0,1]}$.
Commutativity of~(\ref{s22}) follows upon observing that there exists a compact subspace $A \subset X \smallsetminus W$ whose projection $A \to \Delta^k$ is a constructible bundle, and that is fiberwise transverse to the strata of $X$, and such that, for each $0\leq \ell \leq k$, the intersection $A \cap \left( \lag T_\ell \rag \smallsetminus W \right)$ is a subspace isotopic (through properly embedded 1-submanifolds with boundary that are transverse to the strata of $\lag T_\ell \rag$) with the subspace from the paragraph above.

This completes the proof of Proposition~\ref{tDualBord}.

\subsection{Inductive step}
This section establishes that the diagrams
\[
\xymatrix{
\Bord^{[0,0]} \ar[r]^-{0\mapsto j}\ar[d]_-{0\mapsto 0}&\Bord^{[i,j]} \ar[d]^-{0\mapsto 0}\\
\Bord^{[0,1]} \ar[r]_-{0\mapsto j} & \Bord^{[i,j+1]} }
\]
and
\[
\xymatrix{
\Bord^{[0,0]} \ar[r]^-{0\mapsto i}\ar[d]_-{0\mapsto 0}&\Bord^{[i,j]} \ar[d]^-{0\mapsto 0}\\
\Bord^{[1,0]} \ar[r]_-{0\mapsto i} & \Bord^{[i+1,j]} }
\]
are pushouts of monoidal $(\oo,1)$-categories. See Lemma~\ref{lemma.main.basecase.n2}.

\begin{lemma}\label{lemma.monoid.pushout}
Let $A$ and $B$ in $\Alg(\Spaces)$ be monoids. Assume further that the unit maps $\ast\ra A$ and $\ast \ra B$ are monomorphisms, and let $\ov{A}=A\smallsetminus \ast$ and $\ov{B} = B\smallsetminus \ast$ denote the complements of the identity components. 
Then the underlying space of the coproduct of monoids is 
\[
A\star B
\simeq
\ast \amalg  \coprod_{k\geq 0} \ov{A} \times (\ov{B}\times \ov{A})^k 
\amalg 
\ov{B} \times (\ov{A}\times \ov{B})^k 
\amalg
(\ov{A}\times \ov{B})^{k+1}
\amalg
(\ov{B}\times \ov{A})^{k+1}~.
\]
\end{lemma}
\begin{proof}
Consider the adjunction $\Alg^{\sf nu}(\Spaces)\leftrightarrows \Alg(\Spaces)$, where the left adjoint adds a basepoint, sending a nonunital monoid $M$ to $M_\ast = M\amalg \ast$, and the right adjoint is the forgetful functor. The condition that the unit is a monomorphism implies that the monoids $A$ and $B$ are in the image of the left adjoint. That is, there are equivalences $A\simeq \ov{A}_\ast$ and $B\simeq \ov{B}_\ast$. Since left adjoints preserves coproducts, we have an equivalence $A\star B \simeq (\ov{A}\star \ov{B})_\ast$, where $\ov{A}\star \ov{B}$ is the coproduct in nonunital monoids.

The coproduct of nonunital monoids is computed by the standard formula in terms of words alternating in $\ov{A}$ and $\ov{B}$:
\[
\ov{A}\star \ov{B} = \ov{A}\amalg \ov{B} \amalg (\ov{A}\times \ov{B})\amalg (\ov{B}\times \ov{A})
\amalg  (\ov{A}\times \ov{B}\times \ov{A})\amalg  (\ov{B}\times \ov{A}\times \ov{B})\amalg \ldots
\]
the terms of which can be regrouped as:
\[
\ov{A}\star \ov{B}= \coprod_{k\geq 0} \ov{A} \times (\ov{B}\times \ov{A})^k 
\amalg 
\ov{B} \times (\ov{A}\times \ov{B})^k 
\amalg
(\ov{A}\times \ov{B})^{k+1}
\amalg
(\ov{B}\times \ov{A})^{k+1}~.
\]
The formula then follows by taking disjoint union with the identity element.
\end{proof}

\begin{lemma}\label{lemma.points.of.Bord01}
Consider the monomorphism of spaces
\[
\FF_1(\{0,1\}) \simeq\obj\bigl( \Bord^{[0,1]}\bigr) 
=
\Bord^{[0,1]}[0]
\longrightarrow 
 \Bord^{[0,1]}[p]
\]
and let $B^{(0,1]}_{[p]} \subset \Bord^{[0,1]}[p]$ be the submonoid of tangles $W\subset \RR^1\times\DD[p]$ for which no component of $W$ corresponds to the identity morphism of $\{0\}$. That is, $W$ has no component which is a horizontal tangle whose framing is given by the label $\{0\}$.
The natural map from the coproduct of monoids
\[
\Bord^{[0,0]}[p] \star B^{(0,1]}_{[p]}  \longrightarrow  \Bord^{[0,1]}[p]
\]
is an equivalence.
\end{lemma}
\begin{proof}
Given a tangle $W\in  \Bord^{[0,1]}[p]\simeq \Tang^{[0,1]}(\RR^1\times\DD[p])$, we can decompose $W$ as $W \cong W_B\amalg W_0$, where $W_B$ is the maximal sub-tangle of $W$ contained in $B^{(0,1]}_{[p]} $, and $W_0$ is the complement of $W_B$, which is the maximal sub-tangle of $W$ contained in $\Bord^{[0,0]}[p]$. The assignment $W\mapsto W_0$ defines a map of spaces $ \Bord^{[0,1]}[p] \ra  \Bord^{[0,0]}[p]$. For each choice of a tangle $W_0 \in  \Bord^{[0,0]}[p]$, there is a fiber square
\[
\xymatrix{
B^{(0,1]}_{[p]} \cap \Tang^{[0,1]}(\RR^1\times\DD[p] \smallsetminus W_0)\ar[r]\ar[d]& \Bord^{[0,1]}[p]\ar[d]\\
\{W_0\}\ar[r]&  \Bord^{[0,0]}[p]}
\]
where $B^{(0,1]}_{[p]} \cap \Tang^{[0,1]}(\RR^1\times\DD[p] \smallsetminus W_0)$ is the space of tangles $W'\subset (\RR^1\times\DD[p] \smallsetminus W_0)$ no component of which is a horizontal tangle labeled by $\{0\}$.
There is a diffeomorphism
\[
\RR^1\times\DD[p] \smallsetminus W_0 \cong\coprod_{|\pi_0W_0|+1} \DD[p] \times\RR^1
\]
so there exists a equivalence of the fiber
\[
B^{(0,1]}_{[p]} \cap \Tang^{[0,1]}(\RR^1\times\DD[p] \smallsetminus W_0)
\simeq
\prod_{|\pi_0W_0|+1} B^{(0,1]}_{[p]} ~.
\]
Likewise, we can consider the fiber over an element $W_0 \in \Bord^{[0,0]}[p]$ of the canonical map
\[
\Bord^{[0,0]}[p] \star B^{(0,1]}_{[p]}  \longrightarrow  \Bord^{[0,0]}[p]
\]
defined by the monoid map $B^{(0,1]}_{[p]} \ra \ast$. 
By Definition~\ref{def.Bord.ij}, the space 
\[
\Bord^{[0,0]}[p]
\simeq
\coprod_{k\geq 0} \Conf_k(\RR^1) \underset{\Sigma_k}{\times}\{0\}^k
\simeq
\FF_1(\{0\})
\]
is the free monoid on a single element $\{0\}$. So we are considering, for $A = \FF_1(\{0\})$, the fiber of the map
\[
\FF_1(\{0\})\star B^{(0,1]}_{[p]}  \longrightarrow \FF_1(\{0\})
\]
over $W_0 \simeq \left( 0 , \dots , 0 \right)$, which is a length $k$ word in the alphabet $\{0\}$ for some $k\geq 0$. 
Since the unit element of $B^{(0,1]}_{[p]} $ (indexing the empty tangle) is a disjoint and contractible component, the fiber is $(B^{(0,1]}_{[p]})^{k+1}$.

Consequently, the fiber of the map $\Bord^{[0,0]}[p] \star B^{(0,1]}_{[p]}  \ra  \Bord^{[0,1]}[p]$ over each $W_0\in \Bord^{[0,0]}[p]$ induces an equivalence on fibers
\[
\bigl(B^{(0,1]}_{[p]}\bigr)^{k+1} \simeq \prod_{|\pi_0W_0|+1} B^{(0,1]}_{[p]} 
\]
for $k = |\pi_0W_0|$. This implies that the map $\Bord^{[0,0]}[p] \star B^{(0,1]}_{[p]}  \ra  \Bord^{[0,1]}[p]$ is an equivalence.
\end{proof}

\begin{cor}\label{cor.lemma.points.of.Bord01}
Recall the notation of Lemma~\ref{lemma.points.of.Bord01}, where $B^{(0,1]}_{[p]}  \subset \Bord^{[0,1]}[p]$ is again the subspace of tangles $W\subset \RR^1\times\DD[p]$ for which no component of $W$ is in the image of $\{0\}$ under the map $\FF_1(\{0,1\}) \simeq\obj\bigl( \Bord^{[0,1]}\bigr) \ra  \Bord^{[0,1]}[p]$. Then there is a pushout of monoids:
\[
\xymatrix{
\Bord^{[0,0]}[p] \ar[d]^-{0\mapsto j}\ar[r]_-{0\mapsto 0}&\Bord^{[0,1]}[p] \ar[d]^-{0\mapsto j}\\
\Bord^{[i,j]}[p] \ar[r]^-{0\mapsto 0} & \Bord^{[i,j]}[p]\star B^{(0,1]}_{[p]}   }
\]
where $\Bord^{[i,j]}[p]\star B^{(0,1]}_{[p]} $ is the coproduct of $B^{(0,1]}_{[p]} $ and $\Bord^{[i,j]}[p]$ in $\Alg(\Spaces)$.
\end{cor}
\begin{proof}
Using the equivalence $\Bord^{[0,0]}[p]\star B^{(0,1]}_{[p]}  \simeq  \Bord^{[0,1]}[p]$ of Lemma~\ref{lemma.points.of.Bord01}, we obtain
\[
\Bord^{[i,j]}[p] \underset{\Bord^{[0,0]}[p]}\star\Bord^{[0,1]}[p]
\simeq
\Bord^{[i,j]}[p] \underset{\Bord^{[0,0]}[p]}\star\bigl( \Bord^{[0,0]}[p]\star B^{(0,1]}_{[p]}  \bigr)\simeq
\Bord^{[i,j]}[p]\star B^{(0,1]}_{[p]} ~.
\]
\end{proof}

\begin{lemma}\label{lemma.main.basecase.n2}
The commutative diagram
\[
\xymatrix{
\Bord^{[0,0]} \ar[d]^-{0\mapsto j}\ar[r]_-{0\mapsto 0}&\Bord^{[0,1]} \ar[d]^-{0\mapsto j}\\
\Bord^{[i,j]} \ar[r]^-{0\mapsto 0} & \Bord^{[i,j+1]}  }
\]
is a pushout of monoidal $(\oo,1)$-categories.
\end{lemma}
\begin{proof}
This proof relies on Lemmas~\ref{lemma.claimB} and~\ref{lemma.claimC}, which are stated and proved below.

Let $\Bord^{[a,b]}_\bullet$ denote the simplicial monoidal space, given by the value of the forgetful functor
\[
\Alg(\Cat_{(\oo,1)}) \ra \Fun\bigl(\bDelta^{\op}, \Alg(\Spaces)\bigr)
\]
on the monoidal $(\oo,1)$-category $\Bord^{[a,b]}$. Consider the pushout
\[
\Bord^{[i,j]}_\bullet \underset{\Bord^{[0,0]}_\bullet}\star\Bord^{[0,1]}_\bullet
\]
in $\Fun(\bDelta^{\op}, \Alg(\Spaces))$, which takes values
\[
[p]\mapsto \Bord^{[i,j]}[p] \underset{\Bord^{[0,0]}[p]}\star\Bord^{[0,1]}[p]
~,
\]
the pushout in $\Alg(\Spaces)$ of the monoidal spaces of $[p]$-simplices. We prove that the natural morphism
\begin{equation}\label{Seg.to.Bord}
\Seg\Bigl(
\Bord^{[i,j]}_\bullet \underset{\Bord^{[0,0]}_\bullet}\star\Bord^{[0,1]}_\bullet
\Bigr)
\longrightarrow
\Bord^{[i,j+1]}_\bullet
\end{equation}
is an equivalence. By Lemma~\ref{lemma.Seg.universal}, this will imply that $\Bord^{[i,j+1]}_\bullet$ is the universal Segal space under the colimit of $\Bord^{[0,1]}_\bullet \la{\Bord^{[0,0]}_\bullet}\ra\Bord^{[i,j]}_\bullet$. Since Segal-completion preserves colimits, being a left adjoint, this will show that $\Bord^{[i,j+1]}$ is the colimit of $\Bord^{[0,1]} \la{\Bord^{[0,0]}}\ra\Bord^{[i,j]}$ in monoidal $(\oo,1)$-categories.

Consider the functor $F$ defined as the composite in the following diagram:
\begin{equation}
\label{e.F}
\xymatrix{
\TwAro(\bDelta^{\op})_{/[p]^\circ}\ar[d]\ar[rrrr]^-F&&&&\spaces\\
\TwAro(\bDelta^{\op})\ar[r]^-S&\Morita[\bDelta^{\op}]\ar[rr]^-{\Bord^{[i,j]} \underset{\Bord^{[0,0]}}\star\Bord^{[0,1]}}
&&\cM[\Spaces]\ar[r]^-\sL&\Morita[\Spaces]\ar[u]^-{\sf LIM}}
\end{equation}
To prove that the map~(\ref{Seg.to.Bord}) is an equivalence, we show that the colimit of the composite functor $F$ gives the $p$-simplices $\Bord^{[i,j+1]}[p]$. That is, by construction of the functor $\Seg$, it will imply that the map~(\ref{Seg.to.Bord}) is an equivalence if the map
\[
\colim\Bigl(
\TwAro(\bDelta^{\op})_{/[p]^\circ}\xra{F}\Spaces
\Bigr)
\longrightarrow
\Bord^{[i,j+1]}[p]
\]
is an equivalence. We will compute this colimit after restricting along a final functor. 
Consider a closed interval $I$, and fix an object $P\in \sD_{[0,1]}$ with $\pi_0([0,1]\smallsetminus P) = [p]$. From Lemma~\ref{localize.D.Delta}, the colimit of the functor $F$ can be computed after restriction along
\[
\TwAr(\sD_I)^{\op}_{/P} \ra \TwAr(\bDelta^{\op})^{\op}_{/[p]^\circ}~.
\]
We are thereby reduced to proving the map
\[
\colim\Bigl(
\TwAr(\sD_I)^{\op}_{/P}\xra{F}\Spaces
\Bigr)
\longrightarrow
\Bord^{[i,j+1]}[p]
\]
is an equivalence,
where $F$ is the restriction of $F$ to $\TwAr(\sD_I)^{\op}_{/P}$.

First, we make use of an explicit topological space $\Tang_1^{[i,j+1]}(\RR^1\times[0,1])_P$ for which there is a natural homotopy equivalence $\Bord^{[i,j+1]}[p]\simeq \Tang_1^{[i,j+1]}(\RR^1\times[0,1])_P$. See Definition~\ref{def.TangP}.

Second, in Definition~\ref{def.U} we construct a functor
\[
\TwAr(\sD_I)^{\op}_{/P}\overset{\fU}\longrightarrow\Opens\Bigl(\Tang_1^{[i,j+1]}(\RR^1\times[0,1])_P\Bigr)
\]
assigning to an object $P\la A\ra B$ an open subspace $\fU_{A\ra B}$ of $\Tang_1^{[i,j+1]}(\RR^1\times[0,1])_P$. We then prove:

\begin{itemize}
\item Lemma~\ref{lemma.claimB}: For any $W\in \Tang_1^{[i,j+1]}(\RR^1\times[0,1])_P$, then the subcategory
\[
\Bigl(\TwAro(\sD_{[0,1]})_{/P}\Bigr)_W
\subset
\TwAro(\sD_{[0,1]})_{/P}
\]
of $P\la A\ra B$ such that there is containment $W\in \fU_{A\ra B}$ has contractible classifying space. 
Consequently, applying the higher Seifert--van Kampen theorem of \cite{HA} (specifically Theorem~A.3.1), there is a natural homotopy equivalence
\[
\Tang_1^{[i,j+1]}(\RR^1\times[0,1])_P
\simeq
\colim\Bigl(
\TwAr(\sD_I)^{\op}_{/P}
\xra{\fU}
\Opens\bigl(\Tang_1^{[i,j+1]}(\RR^1\times[0,1])_P\bigr)
\xra{\rm forget}
\Spaces
\Bigr)~.
\]
\item Lemma~\ref{lemma.claimC}:
There is a natural equivalence
\[
\fU_{A\ra B} \simeq F(P\la A \ra B)
\]
as functors from $\TwAro(\sD_I)_{/P}$ to $\Spaces$.
\end{itemize}
The above results, Lemma~\ref{lemma.claimB} and Lemma~\ref{lemma.claimC}, then imply the equivalence
\[
\colim\Bigl(\TwAr(\sD_I)^{\op}_{/P}\xra{F}\Spaces\Bigr)\simeq \Tang_1^{[i,j+1]}(\RR^1\times[0,1])_P
\]
which completes the proof.
\end{proof}

The following definitions and results support the above proof of Lemma~\ref{lemma.main.basecase.n2}, starting with Definition~\ref{def.TangP} and leading to Lemma~\ref{lemma.claimCprime}.

\begin{definition}\label{def.TangP}
Let $P\subset [0,1]$ be an object of $\sD_{[0,1]}$, with $\iota P\subset P$ the components of $P$ which map to the interior of $[0,1]$. The subspace
\[
\Tang_1^{[i,j]}(\RR^1\times[0,1])_P
\subset
\Tang_1^{\fr}(\RR^1\times[0,1])
\]
consists of those tangles $W$ such that: 
\begin{itemize}
\item the critical values of the composite projection $W\subset \RR^1\times[0,1] \ra [0,1]$ are contained in $\iota P$; 
\item the intersection $W\cap \RR^1\times\{t\}$ is an object of $\Bord^{[i,j]}$ for $t\in \{0,1\}$.
\end{itemize}
\end{definition}

\begin{definition}\label{def.U}
The functor
\[
\xymatrix{
\TwAro(\sD_{[0,1]})_{/P}
\ar[r]^-\fU&
\Opens\Bigl(
\Tang_1^{[i,j+1]}(\RR^1\times[0,1])_P
\Bigr)
}
\]
sends an object $\bigl(P\la A\ra B\bigr)$ to the open subspace $\fU_{A\ra B}$, defined by:
\[
\fU_{A\ra B} := \Bigl\{(W\subset\RR^1\times[0,1])\in\Tang_1^{[i.j+1]}(\RR^1\times[0,1])_A\Big| \ j\text{-splitting relative } B  \Bigr\}
\]
where $W$ is a properly-embedded 1-dimensional submanifold with no closed components, 
the projection $W \to [0,1]$ is transverse to every point in the complement of $A$ and satisfying the following \bit{$j$-splitting condition} relative to $B$:
\begin{itemize}
\item for each pair of consecutive components of $[0,1]\smallsetminus B$, there exists elements $t, t'$ of these components such that the intersection
\[
W\cap \bigl( \RR^1\times[t, t']\bigr)
\]
has the property that there is a finite subset $S\subset \RR^1$ for which:
\begin{itemize}
\item $W\cap \bigl( \RR^1\times[t, t']\bigr)$ is disjoint from $S\times[0,1]$;
\item for any consecutive elements $s$, $s'$ in $S\cup\{-\oo,\oo\}$, the intersection 
$W \cap [s,s']\times[t,t']$ is either $[i,j]$-framed or a component of this intersection is $\{j+1\}$-framed if its end-points map surjectively onto $\{t,t'\}$.
\end{itemize}
\end{itemize}
\end{definition}
\begin{remark}
Intuitively, the $[i,j+1]$-framed tangle $W\subset \RR^1 \times [0,1]$ satisfies the $j$-splitting condition relative to $B$ if, over each connected component of $B$ in the interior of $[0,1]$, it can be split as a disjoint union of tangles each of which is either $[i,j]$-framed or $\{j,j+1\}$-framed.
See Figure~\ref{fig12}.
\end{remark}

\begin{figure}[H]
  \includegraphics[width=\linewidth, trim={0 {2.2in} {0in} {.7in}}, clip]{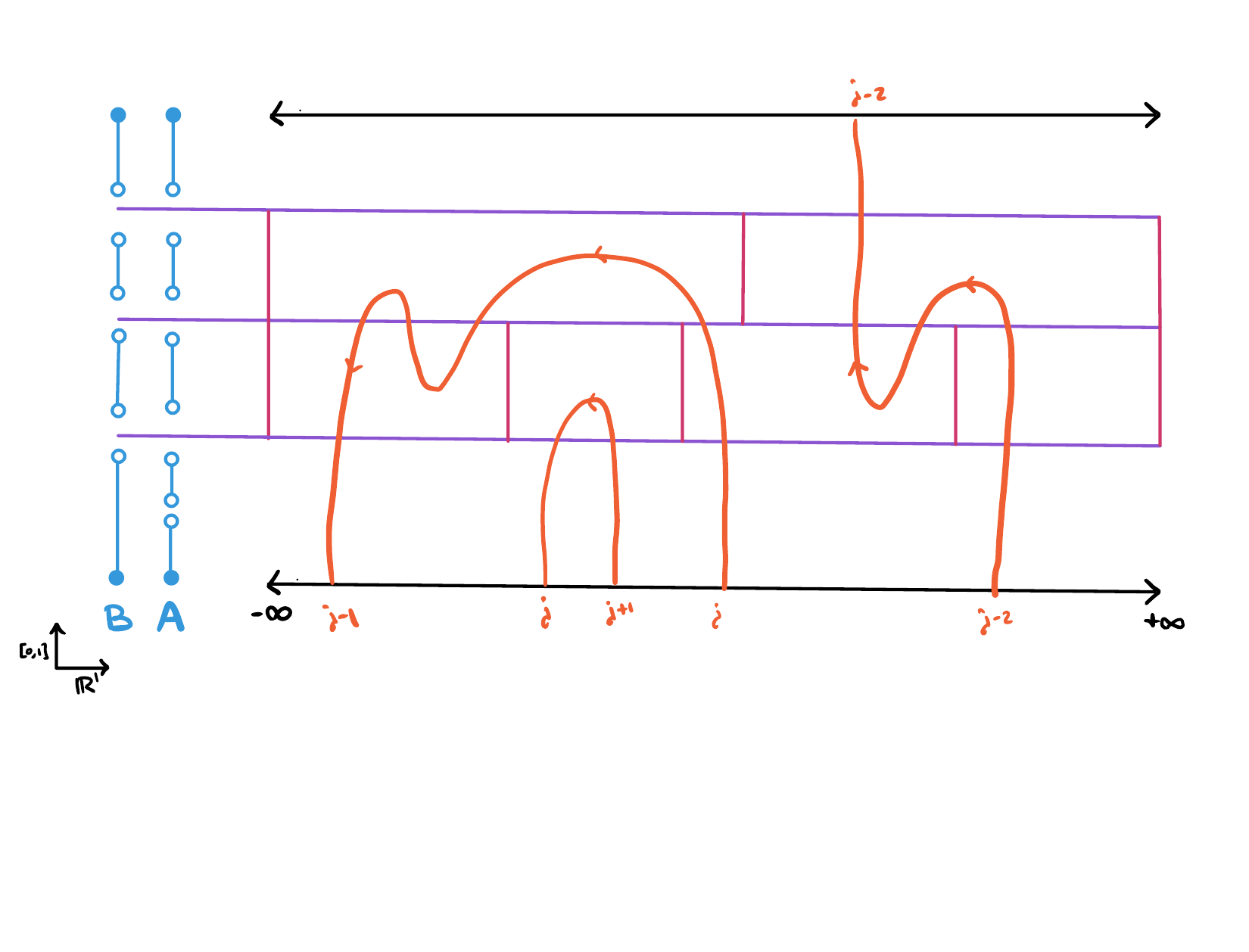}
  \caption{A $[i,j+1]$-framed tangle (in orange) $W \subset \RR^1 \times [0,1]$ such that the critical values of the projection $W \to [0,1]$ are contained in $A$ and that satisfies the $j$-splitting condition relative to $B$.
The (purple) horizontal lines depict representatives $t\in [0,1]\smallsetminus B$ of each connected component, as in Definition~\ref{def.TangP}.
The (red) vertical line segments depict a choice of $S \cup \{\pm \infty\} \subset \RR^1 \cup \{\pm \infty\}$ for each connected component of $B$.}
  \label{fig12}
\end{figure}

\begin{lemma}\label{lemma.claimB}
For any $W\in \Tang_1^{[i,j+1]}(\RR^1\times[0,1])_P$, the full subcategory
\[
\Bigl(\TwAro(\sD_{[0,1]})_{/P}\Bigr)_W
~\subset~
\TwAro(\sD_{[0,1]})_{/P}~,
\]
consisting of those $(P\la A\ra B)$ such that $\fU_{A\ra B}$ contains $W$, has contractible classifying space.

\end{lemma}
\begin{proof}
First, we prove the category $\bigl(\TwAro(\sD_{[0,1]})_{/P}\bigr)_W$ is nonempty. For $B\in \sD_{[0,1]}$, denote by $\iota B \subset B$ the components of $B$ contained in the interior of $[0,1]$. We claim that there exists $B\in \sD_{[0,1]}$ such that:
\begin{itemize}
\item $\iota B$ contains the critical values of the composite projection $W\subset \RR^1\times[0,1] \ra [0,1]$;
\item In any consecutive components of $[0,1]\smallsetminus B$ there exists elements $t$ and $t'$ and a finite set $S\subset \RR^1$ such that for any successive elements $s, s'$ of $S$, the intersection of $W$ with $[t,t'] \times [s,s']$ is either empty or connected.
\end{itemize}
Before proving that such a $B$ exists, we first observe that this existence implies that $\bigl(\TwAro(\sD_{[0,1]})_{/P}\bigr)_W$ is nonempty. Given such a $B$, define $A: = B\cap P$, and note that this defines an object of $\sD_{[0,1]/P}$, and thus that $P\la A \ra B$ then defines an object of $\TwAro(\sD_{[0,1]})_{/P}$. 
The subset $A$ contains the critical values of the projection $W\subset \RR^1\times[0,1] \ra [0,1]$, since both $P$ and $B$ do. 
The $j$-splitting condition for $W$ relative to $B$ follows immediately, because: 
every framed connected tangle is $\{k,k+1\}$-framed for some integer $k$, and therefore every every connected $[i,j+1]$-framed tangle satisfies the $j$-splitting condition since it is either $[i,j]$-framed or $\{j,j+1\}$-framed.

We now prove $B$ exists. For each element $x\in [0,1]$, there exists a closed neighborhood $D_x$ such that the connected components of $W\cap \RR^1 \times D_x$ are split: that is, there exists a finite subset $S\subset\RR^1$ such that for any successive elements $s, s'$ of $S$, the intersection of $W$ with $[s,s'] \times D_x$ is either empty or connected. By compactness of $[0,1]$, there exists a minimal finite subset $T \subset [0,1]$ for which, for each $x\in T$, the interiors $\overset{\circ}D_x$ cover the interval $(0,1)$.  For each pair of consecutive elements $x,x'$ of $T$, choose an element $y \in \overset{\circ} D_x\cap \overset{\circ} D_{x'}$ which is also a regular value of the composite projection $W\subset \RR^1\times[0,1] \ra [0,1]$, and set $Y$ to be this finite set of elements $y$. Setting $B$ to be the complement $B:= [0,1] \smallsetminus Y$, we see that $W$ satisfies splitting with respect to $B$, since any consecutive element $y,y'$ of $Y = [0,1]\smallsetminus B$ lie in a single open $\overset{\circ} D_x$ for some $x\in T$.

Second, we now prove $\bigl(\TwAro(\sD_{[0,1]})_{/P}\bigr)_W$ has contractible classifying space. We use the following condition for contractibility:
\begin{itemize}
\item Let $\cT$ be a connected category such that for every functor $G:\cK \ra \cT$ from a connected poset $\cK$, there exists a factorization
\[
\xymatrix{
\cK\ar[rd]\ar[rr]^-G&&\cT\\
&\ov{\cK}\ar@{-->}[ur]_-{\ov{G}}}
\]
where $\ov{\cK}$ is a category with contractible classifying space. Then $\cT$ has contractible classifying space.
\end{itemize}
We establish this factorization for $\cT = \bigl(\TwAro(\sD_{[0,1]})_{/P}\bigr)_W$ using the following construction: For any pair of objects $P\la A \ra B$ and $P\la A' \ra B'$ of $\bigl(\TwAro(\sD_{[0,1]})_{/P}\bigr)_W$, observe that both $P\la A\cap A'\ra B$ and $P\la A\cap A'\ra B\cap B'$ form objects of
$\TwAro(\sD_{[0,1]})_{/P}$. Further, $\iota (A\cap A')$ contains the critical values of $W \ra [0,1]$, and $W$ satisfies the $j$-splitting condition relative to $B\cap B'$.

We can thus form a pair of cospans in $\bigl(\TwAro(\sD_{[0,1]})_{/P}\bigr)_W$
\[
\xymatrix{
A\ar[d]&\ar[d]A\cap A'\ar[l]\ar[r]&A\cap A'\ar[d]&A\cap A'\ar[r]\ar[d]\ar[l]&A'\ar[d]\\
B\ar[r]&B&\ar[l] B\cap B'\ar[r]& B'&\ar[l]B'
}
\]
which shows that $\bigl(\TwAro(\sD_{[0,1]})_{/P}\bigr)_W$ is connected.

Now let $G:\cK \ra \bigl(\TwAro(\sD_{[0,1]})_{/P}\bigr)_W$ be any functor from a connected poset. We can construct a factorization:
\[
\xymatrix{
\cK\ar[rd]\ar[rr]^-G&&\bigl(\TwAro(\sD_{[0,1]})_{/P}\bigr)_W\\
&\ov{\cK}\ar@{-->}[ur]_-{\ov{G}}}
\]
with
\[
\ov{\cK} = \ast\underset{\cK\times \{-\}}\coprod\cK\times \{-\ra 0\la +\}
\]
where $\{-\ra 0\la +\}$ is the 3-object category which corepresents cospans, and the functor $\cK \ra \ov{\cK}$ is given by the inclusion of $\cK\times\{+\}$. The functor $\ov{G}$ is characterized by declaring, for each $k\in \cK$, its restriction along $\{k\},  \{-\ra 0\la +\} \ra \ov{\cK}$ to be the cospan
\[
\xymatrix{
\underset{x\in \cK}\bigcap \ev_s( G(x) ) 
\ar[d]&\ar[l]\underset{x\in \cK}\bigcap \ev_s(G(x)) 
\ar[d]\ar[r]&
\ev_s(G(k)) 
\ar[d]\\
\underset{x\in \cK}\bigcap 
\ev_t(G(x))\ar[r]& \ev_t(G(k)) &\ar[l] \ev_t(G(k))
.
}
\]
The classifying space of $\ov{\cK}$ is contractible, so this proves that the classifying space of $\bigl(\TwAro(\sD_{[0,1]})_{/P}\bigr)_W$ contractible.
\end{proof}

Recall the composite functor $F$ from~(\ref{e.F}).
\begin{lemma}\label{lemma.claimC}
Let $A \ra B$ be an object of $\TwAro(\sD_{[0,1]})$. There is a natural equivalence
$\fU_{A\ra B} \simeq F(\pi A \ra \pi B)$
between the open $\fU_{A\ra B}\subset \Tang_1^{[i,j+1]}(\RR^1\times[0,1])_A$ and the value on $A\ra B$ of the composite
\[
\TwAro(\sD_{[0,1]})\overset{\pi}\longrightarrow \TwAro(\bDelta^{\op})\overset{F}\longrightarrow\Spaces~.
\]
\end{lemma}
\begin{proof}
Denote the morphism $\pi A \to \pi B$ in $\bDelta^{\op}$ as $[a]^{\circ} \xra{g^{\circ}} [b]^{\circ}$.  
Here we reduce the proof for general morphisms $[b]\xra{g}[a]$ to the case $b=1$ and $g$ the unique active morphism. This special case is then proved in Lemma~\ref{lemma.claimCprime}.

By Observation~\ref{obs.Segal.basecase.contact}, the value $F([a]^\circ\xra{g^\circ}[b]^\circ)$ is given by the iterated fiber product:
\[
F([0,g(0)])
\underset{F(g(0))}\times
F([g(0),g(1)])
\underset{F(g(1))}\times
\ldots
\underset{F(g(b))}\times
F([g(b),a])
\]
where 
\[
F([k,l]) = \Bord^{[i,j]}[k,l]\underset{\Bord^{[0,0]}[k,l]}\amalg \Bord^{[0,1]}[k,l]
\]
and
\[
F(\{k\}) = \Bord^{[i,j]}\{k\}\underset{\Bord^{[0,0]}\{k\}}\amalg \Bord^{[0,1]}\{k\}~.
\]
Likewise, $\fU_{A\ra B}$ has a cover by smaller values of $\fU$: for $B\subset [0,1]$ with $b+1$ connected components, write
\[
B = \coprod_{0\leq k \leq b} B_k
\]
as the ordered disjoint union of connected components. Define 
\[
\widehat{B_0} = \bigl\{ t\in [0,1]\ \big| \ 0 \leq t \leq {\sf inf}(B_{1})\bigr\}
\]
\[
\widehat{B_k} = \bigl\{ t\in [0,1]\ \big| \ {\sf sup}(B_{{\mathit{k}\text{-}1}}) \leq t \leq {\sf inf}(B_{k+1})\bigr\}
\]
\[
\widehat{B_{b}} = \bigl\{ t\in [0,1]\ \big| \ {\sf sup}(B_{b-1}) \leq t \leq 1\bigr\}
\]
for the largest connected sub-interval of $[0,1]$ containing $B_k$ which does not intersect the adjacent components $B_{{\mathit{k}\text{-}1}}$ or $B_{k+1}$. Note that the $\widehat{B_k}$ form a closed cover of $[0,1]$. We can thus decompose $\Tang_1^{[i,j+1]}(\RR^1\times[0,1])_A$ as the fiber product
\[
\Tang_1^{[i,j+1]}(\RR^1 \times \widehat{B_0})_{A_0}
\underset{\Tang_1^{[i,j+1]}( \RR^1 \times \widehat{B_0}\cap\widehat{B_1} )_{\emptyset}}{\times}
\ldots
\underset{\Tang_1^{[i,j+1]}( \RR^1 \times \widehat{B_{b-1}}\cap\widehat{B_b} )_{\emptyset}}{\times}
\Tang_1^{[i,j+1]}( \RR^1 \times \widehat{B_n} )_{A_b}
\]
where $\Tang_1^{[i,j+1]}(\RR^1 \times \widehat{B_k} )_{A_k}$ consists of tangles in $\RR^1 \times\widehat{B_k}$ whose critical values are contained in $A_k := A \cap B_k$, and $\Tang_1^{[i,j+1]}(\RR^1 \times\widehat{B_{\n1}}\cap\widehat{B_n})_{\emptyset}$ consists of tangles in $\RR^1 \times\widehat{B_{\n1}}\cap\widehat{B_n}$ which have no critical points. Note that $\Tang_1^{[i,j+1]}(\RR^1 \times\widehat{B_{\n1}}\cap\widehat{B_n})_{\emptyset}$ is homotopy equivalent to the space of objects of $\Bord^{[i,j+1]}$.

Restricting to the subspace $\fU_{A\ra B} \subset   \Tang_1^{[i,j+1]}(\RR^1\times[0,1])_A$, it is immediate that a tangle $W$ satisfies the $j$-splitting condition relative to $B$ if and only if for each $k$ each restriction $W_{|\RR^1 \times \widehat{B_k} }$ satisfies the $j$-splitting condition relative to $B_k$. We can thus decompose $\fU_{A\ra B}$ as a fiber product
\[
\fU_{A_1\ra B_1}\underset{\fU_{\emptyset \to C_1}}\times\ldots \underset{\fU_{\emptyset \to C_b}}\times\fU_{A_b\ra B_b}~,
\]
where $C_k$ is the interior of $\widehat{B_{k\text{-}1}}\cap \widehat{B_k}$.
By the parameterized isotopy extension theorem, the restrictions maps in this fiber product are fibrations, so the point-set fiber product agrees with the homotopy fiber product.

To complete the proof, we can thereby reduce to the case where each $A_i\ra B_i$ corresponds to an active morphism $\pi A = [a]^\circ \ra [1]^\circ = \pi B$. This case is proved in Lemma~\ref{lemma.claimCprime}.
\end{proof}

We have reduced Lemma~\ref{lemma.claimC} to the simpler subordinate case in which the object $[a]^\circ\ra [b]^\circ$ is the active morphism $[a]^\circ\ra[1]^\circ$.
\begin{lemma}\label{lemma.claimCprime}
For $(A\ra B)\in \TwAr(\sD_{[0,1]})$ an object with $\pi B =\{0<a\}$ and $\pi A = [a]^\circ\xra{\act}\{0<a\}^\circ = \pi B$ the active morphism, the space $\fU_{A\ra B}$ is homotopy equivalent to the pushout in $\Alg(\Spaces)$
\[
\fU_{A\ra B}
\simeq
\Bord^{[i,j]}[a]\underset{\Bord^{[0,0]}[a]}\star \Bord^{[0,1]}[a]
\]
of the monoidal spaces of $[a]$-points.
\end{lemma}
\begin{proof}
In the case at hand, the space $\fU_{A\ra B}$ consists of tangles $W$ with critical values contained in $A$ such that there exists elements $t, t'$ (with $\iota A\subset \iota B\subset [t,t']$) of these components such that the intersection
\[
W\cap \bigl( \RR^1\times[t, t']\bigr)
\]
has the property that there is a finite subset $S\subset \RR^1$ for which  $W\cap \bigl( \RR^1\times[t, t']\bigr)$ is disjoint from $S\times[0,1]$, and for any successive elements $s$, $s'$ in $S$, the intersection of $W$ with $[s,s']\times[t,t']$ satisfies the $j$-splitting condition.

For each $r\geq 0$, consider the subspaces
\[
\fU^{\{1,\ldots, r\}}_{A\ra B}\subset \fU^r_{A\ra B} \subset \fU_{A\ra B}
\]
where:
\begin{itemize}
\item $\fU^r_{A\ra B}\subset \fU_{A\ra B}$ consists of those tangles $W$ such that there exists a cardinality $r$ subset $S\subset \RR^1$ with the  $j$-splitting property relative to $W$ (see Definition~\ref{def.U}), and any other subset $S' \subset \RR^1$ with the $j$-splitting property relative to $W$ has cardinality at least $r$;
\item $\fU^{\{1,\ldots, r\}}_{A\ra B}\subset \fU^r_{A\ra B}$ consists of those $W$ which have the $j$-splitting property relative to the fixed subset $S=\{1,\ldots,r\}\subset\RR^1$.
\end{itemize}
Note that $\fU^r_{A\ra B}$ is a union of components of $\fU_{A\ra B}$: A tangle $W\in \fU_{A\ra B}$ belongs to $\fU^r_{A\ra B}$ if and only if every isotopic tangle $W'\in \fU_{A\ra B}$ belongs to $\fU^r_{A\ra B}$. Consequently, there is a decomposition
\[
\fU_{A\ra B}
\simeq
\{\emptyset\}\amalg\coprod_{r\geq 0} \fU^r_{A\ra B}~.
\]
Further, via a straight-line ambient isotopy, there is a deformation retraction of $\fU^r_{A\ra B}$ onto $\fU^{\{1,\ldots, r\}}_{A\ra B}$. For $r=2k-1$ or $r=2k$, respectively, the space $\fU^{\{1,\ldots, r\}}_{A\ra B}$ is homotopy equivalent to 
\[
\Bigl(\ov{\Bord^{[i,j]}[a]}\times \ov{B^{(0,1]}_{[a]}}\Bigr)^{k}
\amalg
\Bigl(\ov{B^{(0,1]}_{[a]}}\times \ov{\Bord^{[i,j]}[a]}\Bigr)^{k}
\]
or, respectively,
\[
\ov{\Bord^{[i,j]}[a]} \times \Bigr(\ov{B^{(0,1]}_{[a]}}\times \ov{\Bord^{[i,j]}[a]}\Bigl)^k
\amalg \
\ov{B^{(0,1]}_{[a]}}\times \Bigl(\ov{\Bord^{[i,j]}[a]}\times \ov{B^{(0,1]}_{[p]}}\Bigr)^k 
\]
where $\ov{B^{(0,1]}_{[a]}}$ and  $\ov{\Bord^{[i,j]}[a]}$ denote the spaces $\ov{B^{(0,1]}_{[a]}}$ and $\ov{\Bord^{[i,j]}[a]}$ minus the unit (i.e., the 1-point component $\{\emptyset\}$ containing the empty tangle).

By Corollary~\ref{cor.lemma.points.of.Bord01}, we have an equivalence between spaces:
\[
\Bord^{[i,j]}[a]\underset{\Bord^{[0,0]}[a]}\star \Bord^{[0,1]}[a]\simeq \Bord^{[i,j]}[a]\star B^{(0,1]}_{[a]}~.
\]
The equivalence $\fU_{A\ra B}\simeq \Bord^{[i,j]}[a]\star B^{(0,1]}_{[a]}$ now follows by Lemma~\ref{lemma.monoid.pushout}, which implies the result.
\end{proof}

We can now prove the main theorem of this section, which identifies the rigid monoidal $(\oo,1)$-category $\Dual^{[\oo,\oo]}$ freely generated by a single object. For earlier work on freely generated monoidal $(1,1)$-categories, see~\cite{joyal.street1} and~\cite{joyal.street2}.

\begin{theorem}\label{thm.n2}
The natural functor
\[
\Dual^{[\oo,\oo]}\longrightarrow \Bord_1^{\fr}(\RR^1)
\]
is an equivalence. That is, $\Bord_1^{\fr}(\RR^1)$ is equivalent to the universal monoidal $(\oo,1)$-category containing an object which has all left-duals and all right-duals.
\end{theorem}
\begin{proof}
By Proposition~\ref{tDualBord}, the natural functors
\[
\Dual^{[0,1]} \ra \Bord^{[0,1]}
\]
\[
\Dual^{[1,0]} \ra \Bord^{[1,0]}
\]
are equivalences.

By Lemma~\ref{lemma.main.basecase.n2}, the commutative diagram
\[
\xymatrix{
\Bord^{[0,0]} \ar[d]^-{0\mapsto j}\ar[r]_-{0\mapsto 0}&\Bord^{[0,1]} \ar[d]^-{0\mapsto j}\\
\Bord^{[i,j]} \ar[r]^-{0\mapsto 0} & \Bord^{[i,j+1]}  }
\]
is a pushout of monoidal $(\oo,1)$-categories. For $0\leq i,j< \oo$, the diagram
\[
\xymatrix{
\Dual^{[0,0]}\ar[r]^-{\langle D^{\sR^j}\rangle}\ar[d]^-{D}&\Dual^{[i,j]}\ar[d]^-{D}\\
\Dual^{[0,1]}\ar[r]^-{\langle D^{\sR^j}\rangle}&\Dual^{[i,j+1]}}
\]
is a pushout in the $(\oo,1)$-category $\Alg(\Cat_{(\oo,1)})$ of monoidal $(\oo,1)$-categories. This follows immediately from the universal property defining $\Dual^{[i,j]}$. Consequently, we obtain that the natural transformation 
\[
\Dual^{[\bullet,\bullet]}\longrightarrow\Bord^{[\bullet,\bullet]}
\]
is an equivalence in $\Fun(\NN\times\NN,\Alg(\Cat_{(\oo,1)}))$. Lastly, the natural functor
\[
\xymatrix{
\underset{(i,j)\in \NN\times\NN}\colim \Bord^{[i.j]}\ar[r]& \Bord_1^{\fr}(\RR^1)}
\]
is an equivalence by inspection: Each of the functors $\Bord^{[i,j]}\ra \Bord_1^{\fr}(\RR^1)$ is fully-faithful, and every object of $\Bord_1^{\fr}(\RR^1)$ belongs to $\Bord^{[i,j]}$ for some $i$ and $j$. That the likewise colimit
\[
\xymatrix{
\underset{(i,j)\in \NN\times\NN}\colim \Dual^{[i.j]}\ar[r]& \Dual^{[\oo,\oo]}}
\]
is an equivalence again follows immediately from the universal property of $\Dual^{[i,j]}$. This implies the stated equivalence.
\end{proof}

\section{Induction from $\cE_{\n1}$-algebras to $\cE_n$-algebras}

In this section, we construct a model for the induction functor
\[
\Ind_{\n1}^n: \Alg_{\n1}(\Spaces)\longrightarrow \Alg_n(\Spaces)
\]
defined as the left adjoint to the functor given by restriction along the standard morphism $\cE_{\n1}\ra \cE_n$. We then analyze the value of operadic induction on the spaces of tangles
\[
\Ind_{\n1}^n\Bigl(
\Bord_1^{\fr}(\RR^{\n1})[\bullet]
\Bigr)
\]
to deduce the $\cE_n$-monoidal Tangle Hypothesis for $\Bord_1^{\fr}(\RR^n)$ from the $\cE_{\n1}$-monoidal Tangle Hypothesis for $\Bord_1^{\fr}(\RR^{\n1})$.

\subsection{$\cE_n$-Induction}\label{sec.n.ind}

Recall the $(\oo,1)$-category $\Disk_k^{\fr}$ of framed $k$-dimensional Euclidean spaces: An object is a finite disjoint union of $\RR^k$, and a morphism is a framed embedding. See Definitions~2.1, 2.7, and 2.9 from \cite{oldfact}. For each $X\in \Disk_{\n1}^{\fr}$ and $Y\in \Disk_1^{\fr}$, consider the pullback
\[
\xymatrix{
\bigl(\Disk_{\n1/X}^{\fr}\times \Disk_{1/Y}^{\fr}\bigr)\underset{\Fin}\times \Ar^{\sf mono}(\Fin)\ar[d]\ar[rr]&&\Ar^{\sf mono}(\Fin)\ar[d]^-{\ev_t}\\
\Disk_{\n1/X}^{\fr}\times \Disk_{1/Y}^{\fr}\ar[r]^-{\pi_0\times\pi_0}&\Fin\times\Fin\ar[r]^-\times&\Fin}
\]
where $\Ar^{\sf mono}(\Fin)$ is the full subcategory of $\Ar(\Fin)$ whose objects are the injections of finite sets.
Note that an object of the pullback consists of a triple
\[
\bigl(U\in \Disk_{\n1/X}^{\fr}, V\in \Disk_{1/Y}^{\fr}, S\subset \pi_0U\times\pi_0V\bigr)~.
\]

\begin{definition}\label{def.P}
The $(\oo,1)$-category $\cP_{X,Y}$ is the $\oo$-subcategory of
\[
\bigl(\Disk_{\n1/X}^{\fr}\times \Disk_{1/Y}^{\fr}\bigr)\underset{\Fin}\times \Ar^{\sf mono}(\Fin)
\]
consisting of:
\begin{itemize}
\item those objects $(U\in \Disk_{\n1/X}^{\fr}, V\in \Disk_{1/Y}^{\fr}, S\subset \pi_0U\times\pi_0V)$ such that the composite maps $S\ra \pi_0U\times\pi_0V\ra \pi_0U$ and $S\ra \pi_0U\times\pi_0V\ra \pi_0V$ are surjective;
\item those morphisms $f=(U\ra U', V\ra V', S\ra S')$ such that for all $\alpha \in \pi_0U$, the restriction of $f$ to
\[
S\cap(\{\alpha\}\times\pi_0V)\xra{f_|} S'
\]
is injective.
\end{itemize}
\end{definition}
\begin{lemma}\label{lemma.functoriality.of.cP}
Given embeddings $X\hookrightarrow X'$ and $Y\hookrightarrow Y'$, the natural functor
\[
\bigl(\Disk_{\n1/X}^{\fr}\times \Disk_{1/Y}^{\fr}\bigr)\underset{\Fin}\times \Ar^{\sf mono}(\Fin)
\longrightarrow
\bigl(\Disk_{\n1/X'}^{\fr}\times \Disk_{1/Y'}^{\fr}\bigr)\underset{\Fin}\times \Ar^{\sf mono}(\Fin)
\]
restricts to a functor of $\oo$-subcategories $\cP_{X,Y}\ra \cP_{X',Y'}$. The values $\cP_{X,Y}$ thereby define a functor
\[
\Disk_{\n1}^{\fr}\times \Disk_{1}^{\fr}\longrightarrow \Cat
\]
assigning $(X,Y)\mapsto \cP_{X,Y}$.
\end{lemma}

\begin{proof}
We observe that the objects and morphisms of Definition~\ref{def.P} are preserved under the functor induced by $X\hookrightarrow X'$ and $Y\hookrightarrow Y'$.

Given an object $(U\in \Disk_{\n1/X}^{\fr}, V\in \Disk_{1/Y}^{\fr}, S\subset \pi_0U\times\pi_0V)$ of $\cP_{X,Y}$, observe that the induced triple $(U\in \Disk_{\n1/X'}^{\fr}, V\in \Disk_{1/Y'}^{\fr}, S\subset \pi_0U\times\pi_0V)$ retains the property that the composite maps $S\ra \pi_0U\times\pi_0V\ra \pi_0U$ and $S\ra \pi_0U\times\pi_0V\ra \pi_0V$ are surjective, since this property is independent of the induced embeddings of $U$ and $V$ into $X'$ and $Y'$.

Likewise, given a morphism $(U\ra U', V\ra V', f:S\ra S')$ of $\cP_{X,Y}$, observe that the induced morphism in $\bigl(\Disk_{\n1/X'}^{\fr}\times \Disk_{1/Y'}^{\fr}\bigr)\underset{\Fin}\times \Ar^{\sf mono}(\Fin)$ again has the property that for every $\alpha \in \pi_0U$ the restriction $f_|: S\cap(\{\alpha\}\times\pi_0V)\ra S'$ is injective, since again this property is independent of the induced embeddings of $U$ and $V$ into $X'$ and $Y'$.
\end{proof}

Observe that there exists a functor
\[
 \bigl(\Disk_{\n1/X}^{\fr}\times \Disk_{1/Y}^{\fr}\bigr)\underset{\Fin}\times \Ar^{\sf mono}(\Fin)
 \longrightarrow
 \Disk^{\fr}_{\n1}
\]
sending a triple $(U,V,S)$ to the disjoint union
\[
\bigsqcup_{s\in S} U_{\pr_U(s)}
\]
where $U_{\pr_U(s)}\subset U$ is the component of $U$ indexed by the image of $s$ under the composite $S\ra \pi_0U\times\pi_0V\ra \pi_0U$.

For a functor $A:\Disk^{\fr}_{\n1}\ra \Spaces$ and for each $(X,Y)\in\Disk_{\n1}^{\fr}\times \Disk_{1}^{\fr}$, consider the composite functor
\[
\cP_{X,Y}\longrightarrow
\bigl(\Disk_{\n1/X}^{\fr}\times \Disk_{1/Y}^{\fr}\bigr)\underset{\Fin}\times \Ar^{\sf mono}(\Fin)
 \longrightarrow
 \Disk^{\fr}_{\n1}
 \xra{A} 
 \Spaces
\]
sending an object $(U,V,S)$ to the value $A\Bigl(\underset{s\in S}\coprod U_{\pr_U(s)}\Bigr)$. Note that if $A$ is monoidal (i.e., $A$ is an $\cE_{\n1}$-algebra), then the value of this composite functor on $(U,V,S)$ is
\[
\underset{s\in S}\prod A(U_{\pr_U(s)})= \prod_{\alpha\in \pi_0U}A(U_\alpha)^{S\cap (\{\alpha\}\times \pi_0V)}~.
\]

\begin{definition}\label{def.II}
For $A$ an $\cE_{\n1}$-algebra, the functor
\[
\II(A):\Disk_{\n1}^{\fr}\times \Disk_{1}^{\fr}\longrightarrow\Spaces
\]
sends $(X,Y)$ to the colimit
\[
\II(A)(X,Y) := \colim\Bigl(
\cP_{X,Y}\longrightarrow
\bigl(\Disk_{\n1/X}^{\fr}\times \Disk_{1/Y}^{\fr}\bigr)\underset{\Fin}\times \Ar^{\sf mono}(\Fin)
 \longrightarrow
 \Disk^{\fr}_{\n1}
 \xra{A} 
 \Spaces\Bigr)~.
\]
To a morphism $(X,Y)\ra (X',Y')$, $\II(A)$ assigns the map $\II(A)(X,Y) \ra \II(A)(X',Y')$ induced by the functor $\cP_{X,Y} \ra \cP_{X',Y'}$ of Lemma~\ref{lemma.functoriality.of.cP}.
\end{definition}

\begin{prop}
For $A$ an $\cE_{\n1}$-algebra, the functor $\II(A)$ obtains the structure of an $\cE_n$-algebra via a canonical factorization through the functor $\Disk_{\n1}^{\fr}\times \Disk_{1}^{\fr}\ra \Disk^{\fr}_n$. The assignment $A\mapsto \II(A)$ of Definition~\ref{def.II} thereby defines a functor
\[
\Alg_{\n1}(\Spaces)\xra{\II} \Alg_n(\Spaces)~.
\]
\end{prop}
\begin{proof}
First, we show that for each $\cE_{\n1}$-algebra $A$, the functor $\II(A): \Disk_{\n1}^{\fr}\times \Disk_{1}^{\fr}\ra \Spaces$ is monoidal separately in each variable. For a finite disjoint union of 1-disks $Y = \bigsqcup_{j\in J} Y_j\in \Disk^{\fr}_1$, we exhibit a canonical equivalence
\[
\II(A)(X,Y) \simeq \prod_{j\in J}\II(A)(X,Y_j)~.
\]
To do this, note the equivalence of $(\oo,1)$-categories
\[
\Disk^{\fr}_{1/Y} \simeq \prod_{j\in J} \Disk^{\fr}_{1/Y_j}
\]
which restricts to a equivalence
\[
\cP_{X,Y}\simeq \prod_{j\in J} \cP_{X,Y_j}~.
\]
This gives a canonical equivalence
\[
\II(A)(X,Y)\simeq \colim\Bigl(
\prod_{j\in J}\cP_{X,Y_j}\xra{A}\Spaces
\Bigr)
\simeq
\prod_{j\in J}\colim\Bigl(\cP_{X,Y_j}\xra{A}\Spaces\Bigr)
\simeq
\prod_{j\in J} \II(A)(X,Y_j)~.
\]
This implies that the functor $\II(A)(X,-)$ is monoidal for each $X$. The functor $\II(A)(-,Y)$ is monoidal for each $Y$, by identical reasoning. $\II(A)$ thus defines an $\cE_{\n1}$-algebra in $\cE_1$-algebras. By Dunn--Lurie additivity \cite{HA}, this establishes that $\II(A)$ is an $\cE_n$-algebra. The assignment $A\mapsto \II(A)$ is natural in $A$, which thus gives the stated functor $\II: \Alg_{\n1}(\Spaces)\ra \Alg_n(\Spaces)$.

\end{proof}

\begin{observation}
\label{t.100}
For $A$ an $\cE_{\n1}$-algebra, there exists a natural transformation
\[
\xymatrix{
\Disk_{\n1}^{\fr}\ar[ddr]_-{(-, \RR)}\ar[rrr]^-A&&&\Spaces\\
&\Downarrow
\\
&\Disk_{\n1}^{\fr}\times\Disk_1^{\fr}\ar[uurr]_-{\II(A)}
&
,
}
\]
given, for each $X \in \Disk^{\fr}_{\n1}$, by 
\[
A(X)
\longrightarrow
\II(A)(X,\RR)
~.
\]
\end{observation}

By the universal property of induction, we then have the following.
\begin{cor}\label{cor.Ind.to.II}
There is a canonical natural transformation
\[
\Ind_{\n1}^{n} \ra \II\]
of functors $\Fun\bigl(\Alg_{\n1}(\Spaces),\Alg_{n}(\Spaces)\bigr)$. 
\end{cor}

\begin{lemma}\label{lemma.II.sifted.colim}
The functor $\II: \Alg_{\n1}(\Spaces)\ra \Alg_{n}(\Spaces)$ preserves sifted colimits.
\end{lemma}
\begin{proof}
The assigment $A \mapsto \underset{s\in S}\prod A(U_{\pr_U(s)})$ preserves sifted colimits for each $(U,V,S)$, since the forgetful functor $\Alg_{\n1}(\Spaces) \ra\Spaces$ preserves sifted colimits and products also preserve sifted colimits. The functor $\colim_{\cP_{X,Y}}$ then commutes with all colimits, and in particular sifted colimits, from which the result follows.
\end{proof}

The goal of this section is to show that this natural transformation of Corollary~\ref{cor.Ind.to.II} is an equivalence. 
Now, every $\cE_{\n1}$-algebra admits a free resolution, thereby witnessing it as a sifted colimit of free $\cE_{\n1}$-algebras.
So, by preservation of sifted colimits from Lemma~\ref{lemma.II.sifted.colim}, 
it suffices to prove
the natural transformation of Corollary~\ref{cor.Ind.to.II} is an equivalence when $A= \FF_{\n1}(Z)$ is a free $\cE_{\n1}$-algebra generated by a space $Z$. 
For this, we will use a hypercover argument. 
To implement this hypercover argument, we now replace the indexing $(\oo,1)$-category $\cP_{X,Y}$ with a suitable poset.
First, for $M$ a framed $k$-manifold, denote the full subposet
\[
\sD_{k/M}
~\subset~
\Opens(M)
\]
consisting of those open subsets of $M$ that are diffeomorphic with a finite disjoint union of Euclidean spaces.
Proposition~2.19 of \cite{oldfact} states the evident functor
\[
\sD_{k/M}
\longrightarrow
\Disk^{\fr}_{k/M}
\]
is a localization.
\begin{definition}
For any $(X,Y)\in \Disk^{\fr}_{\n1}\times \Disk^{\fr}_1$, $P_{X,Y}$ is the limit in the following diagram:
\[
\xymatrix{
P_{X,Y}\ar[r]\ar[d]&\cP_{X,Y}\ar[d]\\
\disk_{\n1/X}\times \disk_{1/Y}\ar[r]^-{\rm loc}&\Disk^{\fr}_{\n1/X}\times \Disk^{\fr}_{1/Y}~.}
\]
\end{definition}
\begin{observation} $P_{X,Y}$ can alternatively be defined as the subcategory of $\disk_{\n1/X}\times \disk_{1/Y}\underset{\Fin}\times \Ar^{\sf mono}(\Fin)$ satisfying the conditions of Definition~\ref{def.P}. Consequently $P_{X,Y}$ is a poset.
\end{observation}

The finality of the functor $P_{X,Y}\ra \cP_{X,Y}$ will be essential to our hypercover proof of the equivalence of $\Ind_{\n1}^n\ra \II$. This finality is a consequence of the following stronger localization property:

\begin{lemma} The functor $P_{X,Y}\ra\cP_{X,Y}$ is a localization.
\end{lemma}
\begin{proof}
Let $\cI\subset P_{X,Y}$ be the subcategory of isotopy equivalences, which is the inverse image of the isomorphisms in $\cP_{X,Y}$ under the functor $P_{X,Y}\ra\cP_{X,Y}$.
By Theorem 3.8 of \cite{mazel.gee.localization}, it suffices to prove that, for each $[p]\in \bDelta$, the natural functor 
\begin{equation}
\label{e.101}
\Fun^\cI\bigl([p], P_{X,Y}\bigr) \longrightarrow \Map\bigl([p],\cP_{X,Y}\bigr)
\end{equation}
is a localization, where $\Fun^\cI([p], P_{X,Y})\subset \Fun([p], P_{X,Y})$ is the subcategory of natural transformations contained in $\cI$ as in Notation~\ref{d.fun.w}.
Consider the subcategory $\cI' \subset \disk_{\n1/X}\times \disk_{1/Y}$ that is the preimage of the isomorphisms by the functor 
$
\disk_{\n1/X}\times \disk_{1/Y}
\to
\Disk^{\fr}_{\n1/X}\times \Disk^{\fr}_{1/Y}
$.
By the proof of Proposition~2.19 of \cite{oldfact}, the functor
\[
\Fun^{\cI'}\bigl([p],\disk_{\n1/X}\times \disk_{1/Y}\bigr) \longrightarrow \Map\bigl([p],\Disk^{\fr}_{\n1/X}\times \Disk^{\fr}_{1/Y}\bigr)
\]
is a localization and in particular initial.
Now, consider the subcategory $\cI'' \subset \disk_{\n1/X}\times \disk_{1/Y}\underset{\Fin}\times \Ar^{\sf mono}(\Fin)$ that is the preimage of the isomorphisms by the functor 
\[
\disk_{\n1/X}\times \disk_{1/Y}\underset{\Fin}\times \Ar^{\sf mono}(\Fin)
\to
\Disk^{\fr}_{\n1/X}\times \Disk^{\fr}_{1/Y}\underset{\Fin}\times \Ar^{\sf mono}(\Fin)
~
\]
Because monomorphisms in $\Fin$ are preserved by base change, the functor
$\Ar^{\sf mono}(\Fin) \xra{\ev_t} \Fin$ is a right fibration, and in particular a Cartesian fibration.
By Corollary~5.15 of~\cite{fibrations}, the functor between pullbacks
\[
\Fun^{\cI''}\Bigl([p],\disk_{\n1/X}\times \disk_{1/Y}\underset{\Fin}\times \Ar^{\sf mono}(\Fin)\Bigr) \longrightarrow \Map\Bigl([p],\Disk^{\fr}_{\n1/X}\times \Disk^{\fr}_{1/Y}\underset{\Fin}\times \Ar^{\sf mono}(\Fin)\Bigr)
\]
is also initial.
Because the codomain is a space, this functor is a localization.
Observe the relations among subcategories of $\disk_{\n1/X}\times \disk_{1/Y}\underset{\Fin}\times \Ar^{\sf mono}(\Fin)$
\[
\cI
~=~
P_{X,Y}
\cap 
\cI''
~\subset~
\cI''
\]
in which the inclusion is full.
This implies the functor~(\ref{e.101}) is a localization.

\end{proof}

\begin{notation}
Let $Z$ be a topological space.
We denote by
\[
\Conf^Z(M) := \coprod_{r\geq 0} \Conf_r(M)\underset{\Sigma_r}\times Z^r
~=~
\left\{
M \supset T \xra{\ell} Z
\right\}
\]
the topological space of $Z$-labeled configurations of points in a topological space $M$, whose underlying set consists of finite subsets of $M$ equipped with a map to $Z$.
\end{notation}

In what follows, for $(U,V,S) \in P_{X,Y}$, we denote
\[
(U\times V)_S
~\subset~
U\times V
\]
for the union of those connected components of $U\times V$ that belong to $S \subset \pi_0(U) \times \pi_0(V)$.
\begin{prop}\label{prop.hypercover.ind}
Let $Z$ be a topological space.
The assignment
\[
(U,V,S)
\longmapsto
\Bigl\{
X\times Y\supset T\xra{\ell} Z
\Big| \
T\subset (U\times V)_S \ {\rm and} \
\forall \beta\in \pi_0V, \ T\cap(U\times V_\beta) \xra{\pr}U \ {\rm is \ injective}
\Bigr\}~,
\]
defines a functor
\[
P_{X,Y}\overset{F}\longrightarrow \Opens\Bigl(
\Conf^Z(X\times Y)
\Bigr)
~.
\]
Furthermore, the resulting map between spaces
\[
\colim\Bigl(
P_{X,Y}\xra{F} \Opens\bigl(
\Conf^Z(X\times Y)\bigr)\ra \Spaces
\Bigr)
\longrightarrow
\Conf^Z(X\times Y)
\]
is an equivalence.
\end{prop}
\begin{proof}
Since $\Opens\bigl(\Conf^Z(X\times Y)\bigr)$ is a poset, to establish the existence of the asserted functor $F$ it suffices to show that given a morphism $f=(U\subset U', V\subset V', S\ra S')$ of $P_{X,Y}$, 
there is containment $F(U,V,S) \subset F(U',V',S')$ between subsets of $\Conf^Z(X\times Y)$.

To show this, assume that we are given $X\times Y\supset T\xra{\ell} Z$ belonging to $F(U,V,S)$, and we are given containments $U\subset U'$, $V\subset V'$ compatibly with $f:S\ra S'$ such that for each $\alpha \in \pi_0U$, the map $S\cap (\{\alpha\} \times \pi_0V) \ra S'$ is injective. The containment $(U\times V)_S \subset (U'\times V')_{S'}$ immediately implies the first condition, that $T$ is contained in $(U'\times V')_{S'}$. We now show that for each element $\beta' \in \pi_0V'$, the subset $T\cap(U'\times V'_{\beta'})$ maps injectively to $U'$ under the projection map $\pr:U' \times V' \ra U'$. Suppose that there exists elements $w_1$ and $w_2$ of $T$ which map to to same element of $U$ under the composite projection $T\ra U\times V \ra U$, and denote by $\alpha$ the component of $U$ that contains this element. Since $T$ is assumed to belong to $F(U,V,S)$, then necessarily $w_1$ and $w_2$ belong to different components $V_1$ and $V_2$ of $V$, under the composite projection $T\ra U\times V \ra V$. Now, the map
\[
S\cap (\{\alpha\}\times \pi_0V) \ra S'
\]
was assumed to be injective, and therefore the map $\pi_0 V\ra  \pi_0 V'$ is injective when restricted to the subset $\pi_0(V_1\amalg V_2)$. Consequently the images of $w_1$ and $w_2$ belong to different components of $V'$ under the projection $T\ra U\times V\ra V$. This implies that $T$ belongs to $F(U',V',S')$, and thereby establishes the existence of the functor $F$.

We now show that $F$ is a hypercovering of the space $\Conf^Z(X\times Y)$. For $(T,\ell)\in \Conf^Z(X\times Y)$, denote by
\[
\bigl(P_{X,Y}\bigr)_{(T,\ell)}
\]
the subposet of $P_{X,Y}$ consisting of those $(U,V,S)$ for which the topological space $F(U,V,S)$ contains the point $(T,\ell)$. 
By the Siefert--van Kampen Theorem of~\cite{HA} (specifically Theorem~A.3.1), to check that the colimit of $F$ is $\Conf^Z(X\times Y)$, it suffices to show the following hypercovering criterion:

\begin{itemize}
\item For every $(T,\ell)\in \Conf^Z(X\times Y)$, the classifying space of the poset $(P_{X,Y})_{(T,\ell)}$ is contractible:
\[
\sB\Bigl(\bigr(P_{X,Y}\bigr)_{(T,\ell)}\Bigr) \simeq \ast~.
\]
\end{itemize}
We prove this in two steps. First, we prove that $(P_{X,Y})_{(T,\ell)}$ is nonempty for all $(T,\ell)\in \Conf^Z(X\times Y)$. That is, we construct $(U,V,S)\in P_{X,Y}$ such that: $T$ is contained in $(U\times V)_S$; and for each $\beta \in \pi_0V$, the projection $T\cap (U\times V_\beta)\ra U$ is injective. Let $\pr_X(T)$ and $\pr_Y(T)$ denote the images of $T$ in $X$ and $Y$, respectively. For each $x\in \pr_X(T)$ and $y\in \pr_Y(T)$, choose connected neighborhoods $U_x$ and $U_y$, respectively, such that $U_x$ is disjoint from $U_{x'}$ for distinct element $x,x' \in \pr_X(T)$ and $V_y$ is disjoint from $V_{y'}$ for distinct element $y,y' \in \pr_Y(T)$. We can then set
\[
U := \coprod_{x\in \pr_X(T)} U_x~, \ V := \coprod_{y\in \pr_Y(T)} V_y~, \ {\rm and} \ S := T ~.
\]
Second, we show that for two objects  $(U',V',S')$ and  $(U'',V'',S'')$ of $(P_{X,Y})_{(T,\ell)}$, there exists an object $(U,V,T)$ of $(P_{X,Y})_{(T,\ell)}$ with morphisms
\[
(U',V',S')\la (U,V,T)\ra (U'',V'',S'')
\]
in $(P_{X,Y})_{(T,\ell)}$. Since $(P_{X,Y})_{(T,\ell)}$ is a poset, these assertions imply that $(P_{X,Y})_{(T,\ell)}$ is cofiltered and thus that its classifying space is contractible.

Set $U= \coprod_{x\in \pr_X(T)} U_x$, $V= \coprod_{y\in \pr_Y(T)} V_y$, and $S = T$ as in the first step, and additionally shrink each $U_x$ and $V_y$ so as to satisfy the containments $U\subset U'\cap U''$ and $V\subset V'\cap V''$.

Note that by construction, $(U,V,T)$ defines an object of $(P_{X,Y})_{(T,\ell)}$. Further, there are containments $U\subset U'$ and $V\subset V'$. The containment $T\subset (U'\times V')_{S'}$ defines a map $T \ra S'$ lying over the map $\pi_0U\times\pi_0V \ra \pi_0U'\times\pi_0V'$.  To establish the existence of a morphism from $(U,V,T)$ to $(U',V',S')$, we must lastly check the condition that for each $\alpha \in  \pi_0 U$ the restriction
\[
T\cap(\{\alpha\}\times\pi_0V)\ra S'
\]
is injective. By construction, the intersection $T\cap(\{\alpha\}\times\pi_0V)$ consists of elements $\{w_1, \ldots, w_k\}$ whose image in $X$ are the same: $\pr_X(w_i) = \pr_X(w_j)$. Since for every $\beta' \in \pi_0 V'$, the projection $T\cap (U' \times V'_{\beta'}) \ra U'$ is injective, consequently at most one of the $w_i \in T\cap(\{\alpha\}\times\pi_0V)$ belong to $T\cap (U' \times V'_{\beta'})$ for any given connected component $V'_{\beta'}\subset V'$. Therefore the map $T\cap(\{\alpha\}\times\pi_0V) \ra \pi_0 U' \times \pi_0 V'$ is injective. As this map lifts to $S'$, therefore the map $T\cap(\{\alpha\}\times\pi_0V) \ra S'$ is injective, as required. $(U,V,T)$ consequently maps to $(U',V',S')$ in $(P_{X,Y})_{(T,\ell)}$ and, by symmetry,  $(U,V,T)$ also maps to $ (U'',V'',S'')$.

\end{proof}

\begin{lemma}\label{lemma.F.computes.II}
There exists a commutative diagram
\[
\xymatrix{
P_{X,Y}\ar[rr]^-F\ar[d]&&\Opens\Bigl( \Conf^Z(X\times Y)\Bigr)\ar[rd]^-{\rm forget}\\
\cP_{X,Y}\ar[r]&\bigl(\Disk_{\n1/X}^{\fr}\times \Disk_{1/Y}^{\fr}\bigr)\underset{\Fin}\times \Ar^{\sf mono}(\Fin)\ar[r]& \Disk^{\fr}_{\n1}\ar[r]^-{\FF_{\n1}(Z)}& \spaces}
\]
where the functor $F: P_{X,Y}\ra \Opens( \Conf^Z(X\times Y))$ is that defined in Proposition~\ref{prop.hypercover.ind}, and the composite functor $\cP_{X,Y}\ra \Spaces$ is that whose colimit defines $\II(\FF_{\n1}(Z))(X,Y)$ according to Definition~\ref{def.II}, .
\end{lemma}
\begin{proof}
The definitions of the two composite functors of $(\oo,1)$-categories $P_{X,Y}\ra \Spaces$ both come as functors of ordinary categories $P_{X,Y}\ra \Top$, where $\Top$ is the ordinary category of topological spaces. 
To construct a natural equivalence making the above diagram commute, it therefore suffices to construct a natural transformation of functors $P_{X,Y}\ra \Top$ whose value on every object of $P_{X,Y}$ is a homotopy equivalences of topological spaces.
First note that the composite functor
\[
{\rm forget} \circ F: P_{X,Y} \ra \Spaces
\]
sends isotopy equivalence to homotopy equivalences, and so factors through the localization $P_{X,Y}\ra \cP_{X,Y}$. 
It thereby suffices to show the natural equivalence
\[
F(U,V,S) 
\simeq \underset{s\in S}\prod  \FF_{\n1}(Z)(U_{\pr_U(s)})
\]
with the values of the functor $\FF_{\n1}(Z):\Disk^{\fr}_{\n1}\ra \Spaces$ given by the free $\cE_{\n1}$-algebra generated by the space $Z$.

To show this, it further reduces to first showing the special case for $S=\{s\}$ a singleton, of 
\[
F(U,V,\{s\}) \simeq \FF_{\n1}(Z)(U_{\pr_U(s)})
\]
and the second equivalence
\[
F(U,V,S) \simeq \underset{s\in S}\prod F(U_{\pr_U(s)}, V, \{s\})~.
\]
To see the first equivalence, recall from the definition of $F$ that
\[
F(U,V,\{s\}) = 
F(U_{\pr_U(s)}, V_{\pr_V(s)}, \{s\})
\]
\[
=
\Bigl\{
X\times Y\supset T\xra{\ell} Z
\Big| \
T\subset U_{\pr_U(s)}\times V_{\pr_V(s)} \ {\rm and} \ T \xra{\pr}U_{\pr_U(s)} \ {\rm is \ injective}
\Bigr\}~.
\]
Since $V_{\pr_V(s)}\cong \RR^1$ is contractible, the space of $T$ subsets of $U_{\pr_U(s)}\times V_{\pr_V(s)}$ whose projection to $U$ is injective is homotopy equivalent to the space of $T$ subset of $U_{\pr_U(s)}$. Thus:
\[
\Bigl\{
X\times Y\supset T\xra{\ell} Z
\Big| \
T\subset U_{\pr_U(s)}\times V_{\pr_V(s)} \ {\rm and} \ T \xra{\pr}U_{\pr_U(s)} \ {\rm is \ injective}
\Bigr\}
\simeq
\Bigl\{
U_{\pr_U(s)}\supset T\xra{\ell} Z
\Bigr\}
~.
\]
This righthand space is homotopy equivalent to the value $\FF_{\n1}(Z)(U_{\pr_U(s)})$, which establishes the first equivalence.

The second equivalence is immediate from the definition of $F$, that it splits as a product over the set $S$. That is, the projection $T\cap U\times V_\beta\ra U$ is injective for every $\beta \in V_\beta$ if and only if it is true for every component of $U$, that the projection $T\cap U_\alpha \times V_\beta \ra U_\alpha$ is injective. 
Since $T$ is contained in the components of $U \times V$ indexed by the elements of $S$,
this reformulation is true if and only if it is true for each element of $S$, that the projection $T \cap U_{\pr_U(s)}\times V_{\pr_V(s)} \ra U_{\pr_U(s)}$ is injective for each $s\in S$.
\end{proof}

We obtain the following corollary.

\begin{cor}\label{cor.II.free.case}
There natural equivalences:
\[
\FF_n(Z)(X\times Y)
\simeq
\colim\Bigl(P_{X,Y}\xra{F}\spaces\Bigr)
\simeq
\colim\Bigl(\cP_{X,Y}\xra{\FF_{\n1}(Z)}\spaces\Bigr)
=:
\II(\FF_{\n1}(Z))(X,Y)~.
\]
\end{cor}
\begin{proof}
The equivalence $\FF_n(Z)(X\times Y) \simeq \colim(P_{X,Y}\xra{F}\spaces)$ follows from the hypercover condition established in Proposition~\ref{prop.hypercover.ind}. The natural map
\[
\colim\Bigl(P_{X,Y}\xra{F}\spaces\Bigr)
\longrightarrow
\colim\Bigl(\cP_{X,Y}\xra{\FF_{\n1}(Z)}\spaces\Bigr)
\]
is an equivalence by the finality of $P_{X,Y} \ra \cP_{X,Y}$ and the identification
\[
F(U,V,S) 
\simeq \underset{s\in S}\prod  \FF_{\n1}(Z)(U_{\pr_U(s)})
\]
of Lemma~\ref{lemma.F.computes.II}.
\end{proof}

We now assemble the preceding results to establish our main technical result of this section.
\begin{lemma}\label{lemma.II.is.Ind}
The natural transformation $\Ind_{\n1}^n \ra \II$ is an equivalence. In particular, for any $A\in \Alg_{\n1}(\Spaces)$ and any $D\in \Disk^{\fr}_n$ identified as $D\cong X\times Y$ for $(X,Y)\in \Disk^{\fr}_{\n1}\times\Disk^{\fr}_1$, there are equivalences
\[
\Ind_{\n1}^n(A)(D) \simeq
\colim\Bigl(
\cP_{X,Y}\longrightarrow
\bigl(\Disk_{\n1/X}^{\fr}\times \Disk_{1/Y}^{\fr}\bigr)\underset{\Fin}\times \Ar^{\sf mono}(\Fin)
 \longrightarrow
 \Disk^{\fr}_{\n1}
 \xra{A} 
 \Spaces\Bigr)
\]
\[
\simeq
\colim\Bigl(
P_{X,Y}\longrightarrow
\bigl(\disk_{\n1/X}^{\fr}\times \disk_{1/Y}^{\fr}\bigr)\underset{\Fin}\times \Ar^{\sf mono}(\Fin)
 \longrightarrow
 \Disk^{\fr}_{\n1}
 \xra{A} 
 \Spaces\Bigr)
\]
for the value of the functor $\Ind_{\n1}^n(A): \Disk^{\fr}_n\ra \Spaces$ on the $n$-disk $D\cong X\times Y$.
\end{lemma}
\begin{proof}
By Corollary~\ref{cor.II.free.case} the natural map is an equivalence for a free $\cE_{\n1}$-algebra $A = \FF_{\n1}(Z)$. The functor $\Ind^n_{\n1}$ preserves sifted colimits, being a left adjoint. The functor $\II$ preserves sifted colimits, by Lemma~\ref{lemma.II.sifted.colim}. Since any $\cE_{\n1}$-algebra is a sifted colimit of free algebras, the equivalence follows for general $A$.
\end{proof}

\section{Operadic induction and projection-immersion tangles}
In this section, we prove that the induction functor $\Alg_{\n1}(\Cat_{(\oo,1)}) \ra \Alg_n(\Cat_{(\oo,1)})$ sends $\Bord_1^{\fr}(\RR^{\n1})$ to $\Bord_1^{\fr}(\RR^n)$:
\[
\Ind_{\n1}^n\Bigl(\Bord_1^{\fr}(\RR^{\n1})^{\wedge}_{\sf unv}\Bigr)
\simeq \Bord_1^{\fr}(\RR^n)^{\wedge}_{\sf unv}~.
\]
We outline our argument:
\begin{enumerate}
\item Induction of spaces of $p$-simplices gives slice-wise projection-embeddings: Corollary~\ref{cor1}.
\item Segaling of slice-wise projection-embeddings gives projection-immersions: Lemma~\ref{lemma.proj.emb.to.proj.imm}.
\item Univalent-completion of projection-immersions gives that of the usual tangle category: Lemma~\ref{lemma.univ.cplt.of.proj.imm}.
\end{enumerate}

By universal properties, the following diagram commutes:
\[
\xymatrix{
\Alg_{\n1}(\Cat_{(\oo,1)})\ar@{^{(}->}[d]_-{\rm forget}\ar[rrrr]^{\Ind^n_{\n1}}&&&&\Alg_n(\Cat_{(\oo,1)})\\
\Alg_{\n1}(\Fun(\bDelta^{\op},\Spaces))\ar[rr]^{\Ind^n_{\n1}}_-{({\rm See~Cor}~\ref{cor1})}&&\Alg_n(\Fun(\bDelta^{\op},\Spaces))\ar[rr]^-{\rm Segaling}_-{({\rm See~Lem}~\ref{lemma.proj.emb.to.proj.imm})}&&\Alg_n(\fCat_{(\oo,1)})\ar[u]^-{(-)^{\wedge}_{\sf unv}}_-{({\rm See~Lem}~\ref{lemma.univ.cplt.of.proj.imm})}
}
\]

\subsection{Slice-wise projection-embedding tangles}

Fix $X\in \Disk_{\n1}^{\fr}$ and $Y\in \Disk_1^{\fr}$ with resulting product $X\times Y\in \Disk_n^{\fr}$. In the particular case of interest $X=\RR^{\n1}$ and $Y=\RR^1$, we further fix the usual identification $\RR^n \cong\RR^{\n1}\times\RR^{\{n\}}$.

\begin{definition}\label{def.a.proj.emb.tangle}
For $W$ a 1-manifold with boundary, $p\geq 0$, an embedding $W\hookrightarrow X\times Y\times\DD[p]$ is a \bit{slice-wise projection-embedding} if:
\begin{itemize}
\item the composite projection $W\ra X\times Y\times\DD[p]\ra \DD[p]$ is transverse to $[p]\subset\DD[p]$; 
\item 
for each connected component $Y_\alpha \subset Y$, 
there exists a finite subset $S\subset Y_\alpha$ for which:
\begin{itemize}
\item $X\times S\times \DD[p]$ is disjoint from $W$; and
\item for each pair of successive elements $s$ and $s'$ of $S\cup\{-\oo,\oo\}$, the projection
\[
W\cap\Bigl(X\times[s,s']\times\DD[p]\Bigr) \longrightarrow X\times\DD[p]
\]
is an embedding.
\end{itemize}
\end{itemize}
The topological space 
\[
\Emb(W, X\times Y\times\DD[p])^{\sf s.pr.emb}_{\pitchfork[p]}
~\subset~
\Emb(W,X\times Y\times\DD[p])
\]
is the topological subspace consisting of slice-wise projection-embedding.
\end{definition}

Observe that the slice-wise projection-embedding condition implies that the composite projection $W\ra X\times Y\times\DD[p] \ra X\times\DD[p]$ is an immersion. 
Therefore, projection onto $X \times \DD[p]$ defines a continuous map
\[
\Emb(W, X\times Y\times\DD[p])^{\sf s.pr.emb}_{\pitchfork[p]}\longrightarrow \Imm(W, X\times\DD[p])
\]
to the topological space of immersions.
This allows for the following notion of a framing on a tangle that satisfies this slice-wise embedding condition.

\begin{definition}
A \bit{framing} of a slice-wise projection-embedding $W\hookrightarrow X\times Y\times\DD[p]$ consists of a nullhomotopy $\varphi$ of the Gauss map ${\sf Gauss}_{(W\ra X\times\DD[p])}: W \ra \Gr_1(n)$ associated to the composite immersion
\[
W ~\hookrightarrow~ X\times Y\times\DD[p] \xra{~\pr~} X\times\DD[p]
\]
and the framing of $X\times\DD[p]$.
\end{definition}

We topologize the set of all slice-wise projection-embeddings in the following way.
\begin{definition}\label{def.proj.emb.tang}
The topological space of framed slice-wise projection-embedding tangles is
\[
\Tang_1^{\fr}(X\times Y\times\DD[p])_{\pitchfork[p]}^{\sf s.pr.emb}
:=
\coprod_{[W]}\Bigl( \Emb\bigl(W,X\times Y\times\DD[p]\bigr)^{\sf s.pr.emb}_{\pitchfork[p]}\underset{\Imm\bigl(W,X\times\DD[p]\bigr)} \times \Imm^{\fr}\bigl(W,X\times\DD[p]\bigr)\Bigr)_{\Aut(W)}
\]
where $\Imm^{\fr}(W,X\times\DD[p])$ is the space of immersions equipped with a trivialization of the Gauss map in the sense of Definition~\ref{def.framed.tangle}.
\end{definition}

The following lemma shows how the homotopy type of the space $\Tang_1^{\fr}(\RR^n\times\DD[p])_{\pitchfork[p]}^{\sf s.pr.emb}$ arises from induction.

\begin{lemma}\label{lemma.hypercover.tang.proj.emb}
There exists a functor
\[
P_{X,Y}\overset{F}\longrightarrow \Opens\Bigl(
\Tang_1^{\fr}\bigl(X\times Y \times\DD[p]\bigr)^{\sf s.pr.emb}_{\pitchfork[p]}
\Bigr)
\]
whose value on an object $(U,V,S)$ is the subspace of tangles $W\subset X\times Y\times\DD[p]$ for which:

\begin{enumerate}
\item $W\subset (U\times V)_S \times\DD[p]$; and
\item for each connected component $V_\beta\subset V$, the projection
\[
W\cap (U\times V_\beta\times\DD[p]) \ra  U \times\DD[p]
\]
is an embedding.
\end{enumerate}
Furthermore, the resulting map between spaces
\[
\colim\Bigl(
P_{X,Y}\xra{F} \Opens\bigl(
\Tang_1^{\fr}\bigl(X\times Y \times\DD[p]\bigr)^{\sf s.pr.emb}_{\pitchfork[p]}\bigr)\ra \Spaces
\Bigr)
\longrightarrow
\Tang_1^{\fr}\bigl(X\times Y \times\DD[p]\bigr)^{\sf s.pr.emb}_{\pitchfork[p]}
\]
is an equivalence.
\end{lemma}
\begin{proof}
As in the proof of Lemma~\ref{prop.hypercover.ind}, since $\Opens\bigl(\Tang_1^{\fr}\bigl(X\times Y \times\DD[p]\bigr)^{\sf s.pr.emb}_{\pitchfork[p]}\bigr)$ is a poset, to establish the existence of the asserted functor $F$ it suffices to show that given a morphism $f=(U\subset U', V\subset V', S\ra S')$ of $P_{X,Y}$, the topological space $F(U,V,S)$ is contained in the topological space $F(U',V',S')$, as subspaces of $\Tang_1^{\fr}\bigl(X\times Y \times\DD[p]\bigr)^{\sf s.pr.emb}_{\pitchfork[p]}$. 
The argument here is now similar to that given in Proposition~\ref{prop.hypercover.ind}.

We now show that $F$ is a hypercovering of the space $\Tang_1^{\fr}\bigl(X\times Y \times\DD[p]\bigr)^{\sf s.pr.emb}_{\pitchfork[p]}$. For $(W,\varphi)\in \Tang_1^{\fr}\bigl(X\times Y \times\DD[p]\bigr)^{\sf s.pr.emb}_{\pitchfork[p]}$, denote by
\[
\bigl(P_{X,Y}\bigr)_{(W,\varphi)}
\]
the subposet of $P_{X,Y}$ consisting of those $(U,V,S)$ for which the topological space $F(U,V,S)$ contains the point $(W,\varphi)$. 
By the Siefert--van Kampen Theorem of~\cite{HA} (specifically Theorem~A.3.1), to check that the colimit of $F$ is $\Tang_1^{\fr}\bigl(X\times Y \times\DD[p]\bigr)^{\sf s.pr.emb}_{\pitchfork[p]}$, it suffices to show the following hypercovering criterion:

\begin{itemize}
\item For every $(W,\varphi)\in \Tang_1^{\fr}\bigl(X\times Y \times\DD[p]\bigr)^{\sf s.pr.emb}_{\pitchfork[p]}$, the classifying space of the poset $(P_{X,Y})_{(W,\varphi)}$ is contractible:
\[
\sB\Bigl(\bigr(P_{X,Y}\bigr)_{(W,\varphi)}\Bigr) \simeq \ast~.
\]
\end{itemize}
The argument here is again similar to that given in Proposition~\ref{prop.hypercover.ind}.
\end{proof}

\begin{lemma}\label{lemma.F.computes.Ind.Tang}
There exists a commutative diagram
\[
\xymatrix{
P_{X,Y}\ar[rr]^-F\ar[d]&&\Opens\Bigl( \Tang_1^{\fr}\bigl(X\times Y \times\DD[p]\bigr)^{\sf s.pr.emb}_{\pitchfork[p]}\Bigr)\ar[rd]^-{\rm forget}\\
\cP_{X,Y}\ar[r]&\bigl(\Disk_{\n1/X}^{\fr}\times \Disk_{1/Y}^{\fr}\bigr)\underset{\Fin}\times \Ar^{\sf mono}(\Fin)\ar[r]& \Disk^{\fr}_{\n1}\ar[r]_-{\Bord_1^{\fr}(\RR^{\n1})[p]}& \spaces}
\]
where the functor $F: P_{X,Y}\ra \Opens(\Tang_1^{\fr}\bigl(X\times Y \times\DD[p]\bigr)^{\sf s.pr.emb}_{\pitchfork[p]})$ is that defined in Lemma~\ref{lemma.hypercover.tang.proj.emb}, and the composite functor $\cP_{X,Y}\ra \Spaces$ is that whose colimit is
\[
\Ind_{\n1}^n\Bigl(\Bord_1^{\fr}(\RR^{\n1})[p]\Bigr)(X\times Y)
\]
according to Definition~\ref{def.II} and Lemma~\ref{lemma.II.is.Ind}.
\end{lemma}
\begin{proof}
As in the proof of Lemma~\ref{lemma.F.computes.II}, the composite functor
\[
{\rm forget} \circ F: P_{X,Y} \ra \Spaces
\]
sends isotopy equivalence to homotopy equivalences, and so factors through the localization $P_{X,Y}\ra \cP_{X,Y}$. It thereby suffices to show the natural equivalence
\[
F(U,V,S) 
\simeq \underset{s\in S}\prod  \Bord_1^{\fr}(U_{\pr_U(s)})[p]
\]
with the values of the functor $\Bord_1^{\fr}(\RR^{\n1})[p]:\Disk^{\fr}_{\n1}\ra \Spaces$. Recall that the value of this functor on $U\in \Disk^{\fr}_{\n1}$ is
\[
\Bord_1^{\fr}(\RR^{\n1})[p](U):= \Bord_1^{\fr}(U)[p]\simeq \Tang^{\fr}_1(U\times\DD[p])_{\pitchfork[p]}~.
\]
Again as in the proof of Lemma~\ref{lemma.F.computes.II}, we first show the equivalence for $S=\{s\}$ a singleton, of 
\[
F(U,V,\{s\}) \simeq \Bord_1^{\fr}(U_{\pr_U(s)})[p]
\]
and the second equivalence
\[
F(U,V,S) \simeq \underset{s\in S}\prod F(U_{\pr_U(s)}, V, \{s\})~.
\]
To see the first equivalence, note from the definition of $F$ that
\[
F(U,V,\{s\}) = 
F(U_{\pr_U(s)}, V_{\pr_V(s)}, \{s\})~.
\]
Given the isomorphisms $U_{\pr_U(s)}\cong\RR^{\n1}$ and $ V_{\pr_V(s)}\cong\RR^1$, it thus suffices to show the equivalence
\[
F(\RR^{\n1},\RR^1, \{s\}) \simeq \Tang^{\fr}_1(\RR^{\n1}\times\DD[p])_{\pitchfork[p]}~.
\]
The value $F(\RR^{\n1},\RR^1, \{s\})$ can be expressed in terms of the topological subspace which we will briefly denote
\[
\Emb\bigl(W,\RR^{\n1}\times \RR^1\times\DD[p]\bigr)^{\sf t.pr.emb}_{\pitchfork[p]}
\subset
\Emb(W,\RR^{\n1}\times \RR^1\times\DD[p])_{\pitchfork[p]}
\]
of total (and {\it not} slice-wise) projection-embeddings: This consists of those embeddings $W\ra \RR^{\n1}\times \RR^1\times\DD[p]$ for which the composite $W\ra \RR^{\n1}\times \RR^1\times\DD[p] \ra \RR^{\n1}\times\DD[p]$ is again an embedding. Direct from the definition of $F$ and Definition~\ref{def.proj.emb.tang}, the value $F(\RR^{\n1},\RR^1, \{s\})$ is equivalent to
\[
\coprod_{[W]}\Bigl( \Emb\bigl(W,\RR^{\n1}\times \RR^1\times\DD[p]\bigr)^{\sf t.pr.emb}_{\pitchfork[p]}\underset{\Emb\bigl(W,\RR^{\n1}\times\DD[p]\bigr)} \times \Emb^{\fr}\bigl(W,\RR^{\n1}\times\DD[p]\bigr)\Bigr)_{\Aut(W)}~.
\]
We can now observe the existence of the fiber square among topological spaces,
\[
\xymatrix{
\Map(W,\RR^1)\ar[r]\ar[d]&\Emb\bigl(W,\RR^{\n1}\times \RR^1\times\DD[p]\bigr)^{\sf t.pr.emb}_{\pitchfork[p]}\ar[d]\\
\bigl\{W\overset{f}\hookrightarrow\RR^{\n1}\times\DD[p]\bigr\}\ar[r]&\Emb\bigl(W,\RR^{\n1}\times\DD[p]\bigr)_{\pitchfork[p]}}
~.
\]
Observe that the right vertical map is a fibration, and therefore this square is a homotopy pullback.
Since $\Map(W,\RR^1)$ is contractible, the the right vertical map is a weak homotopy equivalence. From this, obtain the further equivalences
\[
 \Emb\bigl(W,\RR^{\n1}\times \RR^1\times\DD[p]\bigr)^{\sf t.pr.emb}_{\pitchfork[p]}\underset{\Emb\bigl(W,\RR^{\n1}\times\DD[p]\bigr)} \times \Emb^{\fr}\bigl(W,\RR^{\n1}\times\DD[p]\bigr)
 \simeq
  \Emb^{\fr}\bigl(W,\RR^{\n1}\times\DD[p]\bigr)_{\pitchfork[p]}
\]
and thus
\[
F(\RR^{\n1},\RR^1, \{s\}) \simeq \Tang^{\fr}_1(\RR^{\n1}\times\DD[p])_{\pitchfork[p]}
\]
which proves the first equivalence.

The second equivalence is immediate from the definition of $F$, that it splits as a product over the set $S$. That is, the projection $W\cap U\times V_\beta\ra U$ is an embedding for every $\beta \in V_\beta$ if and only if is true for every component $U_\alpha \subset U$, that the projection $W\cap U_\alpha \times V_\beta \ra U_\alpha$ is an embedding. Since $W\subset U \times V$ is contained in those connected components that are indexed by elements in the subset $S$, this reformulation is true if and only if it is true for each element of $S$, that the projection $W \cap U_{\pr_U(s)}\times V_{\pr_V(s)} \ra U_{\pr_U(s)}$ is an embedding for each $s\in S$.
\end{proof}

Combined with the results of the previous section, Lemma~\ref{lemma.hypercover.tang.proj.emb} and Lemma~\ref{lemma.F.computes.Ind.Tang} have the following consequence.

\begin{cor}\label{cor1}
For each $[p]$, there is a natural equivalence
\[
\Ind_{\n1}^{n}\Bigl( \Bord_1^{\fr}(\RR^{\n1})[p]\Bigr)
\overset{\simeq}\longrightarrow
\Tang_1^{\fr,\pitchfork[p]}(\RR^n\times\DD[p])^{\sf s.pr.emb}
\]
where $\Bord_1^{\fr}(\RR^{\n1})[p]$ is the $\cE_{\n1}$-algebra in spaces given by the space of $p$-simplices of $\Bord_1^{\fr}(\RR^{\n1})$, and $\Ind_{\n1}^{n}$ is the induction functor $\Alg_{\n1}(\spaces)\ra\Alg_n(\spaces)$.
\end{cor}
\qed

\begin{definition}\label{def.proj.emb.simplicial.space}%lemma.proj.emb.to.proj.imm
The $\cE_n$-algebra in simplicial spaces
\[
\Tang_1^{\fr}(\RR^n\times\DD[\bullet])^{\sf s.pr.emb}_{\pitchfork[p]}\in \Alg_n(\Fun(\bDelta^{\op},\Spaces))
\]
is the $\cE_n$-induction $\Ind_{\n1}^n
\Bigl(
\Bord_1^{\fr}(\RR^{\n1})
\Bigr)$.
By Corollary~\ref{cor1}, its space of $p$-simplices are the slice-wise projection-embedding tangles in $\RR^n\times\DD[p]$.
\end{definition}

\subsection{Projection-immersion tangles}

\begin{definition}\label{def.proj.imm.emb}
Let $W$ be a 1-manifold with boundary; let $X\in \Disk_{\n1}^{\fr}$ and $Y\in \Disk_1^{\fr}$. The topological space of \bit{projection-immersion embeddings} into $X\times Y\times\DD[p]$ is the topological subspace
\[
\Emb\bigl(W,X\times Y\times\DD[p]\bigr)^{\sf pr.imm}_{\pitchfork[p]} \subset
\Emb\bigl(W,X\times Y\times\DD[p]\bigr)_{\pitchfork[p]}
\]
of boundary-preserving embeddings $W\ra X\times Y\times\DD[p]$ transverse to $[p]\subset \DD[p]$ such that the composite projection
\[
W\ra X\times Y\times\DD[p]\ra X\times\DD[p]
\]
is an immersion.
\end{definition}

Compare the following definition with Definition~\ref{def.proj.emb.tang}.

\begin{definition}\label{def.proj.imm.tang}
Let $X\in \Disk_{\n1}^{\fr}$ and $Y\in \Disk_1^{\fr}$.
The topological space of framed projection-immersion tangles is
\[
\Tang_1^{\fr}(X\times Y\times\DD[p])_{\pitchfork[p]}^{\sf pr.imm}
:=
\coprod_{[W]}\Bigl( \Emb\bigl(W,X\times Y\times\DD[p]\bigr)^{\sf pr.imm}_{\pitchfork[p]}\underset{\Imm\bigl(W,X\times\DD[p]\bigr)} \times \Imm^{\fr}\bigl(W,X\times\DD[p]\bigr)\Bigr)_{\Aut(W)}
\]
where $\Imm^{\fr}(W,X\times\DD[p])$ is the space of immersions equipped with a trivialization of the Gauss map in the sense of Definition~\ref{def.framed.tangle}.
\end{definition}

\begin{observation}
\label{t.102}
Let $X\in \Disk_{\n1}^{\fr}$ and $Y\in \Disk_1^{\fr}$.
The spaces of framed projection-immersion tangles of Definition~\ref{def.proj.imm.tang} organize as a simplicial space that satisfies the Segal condition:
\[
\Tang_1^{\fr}(X\times Y\times\DD[\bullet])_{\pitchfork[\bullet]}^{\sf pr.imm}
~\in~
\fCat_{(\infty,1)}
~\subset~
\Fun(\bDelta^{\op} , \Spaces)
~.
\]

\end{observation}

\begin{lemma}\label{lemma.proj.emb.to.proj.imm}
The canonical map
\[
\Seg
\Bigl(
\Tang_1^{\fr}\bigl(\RR^n \times\DD[\bullet]\bigr)^{\sf s.pr.emb}_{\pitchfork[\bullet]}
\Bigr)[p]
\longrightarrow
\Tang_1^{\fr}(\RR^n\times\DD[p])_{\pitchfork[p]}^{\sf pr.imm}
\]
is an equivalence, 
and moreover identifies
$\Tang_1^{\fr}(\RR^n\times\DD[p])_{\pitchfork[\bullet]}^{\sf pr.imm}$
as the Segal-completion of 
$
\Tang_1^{\fr}\bigl(\RR^n \times\DD[\bullet]\bigr)^{\sf s.pr.emb}_{\pitchfork[\bullet]}
$.
Here $\Tang_1^{\fr}\bigl(\RR^n \times\DD[\bullet]\bigr)^{\sf s.pr.emb}_{\pitchfork[\bullet]}$ is the $\cE_n$-algebra in $\Fun(\bDelta^{\op},\Spaces)$ whose space of $p$-simplices is the space $\Tang_1^{\fr}\bigl(\RR^n \times\DD[p]\bigr)^{\sf s.pr.emb}_{\pitchfork[p]}$ as in Definition~\ref{def.proj.emb.simplicial.space}; $\Seg$ is as in Definition~\ref{def.Seg}.

\end{lemma}

\begin{proof}
This proof relies on Lemmas~\ref{lemma.hypercover.proj.imm.by.proj.emb} and~\ref{lemma.U.is.F}, which are stated and proved below.

Consider the functor $F$ defined as the composite in the following diagram:
\[
\xymatrix{
\TwAro(\bDelta^{\op})_{/[p]^\circ}\ar[d]\ar[rrrrr]^-F&&&&&\spaces\\
\TwAro(\bDelta^{\op})\ar[r]^-S&\Morita[\bDelta^{\op}]\ar[rrr]^-{\Tang_1^{\fr}\bigl(\RR^n \times\DD[\bullet]\bigr)^{\sf s.pr.emb}_{\pitchfork[\bullet]}}
&&&\cM[\Spaces]\ar[r]^-\sL&\Morita[\Spaces]\ar[u]^-{\sf LIM}
.
}
\]
To prove the claimed equivalence, we show that the colimit of the composite functor $F$ gives the $p$-simplices $\Tang_1^{\fr}\bigl(\RR^n \times\DD[p]\bigr)^{\sf pr.imm}_{\pitchfork[p]}$. That is, by construction of the functor $\Seg$, it will imply the claimed equivalence if the map
\[
\colim\Bigl(
\TwAro(\bDelta^{\op})_{/[p]^\circ}\xra{F}\Spaces
\Bigr)
\longrightarrow
\Tang_1^{\fr}\bigl(\RR^n \times\DD[p]\bigr)^{\sf pr.imm}_{\pitchfork[p]}
\]
is an equivalence: 
By Observation~\ref{t.102}, this result implies that $\Seg\bigl(\Tang_1^{\fr}\bigl(\RR^n \times\DD[\bullet]\bigr)^{\sf s.pr.emb}_{\pitchfork[\bullet]}\bigr)$ satisfies the Segal condition, and then by Lemma~\ref{lemma.Seg.universal} we conclude that $\Seg\bigl(\Tang_1^{\fr}\bigl(\RR^n \times\DD[\bullet]\bigr)^{\sf s.pr.emb}_{\pitchfork[\bullet]}\bigr)$ is the Segal-completion of the simplicial space $\Tang_1^{\fr}\bigl(\RR^n \times\DD[\bullet]\bigr)^{\sf s.pr.emb}_{\pitchfork[\bullet]}$.
We will compute this colimit after restricting along a final functor. 
Consider a closed interval $I$, and fix an object $P\in \sD_{\DD[p]}$ with $\pi_0(\DD[p]\smallsetminus P) = [p]$. Let 
\[
\Tang_1^{\fr}\bigl(\RR^n \times\DD[p]\bigr)^{\sf pr.imm}_{P}
~\subset~
\Tang_1^{\fr}\bigl(\RR^n \times\DD[p]\bigr)^{\sf pr.imm}
\]
denote the topological subspace of framed tangles $W \subset \RR^n \times\DD[p]$ such that the critical values of the projection $W\subset \RR^n\times\DD[p]\ra \DD[p]$ are contained in $P$.
Note the immediate consequence of isotopy extension that	
\[
\Tang_1^{\fr}\bigl(\RR^n \times\DD[p]\bigr)^{\sf pr.imm}_{P}
\simeq 
\Tang_1^{\fr}\bigl(\RR^n \times\DD[p]\bigr)^{\sf pr.imm}_{\pitchfork[p]}~.
\]
From Lemma~\ref{localize.D.Delta}, the colimit of the functor $F$ can be computed after restriction along
\[
\TwAr(\sD_{\DD[p]})^{\op}_{/P} \ra \TwAr(\bDelta^{\op})^{\op}_{/[p]^\circ}~.
\]
We are thereby reduced to proving the map
\[
\colim\Bigl(
\TwAr(\sD_{\DD[p]})^{\op}_{/P}\xra{F}\Spaces
\Bigr)
\longrightarrow
\Tang_1^{\fr}\bigl(\RR^n \times\DD[p]\bigr)^{\sf pr.imm}_{P}
\]
is an equivalence,
where $F$ is the restriction of $F$ to $\TwAr(\sD_{\DD[p]})^{\op}_{/P}$.

We construct a functor
\[
\fZ 
\colon
\TwAr(\sD_{\DD[p]})^{\op}_{/P}
\longrightarrow
\Opens\Bigl(
\Tang_1^{\fr}\bigl(\RR^n \times\DD[p]\bigr)^{\sf pr.imm}_{P}
\Bigr)
\]
assigning to an object $P\la A\ra B$ an open subspace $\fZ_{A\ra B}$ of $\Tang_1^{\fr}\bigl(\RR^n \times\DD[p]\bigr)^{\sf pr.imm}_{P}$ of those $W$ such that:
\begin{itemize}
\item The critical values of the composite projection $W\ra \DD[p]$ are contained in $A$. That is, $W$ is contained in the subspace
\[
\Tang_1^{\fr}\bigl(\RR^n \times\DD[p]\bigr)^{\sf pr.imm}_{A}
\subset
\Tang_1^{\fr}\bigl(\RR^n \times\DD[p]\bigr)^{\sf pr.imm}_{P}~.
\]
\item 
For each pair of consecutive components of $[0,1]\smallsetminus B$, there exists elements $t, t'$ of these components
such that the intersection
\[
W\cap \bigl( \RR^n\times [t,t']\bigr)
\]
is a slice-wise projection-embedding tangle in $\RR^n\times[t,t']$. 

\end{itemize}

We now prove:

\begin{enumerate}
\item Lemma~\ref{lemma.hypercover.proj.imm.by.proj.emb}: For any $W\in \Tang_1^{\fr}\bigl(\RR^n \times\DD[p]\bigr)^{\sf pr.imm}_{P}$, the subcategory
\[
\Bigl(\TwAro(\sD_{[0,1]})_{/P}\Bigr)_W
\subset
\TwAro(\sD_{[0,1]})_{/P}
\]
of $P\la A\ra B$ for which there is containment $W\in \fZ_{A\ra B}$ has contractible classifying space. Consequently, applying the higher Seifert--van Kampen theorem of \cite{HA} (specifically Theorem~A.3.1), there is a natural homotopy equivalence
\[
\Tang_1^{\fr}\bigl(\RR^n \times\DD[p]\bigr)^{\sf pr.imm}_{P}\simeq
\colim\Bigl(
\TwAr(\sD_{\DD[p]})^{\op}_{/P}
\xra{\fZ}
\Opens\bigl(\Tang_1^{\fr}\bigl(\RR^n \times\DD[p]\bigr)^{\sf pr.imm}_{P}\bigr)
\xra{\rm forget}
\Spaces
\Bigr)
\]
\item 
Lemma~\ref{lemma.U.is.F}: There is a natural equivalence
\[
\fZ_{A\ra B} \simeq F(P\la A \ra B)
~.
\]
 \end{enumerate}

The above results then imply the equivalence
\[
\colim\Bigl(\TwAr(\sD_{\DD[p]})^{\op}_{/P}\xra{F}\Spaces\Bigr)\simeq \Tang_1^{\fr}\bigl(\RR^n \times\DD[p]\bigr)^{\sf pr.imm}_{\pitchfork[p]}
\]
which completes the proof.
\end{proof}

The following lemmas were used in the preceding proof, of Lemma~\ref{lemma.proj.emb.to.proj.imm}.

\begin{lemma}\label{lemma.hypercover.proj.imm.by.proj.emb}

The canonical map between spaces
\[
\colim\Bigl(
\TwAr(\sD_{\DD[p]})^{\op}_{/P}
\xra{\fZ}
\Opens\bigl(\Tang_1^{\fr}\bigl(\RR^n \times\DD[p]\bigr)^{\sf pr.imm}_{P}\bigr)
\xra{\rm forget}
\Spaces
\Bigr)
\xra{~\simeq~}
\Tang_1^{\fr}\bigl(\RR^n \times\DD[p]\bigr)^{\sf pr.imm}_{P}
\]
is an equivalence.
\end{lemma}
\begin{proof}
The statement is a direct application of the higher Seifert--van Kampen theorem of \cite{HA} (specifically Theorem~A.3.1).
For this theorem to apply, we must verify the following.
\begin{itemize}
\item[]
For any $W\in \Tang_1^{\fr}\bigl(\RR^n \times\DD[p]\bigr)^{\sf pr.imm}_{P}$, the full subcategory
\[
\Bigl(\TwAro(\sD_{[0,1]})_{/P}\Bigr)_W
\subset
\TwAro(\sD_{[0,1]})_{/P}
~,
\]
consisting of those $P\la A\ra B$ for which there is containment $W\in \fZ_{A\ra B}$, has contractible classifying space. 
\end{itemize}

Fix a linear identification $\DD[p] = [0,p]$. 
First, we prove that for any such $W$, the category $\bigl(\TwAro(\sD_{[0,p]})_{/P}\bigr)_W$ is nonempty. To do so, it suffices to show that there exists $B\in \sD_{[0,p]}$ such that:
\begin{itemize}
\item $B$ contains the critical values of the projection $W \ra [0,p]$.
\item For any two points $t$ and $t'$ in consecutive components of $[0,p] \smallsetminus B$ and such that $[t,t']$ contains $B_0$, there exists a finite subset $S\subset \RR^{\{n\}}$ for which:
\begin{itemize}
\item $\RR^{\n1}\times S\times [t,t']$ is disjoint from $W$; and
\item for every pair of successive elements $s$ and $s'$ of $S\cup\{-\oo,\oo\}$, the projection
\[
W\cap\Bigl(\RR^{\n1}\times[s,s']\times  [t,t']\Bigr) \longrightarrow \RR^{\n1}\times  [t,t']
\]
is an embedding.
\end{itemize}
\end{itemize}

Since the projection $W \subset \RR^n \times \DD[p] \to \RR^{\n1} \times \DD[p]$ is assumed to be an immersion,
compactness of $W$ ensures,
for each element $x\in [0,p]$, there exists a closed neighborhood $D'_x$ such that the immersion
\[
W\cap(\RR^n\times D'_x) \longrightarrow \RR^{\n1}\times D'_x
\]
is an embedding when restricted to each connected component of $W\cap(\RR^n\times D'_x)$. Likewise, for each element $x\in [0,p]$, there exists a closed neighborhood $D''_x$ such that the connected components of $W\cap(\RR^n\times D''_x)$ are split in the $n$th coordinate: That is, there exists $S\subset \RR^{\{n\}}$ such that $W$ is disjoint from $\RR^{\n1}\times S\times D''_x$ and for each pair of successive elements $s$ and $s'$ of $S$, the intersection of $W$ with $\RR^{\n1}\times[s,s']\times D''_x$ is either empty or connected. We now set $D_x := D'_x \cap D''_x$.

By compactness of $[0,p]$, there exists a minimal finite set $T$ of $x\in [0,p]$ for which the interiors $D_x^\circ$ cover the interval $(0,p)$.  For each pair of consecutive elements $x,x'$ of $T$, choose an element $y \in D_x\cap D_{x'}$ which is also a regular value of the composite projection $W\subset \RR^n\times[0,p] \ra [0,p]$, and set $Q$ to be this finite set of elements $y$. Define $B$ to be the complement, $B:= [0,p] \smallsetminus Q$. By construction $B$ satisfies the bulleted conditions, which implies that $(\TwAro(\sD_{[0,p]})_{/P})_W$ is nonempty.

Second, we now prove $\bigl(\TwAro(\sD_{[0,p]})_{/P}\bigr)_W$ has contractible classifying space. We use the following essential property: For any pair of objects $P\la A \ra B$ and $P\la A' \ra B'$ of $\bigl(\TwAro(\sD_{[0,p]})_{/P}\bigr)_W$, observe that both $P\la A\cap A'\ra B$ and $P\ra A\cap A'\ra B\cap B'$ form objects of
$\bigl(\TwAro(\sD_{[0,p]})_{/P}\bigr)_W$. The verification that $\bigl(\TwAro(\sD_{[0,p]})_{/P}\bigr)_W$ has contractible classifying space is now identical to the proof in Lemma~\ref{lemma.claimB}, by checking the condition that for every functor $G:\cK \ra \bigl(\TwAro(\sD_{[0,p]})_{/P}\bigr)_W$ from a poset $\cK$, there exists a factorization
\[
\xymatrix{
\cK\ar[rd]\ar[rr]^-G&&\bigl(\TwAro(\sD_{[0,p]})_{/P}\bigr)_W\\
&\ov{\cK}\ar@{-->}[ur]_-{\ov{G}}
&
.
}
\]
%

%}
\end{proof}

\begin{lemma}\label{lemma.U.is.F}
There is a natural equivalence
\[
\fZ_{A\ra B} \simeq F(P\la A \ra B)
~.
\]
\end{lemma}

\begin{proof}
First, note that for $(A\ra B)\in \TwAr(\sD_{[0,p]})$ an object with $\pi B =\{0<a\}$ and $\pi A = [a]^\circ\xra{\act}\{0<a\}^\circ = \pi B$ the active morphism, the space $\fZ_{A\ra B}$ is homotopy equivalent to 
\[
\fZ_{A\ra B}\simeq \Tang_1^{\fr}\bigl(\RR^n \times[0,a]\bigr)^{\sf s.pr.emb}_{\pitchfork\{0,a\}}  = F([0,a]) 
\]
by the condition that for any $W\in \fZ_{A\ra B}$ there exists $t,t'$ in $[0,p]\smallsetminus B$ for which the intersection
\[
W\cap \bigl( \RR^n\times [t,t']\bigr)
\]
is a slice-wise projection-embedding tangle.

For general $A\ra B$ we decompose $\fZ_{A\ra B}$ as a fiber product
\[
\fZ_{A_1\ra B_1}
\xra{~\cong~}
\underset{\fZ_{\emptyset \to C_1}}\times\ldots \underset{\fZ_{ \emptyset \to C_b}}\times\fZ_{A_b\ra B_b}
\]
as in the proof of Lemma~\ref{lemma.claimC}. 
Again, the restriction maps in the fiber product are fibrations by the parametrized isotopy extension theorem.
Therefore, this iterated pullback among topological spaces is an iterated homotopy pullback.

By Observation~\ref{obs.Segal.basecase.contact}, the value $F([a]^\circ\xra{g^\circ}[b]^\circ)$ is given by the iterated fiber product:
\[
F([0,g(0)])
\underset{F(g(0))}\times
F([g(0),g(1)])
\underset{F(g(1))}\times
\ldots
\underset{F(g(b))}\times
F([g(b),a])
\]
where 
\[
F([k,l]) = \Tang_1^{\fr}\bigl(\RR^n \times[k,l]\bigr)^{\sf s.pr.emb}_{\pitchfork[k,l]}
\]
and
\[
F(\{k\}) = \Tang_1^{\fr}\bigl(\RR^n \times\{k\}\bigr)^{\sf s.pr.emb}~.
\]
The general equivalence $\fZ_{A\ra B}\simeq F(A \ra B)$ follows.
\end{proof}

\begin{definition}\label{def.Bord.pr.imm}
The $\cE_n$-monoidal flagged $(\oo,1)$-category
\[
\Bord_1^{\fr}(\RR^n)^{\sf pr.imm}\in \Alg_n(\fCat_{(\oo,1)})
\]
is defined to be $\Seg
\Bigl(
\Tang_1^{\fr}\bigl(\RR^n \times\DD[\bullet]\bigr)^{\sf s.pr.emb}_{\pitchfork[\bullet]}
\Bigr)$, the Segal-completion of the slice-wise projection-embedding tangle spaces.
By Lemma~\ref{lemma.proj.emb.to.proj.imm}, its space of $p$-composable morphisms is the space
\[
\Bord_1^{\fr}(\RR^n)^{\sf pr.imm}[p]
\simeq
\Tang_1^{\fr}(\RR^n\times\DD[p])^{\sf pr.imm}_{\pitchfork[p]}
\]
from Definition~\ref{def.proj.imm.tang}.
\end{definition}

\subsection{Univalent-completion of projection-immersion tangles}

We now prove that the projection-immersion tangle category is equivalent to the usual tangle category after univalent-completion.
\begin{lemma}\label{lemma.univ.cplt.of.proj.imm}
The functor in $\Alg_n(\fCat_{(\oo,1)})$
\[
\Bord_1^{\fr}(\RR^n)^{\sf pr.imm}
\longrightarrow
\Bord_1^{\fr}(\RR^n)
\]
is an equivalence after univalent-completion.
\end{lemma}
\begin{proof}
This proof relies on Corollary~\ref{cor.hprinciple.framed.embimm}, which is stated and proved in \S\ref{sec.h.prin}.
By the model for univalent-completion given in Proposition~2.4 of \cite{flagged}, it suffices to show that the map on spaces of objects is $\pi_0$-surjective, and all maps on hom-spaces are equivalences.

Recall that the data defining each object $W \in \Bord_1^{\fr}(\RR^n)^{\sf pr.imm}$ consists of: a finite set $W\subset \RR^n$, a trivialization of the Gauss map $W_i \ra \Gr_1(n)$ associated to the Gauss map of the composite immersion $W_i\ra \RR^n\ra\RR^{\n1}$ and the standard solid $n$-framing of $\RR^{\n1}$.
So $\Obj \left( \Bord_1^{\fr}(\RR^n)^{\sf pr.imm} \right) \simeq \FF_{\cE_n}(\Omega \Gr_1(n))$.
Meanwhile, an object $W \in \Bord_1^{\fr}(\RR^n)$ consists of: a finite set $W\subset \RR^n$, a trivialization of the Gauss map $W_i \ra \Gr_1(n+1)$ associated to the Gauss map of the composite immersion $W_i\ra \RR^n$ and the standard solid $(n+1)$-framing of $\RR^{n}$.
So $\Obj \left( \Bord_1^{\fr}(\RR^n) \right) \simeq \FF_{\cE_n}(\Omega \Gr_1(n+1))$.
The condition of $\pi_0$-surjectivity on spaces of objects then follows from surjectivity of the homomorphism $\pi_1( \Gr_1(n)) \to \pi_1(\Gr_1(n+1))$ for $n\geq 2$.

We now show that all maps on hom-space are equivalences.
Let $W_0,W_1 \in \obj(\Bord_1^{\fr}(\RR^n)^{\sf pr.imm}$.
Denote their images $W_0',W_1' \in \obj(\Bord_1^{\fr}(\RR^n)$.
We must show the map
\[
\Map_{\Bord_1^{\fr}(\RR^n)^{\sf pr.imm}}(W_0,W_1) \ra
\Map_{\Bord_1^{\fr}(\RR^n)}(W_0',W_1') 
\]
is an equivalence.

From the topological models of these tangle spaces (see Definition~\ref{def.Bord.pr.imm} and Lemma~\ref{lemma.corep.tang}) this equivalence will follow from showing that, for every $W$ a compact 1-manifold with boundary, the horizontal maps
\[
\xymatrix{
\EmbImm^{\fr}(W, \RR^n\times\DD^1\ra \RR^{\n1}\times\DD^1)\ar[d]
\ar[r]&\Emb^{\fr}(W, \RR^n\times\DD^1)\ar[d]\\
\EmbImm^{\fr}(\partial W, \RR^n\times\partial\DD^1\ra \RR^{\n1}\times\partial\DD^1)\ar[r]&
\Emb^{\fr}(\partial W, \RR^n\times\partial \DD^1)}
\]
induce a weak homotopy equivalence on vertical homotopy fibers. Here, we use the notation for projection-immersions of Definition~\ref{def.emb.imm}: Note that the identification $\EmbImm(W, \RR^n\times\DD^1\ra \RR^{\n1}\times\DD^1)= \Emb\bigl(W,\RR^{\n1}\times \RR^1\times\DD^1\bigr)^{\sf pr.imm}$ is by definition.

 By the h-principle of Corollary~\ref{cor.hprinciple.framed.embimm}, we can replace the spaces of framed embedding-immersions with those of framed embedding-formal-immersions. We thus are reduced to showing that the diagram
\[
\xymatrix{
\EmbfImm^{\fr}(W, \RR^n\times\DD^1\ra \RR^{\n1}\times\DD^1)\ar[d]
\ar[r]&\Emb^{\fr}(W, \RR^n\times\DD^1)\ar[d]\\
\EmbfImm^{\fr}(\partial W, \RR^n\times\partial\DD^1\ra \RR^{\n1}\times\partial\DD^1)\ar[r]&
\Emb^{\fr}(\partial W, \RR^n\times\partial \DD^1)}
\]
induces a weak homotopy equivalence on vertical homotopy fibers. 
Let $g_0 \in \EmbfImm^{\fr}(\partial W, \RR^n\times\partial\DD^1\ra \RR^{\n1}\times\partial\DD^1)$.
Denote the map between vertical homotopy fibers over $g_0$ as
\begin{equation}
\label{e.fibs}
{\sf Fib}_1
\longrightarrow
{\sf Fib}_2
~.
\end{equation}
Forgetting framing defines maps to the space of embeddings of $W$ into $\RR^n \times \DD^1$, relative to the embedding $g_0$ of its boundary:
\begin{equation}
\label{e.fibs.2}
\xymatrix{
{\sf Fib}_1
\ar[rr]^-{(\ref{e.fibs})}
\ar[dr]
&&
{\sf Fib}_2
\ar[dl]
\\
&
\Emb_{g_0}(W, \RR^n\times\DD^1)
&
.
}
\end{equation}
The map~(\ref{e.fibs}) is an equivalence if the map between fibers of~(\ref{e.fibs.2}) is an equivalence.  
So choose an embedding $g:W\hookrightarrow \RR^n\times\DD^1$ extending the embedding $g_0$ of the boundary.  
By definition of framed embeddings, the fiber of ${\sf Fib}_2 \to \Emb_{g_0}(W, \RR^n\times\DD^1)$ over $g$ is the space of lifts of the Gauss map:
\[
\xymatrix{
W\ar@{-->}[rrd]\ar[rr]^-{{\sf Gauss}(g)}&&\Gr_1(n+1)
&
\\
\partial W\ar[u]\ar[rr]&&\ast 
\ar[u]
&
.
}
\]
By definition of framed embedding-formal-immersions, the fiber of ${\sf Fib}_1 \to \Emb_{g_0}(W, \RR^n\times\DD^1)$ over $g$ is the space of pairs of lifts:
\[
\xymatrix{
W\ar@{-->}[rrdd]\ar@{-->}[rrd]\ar[rr]^-{{\sf Gauss}(g)}&&\Gr_1(n+1)\\
&&\Gr_1(n)\ar[u]
&
\\
\partial W\ar[uu]\ar[rr]&&\ast
\ar[u]
&
.
}
\]
Since a trivialization of the Gauss map to $\Gr_1(n+1)$ is equivalent to first lifting the Gauss map to $\Gr_1(n)$ and then trivializing this lift, therefore these spaces of factorizations are homotopy equivalent.

\end{proof}

We now prove the main theorem of this paper.

\begin{theorem}[Tangle Hypothesis]\label{theorem.main}
Let $\cR$ be a rigid $\cE_{\n1}$-monoidal $(\infty,1)$-category, with $n\geq 2$. Let $\Bord_1^{\fr}(\RR^{\n1})$ be the $\cE_{\n1}$-monoidal flagged $(\infty,1)$-category of framed tangles of Definition~\ref{def.bord}.
Evaluation at the object $\ast\in \Bord_1^{\fr}(\RR^{\n1})$ defines an equivalence between spaces
\[
\ev_\ast
\colon
\Map_{\Alg_{\n1}(\fCat_{(\oo,1)})} \bigl(
\Bord_1^{\fr}(\RR^{\n1})
,
\cR
\bigr)
\xra{~\simeq~}
\Obj(\cR)
~,\qquad
Z
\longmapsto
Z(\ast)
~.
\]
between the space of objects of $\cR$ and the space of $\cR$-valued $\cE_{\n1}$-monoidal functors from $\Bord_1^{\fr}(\RR^{\n1})$.
\end{theorem}
\begin{proof}
We prove the result by induction on $n$. 
First, note that the object $\ast \in \Bord_1^{\fr}(\RR^{\n1})$ is in the image of the composite functor 
\[
\Bord_1^{\fr}(\RR^{n \text{-} 2}) 
\to 
\Ind_{{\mathit{n}\text{-}2}}^{\n1}\bigl(\Bord_1^{\fr}(\RR^{{\mathit{n}\text{-}2}})\bigr)
\to
\Bord_1^{\fr}(\RR^{\n1})^{\sf pr.imm}
\to
\Bord_1^{\fr}(\RR^{\n1})
~,
\]
in which the first is $\cE_{n\text{-}2}$-monoidal, and the others are $\cE_{\n1}$-monoidal.
The base case of $n=2$ is Theorem~\ref{thm.n2}.
We now have the following chain of equivalence:
\[
\Map_{\Alg_{\n1}(\fCat_{(\oo,1)})} \bigl(
\Bord_1^{\fr}(\RR^{\n1})
,
\cR
\bigr)
\simeq
\Map_{\Alg_{\n1}(\fCat_{(\oo,1)})} \bigl(
\Bord_1^{\fr}(\RR^{\n1})^{\sf pr.imm}
,
\cR
\bigr)
\]
\[
\simeq
\Map_{\Alg_{\n1}(\fCat_{(\oo,1)})} \Bigl(
\Ind_{{\mathit{n}\text{-}2}}^{\n1}\bigl(\Bord_1^{\fr}(\RR^{{\mathit{n}\text{-}2}})\bigr)
,
\cR
\Bigr)
\simeq
\Map_{\Alg_{{\mathit{n}\text{-}2}}(\fCat_{(\oo,1)})} \bigl(
\Bord_1^{\fr}(\RR^{{\mathit{n}\text{-}2}})
,
\cR
\bigr)
\underset{\ev_\ast}{\xra{\simeq} }
\obj(\cR)~.
\]
The first equivalence follows from Lemma~\ref{lemma.univ.cplt.of.proj.imm}. The second equivalence follows from Definition~\ref{def.Bord.pr.imm}, after Lemma~\ref{lemma.proj.emb.to.proj.imm} and Definition~\ref{def.proj.emb.simplicial.space}. The penultimate equivalence is the adjunction property of induction, and the last equivalence is the inductive hypothesis.
\end{proof}

\section{A relative h-principle for embedding-immersions}\label{sec.h.prin}

We now establish an h-principle for embedding-immersions, which was used in the previous section in proving Lemma~\ref{lemma.univ.cplt.of.proj.imm}, that the univalent-completion of projection-immersion tangles gives the usual tangle category.

\subsection{Relatives sheaves and h-principles}

\begin{definition}\label{def.rel.sheaf}
Let $M$ be a topological space.
For $\cF\ra \cE$ a morphism of presheaves on $\Opens(M)$, then $\cF$ is a relative sheaf if for every covering sieve $\cU \subset \Opens(M)_{/U} = \Opens(U)$ the diagram
\[
\xymatrix{
\cF(U) \ar[r]\ar[d]& \underset{V \in \cU}{\sf lim}\ \cF(V)\ar[d]\\
\cE(U)\ar[r] &  \underset{V \in \cU}{\sf lim}\ \cE(V)}
\]
is a pullback.
\end{definition}
\begin{remark}
Note that if $\cE$ and $\cF$ are both sheaves, then $\cF$ is automatically a relative sheaf since the horizontal morphisms in Definition~\ref{def.rel.sheaf} are both equivalences.
\end{remark}

\begin{definition}
Let $\cF:\Mfld_n^{\partial,\op} \ra \Top$ be a continuous functor of topologically-enriched categories. The h-sheafification of $\cF$ is
\[
\sD\cF := \iota_\ast\iota^\ast \cF
\]
the homotopy right Kan extension of the restriction of $\cF$ along $\iota: \Disk_n^{\partial, \leq 1,\op}\hookrightarrow \Mfld_n^{\partial,\op}$, the inclusion of the full topologically-enriched subcategory consisting of the two objects $\RR^n$ and $\RR^{\n1} \times \RR_{\geq 0}$.
Then $\cF$ adheres to the \bit{h-principle} if the map $\cF \ra \sD\cF$ is a weak homotopy equivalence.
\end{definition}
\begin{remark}
For $M$ an $n$-manifold without boundary, consider the diagonal quotient (i.e., the Borel construction)
\[
M_\cF :={\sf Fr}(M)\underset{\sO(n)}\times \cF(\RR^n)
\]
where ${\sf Fr}(M)$ is the frame bundle of $M$. This can alternatively be constructed as the unstraightening of the functor
\[
M\simeq \Disk^{\leq 1, \op}_{n/M} \xra\cF \Top \ra \Spaces~.
\]
There is then a natural homotopy equivalence with the space of sections
\[
\sD\cF(M) \simeq \Gamma\bigl(M_\cF \ra M\bigr)~.
\]
Given a choice of Riemannian metric on $M$, the natural map $\cF(M) \ra \sD\cF(M)$ is then equivalent to the scanning map of Segal \cite{segal.scanning}.
\end{remark}

\begin{definition}\label{def.rel.h.principle}
For $\cF\ra \cE$ a relative sheaf in the sense of Definition~\ref{def.rel.sheaf}, then $\cF\ra \cE$ adheres to the \bit{relative h-principle} if the diagram
\[
\xymatrix{
\cF \ar[r]\ar[d]& \sD \cF\ar[d]\\
\cE\ar[r] &  \sD\cE}
\]
is a homotopy pullback.
\end{definition}

Recall the following standard notion, flexibility.
\begin{definition}
A presheaf $\cF$ of topological spaces on $M$ is \bit{flexible} if the restriction map $\cF(V) \to \cF(U)$ is a Serre fibration for every pair of $U \subset V$ of subsets of $M$ that satisfy the following condition:
\begin{itemize}
\item[]
Both $U \subset M$ and $V\subset M$ are compact codimension-0 submanifolds of $M$ with boundary.
Furthermore, $U$ is contained in the point-set interior of $V$.

\end{itemize}
\end{definition}

\begin{example}
Let $N$ be a manifold with boundary.
\begin{itemize}
\item The functor of boundary-preserving smooth maps $\Map(-, N)$, endowed with the compact-open smooth topology (also known as the weak $C^\infty$ Whitney topology), is a flexible sheaf.
\item The functor of boundary-preserving embeddings $\Emb(-, N)$ is a flexible presheaf, by the parametrized isotopy extension theorem.
\item The functor of boundary-preserving immersions $\Imm(-, N)$ is a flexible sheaf on open manifolds. This is the main theorem of Smale--Hirsch immersion theory.
\end{itemize}
\end{example}

The essential importance of flexibility for sheaves is reflected in the following standard proposition, which effectively dates to Smale--Hirsch's h-principle for immersions, if not before.

\begin{prop}\label{prop.flex.sheaves.hprin}
For $\cF$ a flexible sheaf on $\Mfld_n^{\partial}$, then the natural map $\cF \ra \sD \cF$ is an equivalence. That is, a flexible sheaf adheres to the h-principle.
\end{prop}
\begin{proof}[Proof sketch.]
By induction on $0\leq k\leq n$, the maps $\cF(S^k\times\RR^{\nk})\ra \sD\cF(S^k\times\RR^{\nk})$ and $\cF(S^k\times\RR_{\geq 0}^{\nk})\ra \sD\cF(S^k\times\RR_{\geq 0}^{\nk})$ are weak homotopy equivalences. By induction on the index of a handle presentation of $M$, the map $\cF(M)\ra \sD\cF(M)$ is then a weak homotopy equivalence, for all $M$.
\end{proof}

\begin{observation}
Let $\cE$ be a flexible presheaf on $n$-manifolds with boundary. If $\cF\ra \cF'$ is a map of presheaves over $\cE$, and both $\cF$ and $\cF'$ adhere to the relative h-principle, then $\cF\ra \cF'$ is a weak homotopy equivalence if and only if the values $\cF(\RR^n) \ra \cF'(\RR^n)$ and $\cF(\RR^n_{\geq 0}) \ra \cF'(\RR^n_{\geq 0})$ are equivalences.

\end{observation}

\begin{lemma}\label{lemma.flex.rel.hprin}
Let $\cF\ra \cE$ be a relative sheaf, where $\cF$ and $\cE$ are flexible presheaves on $M$. Then $\cF\ra \cE$ adheres to the relative h-principle on $M$.
\end{lemma}
\begin{proof}
Let $\sD_{\cE}\cF$ denote the functor assigning to $U\subset M$ the homotopy fiber product of the diagram $\cE(U) \ra \sD\cE(U)\la \sD\cF(U)$. Note that we can reformulate the relative h-principle as the condition that the natural transformation $\cF \ra \sD_\cE\cF$ is an equivalence. Note that $\sD_\cE\cF$ has the following homotopy sheaf property relative to $\cE$: For any codimension-0 submanifolds $U,V\subset M$ with boundary, the diagram
\[
\xymatrix{
\sD_\cE\cF(U\cup V)\ar[r]\ar[d]&\sD_\cE\cF(U)\underset{\sD_\cE\cF(U\cap V)}\times\sD_\cE\cF(V)\ar[d]\\
\cE(U\cup V)\ar[r]&\cE(U)\underset{\cE(U\cap V)}\times\cE(V)}
\]
is a homotopy pullback. To show that that any map $\cF\ra \sD_\cE\cF$ is a weak equivalence, it therefore suffices to show
\begin{enumerate}
\item $\cF(D) \ra \sD_\cE\cF(D)$ is a weak equivalence for $D\in\Disk_{n/M}^{\partial, \leq 1}$;
\item For all $U,V \in \Mfld_{n/M}^\partial$, the diagram
\[
\xymatrix{
\cF(U\cup V)\ar[r]\ar[d]&\cF(U)\underset{\cF(U\cap V)}\times\cF(V)\ar[d]\\
\cE(U\cup V)\ar[r]&\cE(U)\underset{\cE(U\cap V)}\times\cE(V)}
\]
is a homotopy pullback.
\end{enumerate}

The first equivalence is immediate by construction of $\sD$, as homotopy right Kan extension from the full subcategory $\Disk_{n/M}^{\partial, \leq 1}$. We now show the second condition. It suffices to show that for every $U$ and $V$ compact submanifolds of $M$ with boundary, then $\cF(U\cup V)$, which is the limit of the diagram of topological spaces
\[
\xymatrix{
\cF(U)\ar[d]\ar[dr]&\ar[dl]\cE(U\cup V)\ar[dr]&\cF(V)\ar[d]\ar[dl]\\
\cE(U)\ar[dr]&\cF(U\cap V)\ar[d]&\cE(V)\ar[dl]\\
&\cE(U\cap V)~,}
\]
agrees with the homotopy limit of this diagram. The homotopy limit is equivalent to the ordinary limit of
\begin{equation}
\label{e.shf}
\xymatrix{
\cF(U)\underset{\cE(U)}\times\cE(U)^{[0,1]}\ar[d]\ar[dr]&\ar[dl]\cE(U\cup V)\ar[dr]&\cF(V)\underset{\cE(V)}\times\cE(V)^{[0,1]}\ar[d]\ar[dl]\\
\cE(U)\ar[dr]&\cF(U\cap V)\underset{\cE(U\cap V)}\times\cE(U\cap V)^{[0,1]}\ar[d]&\cE(V)\ar[dl]\\
&\cE(U\cap V)}
\end{equation}
since every map in this diagram is a Serre fibration. 
Using that $\cF \to \cE$ is a relative sheaf, the canonical map
\[
\cF(U\cup V) \cong \cE(U\cup V)\underset{\cE(U)\underset{\cE(U\cap V)}\times\cE(V)}\times \Bigl(\cF(U)\underset{\cF(U\cap V)}\times\cF(V)\Bigr)~.
\]
is a homeomorphism.
Therefore, the ordinary limit of~(\ref{e.shf}) is then homeomorphic to the ordinary limit of the diagram
\[
\xymatrix{
\cF(U\cup V)\ar[rd]&& \cE(U)^{[0,1]}\underset{\cE(U\cap V)^{[0,1]}}\times\cE(V)^{[0,1]}\ar[dl]\\
&\cE(U)\underset{\cE(U\cap V)}\times\cE(V)
&
~.
}
\]
The map $ \cE(U)^{[0,1]}\underset{\cE(U\cap V)^{[0,1]}}\times\cE(V)^{[0,1]}\ra \cE(U)\underset{\cE(U\cap V)}\times\cE(V)$ is an acyclic fibration, so therefore this limit is homotopy equivalent to $\cF(U\cup V)$.
\end{proof}

Gromov's Theorem, that a diffeomorphism-invariant microflexible sheaf on open manifolds adheres to the h-principle \cite{gromov}, admits a relative formulation.
\begin{theorem}\label{theorem.relative.Gromov.hprinciple}
Let $E$ be a manifold with boundary.
Consider the presheaf $\Map(-,E) \colon \Mfld_n^{\partial , \op} \to \Top$ of smooth boundary-preserving maps to $E$.
Assume the following:
\begin{itemize}
\item $\cE\subset \Map(-, E)$ is an open subpresheaf which is flexible on open $n$-manifolds.
\item $\cF \subset \cE$ is an open relative sheaf.
\end{itemize}
Then $\cF$ is a flexible presheaf on open manifolds, and $\cF\ra \cE$ adheres to the relative h-principle on open $n$-manifolds.
\end{theorem}
\begin{proof}
By Lemma~\ref{lemma.flex.rel.hprin}, it suffices to show that $\cF$ is a flexible presheaf on open manifolds. 
Let $U \hookrightarrow V$ be an inclusion of compact codimension-0 submanifolds of $M$  with boundary, such that $U$ is contained in the point-set interior of $V$.
Let $C := V \smallsetminus {\sf Int}(U)$ be the complement in $V$ of the point-set interior of $U$.
Regard $C$ as an $n$-dimensional cobordism from $\partial U$ to $\partial V$.
Choose a finite handle-presentation of $C$.
Using the assumption that $M$ is an open $n$-manifold, 
we can assume all handles in this handle-presentation of $C$ have handle-index strictly less than $n$.
Choose a linear order on the finite set of handles.  
Denote by $D_i$ the $i^{th}$ handle of this handle-presentation of $C$, and by $A_i \subset D_i$ its site of handle-attachment.  
So, if $D_i$ has index $k$, there are isomorphisms $D_i \cong \DD^{k} \times \DD^{n\text{-}k}$ and $A_i \cong ( S^{k\text{-}1} \times \DD^1) \times \DD^{n\text{-}k}$.
Consider the filtration $U = H_0 \subset H_1 \subset \dots \subset H_r = V$ where $H_i$ is the union of $U$ with the first $i$ handles of the handle presentation of $C$.
In other words, $H_i = H_{i-1} \cup D_i$ with $H_{i-1} \cap D_i = A_i$.

By induction on $0\leq i \leq r$, we prove that the restriction map $\cF(H_i) \to \cF(U)$ is a Serre fibration for all $0\leq i \leq r$.  
The case $i=r$ establishes that the restriction map $\cF(V) \to \cF(U)$ is a Serre fibration, as desired.
Note the factorization of the restriction map through other restriction maps:
\[
\cF(H_i)
\longrightarrow
\cF(H_{i-1})
\longrightarrow
\dots
\longrightarrow
\cF(H_1)
\longrightarrow
\cF(U)
~.
\]
Using that the composition of Serre fibrations is again a Serre fibration, we are reduced to proving the restriction map $\cF(H_i) \to \cF(H_{i-1})$ is a Serre fibration.
Recall the identities $H_i = H_{i-1} \cup D_i$ with $H_{i-1} \cap D_i = A_i$.
The restriction map $\cF(H_i) \to \cF(H_{i-1})$ fits into a cube among spaces:
% https://q.uiver.app/#q=WzAsOCxbMSwwLCJcXGNGKEhfaSkiXSxbMCwxLCJcXGNGKEhfe2ktMX0pIl0sWzEsMiwiXFxjRihBX2kpIl0sWzIsMSwiXFxjRihEX2kpIl0sWzEsMSwiXFxjRShIX2kpIl0sWzAsMiwiXFxjRShIX3tpLTF9KSJdLFsxLDMsIlxcY0UoQV9pKSJdLFsyLDIsIlxcY0UoRF9pKSJdLFswLDFdLFsxLDJdLFswLDNdLFszLDJdLFs0LDVdLFs1LDZdLFs0LDddLFs3LDZdLFsxLDVdLFswLDRdLFsyLDZdLFszLDddXQ==
\[
\begin{tikzcd}
	& {\cF(H_i)} \\
	{\cF(H_{i-1})} & {\cE(H_i)} & {\cF(D_i)} \\
	{\cE(H_{i-1})} & {\cF(A_i)} & {\cE(D_i)} \\
	& {\cE(A_i)}
	\arrow[from=1-2, to=2-1]
	\arrow[from=1-2, to=2-2]
	\arrow[from=1-2, to=2-3]
	\arrow[from=2-1, to=3-1]
	\arrow[from=2-1, to=3-2]
	\arrow[from=2-2, to=3-1]
	\arrow[from=2-2, to=3-3]
	\arrow[from=2-3, to=3-2]
	\arrow[from=2-3, to=3-3]
	\arrow[from=3-1, to=4-2]
	\arrow[from=3-2, to=4-2]
	\arrow[from=3-3, to=4-2]
\end{tikzcd}
~.
\]
The assumption that $\cE$ is a flexible presheaf ensures each of the maps in the lower square are Serre fibrations.
The assumption that $\cF \to \cE$ is a relative sheaf implies this cube is a limit.
Consequently, the restriction map $\cF(H_i) \to \cF(H_{i-1})$ is a Serre fibration provided the restriction map $\cF(D_i) \to \cF(A_i)$ is a Serre fibration.
Through the identifications $D_i \cong \DD^{k} \times \DD^{n\text{-}k}$ and $A_i \cong ( S^{k\text{-}1} \times \DD^1) \times \DD^{n\text{-}k}$, we are therefore reduced to showing the restriction map
\[
\cF(\DD^{k} \times \DD^{n\text{-}k})
\longrightarrow
\cF( ( S^{k\text{-}1} \times \DD^1) \times \DD^{n\text{-}k})
\]
is a Serre fibration for all $0\leq k <n$.  This is exactly as in Proposition~3 \cite{haefliger} (which is further elucidated in Figure~3.5 in the proof of Proposition~3.14 \cite{geiges}). This gives that the maps
\[
\cF(\DD^k\times \DD^{\nk}) \longrightarrow \cF(S^{{\mathit{k}\text{-}1}}\times\DD^1\times\DD^{\nk})
\]
are Serre fibrations for $k<n$.
\end{proof}

\subsection{Embedding-immersions}

In this section, we fix a smooth fiber bundle $\pi: E\ra N$, where $E$ and $N$ are manifolds with boundary and $\pi$ preserves boundary.

\begin{definition}\label{def.emb.imm}
For $M$ a smooth manifold, an \bit{embedding-immersion} of $M$ into $E\ra N$ is a boundary-preserving embedding $M\hookrightarrow E$ with the property that the composite map $M\ra E\ra N$ is an immersion. The topological space of embedding-immersions of $M$ into $E\ra N$ is the subspace
\[
\EmbImm\bigl(M, E\ra N\bigr) \subset \Emb(M,E)
\]
consisting of embedding-immersions. An \bit{immersion-immersion} of $M$ into $E\ra N$ is a boundary-preserving immersion $M\hookrightarrow E$ with the property that the composite map $M\ra E\ra N$ is an immersion. The topological space of immersion-immersions of $M$ into $E\ra N$ is the subspace
\[
\ImmImm\bigl(M, E\ra N\bigr) \subset \Imm(M,E)
\]
consisting of immersion-immersions. 
\end{definition}

\begin{definition}
Let $M$ be a manifold with boundary. The space of \bit{formal immersions} of $M$ into $N$ is the pullback
\[
\xymatrix{
\fImm(M,N)\ar[r]\ar[d]&{\sf Vect}^{\sf inj}(\sT M ,\sT N)\ar[d]\\
{\sf Vect}^{\sf inj}(\sT\partial M ,\sT\partial N)\ar[r]&{\sf Vect}^{\sf inj}(\sT M_{|\partial M} ,\sT N_{|\partial N})~.}
\]
\end{definition}

\begin{definition}
Let $M$ be a manifold with boundary. The space of formal immersions of $M$ into $E\ra N$ is the subspace
\[
\fImm(M, E\ra N) \subset \fImm(M,E)
\]
consisting of those $\sT M \ra \sT E$ for which the both of the composites $\sT M \ra \sT E \ra \sT N$ and $\sT \partial M \to \sT \partial E \to \sT \partial N$ are injective.  
\end{definition}

\begin{definition}
Let $\pi: E\ra N$ be a smooth fiber bundle, where $E$ and $N$ are manifolds with boundary, and $\pi$ is a boundary-preserving map. For $M$ a smooth manifold, the space of embedding-formal-immersions of $M$ into $E\ra N$ is the homotopy pullback
\[
\xymatrix{
\EmbfImm(M, E\ra N)\ar[d]\ar[r]&\fImm(M, E \ra N)\ar[d]\\
\Emb(M,E)\ar[r]&\fImm(M, E)~.}
\]
The space of immersion-formal-immersions of $M$ into $E\ra N$ is the homotopy pullback
\[
\xymatrix{
\ImmfImm(M, E\ra N)\ar[d]\ar[r]&\fImm(M, E \ra N)\ar[d]\\
\Imm(M,E)\ar[r]&\fImm(M, E)~.}
\]
\end{definition}

\begin{prop}\label{prop.deriv.of.embimm}
The natural maps $\Emb(-, -)\ra \Imm(-,-)\ra \fImm(-,-)$ induce equivalences
\[
\sD\Imm(-,N) \simeq \fImm(-,N)~,
\]
\[
\sD\Emb(-,E)\simeq \fImm(-,E)~, 
\]
\[
\sD\EmbfImm(-,E\ra N)\simeq \fImm(-,E\ra N)~. 
\]
\end{prop}
\begin{proof}
It is sufficient to observe the following criterion for h-sheafification, for $\cF$ a presheaf of topological spaces on $\Mfld_n^{\partial}$: A natural transformation $\cF \ra \cG$ exhibits $\cG$ as the h-sheafification $\sD\cF$ if: first, $\cG$ is a homotopy sheaf; and second, values $\cF(\RR^n)\ra \cG(\RR^n)$ and $\cF(\RR^n_{\geq 0}) \ra \cG(\RR^n_{\geq 0}$ are homotopy equivalences. The result then follows from the observation first that $\fImm(-,E)$ is a homotopy sheaf for all $E$, and second the natural equivalences $\Emb(\RR^n,-)\simeq \Imm(\RR^n,-)\simeq \fImm(\RR^n,-)$ and $\Emb(\RR_{\geq 0}^n,-)\simeq \Imm(\RR_{\geq 0}^n,-)\simeq \fImm(\RR_{\geq 0}^n,-)$.
\end{proof}

\begin{cor}
Let $\pi: E\ra N$ a smooth fiber bundle, where $E$ and $N$ are manifolds with boundary, and $\pi$ is a boundary-preserving map. 
Then the presheaf of embedding-formal-immersions $\EmbfImm(-,E\ra N)$ adheres to the h-principle relative to the presheaf of embeddings $\Emb(-,E)$, in the sense of Definition~\ref{def.rel.h.principle}.\end{cor}
\begin{proof}
We must prove that for every smooth manifold with boundary $M$, the diagram
\[
\xymatrix{
\EmbfImm(M,E\ra N)\ar[r]\ar[d]&\sD\EmbfImm(M,E\ra N)\ar[d]\\
\Emb(M,E)\ar[r]&\sD\Emb(M,E)}
\]
is a homotopy pullback. This follows immediately from the definition and Proposition~\ref{prop.deriv.of.embimm}.
\end{proof}

The following is our application of the relative form of Gromov's h-principle.

\begin{theorem}\label{theorem.hprinciple.embimm}
Let $\pi: E\ra N$ be a smooth fiber bundle, where $E$ and $N$ are manifolds with boundary, and $\pi$ is a boundary-preserving map.
Let $M$ be an open manifold. The natural map
\[
\EmbImm(M, E\ra N) \longrightarrow \EmbfImm(M,E\ra N)
\]
is a weak homotopy equivalence.
\end{theorem}
\begin{proof}
We apply Theorem~\ref{theorem.relative.Gromov.hprinciple} to the case $\cE = \Emb(-, E)$ and $\cF = \EmbImm(-, E\ra N)$. Note that the conditions are satisfied: $\Emb(M,E)$ is open in $\Map(M,E)$ and is flexible by the parametrized isotopy extension theorem. Likewise, $ \EmbImm(M, E\ra N)$ is open in $\Emb(M,E)$, and $\EmbImm(-,E\ra N)$ is a sheaf relative to $\Emb(-,E)$, since an embedding $M\hookrightarrow E$ being an embedding-immersion is a local condition on $M$.
\end{proof}

The Smale--Hirsch h-principle for immersions $\Imm(M, N)$ holds both if the source $M$ is an open manifold, and also if the dimension of the target $N$ is strictly greater than that of the source. The same is true for embedding-immersions, as we show in the next proposition, and the proof is identical.
\begin{prop}\label{prop.emb.imm.n.open}
Let $W$ be a $k$-manifold with boundary, and let $E \xra{\pi} N$ be a boundary-preserving smooth fiber bundle in which $N$ is an $n$-manifold with boundary. Assume further that the inequality $k<n$ holds.
 The map
\[
\EmbImm(W, E\ra N) \longrightarrow \EmbfImm(W,E\ra N)
\]
is a weak homotopy equivalence.

\end{prop}
\begin{proof}
Observe the commutative diagram among topological spaces
% https://q.uiver.app/#q=WzAsNCxbMCwxLCJcXEVtYkltbShXLEVcXGRvd25hcnJvdyBOKSJdLFsxLDEsIlxcRW1iZkltbShXLCBFIFxcZG93bmFycm93IE4pIl0sWzAsMCwiXFx1bmRlcnNldHtbVl19IFxcY29wcm9kIFxcRW1iSW1tKFYsRVxcZG93bmFycm93IE4pIl0sWzEsMCwiXFx1bmRlcnNldHtbVl19IFxcY29wcm9kIFxcRW1iZkltbShWLEUgXFxkb3duYXJyb3cgTikiXSxbMiwzLCIoIFxcc0RfVilfe1tWXX0iXSxbMiwwXSxbMywxXSxbMCwxLCJcXHNEX1ciLDJdXQ==
\begin{equation}
\label{e.end}
\begin{tikzcd}
	{\underset{[V]} \coprod \EmbImm(V,E\ra N)} & {\underset{[V]} \coprod \EmbfImm(V,E \ra N)} \\
	{\EmbImm(W,E\ra N)} & {\EmbfImm(W, E \ra N)}
	\arrow["{( \sD_V)_{[V]}}", from=1-1, to=1-2]
	\arrow[from=1-1, to=2-1]
	\arrow[from=1-2, to=2-2]
	\arrow["{\sD_W}"', from=2-1, to=2-2]
\end{tikzcd}
\end{equation}
in which the coproducts are indexed by isomorphism classes of rank-$(\nk)$ vector bundles over $W$, and the downward maps are restrictions along the zero-sections.  
Theorem~\ref{theorem.hprinciple.embimm} implies the top horizontal map is a weak homotopy equivalence.
Therefore, to prove $\sD_W$ is a weak homotopy equivalence, it suffices to prove the following two assertions.
\begin{enumerate}
\item
The map $\sD_W$ is surjective on path-components.

\item
For each $(W \xra{f} E)\in \EmbImm(W,E \to N)$, and each isomorphism class $[V]$ of rank-$(\nk)$ vector bundles over $W$, the map $\sD_V$ induces a weak homotopy equivalence between vertical homotopy fibers over $f$.  

\end{enumerate}
To prove~(1), 
we first fix $(W \xra{f} E) \in \EmbfImm(W,E \to N)$.
Consider the normal bundle of the formal immersion $W \xra{\pi f} N$: It is a rank-$(\nk)$ vector bundle $V$ over $W$ fitting into a short exact sequence of vector bundles over $W$:
\[
0
\longrightarrow
\sT W
\longrightarrow
(\pi f)^\ast \sT N
\longrightarrow
V
\longrightarrow
0
~.
\]
Choose a splitting of this short exact sequence, along with a splitting of the short exact sequence $0 \ra \Ker(\sD\pi)\ra \sT E\ra\pi^\ast\sT N\ra 0$.
These choices determine a linear embedding of $V$ into the normal bundle of the embedding $W \xra{f} E$.
Applying the tubular neighborhood theorem to $f$, we have an embedding $V \xra{\w{f}} E$.
By construction, the composite $V \xra{\pi \w{f}} N$ is an immersion extending $W \xra{\pi f} N$ and witnessing the normal bundle of $\pi f$. 
We conclude the right vertical map is surjective on path-components.
Theorem~\ref{theorem.hprinciple.embimm} implies the top horizontal map is surjective on path-components.
From the 2-of-3 property of surjections among sets, we conclude that the bottom horizontal map is surjective on path-components, as desired.

We now prove~(2).
Fix $(W \xra{f} E) \in \EmbImm(W,E \to N)$.
Let $[V]$ be a rank-$(\nk)$ vector bundle over $W$.
If $V\to W$ is not the normal bundle of the immersion $W \xra{\pi f} N$, then the homotopy fiber of the restriction map $\EmbImm(V , E \to N) \to \EmbImm(W,E \to N)$ is empty.
So take $V \to W$ to be the normal bundle of $\pi f$.
Denote the respective vertical fibers of~(\ref{e.end}) over $f$ as
\begin{equation}
\label{e.end2}
\EmbImm_f(V,E \to N)
\longrightarrow
\EmbfImm_f(V,E\to N)
~.
\end{equation}
It follows from the parametrized isotopy extension theorem that the vertical maps in~(\ref{e.end}) are Serre fibrations.
So to prove~(2), it is sufficient to prove the map~(\ref{e.end2}) is a weak homotopy equivalence.
Consider the topological subspaces
\[
\ImmImm_f(V,E \to N)
~:=~
\left\{
V \xra{\w{f}} E \in \ImmImm(V,E \to N)
\mid \w{f}_{|W} = f
\right\}
~\subset~
\ImmImm(V,E \to N)
\]
and
\[
\ImmfImm_f(V,E \to N)
~:=~
\left\{
V \xra{\w{f}} E \in \ImmfImm(V,E \to N)
\mid \w{f}_{|W} = f
\right\}
~\subset~
\ImmfImm(V,E \to N)
~.
\]
Observe the inclusions among topological spaces:
\begin{equation}
\label{e.end3}
\xymatrix{
\EmbImm_f(V,E \to N)
\ar[d]
\ar[rr]^-{(\ref{e.end2})}
&&
\EmbfImm_f(V,E \to N)
\ar[d]
\\
\ImmImm_f(V,E \to N)
\ar[rr]
&&
\ImmfImm_f(V,E \to N)
.
}
\end{equation}
to prove the top horizontal map is a weak homotopy equivalence, we prove the three other maps in this diagram are weak homotopy equivalences.
An application of h-principle for immersions of Smale--Hirsch gives that the bottom horizontal map in~(\ref{e.end3}) is a weak homotopy equivalence.
Our arguments proving the vertical maps in~(\ref{e.end3}) are weak homotopy equivalences are identical, so we only provide that for the left vertical map.

Once and for all, choose a smooth map $\RR_{\geq 0} \times (0,1] \xra{\rho} \RR_{\geq 0}$ with the following properties.
\begin{itemize}
\item
For each $t\in (0,1]$, the map $\RR_{\geq 0} \xra{\rho_t} \RR_{\geq 0}$ is an open embedding.

\item
The map $\RR_{\geq 0} \xra{\rho_1} \RR_{\geq 0}$ is the identity map.

\item
For each $0<t < t' \leq 1$, there is containment of the closure $\overline{{\sf Image}(\rho_t)} \subset {\sf Image}(\rho_{t'})$.  

\item
For each open neighborhood $\{0\} \subset O \subset \RR_{\geq 0}$, there exists $t\in (0,1]$ such that ${\sf Image}(\rho_t) \subset O$.

\item
$\rho$ restricts as the projection map near $\{0\} \times (0,1] \subset \RR_{\geq 0} \times (0,1]$.

\end{itemize}
Once and for all, choose an inner product on the vector bundle $V \to W$. 
Consider the smooth map
\[
\varphi
\colon
V \times (0,1] 
\longrightarrow
V
~,\qquad
(v,t)
\longmapsto
\varphi_t(v)
:=
\rho_t(\lVert v \rVert^2) \cdot v
~,
\]
given by scaling by $\rho$.

Let $K$ be a compact manifold with boundary.
Consider a pair of horizontal maps fitting into a solid diagram among topological spaces:
\begin{equation}
\label{e.fin}
\xymatrix{
\partial K 
\ar[d]
\ar[rr]^-{F_0}
&&
\EmbImm_f(V,E\to N)
\ar[d]
\\
K 
\ar@{-->}[urr]_-{\w{F}}
\ar[rr]^-{F}
&&
\ImmImm_f(V,E\to N)
.
}
\end{equation}
To prove the right vertical map in this diagram is a weak homotopy equivalence it is sufficient to construct a filler, up to homotopy, in this diagram.
By smooth approximation, we can homotope $F$ so that it is adjoint to a smooth map $K \times V \xra{F} E$ with the property that, for each $k\in K$, the restriction $V \xra{ F_{|\{k\} \times V} } E$ belongs to $\ImmImm_f(V,E \to N)$, and, for each $k\in \partial K$, the restriction $V \xra{ F_{|\{k\} \times V} } E$ belongs to $\EmbImm_f(V,E \to N)$.
Now, the inverse function theorem implies there exists an open neighborhood of $K \times W \subset K \times V$ on which $F$ is an embedding.
Using that $K$ is compact, the tube lemma ensures the existence of $\epsilon \in (0,1]$ such that the composite map $K \times V \xra{ K \times \varphi_\epsilon} K \times V \xra{F} E$ is an embedding.  
Observe the following facts about this composite $F \circ (K \times \varphi_\epsilon)$, 
all of which follow from the construction of $\varphi_t$.
\begin{itemize}
\item
There is membership $F \circ (K\times \varphi_\epsilon) \in \EmbImm_f(V,E \to N)$.

\item
The family $F \circ (K \times \varphi_t)$ defines a path in $\ImmImm_f(V,E \to N)$ between $F \circ (K\times \varphi_\epsilon)$ and $F$.  

\item
The family $F \circ (\partial K \times \varphi_t)$ defines a path in $\EmbImm_f(V,E \to N)$ between $F \circ (\partial K\times \varphi_\epsilon)$ and $F_0$.  

\end{itemize}
The adjoint of the map $K \times V \xra{F \circ (K \times \varphi_\epsilon)} E$ supplies the desired homotopy filler $\w{F}$ in~(\ref{e.fin}).

\end{proof}

In the proof of the Tangle Hypothesis, we make use of a modification of the preceding result in which framings (in particular, trivializations of a Gauss map) are added. Let $W$ be a $k$-manifold with boundary.
Let $E \xra{\pi} N$ be a boundary-preserving smooth fiber bundle in which $N$ is a framed $n$-manifold with boundary.
The framing of $N$ determines map
\[
\EmbImm(W,E\ra N)
\longrightarrow
\Map(W , \Gr_k(n))
~,\qquad
(W \xra{f} E)
\longmapsto 
\left( W \xra{~{\sf Gauss}_{\pi f}~} \Gr_k(n) \right)
~.
\]
Denote the homotopy fiber of this map over the constant map $W \xra{\ast} \Gr_1(n)$ as
\[
\EmbImm^{\fr}(W, E\ra N)
~.
\]
Note that a point in $\EmbImm^{\fr}(W, E\ra N)$ is an embedding-immersion $W \xra{f} E \xra{\pi} N$ together with a nullhomotopy of the Gauss map $W \xra{{\sf Gauss}_{\pi f}} \Gr_k(n)$. 
Likewise, the framing of $N$ determines map
\[
\EmbfImm(W,E\ra N)
\longrightarrow
\Map(W , \Gr_k(n))
~,\qquad
(W \xra{f} E)
\longmapsto 
\left( W \xra{~{\sf Gauss}_{\pi f}~} \Gr_k(n) \right)
~.
\]
Denote the homotopy fiber of this map over the constant map $W \xra{\ast} \Gr_1(n)$ as
\[
\EmbfImm^{\fr}(W, E\ra N)
~.
\]
We have the following corollary of Proposition~\ref{prop.emb.imm.n.open}. 
\begin{cor}\label{cor.hprinciple.framed.embimm}
Let $\pi: E\ra N$ be a smooth fiber bundle, where $E$ and $N$ are manifolds with boundary, $\pi$ is a boundary-preserving map, and $N$ is framed $n$-manifold.
Let $W$ be a compact $k$-dimensional manifold with boundary, with $k<n$. The natural map
\[
\EmbImm^{\fr}(W, E\ra N) \longrightarrow \EmbfImm^{\fr}(W,E\ra N)
\]
is a weak homotopy equivalence.
\end{cor}
\begin{proof}
Consider the map of homotopy fiber sequences with base $\Map(W, \Gr_k(n))$:
\[
\xymatrix{
\EmbImm^{\fr}(W, E\ra N)\ar[d] \ar[rr]&& \EmbfImm^{\fr}(W,E\ra N)\ar[d]\\
\EmbImm(W, E\ra N)\ar[rr]\ar[dr]&& \EmbfImm(W,E\ra N)\ar[dl]\\
&\Map(W, \Gr_k(n))~.}
\]
Since the bottom horizontal map is a weak homotopy equivalence by Proposition~\ref{prop.emb.imm.n.open}, the top horizontal map is likewise a weak homotopy equivalence.
\end{proof}

\section{The 1-dimensional Cobordism Hypothesis}

Taking products with $\RR$ defines a functor
\[
\RR\times - \colon \bcD_{[\n1,n]}^{\sfr} \longrightarrow \bcD_{[n,n+1]}^{\sfr}
~,\qquad
D \mapsto \RR\times D
~.
\]
This functor carries $\RR^{\n1}$ to $\RR^{n}$, and carries closed covers to closed covers.
This functor determines a natural transformation between representable functors:
% https://q.uiver.app/#q=WzAsMyxbMCwwLCJcXGJjRF97W1xcbjEsbl19XntcXHNmcn0iXSxbMCwyLCJcXGJjRF97W24sbisxXX1ee1xcc2ZyfSJdLFsyLDEsIlxcU3BhY2VzIl0sWzAsMiwiXFxNYXAoXFxSUl57XFxuMX0sLSkiXSxbMSwyLCJcXE1hcChcXFJSXm4sLSkiLDJdLFswLDEsIlxcUlIgXFx0aW1lcyAtIiwyXSxbMCwyLCJcXERvd25hcnJvdyIsMyx7Im9mZnNldCI6NSwic3R5bGUiOnsiYm9keSI6eyJuYW1lIjoibm9uZSJ9LCJoZWFkIjp7Im5hbWUiOiJub25lIn19fV1d
\[\begin{tikzcd}
	{\bcD_{[\n1,n]}^{\sfr}} \\
	&& \Spaces \\
	{\bcD_{[n,n+1]}^{\sfr}}
	\arrow["{\Map(\RR^{\n1},-)}", from=1-1, to=2-3]
	\arrow["\Downarrow"{marking, allow upside down}, shift right=5, draw=none, from=1-1, to=2-3]
	\arrow["{\RR \times -}"', from=1-1, to=3-1]
	\arrow["{\Map(\RR^n,-)}"', from=3-1, to=2-3]
\end{tikzcd}
~.
\]
From Definition~\ref{def.bord} of the tangle $(\infty,1)$-categories, this natural transformation defines a $\cE_{\n1}$-monoidal functor between flagged $(\infty,1)$-categories:
\begin{equation}
\label{e.bords}
\Bord_1^{\fr}(\RR^{\n1})
\longrightarrow
\Bord_1^{\fr}(\RR^{n})
~.
\end{equation}
Explicitly, this functor is given on spaces of objects by
\[
\left(W \subset \RR^{\n1} \right)
\longmapsto
\left(\{0\} \times W\subset \RR^n \right)
\]
and on spaces of morphisms by
\[
\left(W \subset \RR^{\n1} \times \DD^1 \right)
\longmapsto
\left(\{0\} \times W  \subset \RR^n \times \DD^1 \right)
~.
\]
Ignoring monoidal structures, the functors~(\ref{e.bords}) assemble as a diagram among flagged $(\infty,1)$-categories:
\begin{equation}
\label{e.seq.bords}
\Bord_1^{\fr}(\RR^{0})
\longrightarrow
\Bord_1^{\fr}(\RR^{1})
\longrightarrow
\Bord_1^{\fr}(\RR^{2})
\longrightarrow
\dots
~.
\end{equation}
The contractibility of $\Emb(W,\RR^{\oo})$ justifies the following definition.
\begin{definition}\label{def.bord.infty}
The \emph{flagged $(\oo,1)$-category of $1$-framed cobordisms} is the sequential colimit
\[
\Bord_1^{\fr}:= \underset{n \geq 0}\colim \Bord_1^{\fr}(\RR^{\n1})
~.
\]
\end{definition}

\begin{prop}
\label{t.bord.desc}
The flagged $(\infty,1)$-category $\Bord_1^{\fr}$ admits the following description.
\begin{enumerate}

\item
The space of objects
\[
\Obj(\Bord_1^{\fr})
~\simeq~
\left| \left\{ (S,\sigma) \right\} \right|
~\simeq~
\underset{r_+,r_- \geq 0} \coprod \sB \Sigma_{r_+} \times \sB \Sigma_{r_-}
\]
is the moduli space of signed finite sets (i.e., finite sets $S$ equipped with a map $S \xra{\sigma} \{\pm \}$).

\item
The space of morphisms
\[
\Mor(\Bord_1^{\fr})
~\simeq~
\Bigr|\Bigl\{ W \to \DD^1~,~ \varphi  \Bigr\}\Bigl|
~,
\]
is the moduli space of framed 1-dimensional cobordisms, which is the space of the following data:
\begin{itemize}
\item a compact smooth $1$-manifold $W$ with boundary;

\item a smooth map $W\to \DD^1$ such that the boundary of $W$ is the preimage of the boundary of $\DD^1$;

\item a framing $\sT W \overset{\varphi}\cong \epsilon^1_W$ of $W$.

\end{itemize}

\item
The source and target maps are given by:
\[
\Mor(\Bord_1^{\fr})
\xra{~s~}
\Obj(\Bord_1^{\fr})
~,\qquad
(W \to \DD^1 ,\varphi)
\longmapsto
(W_{|\partial_- \DD^1} , \varphi_- )
~,
\]
where, for $w \in W_{|\partial_- \DD^1}$ the sign $\varphi_-(w)$ is $+$ if and only if the inward-pointing vector at $w$ aligns with the framing of $W$ at $w$;
\[
\Mor(\Bord_1^{\fr})
\xra{~t~}
\Obj(\Bord_1^{\fr})
~,\qquad
(W \to \DD^1 ,\varphi)
\longmapsto
(W_{|\partial_+ \DD^1} , \varphi_+ )
~,
\]
where, for $w \in W_{|\partial_+ \DD^1}$ the sign $\varphi_+(w)$ is $+$ if and only if the outward-pointing vector at $w$ aligns with the framing of $W$ at $w$.

\item
The composition map
\[
\Mor(\Bord_1^{\fr})
\underset{\Obj(\Bord_1^{\fr})} \times
\Mor(\Bord_1^{\fr})
\longrightarrow
\Mor(\Bord_1^{\fr})
~,\qquad
\left(
(W \to \DD^1 ,\varphi)
,
(W' \to \DD^1 ,\varphi')
\right)
\longmapsto
( W \cup W' , \varphi \cup \varphi' )
~,
\]
is given by concatenation of framed cobordisms.

\end{enumerate}

\end{prop}

\begin{proof}
Sequential colimits in the $(\infty,1)$-category $\Spaces$ commute with finite limits.
Consequently, the fully-faithful inclusion $\fCat_{(\infty,1)} \hookrightarrow \Fun(\bDelta^{\op} , \Spaces)$ preserves sequential colimits.  
Using that colimits in functor $(\infty,1)$-categories are computed value-wise, we have that the canonical maps between spaces,
\[
\underset{n\geq 0} \colim \Obj(\Bord_1^{\fr}(\RR^{\n1}))
=
\underset{n\geq 0} \colim \Bord_1^{\fr}(\RR^{\n1})[0]
\longrightarrow
\Bord_1^{\fr}[0]
=
\Obj(\Bord_1^{\fr})
\]
and
\[
\underset{n\geq 0} \colim \Mor(\Bord_1^{\fr}(\RR^{\n1}))
=
\underset{n\geq 0} \colim \Bord_1^{\fr}(\RR^{\n1})[1]
\longrightarrow
\Bord_1^{\fr}[1]
=
\Mor(\Bord_1^{\fr})
~,
\]
are equivalences.
There are similar identifications for spaces of $p$-simplices.

Consequently, for $p \geq 0$, we have an identification:
\begin{eqnarray}
\nonumber
\Bord_1^{\fr}[p]
&
:=
&
\underset{n\geq 0} \colim~ \Bord_1^{\fr}(\RR^{\n1})[p]
\\
\nonumber
&
:=
&
\underset{n\geq 0} \colim ~\Map_{\bcD_{[\n1,n]}^{\sfr}}\bigl(\RR^{\n1}, \RR^{\n1}\times\DD[p]\bigr)
\\
\nonumber
&
\underset{\rm Obs~\ref{rem.T.pt.BBord}}{~\simeq~}
&
\underset{n\geq 0} \colim ~   \Bigl| \Bigl\{W\overset{\rm framed}{\underset{{\rm codim\text{-}}(\n1)}\hookrightarrow} \RR^{\n1}\times \DD[p]  \Bigr\} \Bigr|
\\
\nonumber
&
\underset{\rm project}{\xra{~\simeq~}}
&
\Bigr|\Bigl\{ W \to \DD[p]~,~ \varphi  \Bigr\}\Bigl|~.
\end{eqnarray}
Here, the final term is a moduli space of the following data:
\begin{itemize}
\item for $p>0$, a compact smooth 1-manifold $W$ with boundary; for $p=0$, a compact smooth 0-manifold;

\item a smooth boundary-preserving map $W\to \DD[p]$ that is transverse to $[p] \subset \DD[p]$;

\item an injection $\varphi\colon \sT W \hookrightarrow \epsilon^1_W$ of the tangent bundle of $W$ into the trivial rank-$1$ vector bundle over $W$.

\end{itemize}
The result follows by inspecting these identifications.

\end{proof}

\begin{prop}
\label{not.flagged}
The flagged $(\infty,1)$-category $\Bord_1^{\fr}$ is univalent-complete.
In other words, $\Bord_1^{\fr}$ is an $(\infty,1)$-category.

\end{prop}

\begin{proof}
We must show the canonical map from the space of objects to the space of isomorphisms
\begin{equation}
\label{e.univ.map}
\Obj(\Bord_1^{\fr}) 
\longrightarrow
\Iso(\Bord_1^{\fr})
\end{equation}
is an equivalence.
It follows from Proposition~\ref{t.bord.desc} that the subspace
\[
\Iso(\Bord_1^{\fr})
~\subset~
\Mor(\Bord_1^{\fr})
\]
consists of framed 1-dimensional cobordisms that admit a both a left and a right inverse with respect to concatenation of framed cobordisms.
From the classification of compact 1-manifolds with boundary, any such 1-dimensional cobordism $W$ is isomorphic with a finite disjoint union of trivial cobordisms: $W \cong (\DD^1)^{\sqcup S}$ for some finite set $S$.
Therefore, for $W$ a framed 1-dimensional cobordism that admits a left and right inverse with respect to concatenation of framed cobordisms, there is an isomorphism between framed cobordisms $W \cong (\DD^1)^{\sqcup (S,\sigma)}$ for some signed finite set $S \xra{\sigma} \{\pm \}$.
We conclude that the map~(\ref{e.univ.map}) is surjective on path-components.
The result is then proved by showing, for each point $(S,\sigma)\in \Obj(\Bord_1^{\fr})$, the map between spaces of based loops 
\begin{equation}
\label{e.loops}
\Omega_{(S,\sigma)} \Obj(\Bord_1^{\fr})
\longrightarrow
\Omega_{(\DD^1)^{\sqcup (S,\sigma)}} \Iso(\Bord_1^{\fr})
\end{equation}
is an equivalence.

As $\Obj(\Bord_1^{\fr})$ is the moduli space of signed finite sets, and as $\Mor(\Bord_1^{\fr})$ is the moduli space of framed 1-dimensional cobordisms (Proposition~\ref{t.bord.desc}), the map~(\ref{e.loops}) can be identified as the canonical map
\begin{equation}
\label{e.sigma}
\Sigma_{S_+} \times \Sigma_{S_-}
=
\Diff^{\fr}( S,\sigma )
\xra{~ - \times \DD^1~}
\Diff^{\fr}( (\DD^1)^{\sqcup (S,\sigma)} )
~,
\end{equation}
where $S_+ := \sigma^{-1} \{+\}$ and $S_- := \sigma^{-1} \{-\} $.
Every framed diffeomorphism of $(\DD^1)^{\sqcup (S,\sigma)}$ determines a permutation of its path-components $\pi_0\left( (\DD^1)^{\sqcup (S,\sigma)}  \right) = S$, which preserves the partition $S = S_+ \sqcup S_-$.  
In this way, taking path-components defines a retraction of the map~(\ref{e.sigma}), which yields an identification between groups under $\Sigma_{S_+} \times \Sigma_{S_-}$ as a wreath product:
\[
\Diff^{\fr}( (\DD^1)^{\sqcup (S,\sigma)} )
~\simeq~
\left(
\Sigma_{S_+} \times \Sigma_{S_-}
\right)
\wr
\Diff^{\fr}(\DD^1)
~.
\]
To finish, the group of framed diffeomorphisms of $\DD^1$ is contractible.
Therefore, the  map~(\ref{e.sigma}) is an equivalence, as desired.

\end{proof}

\subsection{Symmetric monoidal structure of $\Bord_1^{\fr}$}\label{sec.deloop}
We present a formalism for constructing symmetric monoidal $(\infty,1)$-categories from sequences of compatible $\cE_k$-monoidal $(\infty,1)$-categories.
This is much like constructing an $\infty$-loop space from a pre-spectrum.
We then apply this construction to endow the $(\infty,1)$-category $\Bord_1^{\fr}$ with a symmetric monoidal structure.

For each dimension $k$ consider the symmetric monoidal $(\infty,1)$-category $\Disk_k^{\fr}$ of finite disjoint unions of framed $k$-dimensional vector spaces and framed open embeddings among them, with symmetric monoidal structure given by disjoint union.
Taking products with Euclidean spaces assembles these symmetric monoidal $(\infty,1)$-categories into a functor $\Disk_\bullet^{\fr}\colon \NN \to \Alg_{\sf Com}(\fCat_{(\infty,1)})$.
The colimit of this functor is canonically identified as the symmetric monoidal envelope of the commutative $\infty$-operad.  
Therefore, for each symmetric monoidal $(\infty,1)$-category $\cV$, applying $\Fun^\ot(-,\cV)$ defines a functor
$
\Alg_{-}(\cV)\colon \NN^{\op} \to \fCat_{(\infty,1)}
$
whose limit is $\Alg_{\sf Com}(\cV)$.  
Consider the Cartesian fibration
\[
\Alg_\bullet(\cV) \longrightarrow \NN
\]
which is the unstraightening of this functor.
There is a fully-faithful  functor to the $(\infty,1)$-category of sections of this Cartesian fibration
\begin{equation}\label{as-Carts}
\Alg_{\sf Com}(\cV) \longrightarrow \Gamma\Bigl(\Alg_\bullet(\cV) \longrightarrow \NN\Bigr)
\end{equation}
whose image consists of those sections that carry morphisms to Cartesian morphisms.  
Explicitly, an object in the righthand $(\infty,1)$-category is the data of an $\cE_k$-algebra $A_k$ for each $k\geq 0$ together with a map of $\cE_k$-algebras $A_k \to (A_{k'})_{|\cE_k}$ for each $k\leq k'$, coherently compatibly; while an object in the image of~(\ref{as-Carts}) is one for which each $A_k\to (A_{k'})_{|\cE_k}$ is an equivalence.

\begin{lemma}\label{spectra}
Let $\cV$ be a Cartesian presentable $(\infty,1)$-category, which we regard as a symmetric monoidal $(\infty,1)$-category via the Cartesian product.
The functor~(\ref{as-Carts}) admits a left adjoint that implements a localization.
This left adjoint evaluates on a section $(k\mapsto A_k)$ as the Cartesian section
\[
k~\mapsto~ \underset{\ell \geq k} \colim (A_\ell)_{|k}~: = ~\colim \Bigl( A_k \to (A_{k+1})_{|\cE_{k}} \to (A_{k+2})_{|\cE_k} \to  \dots \Bigr)~.  
\]
\end{lemma}

\begin{proof}
The colimit of the functor $\NN^{-/}\colon \NN^{\op} \to \fCat_{(\infty,1)}$, given by $k\mapsto \NN_{\geq k}$, is canonically identified as $\NN$.
Therefore, to construct the alleged left adjoint it is enough to construct a downward morphism between functors $\NN^{\op} \to \fCat_{(\infty,1)}$
\[
\Small
\xymatrix{
\dots  \ar[r]
&
\Gamma\Bigl(\Alg_{\bullet> k}(\cV) \longrightarrow \NN_{>k}\Bigr)  \ar[d]  \ar[r]
&
\Gamma\Bigl(\Alg_{\bullet\geq k}(\cV) \longrightarrow \NN_{\geq k}\Bigr)  \ar[d]  \ar[r]
&
\dots
\\
\dots  \ar[r]
&
\Alg_{k+1}(\cV)  \ar[r]
&
\Alg_k(\cV)   \ar[r]
&
\dots .
}
\]
and then argue its properties.
Because, for each $k\in \NN$ the object $(k=k)$ is final in $(\NN^{k/})^{\op}$, the colimit of the composite functor $\NN_{\geq k}^{\op} \to \NN^{\op} \xra{\Alg_{-}(\cV)} \fCat_{(\infty,1)}$
is canonically identified as $\Alg_k(\cV)$.  
There results a functor between $(\infty,1)$-categories
$\Gamma\Bigl(\Alg_{\bullet\geq k}(\cV) \longrightarrow \NN_{\geq k}\Bigr)
\xra{(-)_{|\cE_k}}
\Fun\bigl(\NN_{\geq k}^{\op}, \Alg_k(\cV)\bigr).
$  
Postcomposing this functor with the colimit functor $\Fun\bigl(\NN_{\geq k}^{\op}, \Alg_k(\cV)\bigr) \xra{\colim} \Alg_k(\cV)$ defines the functor we seek for each given $k\in \NN$.  
Since the forgetful functor $\Alg_k(\cV) \to \cV$ preserves and creates filtered colimits, this colimit functor indeed exists.  
This also implies that each square in the diagram displayed above canonically commutes.  
We conclude a functor between limit $(\infty,1)$-categories 
$\Gamma\Bigl(\Alg_\bullet(\cV ) \longrightarrow \NN\Bigr)
\to
\Alg_{\sf Com}(\cV)$.

The functor $(-)_{|\cE_k}$ above carries Cartesian sections to constant functors.  
Because $\NN$ has contractible classifying space, the composite functor
$\Alg_{\sf Com}(\cV) \to \Gamma\Bigl(\Alg_\bullet(\cV ) \longrightarrow \NN\Bigr) \to \Alg_{\sf Com}(\cV)$ is canonically identified as the identity functor.  
Constructed by way of a colimit, there is a unit transformation from the composite functor
$\Gamma\Bigl(\Alg_\bullet(\cV ) \longrightarrow \NN\Bigr) \to \Alg_{\sf Com}(\cV) \to \Gamma\Bigl(\Alg_\bullet(\cV ) \longrightarrow \NN\Bigr)$ 
to the identity functor.
Because $\NN$ is filtered, the restriction of this unit transformation to the Cartesian sections is a natural equivalence.  
This completes the proof of the lemma.

\end{proof}

The diagram~(\ref{e.seq.bords}) defines a section:
\begin{equation}\label{tang-sec}
\NN \longrightarrow \Alg_\bullet({\fCat_{(\infty,1)}})~,\qquad k\mapsto \Bord_1^{\fr}(\RR^k)
~.
\end{equation}

\begin{definition}\label{def.Bord-B}
The symmetric monoidal $(\infty,1)$-category of $1$-framed cobordisms
\[
\Bord_1^{\fr}~:=~\underset{k\geq 0} \colim ~\Bord_1^{\fr}(\RR^k)
\]
is the value of the left adjoint functor of Lemma~\ref{spectra} on the section~(\ref{tang-sec}).  

\end{definition}

\subsection{Proof of the Cobordism Hypothesis}
We now state and prove the main result of this section, the 1-dimensional Cobordism Hypothesis, first proved by Hopkins--Lurie~\cite{cobordism}.

\begin{theorem}[1-dimensional Cobordism Hypothesis]\label{B-bord}
Let $\cR$ be a rigid symmetric monoidal $(\infty,1)$-category.
Evaluation at the object $\ast\in \Bord_1^{\fr}$ defines an equivalence between spaces
\[
\ev_\ast
\colon
\Map_{\Alg_{\sf Com}(\fCat_{(\oo,1)})} \bigl(
\Bord_1^{\fr}
,
\cR
\bigr)
\xra{~\simeq~}
\Obj(\cR)
~,\qquad
Z
\longmapsto
Z(\ast)
\]
between $\cR$-valued symmetric monoidal functors and the space of objects of $\cR$.
\end{theorem}

\begin{proof}
We use the definition of $\Bord_n^{\fr}$ as a colimit of tangle categories, and then apply the Tangle Hypothesis given by Theorem~\ref{theorem.main}:
\begin{eqnarray}
\nonumber
\Map_{\Alg_{\sf Com}(\fCat_{(\infty,1)})} \bigl(\Bord_1^{\fr},\cR\bigr)
&
\simeq
&
\underset{k\to \infty}\limit\Map_{\Alg_k(\fCat_{(\infty,1)})}\Bigl(\Bord_1^{\fr}(\RR^k), \cR\Bigr)
\\
\nonumber
&
\simeq
&
\underset{k\to \infty}\limit \obj(\cR)
\\
\nonumber
&
\simeq
&
\obj(\cR)~.
\end{eqnarray}
The result follows, since we obtain a constant sequential limit with value $\obj(\fX)$.

\end{proof}

\section{Invertible field theories and classifying spaces}
We recover the expected classification of \emph{invertible} framed topological quantum field theories.
This is phrased in terms of the groupoid completion of the framed cobordism category.

The fully-faithful  inclusion $\Alg_{\n1}(\Spaces)\hookrightarrow \Alg_{\n1}(\fCat_{(\oo,1)})$ has a left adjoint localization, which is the usual classifying space functor:
\begin{equation}\label{B.def}
\sB \colon  \Alg_{\n1}(\fCat_{(\oo,1)})~\rightleftarrows~ \Alg_{\n1}(\Spaces)~.
\end{equation}

\begin{cor}\label{B.Tang}
There is a canonical identification as $\cE_{\n1}$-spaces,
\[
\sB\bigl(\Bord_1^{\fr}(\RR^{\n1})\bigr)~\simeq~ \Omega^{\n1}S^{\n1}
~,
\]
between the classifying space of the tangle category and the $(\n1)$-fold loop space of the $(\n1)$-sphere. In the $n\mapsto \oo$ limit, there is a canonical identification as $\cE_{\oo}$-spaces,
\[
\sB\bigl(\Bord_1^{\fr}\bigr)~\simeq~ \Omega^{\oo}\Sigma^{\oo}S^0
~,
\]
between the classifying space of the bordism category and the zeroth space of the sphere spectrum. 
\end{cor}

\begin{proof}
Let $\cG$ be a grouplike $\cE_{\n1}$-space.
Via the fully-faithful inclusion $\Spaces \to \fCat_{(\infty,1)}$, which preserves finite products, $\cG$ can be regarded as an $\cE_{\n1}$-monoidal flagged $(\oo,1)$-category.
Regarded as so, $\cG$ is univalent-complete and rigid, with $\Obj(\cG) = \cG$.
We have the following sequence of canonical identifications:
\begin{eqnarray}
\nonumber
\Map_{ \Alg_{\n1}(\Spaces)}\Bigl(\sB\bigl(\Bord_1^{\fr}(\RR^{\n1})\bigr), \cG \Bigr)
&
\simeq
&
\Map_{ \Alg_{\n1}(\fCat_{(\oo,1)})}(\Bord_1^{\fr}(\RR^{\n1}), \cG)
\\
\nonumber
&
\simeq
&
\obj(\cG) = \cG
\end{eqnarray}
The first equivalence follows from the localization~(\ref{B.def}).
The second equivalence follows from Theorem~\ref{theorem.main}, the Tangle Hypothesis.
Evidently, both of these equivalences are functorial in the argument $\cG$.
In this way, we see that $\Bord_1^{\fr}(\RR^{\n1})$ corepresents the forgetful functor from $\Alg_{
\cE_{\n1}}(\Spaces)$ to $\Spaces$. However, by universal properties the forgetful functor is also corepresented by the free grouplike $\cE_{\n1}$-algebra generated by a point, which is $\Omega^{\n1}S^{\n1}$.
The desired canonical equivalence between $\cE_{\n1}$-spaces then follows from the Yoneda lemma. The equivalence $\sB\bigl(\Bord_1^{\fr}\bigr)~\simeq~ \Omega^{\oo}\Sigma^{\oo}S^0$ follows by identical reasoning.

\end{proof}

An immediate consequence of Corollary~\ref{B.Tang} is a classification of invertible 1-dimensional framed topological quantum field theories.
We first introduce the following notation.
Let $0< k \leq \infty$.
For $\cR$ an $\cE_k$-monoidal $(\infty,1)$-category, denote by
\[
{\sf Pic}(\cR) 
~\subset~ 
\Obj(\cR) 
\]
the subspace consisting of those objects that are invertible with respect to the monoidal structure.
\begin{cor}\label{invertible.theories}
\begin{enumerate}
\item[]

\item
Let $\cR$ be an $\cE_{\n1}$-monoidal $(\infty,1)$-category.
There is a monomorphism between spaces,
\[
{\sf Pic}(\cR)~\hookrightarrow~ \Map_{\Alg_{\n1}(\fCat_{(\infty,1)})}\bigl(\Bord_1^{\fr}(\RR^{\n1}),\cR\bigr)
~.
\]
The image consists of those symmetric monoidal functors that carry each morphism in $\Bord_1^{\fr}(\RR^{\n1})$ to an isomorphism in $\cR$.

\item
Let $\cR$ be a symmetric monoidal $(\infty,1)$-category.
There is a monomorphism between spaces,
\[
{\sf Pic}(\cR)~\hookrightarrow~ \Map_{\Alg_{\sf Com}(\Cat_{(\infty,1)})}\bigl(\Bord_1^{\fr} ,\cR\bigr)
~.
\]
The image consists of those symmetric monoidal functors that carry each morphism in $\Bord_1^{\fr}$ to an isomorphism in $\cR$.

\end{enumerate}

\end{cor}

\begin{proof}
We prove the first statement; a proof of the second statement is logically identical.

Consider the subspace
\[
\Map_{\Alg_{\n1}(\fCat_{(\infty,1)})}\bigl(\Bord_1^{\fr}(\RR^{\n1}),\cR\bigr)^{\sf invbl}
~\subset~
\Map_{\Alg_{\n1}(\fCat_{(\infty,1)})}\bigl(\Bord_1^{\fr}(\RR^{\n1}),\cR\bigr)
\]
consisting of those symmetric monoidal functors that carry each morphism in $\Bord_1^{\fr}(\RR^{\n1})$ to an isomorphism in $\cR$.
By definition of the classifying space functor as a left adjoint in~(\ref{B.def}), we identify this subspace:
\[
\Map_{\Alg_{\n1}(\fCat_{(\infty,1)})}\bigl(\Bord_1^{\fr}(\RR^{\n1}),\cR\bigr)^{\sf invbl}
~\simeq~
\Map_{\Alg_{\n1}(\Spaces)}\bigl(\sB \Bord_1^{\fr}(\RR^{\n1}), \Obj(\cR) \bigr)
~.
\]
Now, let $(\Bord_1^{\fr}(\RR^{\n1}) \xra{Z} \cR) \in \Map_{\Alg_{\n1}(\fCat_{(\infty,1)})}\bigl(\Bord_1^{\fr}(\RR^{\n1}),\cR\bigr)^{\sf invbl}$.
Let $( S \subset \RR^{\n1}) \in \Bord_1^{\fr}(\RR^{\n1})$ be an object.
Using that $\Bord_1^{\fr}(\RR^{\n1})$ is rigid, the value $Z(S) \in \cR$ is dualizable.
Furthermore, for $S^\vee \in \Bord_1^{\fr}(\RR^{\n1})$ the dual of this object, the value of $Z$ on the unit and counit in $\Bord_1^{\fr}(\RR^{\n1})$ 
witness $Z(S^\vee)$ as the dual in $\cR$ of $Z(S)$.
Because $Z$ is assumed to carry all morphisms in $\Bord_1^{\fr}(\RR^{\n1})$ to isomorphisms in $\cR$, the value of $Z$ on this unit and counit are isomorphisms in $\cR$.
Therefore, these dualizing data witness the value $Z(S)\in \cR$ as invertible with respect to the monoidal structure of $\cR$.
Therefore, the $\cE_{\n1}$-monoidal functor $Z$ factors:
\[
\Bord_1^{\fr}(\RR^{\n1})
\xra{~Z~}
\Pic(\cR)
~\subset~
\Obj(\cR)
~.
\]
We conclude a further identification:
\[
\Map_{\Alg_{\n1}(\fCat_{(\infty,1)})}\bigl(\Bord_1^{\fr}(\RR^{\n1}),\cR\bigr)^{\sf invbl}
~\simeq~
\Map_{\Alg_{\n1}(\Spaces)}\bigl(\sB \Bord_1^{\fr}(\RR^{\n1}), {\sf Pic}(\cR) \bigr)
~.
\]
By construction, the $\cE_{\n1}$-monoidal space ${\sf Pic}(\cR)$ is grouplike.
So, from Corollary~\ref{B.Tang} we achieve a further identification:
\[
\Map_{\Alg_{\n1}(\fCat_{(\infty,1)})}\bigl(\Bord_1^{\fr}(\RR^{\n1}),\cR\bigr)^{\sf invbl}
~\simeq~
\Map_{\Alg_{\n1}(\Spaces)}\bigl( \Omega^{\n1} S^{\n1} , {\sf Pic}(\cR) \bigr)
~\simeq~
{\sf Pic}(\cR)
~,
\]
in which the final equivalence uses that the grouplike $\cE_{\n1}$-space $\Omega^{\n1} S^{\n1}$ is the free grouplike $\cE_{\n1}$-space generated by a point.
The result follows.

\end{proof}

\end{document}